\documentclass[reqno]{amsart}

\usepackage{a4wide}
\usepackage{amssymb}
\usepackage{graphicx}
\usepackage{hyperref}
\usepackage[mathscr]{euscript}
\usepackage{mathtools}
\usepackage{subfigure}
\usepackage{calc}
\usepackage{tikz}
\usetikzlibrary{cd,arrows}
\usepackage{newtxtext}
\usepackage{newtxmath}
\usepackage{bm}
\usepackage[english]{babel}

\hypersetup{colorlinks=true,linkcolor=blue}

\numberwithin{equation}{section}

\theoremstyle{plain}
\newtheorem{theorem}{Theorem}[section]
\newtheorem{lemma}[theorem]{Lemma}
\newtheorem{proposition}[theorem]{Proposition}
\newtheorem{corollary}[theorem]{Corollary}

\newtheorem{claim}[theorem]{Claim}
\newtheorem*{theorema*}{Theorem A}
\newtheorem*{theoremb*}{Theorem B}
\newtheorem*{theoremc*}{Theorem C}
\newtheorem*{theoremd*}{Theorem D}
\newtheorem*{theoreme*}{Theorem E}
\newtheorem*{theoremf*}{Theorem F}
\newtheorem*{theoremg*}{Conjecture G}
\newtheorem{proposition/definition}[theorem]{Proposition/Definition}
\newtheorem{theorem/definition}[theorem]{Theorem/Definition}

\theoremstyle{definition}
\newtheorem{definition}[theorem]{Definition}
\newtheorem{example}[theorem]{Example}
\newtheorem{example/definition}[theorem]{Example/Definition}
\newtheorem{assumption}[theorem]{Assumption}
\newtheorem{notation}[theorem]{Notation}

\theoremstyle{remark}
\newtheorem{remark}[theorem]{Remark}

\DeclareMathOperator{\augnumber}{aug}

\DeclareMathOperator{\cone}{Cone}

\DeclareMathOperator{\End}{End}

\DeclareMathOperator{\GNR}{GNR}
\let\hom\relax\DeclareMathOperator{\hom}{Hom}
\DeclareMathOperator{\rhom}{RHom}
\DeclareMathOperator{\muhom}{\mu hom}
\DeclareMathOperator{\im}{Image}
\DeclareMathOperator{\MC}{MC}
\DeclareMathOperator{\cMC}{\mathcal{MC}}
\DeclareMathOperator{\NZ}{NZ}
\DeclareMathOperator{\GPS}{GPS}
\DeclareMathOperator{\NRSSZ}{NRSSZ}
\DeclareMathOperator{\Perf}{Perf}
\DeclareMathOperator{\NR}{NR}

\DeclareMathOperator{\Star}{Star}
\DeclareMathOperator{\coker}{Coker}

\DeclareMathOperator{\identity}{Id}

\DeclareMathOperator{\Mat}{Mat}
\DeclareMathOperator{\Ob}{{\categoryfont{O}\mathsf{b}}}
\DeclareMathOperator{\res}{res}
\DeclareMathOperator{\Res}{Res^\mathsf{Ng}}
\DeclareMathOperator{\Tot}{Tot}
\DeclareMathOperator{\Gr}{Gr}
\DeclareMathOperator{\cRM}{RM}

\DeclareMathOperator{\Supp}{Supp}
\DeclareMathOperator{\val}{val}

\newcommand{\FF}{\mathbb{F}}
\newcommand{\KK}{\mathbb{K}}
\newcommand{\NN}{\mathbb{N}}
\newcommand{\PP}{\mathbb{P}}
\newcommand{\RR}{\mathbb{R}}
\newcommand{\ZZ}{\mathbb{Z}}

\newcommand{\bfM}{\mathbf{M}}

\newcommand{\bfU}{\mathbf{U}}

\newcommand{\bfa}{\mathbf{a}}
\newcommand{\bfd}{\mathbf{d}}
\newcommand{\bfe}{\mathbf{e}}
\newcommand{\bff}{\mathbf{f}}

\newcommand{\bfi}{\mathbf{i}}

\newcommand{\bfinfty}{\boldsymbol{\infty}}

\newcommand{\bfgraf}{\mathbf{\graf}}
\newcommand{\bfzero}{\mathbf{0}}
\newcommand{\bfmu}{{\bm{\mu}}}
\newcommand{\bfpi}{{\bm{\pi}}}

\newcommand{\cB}{\mathcal{B}}
\newcommand{\cC}{\mathcal{C}}
\newcommand{\cD}{\mathcal{D}}
\newcommand{\cE}{\mathcal{E}}
\newcommand{\cF}{\mathcal{F}}
\newcommand{\cG}{\mathcal{G}}

\newcommand{\cI}{\mathcal{I}}
\newcommand{\cK}{\mathcal{K}}

\newcommand{\cM}{\mathcal{M}}

\newcommand{\cR}{\mathcal{R}}
\newcommand{\cS}{\mathcal{S}}
\newcommand{\cT}{\mathcal{T}}
\newcommand{\cU}{\mathcal{U}}
\newcommand{\cV}{\mathcal{V}}

\newcommand{\cX}{\mathcal{X}}

\newcommand{\scrA}{\mathscr{A}}

\newcommand{\scrT}{\mathscr{T}}

\newcommand{\scrW}{\mathscr{W}}

\newcommand{\sfA}{\mathsf{A}}

\newcommand{\sfC}{\mathsf{C}}

\newcommand{\sfG}{\mathsf{G}}

\newcommand{\sfI}{\mathsf{I}}

\newcommand{\sfM}{\mathsf{M}}

\newcommand{\sfR}{\mathsf{R}}
\newcommand{\sfS}{\mathsf{S}}
\newcommand{\sfT}{\mathsf{T}}

\newcommand{\sfV}{\mathsf{V}}

\DeclareSymbolFont{sfletters}{OT1}{cmss}{m}{n}
\DeclareMathSymbol{\sTheta}{\mathord}{sfletters}{"02}

\newcommand{\CE}{\mathsf{CE}}

\newcommand{\co}{\mathsf{co}}

\newcommand{\pre}{\mathsf{pre}}

\newcommand{\front}{\mathsf{fr}}
\newcommand{\lag}{\mathsf{Lag}}

\newcommand{\ttB}{\mathtt{B}}
\newcommand{\ttE}{\mathtt{E}}

\newcommand{\ttV}{\mathtt{V}}

\newcommand{\ttv}{\mathtt{v}}

\newcommand{\frK}{\mathfrak{K}}

\newcommand{\fd}{\mathfrak{d}}

\newcommand{\fg}{\mathfrak{g}}
\newcommand{\frh}{\mathfrak{h}}
\newcommand{\fri}{\mathfrak{i}}
\newcommand{\fp}{\mathfrak{p}}
\newcommand{\fr}{\mathfrak{r}}

\renewcommand{\bar}{\overline}
\renewcommand{\emptyset}{\varnothing}
\renewcommand{\hat}{\widehat}
\renewcommand{\tilde}{\widetilde}

\newcommand{\Zmod}[1]{\ZZ/#1\ZZ}

\newcommand{\graf}{\Gamma}
\newcommand{\differential}{\partial}
\newcommand{\boundary}{\partial}
\newcommand{\field}{\KK}
\newcommand{\ring}{\field}

\newcommand{\Left}{\mathsf{L}}
\newcommand{\Right}{\mathsf{R}}

\newcommand{\op}{\mathsf{op}}

\newcommand{\unit}{\mathbf{1}}

\newcommand{\Lcirclearrowright}{{\rotatebox[origin=c]{90}{$\circlearrowright$}}}

\newcommand{\grading}{\ZZ}

\newcommand{\vg}{\xi}

\newcommand{\homotopic}{\simeq}
\newcommand{\isomorphic}{\cong}
\newcommand{\relation}{\sim}

\newcommand{\p}[1]{{(#1)}}
\newcommand{\RM}[1]{(\mathsf{#1})}

\DeclareMathOperator{\Mod}{Mod}

\newcommand{\dgafont}{\mathcal}
\newcommand{\algfont}{\mathsf}

\newcommand{\categoryfont}{\mathscr}
\newcommand{\functorfont}{\mathscr}

\newcommand{\alg}{\algfont{A}}
\newcommand{\dga}{\dgafont{A}}

\newcommand{\aug}{\mathsf{Aug}}

\newcommand{\ruling}{\rho}

\newcommand{\Alg}{{\categoryfont{A}\mathsf{lg}}}
\newcommand{\BAlg}{{\categoryfont{BA}\mathsf{lg}}}

\newcommand{\DGA}{\categoryfont{DGA}}
\newcommand{\BDGA}{\categoryfont{BDGA}}

\newcommand{\Fuk}{\categoryfont{F}\mathsf{uk}}

\newcommand{\Aug}{\categoryfont{A}\mathsf{ug}}

\newcommand{\AugSim}{{\Delta_+}}

\newcommand{\LT}{\categoryfont{LG}}
\newcommand{\BLT}{\categoryfont{BLG}}

\newcommand{\catC}{\categoryfont{C}}

\newcommand{\Ch}{{\categoryfont{C}\mathsf{h}}}

\newcommand{\Fun}{\categoryfont{F}\mathsf{un}}

\newcommand{\Loc}{\categoryfont{L}\mathsf{oc}}
\newcommand{\Op}{{\categoryfont{O}\mathsf{p}}}

\newcommand{\Sh}{\categoryfont{S}\mathsf{h}}

\newcommand{\funF}{\functorfont{F}}

\newcommand{\funI}{\functorfont{I}}

\newcommand*{\DecorationScale}{0.25}

\DeclareRobustCommand*{\crossing}{{
\vcenter{\hbox{\begin{tikzpicture}[scale=\DecorationScale]
\draw[thick,-](-1,1)--(1,-1);
\draw[thick,-](-1,-1)--(1,1);
\end{tikzpicture}}}
}}

\DeclareRobustCommand*{\crossingpositive}{{
\vcenter{\hbox{\begin{tikzpicture}[scale=\DecorationScale]
\draw[thick,-](-1,1)--(1,-1);
\draw[thick,-](-1,-1)--(-0.3,-0.3);
\draw[thick,-](0.3,0.3)--(1,1);
\end{tikzpicture}}}
}}

\DeclareRobustCommand*{\rightkink}{{
\vcenter{\hbox{\begin{tikzpicture}[scale=\DecorationScale]
\draw[thick,-](-1,1) -- (0.3, -0.3) arc (-135: 135: 0.4242);
\draw[thick,-](-0.3,-0.3)--(-1,-1);
\end{tikzpicture}}}
}}

\DeclareRobustCommand*{\leftarc}{{
\vcenter{\hbox{\begin{tikzpicture}[scale=\DecorationScale]
\draw[thick,-](0,1) arc ( 90: 275: 1);
\end{tikzpicture}}}
}}

\DeclareRobustCommand*{\Lcusp}{{
\vcenter{\hbox{\begin{tikzpicture}[scale=\DecorationScale]
\draw[thick,-] (1,1) arc (-45:-90:2.414) arc (90:45:2.414);
\end{tikzpicture}}}
}}

\DeclareRobustCommand*{\Rcusp}{{
\vcenter{\hbox{\begin{tikzpicture}[scale=\DecorationScale]
\draw[thick,-] (-1,1) arc (-135:-90:2.414) arc (90:135:2.414);
\end{tikzpicture}}}
}}

\DeclareRobustCommand*{\frontvertex}{{
\vcenter{\hbox{\begin{tikzpicture}[scale=\DecorationScale]
\draw[thick,-] (-1.414,1) arc (-135:-90:2.414) arc (90:135:2.414);
\draw[thick,-] (1.414,1) arc (-45:-90:2.414) arc (90:45:2.414);
\draw[fill] (0,0.293) circle (5pt);
\end{tikzpicture}}}
}}

\title{Augmentations are sheaves for Legendrian graphs}

\author[B. H. An]{Byung Hee An}
\email{anbyhee@ibs.re.kr}
\address{Center for Geometry and Physics, Institute for Basic Science (IBS), Pohang 37673, Korea}

\author[Y. Bae]{Youngjin Bae}
\email{yjbae@kias.re.kr}
\address{Korea Institute for Advanced Study (KIAS), Seoul 02455, Korea}

\author[T. Su]{Tao Su}
\email{taosu@dma.ens.fr}
\address{Department of Mathematics, \'{E}cole Normale Sup\'{e}rieure, Paris, France}


\begin{document}

\begin{abstract}
In this article, associated to a (bordered) Legendrian graph, we study and show the equivalence between two categorical Legendrian isotopy invariants: the augmentation category, a unital $A_{\infty}$-category, which lifts the set of augmentations of the associated Chekanov-Eliashberg DGA, and a DG category of constructible sheaves on the front plane, with micro-support at contact infinity controlled by the (bordered) Legendrian graph. In other words, generalizing \cite{NRSSZ2015}, we prove ``augmentations are sheaves'' in the singular case. 
\end{abstract}

\maketitle

\tableofcontents

\section{Introduction}
Nowadays, it has been an increasingly rich subject, to study the microlocal sheaf theoretic aspects of symplectic topology, or vice versa. This, has already been of interest in the case of a cotangent bundle $T^*M$. Along this direction, one fundamental result is the microlocalization equivalence \cite{Nadler2009,NZ2009}, between the DG category of constructible sheaves on the base $M$ and the infinitesimally wrapped Fukaya category of the cotangent bundle $T^*M$:
\[
\NZ:\Sh(M;\ring)\xrightarrow[]{\sim}\Fuk^{\epsilon}(T^*M;\ring).
\]
A more refined version of the correspondence involves introducing a Legendrian $\Lambda$ contained in $T^{\infty}M$, the co-sphere bundle (equivalently, the set of points at contact infinity) of $M$. On one side, replace $\Sh(M;\ring)$ by a full subcategory $\Sh_{\Lambda}(M;\ring)$ of constructible sheaves with micro-support at infinity contained in $\Lambda$. On the other side,  replace $\Fuk^{\epsilon}(T^*M;\field)$ by a full subcategory of Lagrangian branes asymptotic to $\Lambda$ at infinity. Then, the Nadler-Zaslow correspondence produces an equivalence:
\[
\NZ:\Sh_{\Lambda}(M;\ring)\xrightarrow[]{\sim}\Fuk_{\Lambda}^{\epsilon}(T^*M;\ring),
\]
Here, both sides are categorical Legendrian isotopy invariants for $\Lambda\subset T^{\infty}M$.
More recently, a correspondence \cite{GPS2017,GPS2018a,GPS2018b} has also been established, between the DG category $\Sh_{\Lambda}(M)^c$ of compact objects in the larger category of (weakly) constructible sheaves on $M$, with micro-support contained in $\Lambda$, and a wrapped Fukaya category $\Perf\mathscr{W}(T^*M,\Lambda)^{\op}$ of $T^*M$, whose Lagrangians are disjoint from $\Lambda$ at infinity:
\[
\GPS:\Sh_{\Lambda}(M)^c\simeq\Perf\scrW(T^*M,\Lambda)^{\op}.
\]
Notice that the latter is of central interest in Homological Mirror Symmetry. 

On the other hand, there is another powerful modern Legendrian isotopy invariant for a Legendrian submanifold $\Lambda$ in any contact manifold $V$. It is called the Legendrian contact homology DGA $A^\CE(V,\Lambda)$, obtained by counting of holomorphic disks. More specifically, the generators of this algebra are the Reeb chords of $\Lambda$, and the differential counts holomorphic disks with punctures on the boundary, in the symplectization $\RR_t\times V$, with boundary living on the Lagrangian $\RR_t\times \Lambda$ and approaching the Reeb chords at the punctures. The DGA $A^\CE(V,\Lambda)$ is a Legendrian isotopy invariant up to homotopy equivalence. 

Motivated by the Nadler-Zaslow correspondence, it is expected that certain representation category of $A^\CE(\Lambda)$ is equivalent to the sheaf category $\Sh_{\Lambda}(N\times\RR_z;\ring)$, when $V=J^1N\cong T^{\infty,-}(N\times\RR_z)$ is a one-jet bundle naturally identified with an open contact submanifold of $T^{\infty}(M=N\times\RR_z)$. 

In this article, the case of the major interest to us is when $N=\RR_x$, in which there is a rich interaction with Legendrian knot theory. That is, $\Lambda\subset J^1\RR_x=\RR_{\mathrm{std}}^3$ with a Maslov potential $\mu$ is now a Legendrian knot in the standard contact three space. In this case, ``augmentations are sheaves'' holds.
\begin{theorem}[\cite{NRSSZ2015}]\label{thm:NRSSZ}
There is an $A_{\infty}$-equivalence: 
\[
\NRSSZ:\Aug_+(\Lambda,\mu;\field)\xrightarrow[]{\sim}\cC_1(\Lambda,\mu;\field)
\]
\end{theorem}
Here, $\Aug_+(\Lambda,\mu;\field)$ is a unital $A_{\infty}$-category whose objects are the set of augmentations of the DGA $A^\CE(\Lambda,\mu)$, a representation category of rank $1$ representations of $A^\CE(\Lambda,\mu)$, and $\cC_1(\Lambda,\mu;\field)$ is the full subcategory of $\Sh_{\Lambda}(\RR_{x,z}^2;\field)$ whose objects are microlocal rank $1$ sheaves with acyclic stalks when $z\ll 0$. Concerning ``representations are sheaves'', or the correspondence between higher rank representations and higher microlocal rank sheaves, an equivalence in the cohomological level \cite{CNS2019} was shown for $(2,m)$-torus knots. In the case when $M$ is of higher dimension, the correspondence in question is widely open. However, see \cite{BC2014,CRGG2016,RS2018,RS2019} and \cite[Sec.6.4]{GPS2018b} for related results along this direction.  

Our input originates from the following speculation in the perspective of microlocal sheaf theory: given a constructible sheaf $\cF$ on $M=\RR_{x,z}^2$, the micro-support of $\cF$ at infinity is in general a singular Legendrian, typically a Legendrian graph in $T^{\infty}M$. So it is natural to seek a contact-geometric interpretation of the sheaf category $\Sh_{\Lambda}(M;\field)$ when $\Lambda$ is a Legendrian graph. More generally, it is natural to ask the same question with essentially no more difficulty, when $N=I_x\subset\RR_{x}$ is an open interval, and $T\subset J^1N$ is a bordered Legendrian graph. 
Our main result gives a positive answer to this question.
\begin{theorem}[Theorems \ref{thm:invariance aug}, \ref{theorem:unitality} and \ref{thm:augs are sheaves for Legendrian graphs}]\label{thm:augs are sheaves}
Let $(\cT,\bfmu)$ be a bordered Legendrian graph $\cT\coloneqq(T_\Left\rightarrow T\leftarrow T_\Right)$ in $J^1I_x\isomorphic T^{\infty,-}(I_x\times\RR_z)$ with a Maslov potential $\bfmu=(\mu_\Left,\mu,\mu_\Right)$.
There is a well-defined diagram of unital $A_{\infty}$-categories:
\begin{eqnarray*}
\Aug_+(\cT,\bfmu;\field)\coloneqq(\Aug_+(T_\Left,\mu_\Left;\field)\leftarrow\Aug_+(T,\mu;\field)\rightarrow\Aug_+(T_\Right,\mu_\Right;\field)),
\end{eqnarray*}
where the objects of $\Aug_+(T,\mu;\field)$ are the augmentations of the LCH DGA $A^\CE(T,\mu)$ defined as in \cite{AB2018}. The diagram is invariant under Legendrian isotopy and basepoint moves up to $A_{\infty}$-equivalence. Moreover, there is an $A_{\infty}$-equivalence between two diagrams of unital $A_{\infty}$-categories:
\begin{eqnarray*}
\Aug_+(\cT,\bfmu;\field)\xrightarrow[]{\sim}\cC_1(\cT,\bfmu;\field),
\end{eqnarray*}
where 
\[
\cC_1(\cT,\bfmu;\field)\coloneqq(\cC_1(T_\Left,\mu_\Left;\field)\leftarrow\cC_1(T,\mu;\field)\rightarrow\cC_1(T_\Right,\mu_\Right;\field))
\]
is a digram of DG categories of constructible sheaves. 
\end{theorem}  

The DG category $\cC_1(T,\mu;\field)$ is the full subcategory of $\Sh_T(I_x\times\RR_z;\field)$ whose objects are microlocal rank $1$ sheaves with acyclic stalks for $z\ll 0$.
In particular, when $T=\Lambda\subset T^{\infty,-}\RR_{x,z}^2$ is a Legendrian graph, we get an $A_{\infty}$-equivalence $\Aug_+(\Lambda,\mu;\field)\xrightarrow[]{\sim}\cC_1(\Lambda,\mu;\field)$, i.e. ``augmentations are sheaves'' holds in the singular case. 

As a consequence of Theorem \ref{thm:augs are sheaves}, up to a normalization, the point-counting of sheaves in $\cC_1(T,\mu;\field)$ (with boundary conditions) over a finite field $\field=\FF_q$, is equivalent to that of augmentations in $\Aug_+(T,\mu;\field)$ (with boundary conditions), called augmentation number and denoted by $\augnumber(\cT,\bfmu;\rho_\Left,\rho_\Right;\FF_q)$. Generalizing the results in \cite{HR2015,Su2017}, our preceding paper \cite{ABS2019count} solves this counting problem.
More precisely, augmentation numbers are computed by ruling polynomials of $T$, defined via the combinatorics of decompositions of the front projection $T$:
\begin{theorem}\cite{ABS2019count}\label{thm:count augmentations}
Let $(\cT,\bfmu)$ be as above. Let $\rho_\Left\in\NR(T_\Left,\mu_\Left)$ and $\rho_\Right\in\NR(T_\Right,\mu_\Right)$ be two boundary conditions (i.e. normal rulings). Then the following two Legendrian isotopy invariants are the same:
\begin{eqnarray*}
\augnumber(\cT,\bfmu;\rho_\Left,\rho_\Right;\FF_q)=q^{-\frac{d+\hat{B}}{2}}z^{\hat{B}}\langle\rho_\Left|R(\cT,\bfmu;q,z)|\rho_\Right\rangle
\end{eqnarray*}
Here, $\langle\rho_\Left|R(\cT,\bfmu;q,z)|\rho_\Right\rangle\in\ZZ[q^{\pm\frac{1}{2}},z^{\pm 1}]$ is the ruling polynomial for $(\cT,\bfmu)$ with boundary conditions $(\rho_\Left,\rho_\Right)$, $d\coloneqq\max\deg_z\langle\rho_\Left|R(\cT,\bfmu;z^2,z)|\rho_\Right\rangle$. In the formula, we take $z=q^{\frac{1}{2}}-q^{-\frac{1}{2}}$, and 
\[
\hat{B}\coloneqq B+\sum_{v\in V(\cT)}\frac{val(v)}{2}
\]
counts the number of ``generalized'' basepoints in $T$.

Moreover, the ruling polynomials satisfy the \emph{gluing property}: If $(\cT,\bfmu)=(\cT^1,\bfmu^1)\circ(\cT^2,\bfmu^2)$ is a composition of two bordered Legendrian graphs, that is, $(T^1_\Right,\mu^1_\Right)=(T^2_\Left,\mu^2_\Left)$, then
\[
\langle\rho_\Left|R(\cT,\bfmu;q,z)|\rho_\Right\rangle=\sum_{\rho_I\in\NR(T^1_\Right,\mu^1_\Right)}\langle\rho_\Left|R(\cT^1,\bfmu^1;q,z)|\rho_I\rangle\langle\rho_I|R(\cT^2,\bfmu^2;q,z)|\rho_\Right\rangle.
\]
\end{theorem}

\addtocontents{toc}{\protect\setcounter{tocdepth}{1}}
\subsection*{Organization}
The article is organized as follows. 
In Section \ref{section:setup}, we review the basic backgrounds from \cite{ABS2019count} on Legendrian graphs and bordered Legendrian graphs $(\cT,\bfmu)$ with a Maslov potential. The main ingredients include the Legendrian contact homology (LCH) DGAs $A^\CE(\cT,\bfmu)=(A^\CE(T_\Left,\mu_\Left)\rightarrow A^\CE(T,\mu)\leftarrow A^\CE(T_\Right,\mu_\Right))$ for bordered Legendrian graphs.

In Section \ref{section:consistent sequences}, we give the algebraic preliminaries before defining augmentation categories. In particular, we introduce two categories of consistent sequences: consistent sequences of bordered Legendrian graphs, and consistent sequences of DGAs. In Section \ref{section:aug cats}, we define augmentation categories for bordered Legendrian graphs, and then show their unitarity and invariance.

In Section \ref{section:sheaf cats}, we firstly give the preliminaries on the microlocal theory of sheaves. Then we construct the necessary combinatorial tools, which we call legible models for  $\Sh(\cT;\field)=(\Sh(T_\Left;\field)\leftarrow\Sh(T;\field)\rightarrow\Sh(T_\Right;\field))$, the diagram of sheaf categories for a bordered Legendrian graph $\cT$ in $T^{\infty,-}(I_x\times\RR_z)$.  As an application, we prove the invariance of $\Sh(\cT;\field)$ via combinatorics, hence the invariance of $\cC_1(\cT,\bfmu;\field)$, the diagram of full-subcategories of $\Sh(\cT;\field)$ whose objects are microlocal rank $1$ objects with acyclic stalks for $z\ll 0$.

In Section \ref{section:augs are sheaves}, we prove our main result ``augmentations are sheaves'' for bordered Legendrian graphs. The basic idea is as follows: By the invariance results in Section \ref{section:aug cats} and Section \ref{section:sheaf cats}, we can assume the vertices in the front projection $T$ are all of type $(0,r)$ for some $r$, all the left cusps and vertices are to the left of the crossings of $T$, and all the right cusps are to the right of the crossings of $T$. Moreover, we can assume all right cusps are marked. Then both of the two diagrams of $A_{\infty}$-categories satisfy the sheaf property over $I_x$. Hence, by decomposing the front diagram $T$ into the composition of elementary pieces, it suffices to show the theorem for each elementary piece. By the results for Legendrian knots in \cite{NRSSZ2015}, it suffices to show the case of an elementary bordered Legendrian graph $(\cV,\bfmu)$ involving only a vertex. This is done by an explicit description for both of the two diagrams $\Aug_+(\cV,\bfmu;\field)$ and $\cC_1(\cV,\bfmu;\field)$. The augmentation side is done in Lemma \ref{lem:aug cat for a vertex}. The sheaf side is a direct application of the legible model in Section \ref{subsubsec:legible model} for $\cC_1(\cV,\bfmu;\field)$. Then we are done.

\addtocontents{toc}{\protect\setcounter{tocdepth}{1}}
\subsection*{Acknowledgements}
We would like to thank RIMS in Japan, IBS-CGP in South Korea, and ENS Paris - CNRS in France for supporting the visits, where much of this project was developed. 
The first author is supported by IBS-R003-D1.
The second author is supported by Korea Institute for Advanced Study and Japan Society for the Promotion of Science International Research Fellowship Program. He thanks Research Institute for Mathematical Sciences, Kyoto University for its warm hospitality.
The third author is supported by ANR-15-CE40-0007. He would like to thank St\'{e}phane Guillermou for the invitation to visit CRM, Montreal, where part of this project was improved. In addition, he is grateful to Vivek Shende, David Nadler, and Lehnard Ng for the help in his early career.

\addtocontents{toc}{\protect\setcounter{tocdepth}{2}}

\section{Setup}
\label{section:setup}

\begin{notation}
For each $m\ge0$, we denote the set $\{1,\cdots, m\}$ equipped with the natural order by $[m]$.
\end{notation}

\subsection{Bordered Legendrian graphs}
In this section, we briefly review the definition of \emph{bordered Legendrian graphs} defined in \cite[\S2]{ABS2019count}.
A graph is a finite regular one dimensional CW complex, whose 0-cells and closed 1-cells are called vertices and edges.
For each vertex $\ttv$, a \emph{half-edge} at $\ttv$ is a small enough restriction of an edge adjacent to $\ttv$.
Then as usual, the \emph{valency} of $\ttv$ is the number of half-edges at $v$ and denoted by $\val(\ttv)$.

A (\emph{based}) \emph{bordered graph} $\graf=(\ttV,\ttV_\Left,\ttV_\Right,\ttB,\ttE)$ of type $(n_\Left,n_\Right)$ consists of the following data:
\begin{itemize}
\item the pair $(\ttV\amalg\ttV_\Left\amalg\ttV_\Right\amalg\ttB,\ttE)$ defines a graph $|\graf|$,
\item each $b\in\ttB$ of $|\graf|$ is bivalent, and
\item two disjoint subsets $\ttV_\Left$ and $\ttV_\Right$ consist of $n_\Left$ and $n_\Right$ univalent vertices of $|\graf|$.
\end{itemize}

Elements in $\ttV,\ttV_\Left,\ttV_\Right,\ttB$ and $\ttE$ will be called \emph{vertices}, \emph{left borders}, \emph{right borders}, \emph{basepoints} and \emph{edges}, respectively.
The \emph{interior} $\mathring{\graf}$ of $\graf$ is define to be the complement of $\ttV_\Left\amalg\ttV_\Right$.

\begin{notation}
In order to emphasize the border structures, we will denote the bordered graph $\graf$ as
\[
\bfgraf=(\graf_\Left\to \graf \leftarrow \graf_\Right),
\]
where $\graf_\Left$ and $\graf_\Right$ are defined by $\ttV_\Left$ and $\ttV_\Right$ by, respectively, and both arrows are inclusions.

From now on, we mean by a \emph{graph} a bordered graph with empty borders $(\emptyset\to \graf\leftarrow\emptyset)$, which will be denoted simply by $\graf$.
\end{notation}

For a closed interval $U=[x_\Left,x_\Right]\subset \RR_x$, let the \emph{bordered one-jet space} $J^1\bfU=(J^1U_\Left\to J^1U\leftarrow J^1U_\Right)$ be the one-jet space $J^1 U\coloneqq(U\times\RR_{yz}, dz-ydx)\subset J^1\RR_x= (\RR^3_{xyz}, dz-ydx)$ together with two contact submanifolds
\[
J^1 U_\Left\coloneqq \left(\{x_\Left\}\times\RR_z,dz\right)\quad\text{ and }\quad
J^1 U_\Right\coloneqq \left(\{x_\Right\}\times\RR_z,dz\right).
\]

\begin{definition}[bordered Legendrian graphs]\label{definition:bordered Legendrian graph}
A \emph{bordered Legendrian graph} $\scrT=(\sfT_\Left\to\sfT\leftarrow \sfT_\Right)$ of a bordered graph $\bfgraf$ of type $(n_\Left,n_\Right)$ in $J^1\bfU$ is defined as an embedding $\sfT:\graf\to J^1U$ such that 
\begin{enumerate}
\item $\sfT$ is transverse to the boundary $\boundary J^1U=\boundary U\times\RR_{yz}$ and the restrictions on the interior $\mathring{\graf}$ and both borders $\graf_\Left$ and $\graf_\Right$ are contained in $J^1\mathring{U}$, $J^1U_\Left$ and $J^1U_\Right$, respectively.
\begin{align*}
\sfT&\pitchfork \boundary J^1U,&
\sfT_\Left&\coloneqq\sfT(\graf_\Left)\subset J^1U_\Left,&
\sfT_\Right&\coloneqq\sfT(\graf_\Right)\subset J^1U_\Right,&
\mathring{\sfT}&\coloneqq\sfT(\mathring{\graf})\subset J^1\mathring{U}.
\end{align*}
\item $\sfT$ on edges are smooth Legendrian with boundary and pairwise non-tangent to each other at all vertices and basepoints, in particular, two edges adjacent to each basepoint form a smooth arc.
\end{enumerate}

By labeling borders in $\sfT_\Left$ and $\sfT_\Right$ in top-to-bottom ways with respect to $z$-coordinates, we identify the left and right border $\sfT_\Left$ and $\sfT_\Right$ with the set $[n_\Left]=\{1,\dots,n_\Left\}$ and $[n_\Right]=\{1,\dots,n_\Right\}$.
\end{definition}

There are two projections for $J^1\RR_x\isomorphic \RR^3_{xyz}$, called the \emph{front} and \emph{Lagrangian} projections $\pi_{\front}:\RR^3_{xyz}\to\RR^2_{xz}$ and $\pi_{\lag}:\RR^3_{xyz}\to\RR^2_{xy}$, respectively.

\begin{definition}[Regular projections]\label{definition:regular projections}
For a bordered Legendrian graph $\scrT=(\sfT_\Left\to\sfT\leftarrow\sfT_\Right)$, the \emph{front} and \emph{Lagrangian projections} $\cT=(T_\Left\to T\leftarrow T_\Right)\coloneqq \pi_\front(\scrT)$ and $\cT_\lag=(T_{\lag,\Left}\to T_\lag\leftarrow T_{\lag,\Right})\coloneqq \pi_\lag(\scrT)$ are said to be \emph{regular} if in their interiors,
\begin{enumerate}
\item there are only finitely many transverse double points, called \emph{crossings}, and
\item no vertices, basepoints or \emph{$x$-extreme points} are crossings,
\item each edge containing a $x$-maximal point must involve at least one vertex or a basepoint,
\end{enumerate}
where a point in the interior $\mathring{\sfT}$ is said to be \emph{$x$-maximal} or \emph{$x$-minimal} if it is maximal or minimal with respect to the $x$-coordinate, and \emph{$x$-extreme} if it is either $x$-maximal or $x$-minimal.\footnotemark

A bordered Legendrian graph of type $(0,0)$ is called a \emph{Legendrian graph} and we denote the sets of regular front and Lagrangian projections of non-bordered and bordered Legendrian graphs by
$\LT, \LT_\lag$ and $\BLT, \BLT_\lag$, respectively.
\end{definition}
\footnotetext{In the front projection, an $x$-extreme point is either a cusp or a vertex of type $(0,n)$ or $(n,0)$.}

\begin{remark}
Due to the Legendrian property, there are no vertical tangencies and no non-transverse double points in the front projection.
Instead, it contains \emph{cusps}, which is obviously, $x$-extreme.
\end{remark}

\begin{notation}
The front and Lagrangian projection of $\sfT=\sfT(\graf)$ with $\graf=(\ttV,\ttV_\Left, \ttV_\Right, \ttB,\ttE)$ will be denoted by $T=(V,V_\Left, V_\Right, B, E)$ and $T_\lag=(V_\lag,V_{\lag,\Left},V_{\lag,\Right},B_\lag,E_\lag)$, respectively.
\end{notation}

For examples of regular and non-regular projection, see Figure~\ref{figure:non-regular projections}.
To avoid any confusion, we denote vertices and basepoints by small dots and bars, respectively.

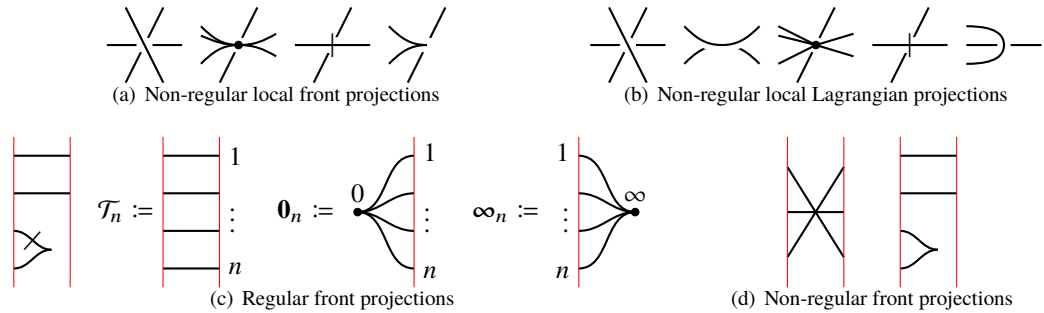
\begin{figure}[ht]
\subfigure[Non-regular local front projections]{\makebox[0.45\textwidth]{$
\begin{tikzpicture}[baseline=-.5ex,scale=0.5]
\begin{scope}
\draw[thick] (-1,0)--(1,0);
\draw[thick] (-.5,-1)--(.5,1);
\draw[white,fill=white] (0,0) circle (0.2);
\draw[thick] (-.5,1)--(.5,-1);
\end{scope}
\begin{scope}[xshift=2.5cm]
\draw[thick] (-.5,-1)--(.5,1);
\draw[white,fill=white] (0,0) circle (0.2);
\draw[thick] (-1,0.5) to[out=-45,in=180] (0,0) to[out=180,in=45] (-1,-0.5);
\draw[thick] (1,0.3) to[out=-150,in=0] (0,0) to[out=0,in=150] (1,-0.3);
\draw[thick] (-1,0.25) to[out=-22.5,in=180] (0,0);
\draw[fill] (0,0) circle (0.1);
\end{scope}
\begin{scope}[xshift=5cm]
\draw[thick] (-.5,-1)--(.5,1);
\draw[white,line width=6] (-1,0) -- (1,0);
\draw[thick] (-1,0)--(0,0) node{$|$} -- (1,0);
\end{scope}
\begin{scope}[xshift=7.5cm]
\draw[thick] (-.5,-1)--(.5,1);
\draw[white,fill=white] (0,0) circle (0.2);
\draw[thick] (-1,0.5) to[out=-45,in=180] (0,0) to[out=180,in=45] (-1,-0.5);
\end{scope}
\end{tikzpicture}
$}}
\subfigure[Non-regular local Lagrangian projections]{\makebox[0.5\textwidth]{$
\begin{tikzpicture}[baseline=-.5ex,scale=0.5]
\begin{scope}
\draw[thick] (-1,0)--(1,0);
\draw[thick] (-.5,-1)--(.5,1);
\draw[white,fill=white] (0,0) circle (0.2);
\draw[thick] (-.5,1)--(.5,-1);
\end{scope}
\begin{scope}[xshift=2.5cm]
\draw[thick] (-1,-0.5) to[out=45,in=180] (0,0) to[out=0,in=135] (1,-0.5);
\draw[white,line width=5] (-1,0.5) to[out=-45,in=180] (0,0) to[out=0,in=-135] (1,0.5);
\draw[thick] (-1,0.5) to[out=-45,in=180] (0,0) to[out=0,in=-135] (1,0.5);
\end{scope}
\begin{scope}[xshift=5cm]
\draw[thick] (-0.5,-1) -- (0.5,1);
\draw[white,fill=white] (0,0) circle (0.2);
\draw[thick] (-1,0.5) -- (0,0) (-1,0.25) -- (0,0) (-1,-0.5) -- (0,0) (1,0.3) -- (0,0) (1,-0.3) -- (0,0);
\draw[fill] (0,0) circle (0.1);
\end{scope}
\begin{scope}[xshift=7.5cm]
\draw[thick] (-.5,-1)--(.5,1);
\draw[white,line width=6] (-1,0) -- (1,0);
\draw[thick] (-1,0)--(0,0) node{$|$} -- (1,0);
\end{scope}
\begin{scope}[xshift=10cm]
\draw[thick] (-1,0) -- (1,0);
\draw[white,line width=5] (-1,-0.5) to[out=0,in=-90] (0,0) to[out=90,in=0] (-1,0.5);
\draw[thick] (-1,-0.5) to[out=0,in=-90] (0,0) to[out=90,in=0] (-1,0.5);
\end{scope}
\end{tikzpicture}
$}}
\subfigure[Regular front projections\label{figure:regular projections}]{\makebox[0.55\textwidth]{$
\begin{tikzpicture}[baseline=-.5ex]
\begin{scope}
\draw[thick] (-0.75,0.75) -- (0,0.75);
\draw[thick] (-0.75,0.25) -- (0,0.25);
\draw[thick](-0.25,-0.5) to[out=180,in=0] node[midway,sloped] {$|$} (-0.75,-0.25);
\draw[thick](-0.25,-0.5) to[out=180,in=0] (-0.75,-0.75);
\draw[red](-0.75,-1)--(-0.75,1);
\draw[red](0,-1)--(0,1);
\end{scope}
\end{tikzpicture}\quad
\cT_n\coloneqq
\begin{tikzpicture}[baseline=-.5ex]
\foreach \i in {-1.5,-0.5,0.5,1.5} {
\draw[thick] (-0.75, {\i*0.5}) -- +(0.75,0);
}
\draw[red](-0.75,-1)--(-0.75,1) (0,-1)--(0,1);
\draw (0,0.75) node[right] {$1$};
\draw (0,0) node[right] {$\vdots$};
\draw (0,-0.75) node[right] {$n$};
\end{tikzpicture}\quad
\bfzero_n\coloneqq
\begin{tikzpicture}[baseline=-.5ex]
\draw[fill](0,0) circle (0.05) node[above] {$0$};
\foreach \i in {0.75, 0.25, -0.25, -0.75} {
\draw[thick, rounded corners](0,0) to[out=0,in=180] (0.75,\i);
}
\draw[red](0.75,-1)--(0.75,1);
\draw (0.75,0.8) node[right] {$1$};
\draw (0.75,0) node[right] {$\vdots$};
\draw (0.75,-0.8) node[right] {$n$};
\end{tikzpicture}\quad
\bfinfty_n\coloneqq
\begin{tikzpicture}[baseline=-.5ex]
\draw[fill](0,0) circle (0.05) node[above] {$\infty$};
\foreach \i in {0.75, 0.25, -0.25, -0.75} {
\draw[thick, rounded corners](0,0) to[out=180,in=0] (-0.75,\i);
}
\draw[red](-0.75,-1)--(-0.75,1);
\draw (-0.75,0.8) node[left] {$1$};
\draw (-0.75,0) node[left] {$\vdots$};
\draw (-0.75,-0.8) node[left] {$n$};
\end{tikzpicture}
$}}
\subfigure[Non-regular front projections]{\makebox[0.4\textwidth]{$
\begin{tikzpicture}[baseline=-.5ex]
\begin{scope}
\draw[thick] (-0.75,0.6) -- (0,-0.6);
\draw[thick] (-0.75,0) -- (0,0);
\draw[thick] (-0.75,-0.6) -- (0,0.6);
\draw[red](-0.75,-1)--(-0.75,1);
\draw[red](0,-1)--(0,1);
\end{scope}
\begin{scope}[xshift=1.5cm]
\draw[thick] (-0.75,0.75) -- (0,0.75);
\draw[thick] (-0.75,0.25) -- (0,0.25);
\draw[thick, rounded corners](-0.25,-0.5) to[out=180,in=0] (-0.75,-0.25);
\draw[thick, rounded corners](-0.25,-0.5) to[out=180,in=0] (-0.75,-0.75);
\draw[red](-0.75,-1)--(-0.75,1);
\draw[red](0,-1)--(0,1);
\end{scope}
\end{tikzpicture}
$}}
\caption{Regular and non-regular projections of bordered Legendrian graphs}
\label{figure:non-regular projections}
\end{figure}

\begin{definition}[Types and orientations]\label{definition:types of vertices}
For a vertex $v$ or a basepoint $b$ of a bordered Legendrian graph, we say that it is of \emph{type $(\ell,r)$} if there are $\ell$ and $r$ half-edges adjacent to $v$ or $b$ on the left and right, respectively. 
We label the set $H_v\coloneqq{\{h_{v,1},\dots,h_{v,n}\}}$ of (small enough) half-edges in front and Lagrangian projections as follows:
\[
\begin{tikzpicture}[baseline=-0.5ex,scale=0.8]
\begin{scope}[xshift=-5cm]
\foreach \i in {4,...,8} {
	\draw[thick] (-1,{(6-\i)/3}) to[out=0,in=180] (0,0);
}
\draw (-1,0.67) node[left] {$h_{v,1}$};
\draw (-1,0.33) node[left] {$h_{v,2}$};
\draw (-1,-0.167) node[left] {$\vdots$};
\draw (-1,-0.67) node[left] {$h_{v,\ell}$};
\foreach \i in {1,2,3} {
	\draw[thick] (1,{(2-\i)/1.5}) to[out=180,in=0] (0,0);
}
\draw (1,0.67) node[right] {$h_{v,\ell+1}$};
\draw (1,0) node[right] {$\vdots$};
\draw (1,-0.67) node[right] {$h_{v,\ell+r}$};
\draw[fill] (0,0) circle (2pt) node[above] {$v$};
\end{scope}
\begin{scope}
\foreach \i in {4,...,8} {
	\draw[thick] (-1,{(\i-6)/3}) -- (0,0);
}
\draw (-1,-0.67) node[left] {$h_{v,1}$};
\draw (-1,-0.33) node[left] {$h_{v,2}$};
\draw (-1,0.167) node[left] {$\vdots$};
\draw (-1,0.67) node[left] {$h_{v,\ell}$};
\foreach \i in {1,2,3} {
	\draw[thick] (1,{(2-\i)/1.5}) -- (0,0);
}
\draw (1,0.67) node[right] {$h_{v,\ell+1}$};
\draw (1,0) node[right] {$\vdots$};
\draw (1,-0.67) node[right] {$h_{v,\ell+r}$};
\draw[fill] (0,0) circle (2pt) node[above] {$v$};
\end{scope}
\begin{scope}[xshift=5cm]
\draw[thick] (-1,0) node[left] {$h_{b,1}$} -- (0,0) node {$|$} node[above] {$b$} -- (1,0) node[right] {$h_{b,2}$};
\end{scope}
\end{tikzpicture}
\quad\cdots
\]

In particular, each basepoint $b\in B$ is assumed to be \emph{oriented} from the half-edge $h_{b,1}$ to $h_{b,2}$ in the above convention.
\end{definition}

\begin{example/definition}[The trivial and vertex bordered Legendrian graphs]\label{example/definition:trivial graph}
Let $n\ge 1$. The front projections of the trivial bordered Legendrian graph $\cT_n=(T_{n,\Left}\to T_n\leftarrow T_{n,\Right})$ of type $(n,n)$ and the vertex bordered graphs $\bfzero_n=(\emptyset\to 0_n\leftarrow 0_{n,\Right})$ and $\bfinfty_n=(\infty_{n,\Left}\to \infty_n\leftarrow\emptyset)$ of type $(0,n)$ and $(n,0)$ as depicted in Figure~\ref{figure:regular projections},
whose left and right borders are points lying on the red lines at the left and right, respectively.

For convenience's sake, we define $\cT_0=\bfzero_0=\bfinfty_0=(\emptyset\to\emptyset\leftarrow\emptyset)$.
\end{example/definition}

\begin{definition}[Equivalences and isomorphisms]\label{definition:isomorphisms}
We say that two bordered Legendrian graphs $\scrT$ and $\scrT'$ are \emph{equivalent} if
they are isotopic, that is, there exists a family of bordered Legendrian graphs
\begin{align*}
\scrT_t&:\graf\times[0,1]\to J^1\bfU_t\subset J^1\RR_x,&
\scrT_0&=\scrT,&
\scrT_1&=\scrT'.
\end{align*}

Two regular front (or Lagrangian) projections $\cT$ (or $\cT_\lag$) and $\cT'$ (or $\cT'_\lag$) of bordered Legendrian graphs $\scrT$ and $\scrT'$ are said to be \emph{isomorphic} if there is an isotopy Lagrangian $\scrT_t$ such that $\cT$ (or $\cT_\lag$) and $\cT'$ (or $\cT'_\lag$) are Lagrangian (or front) projections of $\scrT_0$ and $\scrT_1$ and Lagrangian (or front) projections $\cT_t\coloneqq \pi_\front(\scrT_t)$ (or $\cT_{\lag,t}\coloneqq\pi_\lag(\scrT_t)$) are regular for all $t$.
\end{definition}

\begin{remark}
It is important to note that during the isotopy between two bordered Legendrian graphs, the ambient manifold $J^1\bfU_t$ may changes. For example, any translation along the $x$-axis will give us an equivalence.
\end{remark}

\begin{lemma}
Up to isomorphism, every pair of equivalent front or Lagrangian projections can be connected by a zig-zag sequence of \emph{front} or \emph{Lagrangian Reidemeister moves} depicted in Figures~{\rm\ref{fig:front_RM}} or {\rm\ref{fig:RM}}, respectively.
\end{lemma}
This is well-known and we omit the proof. One may refer \cite[Proposition~4.2]{BI2009} or \cite[Proposition~2.1]{ABK2018}. Notice also that these moves are \emph{not optimal}. Namely, the move $\RM{IV}$ is a special case of $\RM{VI}$.

\begin{notation}
The sets of Lagrangian and front Reidemeister moves will be denoted as follows:
\begin{align*}
\cRM&\coloneqq\left\{\RM{I}, \RM{II}, \RM{III}, \RM{V}, \RM{VI}\right\},&
\cRM_\lag&\coloneqq\left\{\RM{0_a}, \RM{0_b}, \RM{0_c}, \RM{ii}, \RM{iii_a}, \RM{iii_b}, \RM{iv}\right\}.
\end{align*}
\end{notation}

\begin{remark}
All front and Lagrangian Reidemeister moves $\RM{M}\in\cRM$ and $\RM{m}\in\cRM_\lag$ never increase the number of crossings.
\end{remark}

\begin{figure}[ht]
\subfigure[\label{fig:front_RM}Front Reidemeister moves]{$
\begin{tikzcd}[row sep=0pc,ampersand replacement=\&]
\vcenter{\hbox{\includegraphics[scale=0.7]{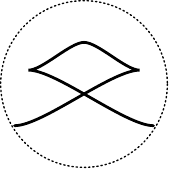}}}\arrow[r,"\RM{I}"]\&
\vcenter{\hbox{\includegraphics[scale=0.7]{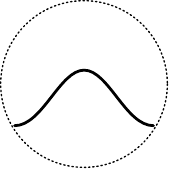}}}
\&
\vcenter{\hbox{\includegraphics[scale=0.7]{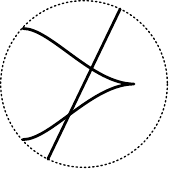}}}\arrow[r,"\RM{II}"]\&
\vcenter{\hbox{\includegraphics[scale=0.7]{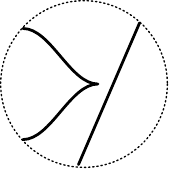}}}
\&
\vcenter{\hbox{\includegraphics[scale=0.7]{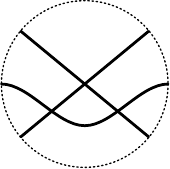}}}\arrow[r,leftrightarrow,"\RM{III}"]\&
\vcenter{\hbox{\includegraphics[scale=0.7]{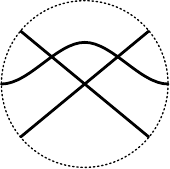}}}
\\
\vcenter{\hbox{\includegraphics[scale=0.7]{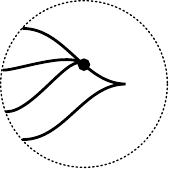}}}\arrow[r,"\RM{IV}"]\&
\vcenter{\hbox{\includegraphics[scale=0.7]{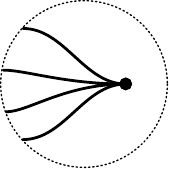}}}
\&
\begin{tikzpicture}[baseline=-.5ex,xscale=-1,scale=0.6]
\draw[densely dotted](0,0) circle (1);
\clip(0,0) circle (1);
\draw[thick] (1,.6) to[out=180,in=0] (-0.5,0);
\draw[thick] (1,.2) to[out=180,in=0] (-0.5,0);
\draw[thick] (1,-.2) to[out=180,in=0] (-0.5,0);
\draw[thick] (1,-.6) to[out=180,in=0] (-0.5,0);
\draw[fill] (-0.5,0) circle (2pt);
\draw[thick] (0,1) -- (0.5,-1);
\end{tikzpicture}\arrow[r,"\RM{V}"]
\&
\begin{tikzpicture}[baseline=-.5ex,xscale=-1,scale=0.6]
\draw[densely dotted](0,0) circle (1);
\clip(0,0) circle (1);
\draw[thick] (1,.6) to[out=180,in=0] (0,0);
\draw[thick] (1,.2) to[out=180,in=0] (0,0);
\draw[thick] (1,-.2) to[out=180,in=0] (0,0);
\draw[thick] (1,-.6) to[out=180,in=0] (0,0);
\draw[fill] (0,0) circle (2pt);
\draw[thick] (-0.5,1) -- (0,-1);
\end{tikzpicture}
\&
\vcenter{\hbox{\includegraphics[scale=0.7]{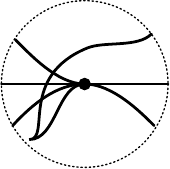}}}\arrow[r,"\RM{VI}"]\&
\vcenter{\hbox{\includegraphics[scale=0.7]{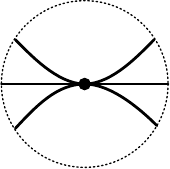}}}
\end{tikzcd}
$}

\subfigure[\label{fig:RM}Lagrangian Reidemeister moves]{$
\begin{tikzcd}[row sep=0pc,ampersand replacement=\&]
\vcenter{\hbox{\includegraphics[scale=0.7]{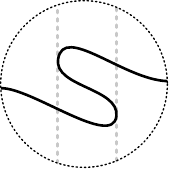}}}\arrow[r,"\RM{0_a}"]\&
\vcenter{\hbox{\includegraphics[scale=0.7]{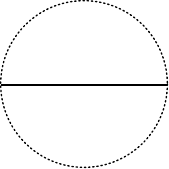}}}
\&
\vcenter{\hbox{\includegraphics[scale=0.7]{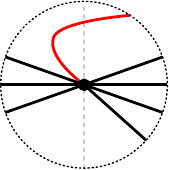}}}\arrow[r,"\RM{0_b}"]\&
\vcenter{\hbox{\includegraphics[scale=0.7]{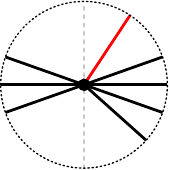}}}
\&
\vcenter{\hbox{\includegraphics[scale=0.7]{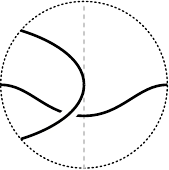}}}\arrow[r,"\RM{0_c}"]\&
\vcenter{\hbox{\includegraphics[scale=0.7]{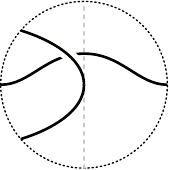}}}
\\
\begin{tikzpicture}[baseline=-.5ex,scale=0.6]
\draw[densely dotted](0,0) circle (1);
\clip(0,0) circle (1);
\draw[thick] (-0.5,-1) -- (0,1);
\draw[white,line width=5] (-1,.6) to[out=0,in=90] (0.5,0) to[out=-90,in=0] (-1,-0.6);
\draw[thick] (-1,.6) to[out=0,in=90] (0.5,0) to[out=-90,in=0] (-1,-0.6);
\end{tikzpicture}\arrow[r,"\RM{ii}"]
\&
\begin{tikzpicture}[baseline=-.5ex,scale=0.6]
\draw[densely dotted](0,0) circle (1);
\clip(0,0) circle (1);
\draw[thick] (0,-1) -- (0.5,1);
\draw[white,line width=5] (-1,.6) to[out=0,in=90] (0,0) to[out=-90,in=0] (-1,-0.6);
\draw[thick] (-1,.6) to[out=0,in=90] (0,0) to[out=-90,in=0] (-1,-0.6);
\end{tikzpicture}
\&
\vcenter{\hbox{\rotatebox[origin=c]{90}{\includegraphics[scale=0.7]{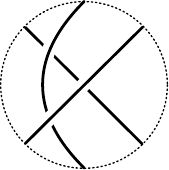}}}}\arrow[r,leftrightarrow,"\RM{iii_a}"]\&
\vcenter{\hbox{\rotatebox[origin=c]{90}{\includegraphics[scale=0.7]{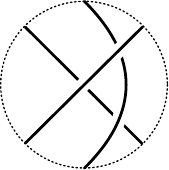}}}}
\&
\vcenter{\hbox{\rotatebox[origin=c]{90}{\includegraphics[scale=0.7]{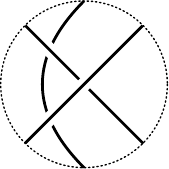}}}}\arrow[r,leftrightarrow,"\RM{iii_b}"]\&
\vcenter{\hbox{\rotatebox[origin=c]{90}{\includegraphics[scale=0.7]{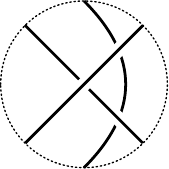}}}}
\\
\begin{tikzpicture}[baseline=-.5ex,scale=0.6]
\draw[densely dotted](0,0) circle (1);
\clip(0,0) circle (1);
\draw[thick] (-0.5,-1) -- (0,1);
\draw[white,line width=5] (-1,.6) -- (0.5,0) (-1,.2) -- (0.5,0) (-1,-.2) -- (0.5,0) (-1,-.6) -- (0.5,0);
\draw[thick] (-1,.6) -- (0.5,0) (-1,.2) -- (0.5,0) (-1,-.2) -- (0.5,0) (-1,-.6) -- (0.5,0);
\draw[fill] (0.5,0) circle (2pt);
\end{tikzpicture}
\arrow[r,"\RM{iv}"]\&
\begin{tikzpicture}[baseline=-.5ex,scale=0.6]
\draw[densely dotted](0,0) circle (1);
\clip(0,0) circle (1);
\draw[thick] (-1,.6) -- (0,0) (-1,.2) -- (0,0) (-1,-.2) -- (0,0) (-1,-.6) -- (0,0);
\draw[fill] (0,0) circle (2pt);
\draw[thick] (0,-1) -- (0.5,1);
\end{tikzpicture}
\end{tikzcd}
$}
\caption{Front and Lagrangian Reidemeister moves: Reflections are possible, the valency of vertex is arbitrary and the vertex can be replaced with a basepoint if it is bivalent.}
\end{figure}
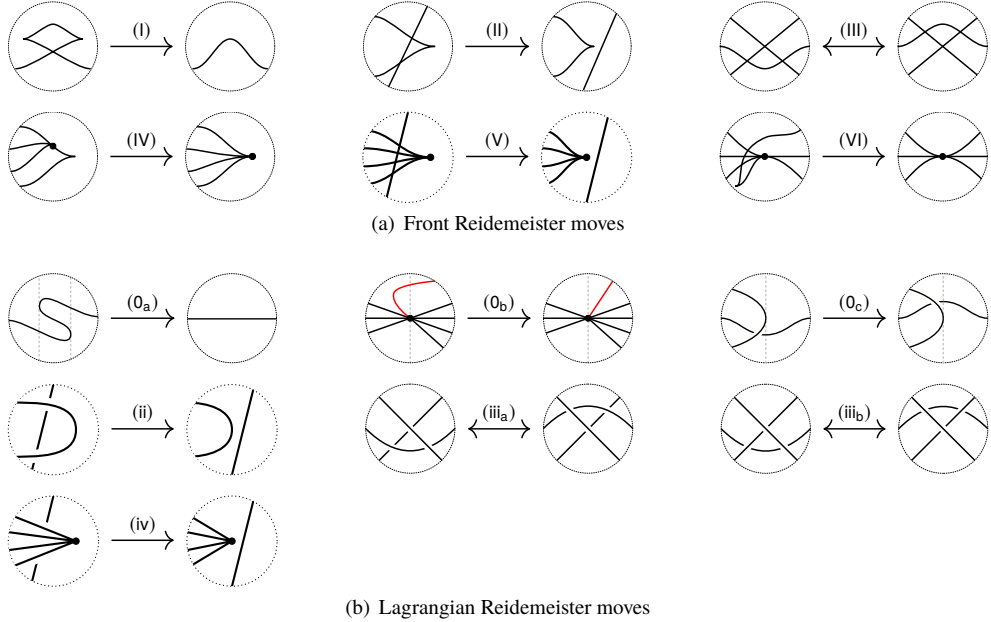

On the other hand, one can consider the weaker equivalence given by the Legendrian isotopy \emph{up to basepoints}. In other words, two bordered Legendrian graphs are Legendrian isotopic after forgetting basepoints.
\begin{lemma}
Two regular front or Lagrangian projections of bordered Legendrian graphs are equivalent up to basepoints if and only if they can be connected by a zig-zag sequence of front or Lagrangian Reidemeister moves together with \emph{basepoint} splittings depicted in Figure~\ref{figure:basepoints}.
\end{lemma}
\begin{proof}
This is obvious.
\end{proof}
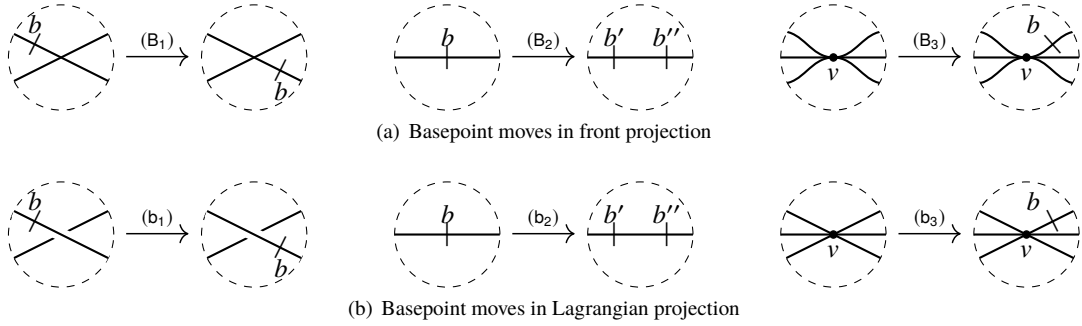
\begin{figure}[ht]
\subfigure[Basepoint moves in front projection]{$
\begin{tikzcd}[ampersand replacement=\&]
\begin{tikzpicture}[baseline=-.5ex,scale=0.7]
\useasboundingbox(-1,-0.5)--(1,0.5);
\draw[dashed](0,0) circle (1);
\clip(0,0) circle (1);
\draw[thick] (-1,0.5)--node[midway,sloped,below=-2ex]{$|$} node[midway,above] {$b$} (0,0) -- (1,-0.5);
\draw[thick] (-1,-0.5) -- (1,0.5);
\end{tikzpicture}\arrow[r,"\RM{B_1}"] \&
\begin{tikzpicture}[baseline=-.5ex,scale=0.7]
\draw[dashed](0,0) circle (1);
\clip(0,0) circle (1);
\draw[thick] (-1,0.5)-- (0,0) --node[midway,sloped,below=-2ex]{$|$} node[midway,below] {$b$} (1,-0.5);
\draw[thick] (-1,-0.5) -- (1,0.5);
\end{tikzpicture}\&
\begin{tikzpicture}[baseline=-.5ex,scale=0.7]
\useasboundingbox(-1,-0.5)--(1,0.5);
\draw[dashed](0,0) circle (1);
\clip(0,0) circle (1);
\draw[thick] (-1,0)-- (0,0) node[below=-2ex]{$|$} node[above] {$b$} -- (1,0);
\end{tikzpicture}\arrow[r,"\RM{B_2}"] \&
\begin{tikzpicture}[baseline=-.5ex,scale=0.7]
\draw[dashed](0,0) circle (1);
\clip(0,0) circle (1);
\draw[thick] (-1,0)-- node[near start, below=-2ex]{$|$}node[near start, above]{$b'$} node[near end, below=-2ex]{$|$} node[near end,above]{$b''$} (1,0);
\end{tikzpicture}\&
\begin{tikzpicture}[baseline=-.5ex,scale=0.7]
\useasboundingbox(-1,-0.5)--(1,0.5);
\draw[dashed](0,0) circle (1);
\clip(0,0) circle (1);
\draw[thick] (-1,-0.5) to[out=0,in=180] (0,0) to[out=0,in=180] (1,0.5);
\draw[thick] (-1,0) to[out=0,in=180] (0,0) to[out=0,in=180] (1,0);
\draw[thick] (-1,0.5) to[out=0,in=180] (0,0) to[out=0,in=180] (1,-0.5);
\draw[fill] (0,0) circle (2pt) node[below] {$v$};
\end{tikzpicture}\arrow[r,"\RM{B_3}"] \&
\begin{tikzpicture}[baseline=-.5ex,scale=0.7]
\draw[dashed](0,0) circle (1);
\clip(0,0) circle (1);
\draw[thick] (-1,-0.5) to[out=0,in=180] (0,0) to[out=0,in=180] node[midway,sloped, below=-2ex] {$|$} node[midway,above left] {$b$} (1,0.5);
\draw[thick] (-1,0) to[out=0,in=180] (0,0) to[out=0,in=180] (1,0);
\draw[thick] (-1,0.5) to[out=0,in=180] (0,0) to[out=0,in=180] (1,-0.5);
\draw[fill] (0,0) circle (2pt) node[below] {$v$};
\end{tikzpicture}
\end{tikzcd}$}
\subfigure[Basepoint moves in Lagrangian projection]{$
\begin{tikzcd}[ampersand replacement=\&]
\begin{tikzpicture}[baseline=-.5ex,scale=0.7]
\useasboundingbox(-1,-0.5)--(1,0.5);
\draw[dashed](0,0) circle (1);
\clip(0,0) circle (1);
\draw[thick] (-1,-0.5) -- (1,0.5);
\draw[white,line width=5] (-1,0.5)--(1,-0.5);
\draw[thick] (-1,0.5)--node[midway,sloped,below=-2ex]{$|$} node[midway,above] {$b$} (0,0) -- (1,-0.5);
\end{tikzpicture}\arrow[r,"\RM{b_1}"] \&
\begin{tikzpicture}[baseline=-.5ex,scale=0.7]
\draw[dashed](0,0) circle (1);
\clip(0,0) circle (1);
\draw[thick] (-1,-0.5) -- (1,0.5);
\draw[white,line width=5] (-1,0.5)--(1,-0.5);
\draw[thick] (-1,0.5)-- (0,0) --node[midway,sloped,below=-2ex]{$|$} node[midway,below] {$b$} (1,-0.5);
\end{tikzpicture}\&
\begin{tikzpicture}[baseline=-.5ex,scale=0.7]
\useasboundingbox(-1,-0.5)--(1,0.5);
\draw[dashed](0,0) circle (1);
\clip(0,0) circle (1);
\draw[thick] (-1,0)-- (0,0) node[below=-2ex]{$|$} node[above] {$b$} -- (1,0);
\end{tikzpicture}\arrow[r,"\RM{b_2}"] \&
\begin{tikzpicture}[baseline=-.5ex,scale=0.7]
\draw[dashed](0,0) circle (1);
\clip(0,0) circle (1);
\draw[thick] (-1,0)-- node[near start, below=-2ex]{$|$}node[near start, above]{$b'$} node[near end, below=-2ex]{$|$} node[near end,above]{$b''$} (1,0);
\end{tikzpicture}\&
\begin{tikzpicture}[baseline=-.5ex,scale=0.7]
\useasboundingbox(-1,-0.5)--(1,0.5);
\draw[dashed](0,0) circle (1);
\clip(0,0) circle (1);
\draw[thick] (-1,-0.5) -- (1,0.5);
\draw[thick] (-1,0)-- (1,0);
\draw[thick] (-1,0.5) -- (1,-0.5);
\draw[fill] (0,0) circle (2pt) node[below] {$v$};
\end{tikzpicture}\arrow[r,"\RM{b_3}"] \&
\begin{tikzpicture}[baseline=-.5ex,scale=0.7]
\draw[dashed](0,0) circle (1);
\clip(0,0) circle (1);
\draw[thick] (-1,-0.5) -- (0,0) -- node[midway,sloped, below=-2ex] {$|$} node[midway,above left] {$b$} (1,0.5);
\draw[thick] (-1,0)-- (1,0);
\draw[thick] (-1,0.5) -- (1,-0.5);
\draw[fill] (0,0) circle (2pt) node[below] {$v$};
\end{tikzpicture}
\end{tikzcd}$}
\caption{Operations on basepoints}
\label{figure:basepoints}
\end{figure}
\begin{remark}
Note that the operations $\RM{B_1}$ and $\RM{b_1}$ that move a basepoint through a crossing below or above can be realized as sequences of front and Lagrangian Reidemeister moves, respectively.
\end{remark}

\begin{definition}[Categories of bordered Legendrian graphs]
We regard $\BLT$ and $\BLT_\lag$ of regular front and Lagrangian projections of isomorphic classes of bordered Legendrian graphs as categories, whose morphisms are \emph{freely generated by} front and Lagrangian Reidemeister moves.
\end{definition}

Therefore, projections \emph{up to zig-zags of morphisms} are the same as those \emph{up to Reidemeister moves}.
Or equivalently, the isomorphism classes in the localized category of $\BLT$ or $\BLT_\lag$ by all Reidemeister moves are the same as the usual equivalent classes of bordered Legendrian graphs.

\begin{example/definition}\label{definition:normal form}
A bordered Legendrian graph $\cT\in\BLT$ is said to be in a \emph{normal form} if 
\begin{enumerate}
\item every vertex is of type $(\val(v),0)$ and on the right above position,
\item every non-vertex $x$-maximum is a basepoint and \textit{vice versa} so that each $x$-maximum or a basepoint looks as follows:
\[
\begin{tikzpicture}[baseline=-.5ex,scale=0.7]
\draw[dashed](0,0) circle (1);
\clip(0,0) circle (1);
\draw[thick] (-1,0.5) to[out=0,in=180] (0.5,0) node {$|$} to[out=180,in=0] (-1,-0.5);
\end{tikzpicture}
\]
\end{enumerate}

Namely, all vertices are lying near $J^1 U_\Right$, moreover, have the larger $z$-coordinate than any point in the right border $T_\Right$ as depicted in Figure~\ref{figure:normal form}
\end{example/definition}

\begin{figure}[ht]
\[
\cT=
\begin{tikzpicture}[baseline=-.5ex,scale=0.8]
\draw[red] (-2,-1)--(-2,1) (2,-1)--(2,1);
\draw (-1.5,-1) rectangle (1,1);
\draw[thick] (-2,0.75) -- +(0.5,0);
\draw[thick] (-2,0.25) -- +(0.5,0) to[out=0,in=180] (1, -0.9);
\draw[thick] (-2,-0.25) -- +(0.5,0);
\draw[thick] (-2,-0.75) -- +(0.5,0) to[out=0,in=180] (1,0.5);
\draw[thick] (-1.5,0.75) to[out=0,in=180] (-0.5, 0.45) node{$|$} to[out=180,in=0] (-1.5,-0.25);
\draw[thick] (2,0) -- +(-1,0);
\draw[thick] (2,-0.3) -- +(-1,0);
\draw[thick] (2,-0.6) -- +(-1,0);
\draw[thick] (2,-0.9) -- +(-1,0);
\draw[fill] (1.75,0.6) circle (2pt);
\draw[thick] (1.75,0.6) to[out=180,in=0] +(-0.75,-0.1);
\draw[thick] (1.75,0.6) to[out=180,in=0] +(-0.75,-0.3) to[out=180,in=0] (0.5, 0) to[out=0,in=180] (1,-0.3);
\draw[thick] (1.75,0.6) to[out=180,in=0] +(-0.75,0.1) to[out=180,in=0] (-1.25, -0.4) to[out=0,in=180] (-0.25, -0.9) to[out=0,in=180] (1,-0.6);
\draw[thick] (1.75,0.6) to[out=180,in=0] +(-0.75,0.3) to[out=180,in=0] (-0.25,0.6) to[out=0,in=180] (0.5, -0.3) to[out=0,in=180] (1, 0);
\end{tikzpicture}\in\BLT
\]
\caption{A bordered Legendrian graph in a normal form}
\label{figure:normal form}
\end{figure}
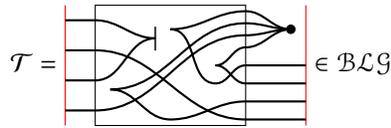

\begin{lemma}\label{lemma:normal form representative}
Let $\cT\in\BLT$ be a bordered Legendrian graph.
Then there exists a sequence of front Reidemeister moves consisting of $\RM{V},\RM{VI}$ together with $\RM{B_*}$'s which transforms a (non-unique) bordered Legendrian graph in a normal form to $\cT$ .
\end{lemma}
\begin{proof}
We first use $\RM{B_*}$ to make each $x$-maximum to be a basepoint and \textit{vice versa}.

For each vertex $v\in V$ of type $(\ell,r)$ with $r>0$, we apply $\RM{VI}$ several times to make $v$ of type $(\ell+r,0)$ in the reverse direction and we move a small neighborhood $U_v$ of each vertex $v$ to the right upward position by applying only front Reidemeister moves $\RM{V}$ in the reverse direction and we are done.
\end{proof}

\subsubsection{Maslov potentials for bordered Legendrian graphs}\label{section:Maslov potentials}
Let $\cT\in\BLT$ and $S=S(T)\subset\mathring{T}$ be the set of $x$-extreme points in its interior.
An \emph{$\grading$-valued Maslov potential} of $T$ is a function $\mu:R\to\grading$ from the set $R\coloneqq\pi_0\left(T\setminus(V\cup S)\right)$ of connected components of the complement of vertices and cusps such that for all $s\in S\setminus V$,
\begin{align}\label{equation:defining equation of Maslov potentials}
\mu(s^+)-\mu(s^-)&=1\in \grading,
\end{align}
where $s^+$ (resp. $s^-$) is the upper (resp. lower) strand near $s$.

For $T_\lag\in\BLT_\lag$, let $S_\lag=S(T_\lag)\subset\mathring{T}_\lag$ be the set of $x$-extreme points.
As before we define the set $R_\lag\coloneqq \pi_0(T_\lag\setminus(V_\lag\cup S_\lag))$ of connected components of the complement of vertices and $x$-extreme points. Then an $\grading$-valued Maslov potential of $T_\lag$ is a function $\mu:R_\lag\to\grading$ satisfying the following condition: for each $s\in S_\lag\setminus V_\lag$,
\begin{align}\label{equation:defining equation of Maslov potentials of Lagrangian projection}
\mu(s^+)-\mu(s^-) = \begin{cases}
1 & s\text{ is $x$-minimal};\\
-1 & s\text{ is $x$-maximal}.
\end{cases}
\end{align}

Diagrammatically, the above definition is depicted in Figures~\ref{figure:defining equation of Maslov potentials for front} and \ref{figure:defining equation of Maslov potentials for Lagrangian}.

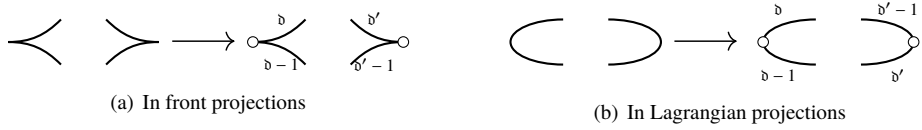
\begin{figure}[ht]
\subfigure[In front projections\label{figure:defining equation of Maslov potentials for front}]{\makebox[0.45\textwidth]{
\begin{tikzcd}[ampersand replacement=\&]
\begin{tikzpicture}[baseline=-.5ex]
\begin{scope}[xshift=-1cm]
\draw[thick] (.7,.3) to[out=225,in=0] (0,0) to[out=0,in=135] (.7,-.3);
\end{scope}
\begin{scope}[xshift=1cm]
\draw[thick] (-.7,.3) to[out=-45,in=180] (0,0) to[out=180,in=45] (-.7,-.3);
\end{scope}
\end{tikzpicture}\arrow[r] \& 
\begin{tikzpicture}[baseline=-.5ex]
\begin{scope}[xshift=-1cm]
\draw[thick] (.7,.3) to[out=225,in=0] node[midway,above] {\tiny$\fd$} (0,0) to[out=0,in=135] node[midway,below] {\tiny$\fd-1$} (.7,-.3);
\draw[fill=white] (0,0) circle (2pt);
\end{scope}
\begin{scope}[xshift=1cm]
\draw[thick] (-.7,.3) to[out=-45,in=180] node[midway,above] {\tiny$\fd'$}  (0,0) to[out=180,in=45] node[midway,below] {\tiny$\fd'-1$} (-.7,-.3);
\draw[fill=white] (0,0) circle (2pt);
\end{scope}
\end{tikzpicture}
\end{tikzcd}
}}
\subfigure[In Lagrangian projections\label{figure:defining equation of Maslov potentials for Lagrangian}]{\makebox[0.45\textwidth]{
\begin{tikzcd}[ampersand replacement=\&]
\begin{tikzpicture}[baseline=-.5ex]
\begin{scope}[xshift=-1cm]
\draw[thick] (.7,.3) arc (90:270:.7 and 0.3);
\end{scope}
\begin{scope}[xshift=1cm]
\draw[thick] (-.7,.3) arc (90:-90:.7 and 0.3);
\end{scope}
\end{tikzpicture}\arrow[r] \& 
\begin{tikzpicture}[baseline=-.5ex]
\begin{scope}[xshift=-1cm]
\draw[thick] (.7,.3) arc (90:180:.7 and 0.3) node[midway,above] {\tiny$\fd$} arc (180:270:.7 and 0.3) node[midway,below] {\tiny$\fd-1$};
\draw[fill=white] (0,0) circle (2pt);
\end{scope}
\begin{scope}[xshift=1cm]
\draw[thick] (-.7,.3) arc (90:0:.7 and 0.3) node[midway,above] {\tiny$\fd'-1$} arc (0:-90:.7 and 0.3) node[midway,below] {\tiny$\fd'$};
\draw[fill=white] (0,0) circle (2pt);
\end{scope}
\end{tikzpicture}
\end{tikzcd}
}}
\caption{Defining diagrams for Maslov potentials}
\end{figure}

Moreover, one can define $\mu_\Left\coloneqq \mu|_{T_\Left}:[n_\Left]\to\grading$ and $\mu_\Right\coloneqq \mu|_{T_\Right}:[n_\Right]\to\grading$ as the restrictions of $\mu$ to the connected components containing $T_\Left$ and $T_\Right$, respectively, under the canonical identifications $T_\Left\isomorphic [n_\Left]$ and $\cT_\Right\isomorphic [n_\Right]$.

\begin{definition}[Maslov potentials for bordered Legendrian graphs]
A Maslov potential for a bordered Legendrian graph $\cT$ is a triple $(\mu_\Left,\mu,\mu_\Right)$ denoted simply by $\bfmu$.

We denote Legendrian graphs with Maslov potentials by using the superscript `$\mu$' such as $\BLT^\mu$.
\end{definition}

\begin{example}\label{example:Maslov potentials for trivial graphs}
Recall the bordered Legendrian graphs $\cT_n$, $\bfzero_n$ and $\bfinfty_n$ defined in Example/Definition~\ref{example/definition:trivial graph}.
Since they have no $x$-extreme points except for a vertex, there are no conditions for Maslov potentials. 
That is, any function $\mu:[n]=\{1,\dots,n\}\to\grading$ can be realized as a Maslov potential for $\cT_n$, $\bfzero_n$ or $\bfinfty_n$.
\end{example}

Then each Lagrangian Reidemeister move induces a unique isotopy between bordered Legendrian graphs with Maslov potentials.
\begin{lemma}\label{lemma:Reidemeister move preserves potential}
Let $\RM{M}:\cT'\to\cT$ and $\RM{m}:\cT'_\lag\to \cT_\lag$ be front and Lagrangian Reidemeister moves.
Then they lift uniquely to $\RM{M}:(\cT',\bfmu')\to(\cT,\bfmu)$ and $\RM{m}:(\cT'_\lag,\bfmu')\to (\cT_\lag,\bfmu)$. 
Namely, for given $\bfmu'$, there is a unique Maslov postential $\bfmu$ on each $\cT$ or $\cT_\lag$ such that 
\[
\RM{M}_{*}\bfmu'=\bfmu\quad\text{ and }\quad
\RM{m}_{*}\bfmu'=\bfmu.
\]
\end{lemma}
\begin{proof}
This is an extension of Theorem~2.21 in \cite{AB2018} for Lagrangian Reidemeister moves.
The proof is straight forward and we omit the proof.
\end{proof}

\begin{definition}[Restrictions of Maslov potentials on vertices]
For each vertex $v$ of type $(\ell,r)$ with $\ell+r=n$, the set $H_v$ of half-edges can be identified with $\Zmod{n}$ by Definition~\ref{definition:types of vertices} and we denote the restriction $\mu|_{H_v}:\Zmod{n}\to\grading$ of a Maslov potential to $H_v$ by $\mu_v$.
\end{definition}

\subsubsection{Concatenations of bordered Legendrian graphs}
For $i=1,2$, let $\cT^i\in\BLT$ be a bordered Legendrian graph of type $\left(n^i_\Left, n^i_\Right\right)$.
Suppose that $n^1_\Right=n^2_\Left$.
Then we can naturally define the \emph{concatenation} $\cT=\cT^1\cdot \cT^2=(T_\Left\to T\leftarrow T_\Right)$ such that $T$ is obtained simply by concatenating and identifying $T^1_\Right$ and $T^2_\Left$ up to small isotopy near borders if necessary, and two borders $T_\Left\coloneqq T^1_\Left$ and $T_\Right\coloneqq T^2_\Right$ come naturally from $\cT^1$ and $\cT^2$, respectively.

\begin{remark}
It is important to note that we will not regard the points of concatenations as vertices of $T$.
Therefore $T$ has $n$-less edges than the disjoint union of $T^1$ and $T^2$.
\end{remark}

\begin{definition}[Closure]
For $\cT\in\BLT$ of type $(n_\Left,n_\Right)$, the \emph{closure} $\hat{\cT}$ is defined by the Legendrian graph obtained by the concatenation
\[
\hat {\cT}\coloneqq \bfzero_{n_\Left}\cdot \cT \cdot \bfinfty_{n_\Right}\in\LT
\]
as depicted in Figure~\ref{figure:closure}.
\end{definition}

\begin{figure}[ht]
\[
\begin{tikzcd}
\cT=\begin{tikzpicture}[baseline=-.5ex]
\foreach \i in {0.75, 0.25, -0.25, -0.75} {
\draw[thick, rounded corners] (-1,\i) -- (-0.75,\i) (0.25,\i) -- (0.5,\i);
}
\draw(-0.75,-1) rectangle (0.25,1);
\draw(-0.25,0) node {$T$};
\draw[red](-1,-1)--(-1,1);
\draw[red](0.5,-1)--(0.5,1);
\end{tikzpicture}\arrow[r,"\hat\cdot",mapsto]&
\begin{tikzpicture}[baseline=-.5ex]
\draw[fill](0.25,0) circle (0.05) node[above] {$0$};
\foreach \i in {0.75, 0.25, -0.25, -0.75} {
\draw[thick, rounded corners](0.25,0) to[out=0,in=180] (1,\i);
}
\draw[red](1,-1)--(1,1);
\end{tikzpicture}
\cdot
\begin{tikzpicture}[baseline=-.5ex]
\foreach \i in {0.75, 0.25, -0.25, -0.75} {
\draw[thick, rounded corners] (-1,\i) -- (-0.75,\i) (0.25,\i) -- (0.5,\i);
}
\draw(-0.75,-1) rectangle (0.25,1);
\draw(-0.25,0) node {$T$};
\draw[red](-1,-1)--(-1,1);
\draw[red](0.5,-1)--(0.5,1);
\end{tikzpicture}
\cdot
\begin{tikzpicture}[baseline=-.5ex]
\draw[fill](-0.25,0) circle (0.05) node[above] {$\infty$};
\foreach \i in {0.75, 0.25, -0.25, -0.75} {
\draw[thick, rounded corners](-0.25,0) to[out=180,in=0] (-1,\i);
}
\draw[red](-1,-1)--(-1,1);
\end{tikzpicture}
\arrow[r,equal]&
\begin{tikzpicture}[baseline=-.5ex]
\begin{scope}[xshift=-2cm]
\draw[fill](0.5,0) circle (0.05) node[above] {$0$};
\foreach \i in {0.75, 0.25, -0.25, -0.75} {
\draw[thick, rounded corners](0.5,0) to[out=0,in=180] (1.25,\i);
}
\end{scope}
\draw(-0.75,-1) rectangle (0.25,1);
\draw(-0.25,0) node {$T$};
\begin{scope}[xshift=1.5cm]
\draw[fill](-0.5,0) circle (0.05) node[above] {$\infty$};
\foreach \i in {0.75, 0.25, -0.25, -0.75} {
\draw[thick, rounded corners](-0.5,0) to[out=180,in=0] (-1.25,\i);
}
\end{scope}
\end{tikzpicture}=\hat \cT
\end{tikzcd}
\]
\caption{The closure of the front projection $\cT$}
\label{figure:closure}
\end{figure}
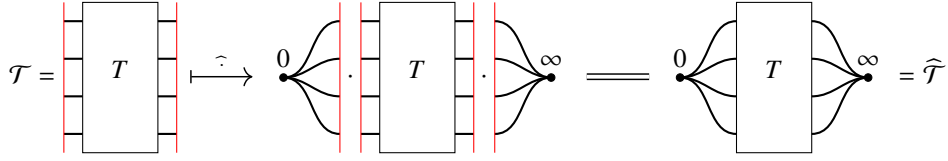

\begin{lemma}\label{lemma:potential extends to the closure}
The closure $\hat\cdot:\BLT^\mu\to \LT^\mu$ is a functor.
\end{lemma}
\begin{proof}
It is obvious that there is a unique way to extend $\bfmu$ on $\cT$ to $\hat \bfmu$ on $\hat \cT$ by definition of the closure, which is well-defined since any function on $[n]$ can be realized as a Maslov potential of $0_n$ or $\infty_n$ as seen in Example~\ref{example:Maslov potentials for trivial graphs}.

Each Reidemeister move $\RM{M}:\cT'\to\cT$ between bordered Legendrian graphs induces the exactly same move $\RM{M}:\hat\cT'\to\hat\cT$ between closures. Therefore it becomes a functor preserving homotopy.
\end{proof}

We introduce a combinatorial way, called the \emph{Ng's resolution} to obtain a regular Lagrangian projection $T_\lag\in\LT_\lag$ for given front projection $T\in\LT$ defining the equivalent Legendrian graphs.
\begin{remark}
This is an extension of the original Ng's resolution for Legendrian knots to Legendrian graphs.
\end{remark}

\begin{definition}[Ng's resolution]\cite[Definition~2.1]{Ng2003}\label{definition:Ng's resolution}
For $T\in\LT$, the \emph{Ng's resolution} $\Res(T)$ is a Lagrangian projection obtained by (combinatorially) replacing the local pieces as follows:
\begin{align*}
\Lcusp&\mapsto \leftarc,&
\Rcusp&\mapsto\rightkink,&
\crossing&\mapsto \crossingpositive,
\end{align*}
and for a vertex $v$ of type $(\ell,r)$ and a basepoint $b$, we take a replacement as follows:
\[
\begin{tikzcd}
\begin{tikzpicture}[baseline=-0.5ex,scale=0.8]
\foreach \i in {4,...,8} {
	\draw[thick] (-1,{(6-\i)/3}) to[out=0,in=180] (0,0);
}
\draw (-1,0.67) node[left] {$h_{v,1}$};
\draw (-1,0.33) node[left] {$h_{v,2}$};
\draw (-1,-0.167) node[left] {$\vdots$};
\draw (-1,-0.67) node[left] {$h_{v,\ell}$};
\foreach \i in {1,2,3} {
	\draw[thick] (1,{(2-\i)/1.5}) to[out=180,in=0] (0,0);
}
\draw (1,0.67) node[right] {$h_{v,\ell+1}$};
\draw (1,0) node[right] {$\vdots$};
\draw (1,-0.67) node[right] {$h_{v,\ell+r}$};
\draw[fill] (0,0) circle (2pt) node[above] {$v$};
\end{tikzpicture}\arrow[r,mapsto] &
\begin{tikzpicture}[baseline=-0.5ex,scale=0.8]
\draw[thick,rounded corners] (-2, -2/3 ) -- (-1.8, -2/3 ) -- (-0.6,2/3) -- (0,0);
\foreach \i in {7,6,5} {
	\draw[white,line width=3,rounded corners] (-2, {(6-\i)/3} ) -- (-1.8, {(6-\i)/3} ) -- ({-1.8+0.3*(8-\i)}, {-2/3} ) -- (-0.7,{(\i-6)/3});
	\draw[thick,rounded corners] (-2, {(6-\i)/3} ) -- (-1.8, {(6-\i)/3} ) -- ({-1.8+0.3*(8-\i)}, {-2/3} ) -- (-0.6,{(\i-6)/3}) --  (0,0);
}
\draw[white,line width=3,rounded corners] (-2, 2/3 ) -- (-1.8, 2/3 ) -- (-0.6,-2/3);
\draw[thick,rounded corners] (-2, 2/3 ) -- (-1.8, 2/3 ) -- (-0.6,-2/3) -- (0,0);
\draw (-2,0.67) node[left] {$h_{v,1}$};
\draw (-2,0.33) node[left] {$h_{v,2}$};
\draw (-2,-0.167) node[left] {$\vdots$};
\draw (-2,-0.67) node[left] {$h_{v,\ell}$};
\foreach \i in {1,2,3} {
	\draw[thick] (1,{(2-\i)/1.5}) -- (0,0);
}
\draw (1,0.67) node[right] {$h_{v,\ell+1}$};
\draw (1,0) node[right] {$\vdots$};
\draw (1,-0.67) node[right] {$h_{v,\ell+r}$};
\draw[fill] (0,0) circle (2pt) node[above] {$v$};
\end{tikzpicture}
\end{tikzcd}
\]
\[
\begin{tikzcd}
\begin{tikzpicture}[baseline=-0.5ex,scale=0.7]
\draw[thick] (-1,0) -- (0,0) node[below=-2ex] {$|$} node[above] {$b$} -- (1,0);
\end{tikzpicture}\arrow[r,mapsto] &
\begin{tikzpicture}[baseline=-0.5ex,scale=0.7]
\draw[thick] (-1,0) -- (0,0) node[below=-2ex] {$|$} node[above] {$b$} -- (1,0);
\end{tikzpicture}&
\begin{tikzpicture}[baseline=-0.5ex,scale=0.7]
\draw[thick] (1,0.5) to[out=180,in=0] (-0.5,0) node[below=-2ex] {$|$} node[above] {$b$} to[out=0,in=180] (1,-0.5);
\end{tikzpicture}\arrow[r,mapsto] &
\begin{tikzpicture}[baseline=-0.5ex,scale=0.7]
\draw[thick] (1,0.5) arc(90:180:1 and 0.5) arc(180:270:1 and 0.5) node[near start,below=-2ex,sloped] {$|$} node[near start, below] {$b$};
\end{tikzpicture}&
\begin{tikzpicture}[baseline=-0.5ex, scale=0.7]
\draw[thick] (-1,0.5) to[out=0,in=180] (0.5,0) node[below=-2ex] {$|$} node[above] {$b$} to[out=180,in=0] (-1,-0.5);
\end{tikzpicture}\arrow[r,mapsto] &
\begin{tikzpicture}[baseline=-0.5ex,scale=0.7]
\draw[thick] (-2, -0.5) to[out=0,in=180] (-1,0.5) arc(90:0:1 and 0.5);
\draw[white,line width=5] (0,0) arc(0:-90:1 and 0.5) to[out=180,in=0] (-2,0.5);
\draw[thick] (0,0) arc(0:-90:1 and 0.5) node[near start,below=-2ex,sloped] {$|$} node[near start, below] {$b$} to[out=180,in=0] (-2,0.5);
\end{tikzpicture}
\end{tikzcd}
\]
\end{definition}

Notice that if we have two equivalent front projections, then the Ng's resolution induces equivalences between resolutions.
Indeed, for each front Reidemeister move $\RM{M}$, we have a sequence of Lagrangian Reidemeister move(s) $\Res\RM{M}$.

One candidate for the choice of $\Res\RM{M}$ for each $\RM{M}\in\cRM$ is given in \cite[Figure~6]{ABS2019count}.

\begin{lemma}\label{lemma:functoriality of Ng's resolution}
The Ng's resolution $\Res:\LT^\mu\to\LT_\lag^\mu$ is a functor.
\end{lemma}
\begin{proof}
The well-definedness is well-known and as observed above, each morphism corresponding to front Reidemeister move $\RM{M}$ will be mapped to the morphism corresponding to the chosen sequence $\Res\RM{M}$ of Lagrangian Reidemeister moves.
Therefore it is functorial.
\end{proof}

\subsection{Bordered Chekanov-Eliashberg DGAs}\label{section:category of bordered DGAs}
Throughout this paper, we mean by a \emph{differential graded algebra} (DGA) a pair $A=(\alg,\differential)$ of a unital associative graded algebra $\alg$ over a field $\ring$ with the unit $\unit\in\ring$.

A DGA $A=(\alg,\differential)$ is said to be \emph{semi-free} and generated by $\sfS$ if its underlying algebra $\alg$ is a tensor algebra of a graded vector space $\ring\langle \sfS \rangle$
\begin{align*}
\alg & = T(\ring\langle \sfS \rangle)=\bigoplus_{\ell\ge0}(\ring\langle \sfS \rangle)^{\otimes\ell}, 
\end{align*}
for a (possibly infinite) \emph{generating set} $\sfS$ with a grading $|\cdot|:\sfS\to\ZZ$.
\begin{remark}
This is the usual notion of the semi-freeness while it is used differently in \cite{NRSSZ2015}.
\end{remark}

We denote the category of semi-free DGAs by $\DGA$.
\begin{assumption}
From now on, we mean by a DGA a semi-free DGA unless mentioned otherwise.
\end{assumption}

\begin{example/definition}[Border DGAs and internal DGAs]\label{example:composable border DGAs}
There are two important examples of DGAs, called the \emph{border DGA} $A_{n}(\mu)=(\alg_{n},\differential_n)$ and the \emph{internal DGA} $I_{n}(\mu)=(\sfI_{n},\differential_n)$ which are edge-graded over $E_n\coloneqq[n]$ and defined as follows: 
\begin{itemize}
\item For a function $\mu:[n]\to\grading$, the algebras $\alg_{n}$ and $\sfI_{n}$ are generated by two sets $K_n$ and $\tilde V_n$
\begin{align*}
K_n&\coloneqq \{k_{ab}\mid 1\le a<b\le n\},&
\tilde V_n&\coloneqq \{ \vg_{a,i}\mid a\in\Zmod{n},i\ge1 \}.
\end{align*}
\item The edge-gradings on generator $k_{ab}$ and $\vg_{a,i}$ are given as $(a,b)$ and $(a,a+i)$, where all indices are modulo $n$.
\item The homological gradings are 
\begin{align*}
|k_{ab}|&\coloneqq \mu(a)-\mu(b)-1,&
|\vg_{a,i}|&\coloneqq 
\mu(a)-\mu(a+i)+N(n,a,i)-1,
\end{align*}
where
\begin{equation}\label{equation:N}
N(n,a,i)=\sum_{j<i} N(n,a+j,1)\quad\text{ and }\quad
N(n,a,1)\coloneqq \begin{cases}
0 & a\neq n;\\
2 & a=n;
\end{cases}
\end{equation}
\item The differentials for $k_{ab}$ and $\vg_{a,i}$ are defined as follows:
\begin{align}
\differential_n(k_{ab})&\coloneqq\sum_{a<c<b}(-1)^{|k_{ac}|-1}k_{ac}k_{cb},\\
\differential_n (\vg_{a,i})&\coloneqq \delta_{i,n}\unit_a+\sum_{i_1+i_2=i}(-1)^{|\vg_{a,i_1}|-1}\vg_{a,i_1}\vg_{a+i_1,i_2}.\label{equation:differential for vertices}
\end{align}
\end{itemize}
\end{example/definition}

Then it is obvious that there is a canonical inclusion $k_{ab}\mapsto \vg_{a,b-a}$ from $A_{n}(\mu)\to I_{n}(\mu)$.
\begin{remark}
We have $A_0(\mu)\isomorphic I_0(\mu)\isomorphic \ring$.
\end{remark}

Especially, when $n=2$ and $\mu$ is constant, say $m$, then the internal DGA $I_{2}(m)=(\sfI_{2}(m),\differential_2)$ is given as
\begin{equation}
\sfI_{2}(m)=\ring\langle \xi_{1,i}, \xi_{2,i}\mid i\ge 1\rangle,
\end{equation}
where $|\xi_{1,i}|=i-1$. Hence, it is independent to the constant $m$ and denoted by $I_{2}$ and moreover, we have the DGA morphism
\begin{align}
I_{2}&\to (\ring[t,t^{-1}],\differential\equiv0)&
|t|&=|t^{-1}|=0,&
b_{a,\ell}&\mapsto\begin{cases}
t & a=1, \ell=1;\\
t^{-1} & a=2, \ell=1;\\
0 & \ell>1.
\end{cases}
\end{align}

\begin{definition}[Tame isomorphisms]
We say that a DGA morphism 
\[
f:A'=(\alg'=T(\ring\langle \sfS\rangle),\differential')\to A=(\alg=T(\ring\langle\sfS\rangle),\differential)
\]
is called an \emph{elementary isomorphism} if for some $g\in\sfR_{ee'}$,
\begin{align*}
f(h)=
\begin{cases}
g+u&\text{ if }h=g\\
h&\text{ if }h\neq g,
\end{cases}
\end{align*}
where $u\in \alg(e,e')$ is a composable word of type $(e,e')$ not containing $g$, and called a \emph{tame isomorphism} if $f$ is a composition of countably many (possibly finite) elementary isomorphisms.
\end{definition}

We consider a \emph{stabilization} in the sense of \cite[Definition~2.16]{EN2015} and a \emph{generalized stabilization} defined in \cite{AB2018} as follows: 
\begin{definition}[Stabilizations]
A \emph{stabilization} of $A=(\alg=T(\ring\langle \sfS\rangle),\differential)\in\DGA$ is a DGA which is tame isomorphic to a DGA $SA=(S\alg,\bar\differential)$ obtained from $A$ by adding a countably many (possibly finite) number of canceling pairs of edge-graded generators $\{\hat e^i, e^i\mid i\in I\}$ for some index set $I$ so that $\hat e^i$ and $e^i$ have the same edge-grading and 
\begin{align*}
S\alg&=T\left(\ring\left\langle\sfS\amalg\{\hat e^i, e^i\mid i\in I\}\right\rangle\right),&
|\hat e^i| &=|e^i| +1,&
\bar\differential(\hat e^i)&= e^i,&
\bar\differential(e^i)&= 0.
\end{align*}
Then the canonical forgetful map $\pi:SA\to A$ sending both $\hat e^i$ and $e^i$ to zero induces the isomorphism on homology groups, whose homotopy inverse is precisely the canonical embedding $\iota:A\to SA$.
\end{definition}

\begin{definition}[Generalized stabilizations]
For $A=(\alg=T(\ring\langle \sfS\rangle), \differential)\in\DGA$ and $\varphi: I_{n}(\mu)\to A$ for some $\mu:\Zmod{n}\to\grading$, the \emph{$\fd$-th positive} or \emph{negative stabilization} of $A$ with respect to $\varphi$ is the DGA $S_\varphi^{\fd \pm}A=\left(S^{\fd}_\varphi\alg,\bar\differential^\pm\right)$, where for $v_{a,i}\coloneqq\varphi(\vg_{a,i})\in A$,
\begin{itemize}
\item the graded algebra $S^{\fd}_\varphi \sfA$ is given as
\begin{align*}
S^{\fd}_\varphi \sfA&\coloneqq T\left(\ring\left\langle\sfS\amalg\{e^1,\dots,e^n\}\right\rangle\right),&
|e^b| &\coloneqq\fd+\sum_{a<b}(|v_{a,1}|+1),
\end{align*}
\item the differentials $\bar\differential^\pm$ for $e^b$ are given as
\begin{align*}
\bar\differential^+ e^b&\coloneqq \sum_{a<b}(-1)^{|e^a|-1}e^a v_{a,b-a},&
\bar\differential^- e^b&\coloneqq \sum_{a<b}v_{n+1-b,b-a} e^a.
\end{align*}
\end{itemize}
\end{definition}

As observed in \cite[Remark~3.8]{AB2018} and proved in \cite[Appendix~A]{ABS2019count}, the generalized stabilization is a composition of a stabilization and a destabilization.
\begin{proposition}\cite[Appendix~A]{ABS2019count}\label{proposition:generalized stabilizations}
There exists a semi-free edge-graded DGA $\tilde S^{\fd\pm}_\varphi A$ which is a common stabilization of both $A$ and $S^{\fd\pm}_\varphi A$.
\end{proposition}

\subsubsection{Bordered DGAs}
We consider DGAs together with additional structures, called \emph{bordered DGAs}.
\begin{definition}[Bordered DGAs]
A \emph{bordered} DGA $\dga=\left(A_\Left\stackrel{\phi_\Left}\longrightarrow A\stackrel{\phi_\Right}\longleftarrow A_\Right\right)$ of type $\left(n_\Left, n_\Right\right)$ consists of DGAs $A, A_\Left$ and $A_\Right$, and two DGA morphisms $\phi_\Left:A_\Left\to A$ and $\phi_\Right:A_\Right\to A$ such that
\begin{enumerate}
\item $\phi_\Left$ is \emph{injective}, and
\item for some $\mu_\Left:[n_\Left]\to\grading$ and $\mu_\Right:[n_\Right]\to\grading$,
\[
A_\Left\isomorphic A_{n_\Left}(\mu_\Left)\quad\text{ and }\quad
A_\Right\isomorphic A_{n_\Right}(\mu_\Right).
\]
\end{enumerate}

A \emph{bordered morphism} $\bff:\dga'\to\dga$ is a triple $(f_\Left, f, f_\Right)$ of DGA morphisms such that they fit into the following commutative diagram
\begin{align*}
\begin{tikzcd}[column sep=4pc, row sep=2pc, ampersand replacement=\&]
\dga'\arrow[d,"\bff"']\\
\dga
\end{tikzcd}&=\ \left(
\begin{tikzcd}[column sep=4pc, row sep=2pc, ampersand replacement=\&]
A'_\Left\arrow[r,"\phi'_\Left"]\arrow[d,"f_\Left"']
\& A'\arrow[d,"f"'] \& A'_\Right\arrow[l,"\phi'_\Right"']\arrow[d,"f_\Right"]\\
A_\Left\arrow[r,"\phi_\Left"] \& A \& A_\Right\arrow[l,"\phi_\Right"'].
\end{tikzcd}\right)
\end{align*}

We say that $\dga$ is a \emph{cofibrant} if $\phi_\Right$ is also injective.
\end{definition}

\begin{definition}[Stabilizations]
We say that $\dga'=(A'_\Left\to A'\leftarrow A'_\Right)$ is a \emph{(weak) stabilization} of $\dga=(A_\Left\to A\leftarrow A_\Right)$ if $A'$ is a stabilization of $A$ and there exists a canonical projection
\[
\bfpi=(\identity_\Left, \pi, \identity_\Right):\dga'\to\dga,
\]
and is a \emph{strong stabilization} if it is a weak stabilization and there exists the canonical embedding $\bfi:\dga\to\dga'$ as well.
\end{definition}

\begin{definition}[Category of bordered DGAs]
The category of bordered DGAs will be denoted by $\BDGA$, and the full subcategory of $\BDGA$ consisting of cofibrants by $\BDGA^c$.
Then the category $\DGA$ is the full subcategory of $\BDGA^c$ consisting of bordered DGAs of type $(0,0)$.
\end{definition}

\begin{definition}[Cofibrant replacements of bordered DGAs]\label{definition:cofibrant replacement}
Let $\dga=\left(A_\Left\stackrel{\phi_\Left}\longrightarrow A\stackrel{\phi_\Right}\longleftarrow A_\Right\right)\in\BDGA$ be a bordered DGA. The \emph{cofibrant replacement} of $\dga$
\[
\hat\dga\coloneqq\left(A_\Left\stackrel{\hat\phi_\Left}\longrightarrow \hat A \stackrel{\hat\phi_\Right}\longleftarrow A_\Right\right)\in\BDGA^c
\]
is defined by the \emph{mapping cylinder construction} as follows: if $A=(\alg=T(\ring\langle \sfS\rangle),\differential)$,
\begin{itemize}
\item the DGA $\hat A=(\hat\alg,\hat\differential)$ and its graded algebra $\hat \alg$ is defined as
\begin{align*}
\hat\alg&\coloneqq T\left(\ring\left\langle R\amalg K_\Right \amalg \hat K_\Right\right\rangle\right),&
\hat K_\Right&\coloneqq\left\{\hat k_{ab}~\middle|~k_{ab}\in K_\Right\right\},&
\left|\hat k_{ab}\right|&\coloneqq |k_{ab}|+1.
\end{align*}
\item The differential $\hat\differential$ for each $k_{ab}$ is the same as $\differential_\Right(k_{ab})$ and for $\hat k_{ab}$ it is defined as 
\begin{align}\label{equation:differential for mapping cylinder}
\hat\differential(\hat k_{ab})&\coloneqq 
k_{ab}-\phi_\Right(k_{ab})+\sum_{a<c<b}(-1)^{|\hat k_{ac}|-1}\hat k_{ac}\phi_\Right( k_{cb})+ k_{ac}\hat k_{cb}.
\end{align}
\item The morphism $\hat\phi_\Left$ is the composition of $\phi_\Left$ and the canonical inclusion $A\to\hat A$, and $\hat\phi_\Right$ is defined by $\hat\phi_\Right(k_{ab})=k_{ab}\in \hat\alg$.
\end{itemize}
\end{definition}

Let $\hat\pi=(\identity, \hat\pi,\identity):\hat\dga\to\dga$ be the bordered morphism which fixes both border DGAs and sends all $\hat k_{ab}$ to zero and each $k_{ab}$ to $\phi_\Right(k_{ab})$.

\begin{lemma}\label{lemma:mapping cylinders are stabilizations}
The bordered DGA $\hat\dga$ is a weak stabilization.
\end{lemma}
\begin{proof}
Notice that in each differential $\hat\differential(\hat k_{ab})$ there is one and only one generator $k_{ab}$.
Therefore one may find a tame automorphism $\Phi$ on $\hat A$ that sends $\hat\differential(\hat k_{ab})$ to $k_{ab}$ so that the DGA $(\hat \alg,\hat\differential_\Phi)$ with the twisted differential $\hat\differential_\Phi\coloneqq \Phi\circ\hat\differential\circ\Phi^{-1}$ becomes a stabilization of $A$.
\end{proof}

\begin{remark}
The cofibrant replacement $\hat\dga$ is \emph{not} a strong stabilization.
Indeed, the canonical inclusion $\bfi$ may involve the DGA homotopy.
However it still satisfies the good property, for example, induces an $A_\infty$-quasi-equivalence between augmentation categories. We will see this in Section~\ref{section:augmentation category}.
\end{remark}

\subsubsection{Bordered Chekanov-Eliashberg DGAs}\label{section:CE DGA}
Let $(\cT,\bfmu)\in\BLT^\mu$ be a front projection of a bordered Legendrian graph of type $(n_\Left,n_\Right)$ with a Maslov potential.
We consider the Ng's resolution $\hat\cT_\lag\coloneqq\Res(\hat\cT)$ of the closure of $\cT$, which has two distinguished vertices $0$ and $\infty$ of valency $n_\Left$ and $n_\Right$, respectively. See Figure~\ref{figure:resolution of the closure}.

\begin{figure}[ht]
\[
\Res(\hat \cT)=
\begin{tikzpicture}[baseline=-0.5ex,scale=0.8]
\begin{scope}
\foreach \t in {-2,...,2} {
\draw[thick,rounded corners] (-1,0) -- (-0.5,\t*0.4) -- (0,\t*0.4);
}
\draw (-0.5,0.8) node[left] {$1$};
\draw (-0.5,-0.8) node[left] {$n_\Left$};
\draw[fill](-1,0) circle (2pt) node[left] {$0$};
\end{scope}
\begin{scope}[xshift=3cm]
\draw[thick,rounded corners] (-1,-0.8)--(-0.8,-0.8) -- (0.8, 0.8) --(1.3,0);
\draw[white,line width=5,rounded corners] (-1,-0.4)--(-0.8,-0.4) -- (-.4,-0.8) -- (0.8, 0.4);
\draw[thick,rounded corners] (-1,-0.4)--(-0.8,-0.4) -- (-.4,-0.8) -- (0.8, 0.4)--(1.3,0);
\draw[white,line width=5,rounded corners] (-1,0)--(-0.8,0) -- (0, -0.8) -- (0.8, 0);
\draw[thick,rounded corners] (-1,0)--(-0.8,0) -- (0, -0.8) -- (0.8, 0)--(1.3,0);
\draw[white,line width=5,rounded corners] (-1,0.4)--(-0.8,0.4) -- (.4, -0.8) -- (0.8, -0.4);
\draw[thick,rounded corners] (-1,0.4)--(-0.8,0.4) -- (.4, -0.8) -- (0.8, -0.4) --(1.3,0);
\draw[white,line width=5,rounded corners] (-1,0.8)--(-0.8,0.8) -- (0.8, -0.8);
\draw[thick,rounded corners] (-1,0.8)--(-0.8,0.8) -- (0.8, -0.8)--(1.3,0);
\draw[fill] (1.3,0) circle (2pt) node[right] {$\infty$};
\draw[dashed,red,thick,rounded corners] (-0.9,-1) rectangle (0.8,1);
\draw (0.8,0.8) node[right] {$n_\Right$};
\draw (0.8,-0.8) node[right] {$1$};
\end{scope}
\draw (0,-1) rectangle (2,1);
\draw (1,0) node {$\Res(T)$};
\end{tikzpicture}
\]
\caption{The Ng's resolution of the closure of $\cT$}
\label{figure:resolution of the closure}
\end{figure}
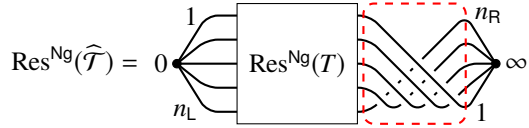

Then it is known that the Chekanov-Eliashberg DGA $A^\CE(\hat\cT_\lag,{\hat\bfmu})$ has the following generating sets:
\begin{enumerate}
\item Crossings $C(\hat\cT_\lag)$;
\item Vertex generators $\tilde V(\hat\cT_\lag)=\{v_{a,i}\mid v\in V(\hat\cT_\lag), a\in\Zmod{\val(v)},i\ge1\}$;
\item Basepoint generators $\tilde B(\hat \cT_\lag)=\{b_{a,\ell}\mid b\in B(\hat \cT_\lag), a=1,2, i\ge 1\}$.
\end{enumerate}
Especially, $C(\hat\cT_\lag)$ has the subset $\hat K_\Right$ of $\binom{n_\Right}2$ crossings coming from the right border by $\hat k_{ab}$ 
\[
\hat K_\Right\coloneqq\left\{\hat k_{ab}~\middle|~ 1\le a<b\le n_\Right\right\},
\]
which are in the dashed box in Figure~\ref{figure:resolution of the closure}.
Similarly, two distinguished vertices $0$ and $\infty$ have the subsets of vertex generators
\[
K_\Left\coloneqq\{k_{a'b'}=0_{a',b'-a'}\mid 1\le a'<b'\le n_\Left\}\quad\text{ and }\quad
K_\Right\coloneqq\{k_{ab}=\infty_{a,b-a}\mid 1\le a<b\le n_\Right\}.
\]

Then it is known that the differential for $\hat k_{ab}$ is given as
\[
\hat\differential(\hat k_{ab})=(-1)^{|\hat k_{ab}|-1}k_{ab}+g_{ab}+\sum_{a<c<b}(-1)^{|\hat k_{ab}|-1}\hat k_{ac} k_{cb}+g_{ac}\hat k_{cb},
\]
which is essentially the same as the equation \eqref{equation:differential for mapping cylinder} by replacing $k_{ab}$ and $g_{ab}$ with $(-1)^{|\hat k_{ab}|-1}\hat k_{ab}$ and $(-1)^{|\hat k_{ab}|}g_{ab}$, respectively.
In particular, the assignment $k_{ab}\mapsto (-1)^{|\hat k_{ab}|}g_{ab}$ defines a DGA morphism.

We define a bordered DGA 
\[
\hat A^\CE(\cT,\bfmu)=\left(
A^\CE(T_\Left,{\mu_\Left})\stackrel{\hat\phi_\Left}\longrightarrow \hat A^\CE(T,\mu)\stackrel{\hat\phi_\Right}\longleftarrow A^\CE(T_\Right,{\mu_\Right})
\right)
\]
associated to $(\cT,\bfmu)$ as follows: two border DGAs are as before and the DGA $\hat A^\CE(T,\mu)=(\hat \alg,\hat\differential)$ is the DG-subalgebra of $A^\CE(\hat\cT_\lag,{\hat\mu})$ generated by all crossings in $\hat\cT_\lag$, vertex generators for all but $0$ and $\infty$ and two sets $K_\Left$ and $K_\Right$.

The \emph{bordered Chekanov-Eliashberg DGA} (CE DGA) or \emph{Legendrian contact homology DGA} (LCH DGA) for $\cT^\mu$ is given as
\[
A^\CE(\cT,\bfmu)\coloneqq\left(A^\CE(T_\Left,{\mu_\Left})\stackrel{\phi_\Left}\longrightarrow A^\CE(T,\mu)\stackrel{\phi_\Right}\longleftarrow A^\CE(T_\Right,{\mu_\Right})\right),
\]
where 
\begin{enumerate}
\item two DGAs $A^\CE(T_\Left,{\mu_\Left})$ and $A^\CE(T_\Right,{\mu_\Right})$  are isomorphic to border DGAs $A_{n_\Left}(\mu_\Left)$ and $A_{n_\Right}(\mu_\Right)$, respectively, whose generating sets $K_\Left$ and $K_\Right$ are identified with
\begin{align*}
K_\Left&\coloneqq\{k_{a'b'}=0_{a',b'-a'}\mid 1\le a'<b'\le n_\Left\},&
K_\Right&\coloneqq\{k_{ab}=\infty_{a,b-a}\mid 1\le a<b\le n_\Right\},
\end{align*}
\item the DGA $A^\CE(T,\mu)$ is the DG-subalgebra generated by crossings, vertex and basepoint generators contained only in $\Res(T)$ together with $K_\Left$, and
\item the DGA morphism $\phi_\Left$ is an obvious inclusion while $\phi_\Right$ is defined as $k_{ab}\mapsto (-1)^{|\hat k_{ab}|}g_{ab}$ as above.
\end{enumerate}

\begin{theorem}\cite[Theorem~3.2.11 and Proposition~3.2.13]{ABS2019count}\label{theorem:zig-zags of stabilizations}
Let $(\cT,\bfmu)\in\BLT^\mu$. The bordered DGA $A^\CE(\cT,\bfmu)$ is well-defined and invariant under Legendrian isotopy up to stabilizations.

Indeed, for each front Reidemeister move $\RM{M}:(\cT',{\bfmu'})\to(\cT,\bfmu)$, there exists a zig-zag
\[
\begin{tikzcd}[column sep=2pc]
A^\CE(\cT',{\bfmu'})\arrow[from=r,"\hat\bfpi'"',->>]& \hat A^\CE(\cT,\bfmu)=\dga_0
\arrow[from=r, "\bfpi_1",->>,yshift=-0.7ex]
\arrow[r,"\bfi_1",yshift=0.7ex,hook]
& \cdots 
\arrow[r, "\bfpi'_{n-1}"',->>,yshift=-0.7ex]
\arrow[from=r,"\bfi_{n-1}'"',yshift=0.7ex,hook']& \dga_n=\hat A^\CE(\cT,{\bfmu}) \arrow[r,"\hat\bfpi'",->>]& A^\CE(\cT',{\bfmu'})
\end{tikzcd}
\]
of stabilizations of bordered DGAs between $A^\CE(\cT',{\bfmu'})$ and $A^\CE(\cT\bfmu)$,
where the bordered DGAs $\hat A^\CE(\cT',{\bfmu'})$ and $\hat A^\CE(\cT,\bfmu)$ are well-defined and the cofibrant replacements of $A^\CE(\cT',{\bfmu'})$ and $A^\CE(\cT,\bfmu)$, respectively, and all but the left- and rightmost are strong stabilizations.
\end{theorem}

We explain briefly the reason why the DG-subalgebra $A^\CE(\Res(\hat\cT,{\hat\bfmu}))$ is well-defined.
Since two vertex $0$ and $\infty$ face the unbounded region in $\RR^2$ especially between two half-edges $h_{0,n_\Left}$ and $h_{0,1}$, and $h_{\infty,n_\Right}$ and $h_{\infty,1}$, there are no crossing generators involving vertex generators at $0$ or $\infty$ which pass these unbounded regions in their differentials.

In other words, the bordered DGA $\hat A^\CE(\cT,\bfmu)$ is the same as the LCH DGA of the Legendrian graph $(\hat\cT,{\hat\bfmu})$ together with two DG-subalgebras $A^\CE(T_\Left,{\mu_\Left})$ and $A^\CE(T_\Right,{\mu_\Right})$ of the distinguished internal DG-subalgebras $I_0$ and $I_\infty$. 
This data is obviously invariant under any front Reidemeister moves fixing two vertices $0$ and $\infty$.

\begin{remark}
The exactly same statement holds for Lagrangian projections. That is,
for each $\RM{m}:(\cT'_\lag,\bfmu')\to(\cT_\lag,\bfmu)$, there is a zig-zag of stabilizations between two bordered LCH DGAs $A^\CE(\cT'_\lag,\bfmu')$ and $A^\CE(\cT_\lag,\bfmu)$.
\end{remark}

More precisely, $\RM{\hat{III}}_*$ is an isomorphism and $\RM{\hat{M}}_*$ is a canonical projection of a stabilization for $\RM{M}\in\{\RM{I}, \RM{II}, \RM{VI}\}$.
However, we have a zig-zag of stabilization for $\RM{VI}$
\[
\begin{tikzcd}[column sep=3pc]
& & \tilde SA\arrow[dl, "\tilde\bfpi'"',->>, yshift=0.7ex] \arrow[from=dl, hook', near end,"\tilde\bfi'"', yshift=-0.7ex]
\arrow[dr, "\tilde\bfpi",->>, yshift=0.7ex] \arrow[from=dr, hook, near end, "\tilde\bfi", yshift=-0.7ex]\\
A^\CE(\cT', \bfmu')\arrow[from=r, "\hat\bfpi'"',->>] & 
\hat\dga' \arrow[rr, "\RM{\hat{VI}}"',->>] & &
\hat\dga \arrow[r, "\hat\bfpi",->>] & 
A^\CE_\co(\cT, \bfmu).
\end{tikzcd}
\]

Similarly, for Lagrangian Reidemeister moves $\RM{m}$, we have an isomorphism $\RM{\hat m}_*$ for $\RM{m}\in\{\RM{0_*}, \RM{iii_*}\}$, a canonical projection of a stabilization for $\RM{ii}$ and a zig-zag of stabilizations for $\RM{iv}$.

\begin{example}[DGAs for bordered Legendrian graphs in a normal form]\label{example:DGAs for graphs in a normal form}
Recall a bordered Legendrian graph $(\cT,\bfmu)\in\BLT^\mu$ in a normal form in Definition~\ref{definition:normal form}.
Let $(\cT_\lag,\bfmu)$ be the Ng's resolution of $(\cT,\bfmu)$.
Then the set of vertex generators at each $v$ can be decomposed into two subsets as follows:
\begin{align*}
\tilde V_{v,\Left}&\coloneqq\left\{v_{a,j}\in \tilde V_v~\middle|~a+j<\val(v)\right\},&
\tilde V_{v,\Lcirclearrowright}&\coloneqq \left\{v_{a,j}\in \tilde V_v~\middle|~a+j\ge\val(v)\right\}
\end{align*}
and we denote their unions by $\tilde V_\Left$ and $\tilde V_{\Lcirclearrowright}$
\begin{align*}
\tilde V_{\Left}&\coloneqq\coprod_{v\in V} \tilde V_{v,\Left},&
\tilde V_{\Lcirclearrowright}&\coloneqq\coprod_{v\in V} \tilde V_{v,\Lcirclearrowright}.
\end{align*}

Notice that the elements in $\tilde V_{v,\Left}$ correspond to the vertex generators which are lying on the left hand side of the vertex $v$:
\begin{align*}
v_{1,2}&=\begin{tikzpicture}[baseline=-.5ex]
\draw[thick,fill] (-1,0) node[left] {$e_2$} -- (0,0) node[above] {$v$} circle (2pt);
\draw[thick] (-1,0.5) node[left] {$e_3$} -- (0,0);
\draw[thick] (-1,-0.5) node[left] {$e_1$} -- (0,0);
\draw[green,-latex',thick] (180+26.56:0.7) arc (180+26.56:180-26.56:0.7);
\end{tikzpicture}\in \tilde V_{v,\Left}&
v_{2,2}&=\begin{tikzpicture}[baseline=-.5ex]
\draw[thick,fill] (-1,0) node[left] {$e_2$} -- (0,0) node[above] {$v$} circle (2pt);
\draw[thick] (-1,0.5) node[left] {$e_3$} -- (0,0);
\draw[thick] (-1,-0.5) node[left] {$e_1$} -- (0,0);
\draw[green,-latex',thick] (180:0.7) arc (180:-180+26.56:0.7);
\end{tikzpicture}\in \tilde V_{v,\Lcirclearrowright}
\end{align*}

Since all vertices are in the upper right position, there are no immersed polygons for differentials involving generators in $\tilde V_{\Lcirclearrowright}$ except for infinitesimal ones.
That is, due to this geometric reason, no generators in $\tilde V_{\Lcirclearrowright}$ appears in differentials for any crossing generator and in the image of the right border generators $k_{ab}$'s under $\phi_\Right$.
\end{example}

Before closing this section, we remark the following.
The LCH DGA $A^\CE(\cT,\bfmu)$ has infinitely many generators $b_{a,\ell}$ for each basepoint $b\in B$ while it usually contributes two generator inverses to each other such as $t_b$ and $t_b^{-1}$.
However, as observed in \cite[Proposition~14]{EL2019}, the internal DGA for each basepoint $I_b$ is quasi-isomorphic to the Laurent polynomial ring $\field[t_b,t_b^{-1}]$.
Therefore one can regard $I_b$ as the free resolution of $\field[t_b, t_b^{-1}]$.

Indeed, $\field[t_b, t_b^{-1}]$ can be obtained by taking a quotient by all $b_{a,\ell}$'s with $\ell>1$, whose degrees are not zero.
Due to this observation, we can get one another consequence such that any DGA morphism from $A^(\CE,\cT,\bfmu)$ to a field $K=(\field,\differential=0)$ factors through the quotient DGA.
Therefore there is an isomorphism between sets\footnote{Indeed, we have an isomorphism between \emph{augmentation varieties}. See \cite{ABS2019count}.} of DGA morphisms.

Now let us consider the effects of operations on basepoints to LCH DGAs.
As seen earlier, there are three more moves on basepoints.
As mentioned earlier, the operation $\RM{b_1}:(\cT'_\lag,\bfmu')\to(\cT_\lag,\bfmu)$ can be viewed as a sequence of front Reidemeister moves. Hence we have a zig-zag of stabilizations which is not easy to describe directly.
\[
\begin{tikzcd}[ampersand replacement=\&]
b_{1,1}c=\begin{tikzpicture}[baseline=-.5ex,scale=0.7]
\useasboundingbox(-1,-0.5)--(1,0.5);
\draw[dashed](0,0) circle (1);
\clip(0,0) circle (1);
\draw[thick] (-1,-0.5) -- (1,0.5);
\draw[white,line width=5] (-1,0.5)--(1,-0.5);
\draw[thick] (-1,0.5)--node[midway,sloped,below=-2ex]{$|$} node[midway,above] {$b$} (0,0) node[below] {$c$} -- (1,-0.5);
\fill[gray, opacity=0.6] (-1,0.5)--(0,0)--(1,0.5)--(1,1) --(-1,1)--cycle;
\end{tikzpicture}\arrow[r,"\RM{b_1}","?"'] \&
c=\begin{tikzpicture}[baseline=-.5ex,scale=0.7]
\draw[dashed](0,0) circle (1);
\clip(0,0) circle (1);
\draw[thick] (-1,-0.5) -- (1,0.5);
\draw[white,line width=5] (-1,0.5)--(1,-0.5);
\draw[thick] (-1,0.5)-- (0,0) node[below] {$c$} --node[midway,sloped,below=-2ex]{$|$} node[midway,below] {$b$} (1,-0.5);
\fill[gray, opacity=0.6] (-1,0.5)--(0,0)--(1,0.5)--(1,1) --(-1,1)--cycle;
\end{tikzpicture}
\end{tikzcd}
\]

On the other hand, two operations $\RM{b_2}$ and $\RM{b_3}$ can be regarded as \emph{tangle replacements} so that they induce DGA morphisms $\RM{b_2}_*, \RM{b_3}_*:A^\CE(\cT',\bfmu')\to A^\CE(\cT,\bfmu)$.
Indeed, for each generator $b_{a,\ell}$ or $v_{a,\ell}$ will be mapped to elements obtained by counting certain polygons in the support of $\RM{b_2}$ or $\RM{b_3}$ in $\cT$.
For example, $b_{1,1}$ or $v_{4,2}$ will be mapped to $b'_{1,1}b''_{1,1}$ and $b_{2,1} v_{4,2}$. See \cite[\S6.5]{AB2018} for detail.
\[
\begin{tikzcd}[ampersand replacement=\&]
b_{1,1}=\begin{tikzpicture}[baseline=-.5ex,scale=0.7]
\useasboundingbox(-1,-0.5)--(1,0.5);
\draw[dashed](0,0) circle (1);
\clip(0,0) circle (1);
\draw[thick] (-1,0)-- (0,0) node[below=-2ex]{$|$} node[below] {$b$} -- (1,0);
\fill[gray,opacity=0.6] (-1,0) -- (1,0) -- (1,1) -- (-1,1) --cycle;
\end{tikzpicture}\arrow[r,"\RM{b_2}"] \&
b'_{1,1}b''_{1,1}=\begin{tikzpicture}[baseline=-.5ex,scale=0.7]
\draw[dashed](0,0) circle (1);
\clip(0,0) circle (1);
\draw[thick] (-1,0)-- node[near start, below=-2ex]{$|$}node[near start, below]{$b'$} node[near end, below=-2ex]{$|$} node[near end,below]{$b''$} (1,0);
\fill[gray,opacity=0.6] (-1,0) -- (1,0) -- (1,1) -- (-1,1) --cycle;
\end{tikzpicture}\&
v_{4,2}=\begin{tikzpicture}[baseline=-.5ex,scale=0.7]
\useasboundingbox(-1,-0.5)--(1,0.5);
\draw[dashed](0,0) circle (1);
\clip(0,0) circle (1);
\draw[thick] (-1,-0.5) -- (1,0.5);
\draw[thick] (-1,0)-- (1,0);
\draw[thick] (-1,0.5) -- (1,-0.5);
\draw[fill] (0,0) circle (2pt) node[below] {$v$};
\fill[gray, opacity=0.6] (1,0.5) -- (0,0) -- (1,-0.5) -- (1,1.5) --cycle;
\end{tikzpicture}\arrow[r,"\RM{b_3}"] \&
b_{2,1}v_{4,2}=\begin{tikzpicture}[baseline=-.5ex,scale=0.7]
\draw[dashed](0,0) circle (1);
\clip(0,0) circle (1);
\draw[thick] (-1,-0.5) -- (0,0) -- node[midway,sloped, below=-2ex] {$|$} node[midway,above left] {$b$} (1,0.5);
\draw[thick] (-1,0)-- (1,0);
\draw[thick] (-1,0.5) -- (1,-0.5);
\draw[fill] (0,0) circle (2pt) node[below] {$v$};
\fill[gray, opacity=0.6] (1,0.5) -- (0,0) -- (1,-0.5) -- (1,1.5) --cycle;
\end{tikzpicture}
\end{tikzcd}
\]

However, when we consider the quotient DGA by all generators $b_{a,\ell}$'s with $\ell>1$, then the move $\RM{b_1}$ induces a DGA isomorphism $\RM{b_1}_*$ which sends $c$ to $t_b^{-1}c$. The induced map $\RM{b_2}_*$ sends $t_b$ to $t_{b'}t_{b''}$ and $\RM{b_3}_*$ sends $v_{a,\ell}$ to either $v_{a,\ell}$, $v_{a,\ell}t_b$ or $t_b^{-1}v_{a,\ell}$ according to where the basepoint $b$ is lying on.

Conversely, one can find a DGA morphism $\RM{b_i^{-1}}_*:A^\CE(\cT_\lag,\bfmu)\to A^\CE(\cT'_\lag,\bfmu')$ for $i=2,3$ which are left inverses of $\RM{b_i}_*$ and defined as
\begin{align*}
\RM{b_2^{-1}}_*(b''_{a,\ell})&\coloneqq\begin{cases}
1 & \ell=1;\\
0 & \ell>1,
\end{cases}&
\RM{b_3^{-1}}_*(b_{a,\ell})&\coloneqq\begin{cases}
1 & \ell=1;\\
0 & \ell>1.
\end{cases}
\end{align*}

In summary, we have the following:
\begin{lemma}\label{lemma:basepoint moves on DGAs}
Let $\RM{b_i}:(\cT'_\lag,\bfmu')\to(\cT_\lag,\bfmu)$ be a basepoint move. Then either
\begin{enumerate}
\item for $i=1$, there exists a zig-zag of stabilizations between $A^\CE(\cT'_\lag,\bfmu')$ and $A^\CE(\cT_\lag,\bfmu)$, or
\item for $i=2$ or $3$, there exists a pair of DGA morphisms
\[
\begin{tikzcd}
A^\CE(\cT'_\lag,\bfmu')\arrow[r, "\RM{b_i}_*", yshift=.7ex]\arrow[from=r, "\RM{b_i^{-1}}_*", yshift=-.7ex] & A^\CE(\cT_\lag,\bfmu),
\end{tikzcd}
\]
where $\RM{b_i^{-1}}_*$ is a left inverse of the DGA morphism $\RM{b_i}_*$ induced from the tangle replacement.
\end{enumerate}
\end{lemma}
\begin{remark}\label{remark:basepoint move commutativity}
For $i=2$ or $3$, there is no issue on the commutativity of $\RM{b_i^{-1}}_*$with the structure morphisms $\phi_\Left$ and $\phi_\Right$.
\end{remark}

\section{Consistent sequences}\label{section:consistent sequences}
In this section, we first briefly review the definition of a consistent sequence of DGAs and the construction of the $A_\infty$-category associated to the consistent sequence of DGAs and we will basically follow the definition and construction described in \cite[\S3.3]{NRSSZ2015} but may use the different notations or conventions.

Let $\catC$ be a category. An \emph{$m$-component object} $C^{\p m}$ in $\catC$ is a functor from $\Op([m])$ to $\catC$. 
That is, for any $V\subset U\subset[m]$, there exist \emph{restrictions} $C^{\p m}|_U, C^{\p m}|_V$ and a \emph{restriction morphism} $C^{\p m}|_U\to C^{\p m}|_V$ in $\catC$.
A morphism $f^{\p m}:C'^{\p m}\to C^{\p m}$ is a natural transformation between $C'^{\p m}$ and $C^{\p m}$.
We denote the category of $m$-component objects in $\catC$ by $\catC^{\p m}$.

Let $\AugSim$ be the category of all finite ordinals
\begin{align*}
[1]&=\{1\},&
[2]&=\{1,2\},&
[3]&=\{1,2,3\},&\cdots,
\end{align*}
whose morphisms are order-preserving inclusions.

A sequence $C^{\p\bullet}$ of $m$-component objects in $\catC$ is \emph{semi-simplicial} if there is a morphism $C^h|_U:C^{\p n}|_{h(U)}\to C^{\p m}|_U$ in $\catC$ for each $(h:[m]\to[n])\in\AugSim$ and $U\subset[m]$, and is \emph{consistent} if it is semi-simplicial and each morphism $C^h|_U$ is an isomorphism.

For two consistent sequences $C'^{\p\bullet}$ and $C^{\p\bullet}$, a \emph{consistent morphism} $f^\p\bullet:C'^{\p\bullet}\to C^{\p\bullet}$ is a collection of morphisms between $m$-component objects satisfying the \emph{consistency}: for each $(h:[m]\to[n])\in\AugSim$ and $U\subset[m]$, the following diagram is commutative
\[
\begin{tikzcd}[column sep=4pc,row sep=2pc]
C'^{\p n}|_{h(U)}\arrow[d,"f^{\p n}|_{h(U)}"']\arrow[r,"C'^{h}|_U","\isomorphic"'] &C'^{\p m}|_U \arrow[d,"f^{\p m}|_U"]\\
C^{\p n}|_{h(U)}\arrow[r,"C^h|_U","\isomorphic"'] &C^{\p m}|_U.
\end{tikzcd}
\]

The composition of two consistent morphisms $f^{\p\bullet}:C'^{\p\bullet}\to C^{\p\bullet}$ and $g^{\p\bullet}:C''^{\p\bullet}\to C'^{\p\bullet}$ is given as
\[
f^{\p\bullet}\circ g^{\p\bullet}\coloneqq
\left(f^{\p m}\circ g^{\p m}\right)_{m\ge 1}:C''^{\p\bullet}\to C^{\p\bullet}.
\]

\begin{notation}
For each $h:[m]\to [n]$, $U\subset [m]$ and $i\in [m]$, we may denote $C^h|_U$, $C^{\p m}|_U$ and $C^{\p m}|_{\{i\}}$ by $C(h)^U$, $C^U$ and $C^i$, respectively.
We further denote $C^{\p1}$ and $f^{\p1}$ simply by $C$ and $f$.
\end{notation}

\begin{definition}[Category of consistent sequences]
The category of all consistent sequences and morphisms will be denoted by $\catC^\p\bullet$ and called a category of \emph{consistent sequences} in $\catC$.
\end{definition}

\begin{lemma}\label{lemma:isom}
Let $C^{\p\bullet}\in\catC^{\p\bullet}$ be a consistent sequence in $\catC$.
Then for each $m\ge1$ and $i\in[m]$, we have
\[
C^i= C^{\p m}_{i}\isomorphic C^{\p 1}_{1}=C^1. 
\]
\end{lemma}
\begin{proof}
This is obvious from the definition by considering $h_i:[1]\to[m]$ with $i=h_i(1)$.
\end{proof}

\subsection{Consistent sequences of bordered Legendrian graphs}
We define categories of consistent sequences of bordered Legendrian graphs given in terms of both Lagrangian and front projections, which are related via Ng's resolution as usual.

For a bordered graph $\bfgraf=(\graf_\Left\to \graf\leftarrow\graf_\Right)$, an \emph{$m$-component bordered graph $\bfgraf^{\p m}$} is a bordered graph defined as
\[
\bfgraf^{\p m}\coloneqq (\graf_\Left^{\p m}\to \graf^{\p m}\leftarrow \graf_\Right^{\p m}),
\]
where the inclusions preserve the label.
Then for each $U\subset [m]$ and $i\in[m]$, let the restriction $\bfgraf^i\coloneqq (\graf^i_\Left\to\graf^i\leftarrow\graf^i_\Right)$ and $\bfgraf^U\coloneqq \coprod_{i\in U} \bfgraf^i$,
where $\graf^i_*=\{i\}\times\graf_*$ for $*=\Left,\Right$ or empty.

For each $m$-component bordered graph $\bfgraf^{\p m}$, its Legendrian embedding defines an $m$-component bordered Legendrian graph $\scrT^{\p m}=\left(\sfT^{\p m}_\Left\to \sfT^{\p m}\leftarrow \sfT^{\p m}_\Right\right)=\left(\bfgraf^{\p m}\to J^1 \bfU\right)$, whose restriction $\scrT^U$ on $U\subset[m]$ is obviously defined as $\scrT^{\p m}|_{\bfgraf^U}$.

We can naturally consider the front and Lagrangian projections of $m$-component bordered Legendrian graphs which will be denoted by 
\[
\cT^{\p m}=(T_\Left^{\p m}\to T^{\p m}\leftarrow T_\Right^{\p m})\coloneqq\pi_\front(\scrT^{\p m})\quad\text{ and }\quad
\cT^{\p m}_\lag=(T_{\lag,\Left}^{\p m}\to T_\lag^{\p m}\leftarrow T_{\lag,\Right}^{\p m})\coloneqq\pi_\lag(\scrT^{\p m})
\]
and their restrictions on $U\subset[m]$ are denoted by $\cT^U$ and $\cT^U_\lag$ as before.

One can equip a Maslov potential $\bfmu^{\p m}$ on $\cT^{\p m}$ or $\cT_\lag^{\p m}$ so that the restriction $\bfmu^U$ is the restriction of $\bfmu^{\p m}$ on $\cT^U$ or $\cT^U_\lag$.

\begin{notation}
We denote the sets of all regular front and Lagrangian projections of $m$-component bordered Legendrian graphs with Maslov potentials by $\BLT^{\mu,\p m}$ and $\BLT_\lag^{\mu,\p m}$, respectively.
\end{notation}

For a sequence $(\cT^{\p m},\bfmu^{\p m})_{m\ge 1}$ of $m$-component front projections $(\cT^{\p m},\bfmu^{\p m})\in\BLT^{\mu,\p m}$, the consistency is as follows: for each $(h:[m]\to[n])\in\AugSim$ and $U\subset[m]$, there is an isomorphism between two front projections in the sense of Definition~\ref{definition:isomorphisms}
\[
(\cT(h)^U,\bfmu(h)^U):(\cT^{h(U)},\bfmu^{h(U)})\stackrel{\isomorphic}\longrightarrow(\cT^U,\bfmu^U).
\]

\begin{definition}[Consistent moves]\label{definition:consistent Reidemeister moves}
A \emph{consistent front (or Lagrangian) Reidemeister move} or \emph{basepoint move} $\RM{M}^{\p \bullet}:(\cT'^{\p\bullet},\bfmu'^{\p\bullet})\to(\cT^{\p\bullet},\bfmu^{\p\bullet})$ between two consistent sequences of front (or Lagrangian) projections is a sequence of sequences of front (or Lagrangian) Reidemeister moves or basepoint moves satisfying the following conditions:
\begin{enumerate}
\item for each $m\ge 1$, $\RM{M}^{\p m}:(\cT'^{\p m},\bfmu'^{\p m})\to(\cT^{\p m},\bfmu^{\p m})$ is a sequence of front (or Lagrangian) Reidemeister moves or basepoint moves and in particular, $\RM{M}^{\p1}=\RM{M}\in\cRM$ is a (possibly empty) front (or Lagrangian) Reidemeister move or a basepoint move;
\item it is compatible with restrictions, i.e., for each $U\subset[m]$, we have a sequence of front (or Lagrangian) Reidemeister moves or basepoint moves
\[
\RM{M}^U:(\cT'^U,\bfmu'^U)\to (\cT^U, \bfmu^U);
\]
\item it satisfies the consistency, i.e., for each $(h:[m]\to[n])\in\AugSim$ and $U\subset[m]$, the following diagram is commutative:
\[
\begin{tikzcd}[column sep=6pc, row sep=2pc,ampersand replacement=\&]
(\cT'^{h(U)},\bfmu'^{h(U)}) \arrow[r,"{(\cT'(h)^U,\bfmu'(h)^U)}"] \arrow[d, "\RM{M}^{h(U)}"'] \& (\cT'^U,\bfmu'^U) \arrow[d,"\RM{M}^U"]\\
(\cT^{h(U)},\bfmu'^{h(U)})\arrow[r,"{(\cT(h)^U,\bfmu(h)^U)}"] \& (\cT^U,\bfmu^U).
\end{tikzcd}
\]
\end{enumerate}

We furthermore say that a consistent move $\RM{M}^{\p \bullet}$ is \emph{elementary} if it consists of Reidemeister moves or basepoint moves of only one type.

Two consistent sequences of front or Lagrangian projections are said to be (consistently) \emph{equivalent} if and only if they are connected by a zig-zag sequence of consistent front or Lagrangian Reidemeister moves.
\end{definition}

\begin{notation}
From now on, we will denote the elementary consistent Reidemeister move or basepoint move by the double arrow `$\Rightarrow$' such as
\[
\RM{M}^{\p\bullet}:\cT'\Longrightarrow \cT.
\]
\end{notation}

\begin{example}[Elementary consistent Reidemeister moves]\label{example:elementary iv}
One example for elementary consistent Reidemeister move is a consistent Lagrangian Reidemeister move as follows: for each $m\ge 1$,
\begin{equation}\label{equation:consistent sequence for iv_b}
\begin{tikzcd}
\begin{tikzpicture}[baseline=-.5ex,scale=0.7]
\draw[dashed] (0,0) circle (1);
\clip (0,0) circle (1);
\draw[thick,black] (-0.7,-1)--(-0.2,1);
\draw[thick,red] (-0.6,-1)--(-0.1,1);
\draw[thick,blue] (-0.5,-1)--(0,1);
\foreach[count=\index] \c in {blue, red, black} {
\begin{scope}[yshift=(\index-2)*0.5 cm, color=\c]
\draw[white,line width=5] (-1,0.15) - - (0,0) (-1,0) -- (0,0) (-1,-0.15) -- (0,0);
\draw[thick] (-1,0.15) - - (0,0) (-1,0) -- (0,0) (-1,-0.15) -- (0,0);
\draw[fill] (0,0) circle (2pt);
\end{scope}
}
\end{tikzpicture}\arrow[rrr, Rightarrow, "m^2\RM{iv}"] \arrow[d, equal]& & &
\begin{tikzpicture}[baseline=-.5ex,scale=0.7]
\draw[dashed] (0,0) circle (1);
\clip (0,0) circle (1);
\draw[thick,black] (0.3,-1)--(0.8,1);
\draw[thick,red] (0.4,-1)--(0.9,1);
\draw[thick,blue] (0.5,-1)--(1,1);
\foreach[count=\index] \c in {blue, red, black} {
\begin{scope}[yshift=(\index-2)*0.5 cm, color=\c]
\draw[white,line width=5] (-1,0.15) - - (0,0) (-1,0) -- (0,0) (-1,-0.15) -- (0,0);
\draw[thick] (-1,0.15) - - (0,0) (-1,0) -- (0,0) (-1,-0.15) -- (0,0);
\draw[fill] (0,0) circle (2pt);
\end{scope}
}
\end{tikzpicture}\arrow[d,equal]\\
\begin{tikzpicture}[baseline=-.5ex,scale=0.7]
\draw[dashed] (0,0) circle (1);
\clip (0,0) circle (1);
\draw[thick,black] (-0.7,-1)--(-0.2,1);
\draw[thick,red] (-0.6,-1)--(-0.1,1);
\draw[thick,blue] (-0.5,-1)--(0,1);
\foreach[count=\index] \c in {blue, red, black} {
\begin{scope}[yshift=(\index-2)*0.5 cm, color=\c]
\draw[white,line width=5] (-1,0.15) - - (0,0) (-1,0) -- (0,0) (-1,-0.15) -- (0,0);
\draw[thick] (-1,0.15) - - (0,0) (-1,0) -- (0,0) (-1,-0.15) -- (0,0);
\draw[fill] (0,0) circle (2pt);
\end{scope}
}
\end{tikzpicture}\arrow[r,"m\RM{iv}"]&
\begin{tikzpicture}[baseline=-.5ex,scale=0.7]
\draw[dashed] (0,0) circle (1);
\clip (0,0) circle (1);
\draw[thick,black] (-0.7,-1)--(-0.2,1);
\draw[thick,red] (-0.6,-1)--(-0.1,1);
\draw[thick,blue] (0.5,-1)--(1,1);
\foreach[count=\index] \c in {blue, red, black} {
\begin{scope}[yshift=(\index-2)*0.5 cm, color=\c]
\draw[white,line width=5] (-1,0.15) - - (0,0) (-1,0) -- (0,0) (-1,-0.15) -- (0,0);
\draw[thick] (-1,0.15) - - (0,0) (-1,0) -- (0,0) (-1,-0.15) -- (0,0);
\draw[fill] (0,0) circle (2pt);
\end{scope}
}
\end{tikzpicture}\arrow[r,"m\RM{iv}"]&\cdots\arrow[r,"m\RM{iv}"]&
\begin{tikzpicture}[baseline=-.5ex,scale=0.7]
\draw[dashed] (0,0) circle (1);
\clip (0,0) circle (1);
\draw[thick,black] (0.3,-1)--(0.8,1);
\draw[thick,red] (0.4,-1)--(0.9,1);
\draw[thick,blue] (0.5,-1)--(1,1);
\foreach[count=\index] \c in {blue, red, black} {
\begin{scope}[yshift=(\index-2)*0.5 cm, color=\c]
\draw[white,line width=5] (-1,0.15) - - (0,0) (-1,0) -- (0,0) (-1,-0.15) -- (0,0);
\draw[thick] (-1,0.15) - - (0,0) (-1,0) -- (0,0) (-1,-0.15) -- (0,0);
\draw[fill] (0,0) circle (2pt);
\end{scope}
}
\end{tikzpicture}
\end{tikzcd}
\end{equation}
\end{example}

\begin{definition}[Categories of consistent sequences of bordered Legendrian graphs]
The categories of consistent sequences and Reidemeister moves of regular front and Lagrangian projections of bordered Legendrian graphs with Maslov potentials will be denoted by $\BLT^{\mu,\p \bullet}$ and $\BLT^{\mu,\p \bullet}_\lag$, and two full subcategories consisting of consistent sequences of non-bordered Legendrian graphs will be denoted by $\LT^{\mu,\p\bullet}$ and $\LT^{\mu,\p\bullet}_\lag$.
\end{definition}

Recall the Ng's resolution $\Res:\BLT^\mu\to\BLT^\mu_\lag$ which is a functor as seen earlier and therefore it preserves isomorphisms.
In other words, each isomorphic pair of front projections will be mapped to an isomorphic pair of Lagrangian projections.
It is obvious that $\Res$ sends each component to the corresponding component, it induces a functor between consistent sequences of front and Lagrangian projections.
\begin{lemma}
The Ng's resolution induces a functor
\begin{align*}
\Res^{\p\bullet}&:\BLT^{\mu,\p\bullet}\to\BLT_\lag^{\mu,\p\bullet}
\end{align*}
which preserves the homotopy relation.
\end{lemma}
\begin{proof}
This is nothing but the Lemma~\ref{lemma:functoriality of Ng's resolution}.
\end{proof}

\subsubsection{Canonical front and Lagrangian parallel copies}\label{section:canonical copies}
We introduce a canonical way to obtain a consistent sequence for a bordered Legendrian graph in terms of front projections or Lagrangian projections.

\begin{definition}[Canonical consistent sequences]
Let $(\cT,\bfmu)\in\BLT^\mu$ and $(\cT_\lag,\bfmu)\in\BLT_\lag^\mu$ be regular front and Lagrangian projections.
Then the \emph{canonical consistent sequences} $(\cT^{\p \bullet},\bfmu^{\p \bullet})\in\BLT^{\mu,\p \bullet}$ and $(\cT^{\p \bullet}_\lag,\bfmu^{\p \bullet})\in\BLT^{\mu, \p\bullet}_\lag$ are given by the $z$- and $y$-translations as depicted in Figures~\ref{figure:front copy} and \ref{figure:Lagrangian copy}, respectively.
\end{definition}

\begin{figure}[ht]
\subfigure[Left border]{\makebox[0.175\textwidth]{$
\begin{tikzpicture}[baseline=-.5ex,scale=0.75]
\useasboundingbox (-0.5,-1.5)--(0.5,1.5);
\draw[red] (0,-1.5)--(0,1.5);
\foreach \y in {1.2,0.2,-0.8} {
	\draw[thick] (0,\y) -- +(0.5,0);
}
\end{tikzpicture}\mapsto
\begin{tikzpicture}[baseline=-.5ex,scale=0.75]
\useasboundingbox (-0.5,-1.5)--(0.5,1.5);
\draw[red] (0,-1.5)--(0,1.5);
\foreach \y in {1.2,0.2,-0.8} {
	\draw[thick] (0,\y) -- +(0.5,0);
	\draw[red,thick] (0,\y-0.2) -- +(0.5,0);
	\draw[blue,thick] (0,\y-0.4) -- +(0.5,0);
}
\end{tikzpicture}$}}
\subfigure[Right border]{\makebox[0.175\textwidth]{$
\begin{tikzpicture}[baseline=-.5ex,scale=0.75]
\useasboundingbox (-0.5,-1.5)--(0.5,1.5);
\draw[red] (0,-1.5)--(0,1.5);
\foreach \y in {1.2,0.2,-0.8} {
	\draw[thick] (0,\y) -- +(-0.5,0);
}
\end{tikzpicture}\mapsto
\begin{tikzpicture}[baseline=-.5ex,scale=0.75]
\useasboundingbox (-0.5,-1.5)--(0.5,1.5);
\draw[red] (0,-1.5)--(0,1.5);
\foreach \y in {1.2,0.2,-0.8} {
	\draw[thick] (0,\y) -- +(-0.5,0);
	\draw[red,thick] (0,\y-0.2) -- +(-0.5,0);
	\draw[blue,thick] (0,\y-0.4) -- +(-0.5,0);
}
\end{tikzpicture}$}}
\subfigure[Basebasepointpoint]{\makebox[0.275\textwidth]{$
\begin{tikzpicture}[baseline=-.5ex,scale=0.5]
\useasboundingbox (-1.5,-1.5) -- (1.5,1.5);
\draw[dashed,gray] (0,0) circle (1.5);
\clip (0,0) circle (1.5);
\draw[thick] (-2,.6) -- node[midway] {$\scriptstyle|$} +(4,0);
\end{tikzpicture}\mapsto
\begin{tikzpicture}[baseline=-.5ex,scale=0.5]
\useasboundingbox (-1.5,-1.5) -- (1.5,1.5);
\draw[dashed,gray] (0,0) circle (1.5);
\clip (0,0) circle (1.5);
\draw[thick] (-2,.6) -- node[midway] {$\scriptstyle|$} +(4,0);
\draw[red,thick] (-2,0) -- node[midway] {$\scriptstyle|$} +(4,0);
\draw[blue,thick] (-2,-.6) -- node[midway] {$\scriptstyle|$} +(4,0);
\end{tikzpicture}$}}
\subfigure[Crossing]{\makebox[0.275\textwidth]{$
\begin{tikzpicture}[baseline=-.5ex,scale=0.5]
\useasboundingbox (-1.5,-1.5) -- (1.5,1.5);
\draw[dashed,gray] (0,0) circle (1.5);
\clip (0,0) circle (1.5);
\draw[thick] (-2,-1.5) -- +(4,4);
\draw[thick] (-2,2.5) -- +(4,-4);
\end{tikzpicture}\mapsto
\begin{tikzpicture}[baseline=-.5ex,scale=0.5]
\useasboundingbox (-1.5,-1.5) -- (1.5,1.5);
\draw[dashed,gray] (0,0) circle (1.5);
\clip (0,0) circle (1.5);
\draw[thick] (-2,-1.5) -- +(4,4);
\draw[red,thick] (-2,-2) -- +(4,4);
\draw[blue,thick] (-2,-2.5) -- +(4,4);
\draw[thick] (-2,2.5) -- +(4,-4);
\draw[red,thick] (-2,2) -- +(4,-4);
\draw[blue,thick] (-2,1.5) -- +(4,-4);
\end{tikzpicture}
$}}
\subfigure[Left cusp]{\makebox[0.3\textwidth]{$
\begin{tikzpicture}[baseline=-.5ex,scale=0.5]
\useasboundingbox (-1.5,-1.5) -- (1.5,1.5);
\draw[dashed,gray] (0,0) circle (1.5);
\clip (0,0) circle (1.5);
\draw[thick] (1.5,-0.2) to[out=180,in=0] (-0.2,0.3) to[out=0,in=180] (1.5,0.8);
\end{tikzpicture}\mapsto
\begin{tikzpicture}[baseline=-.5ex,scale=0.5]
\useasboundingbox (-1.5,-1.5) -- (1.5,1.5);
\draw[dashed,gray] (0,0) circle (1.5);
\clip (0,0) circle (1.5);
\draw[thick] (1.5,-0.2) to[out=180,in=0] (-0.2,0.3) to[out=0,in=180] (1.5,0.8);
\draw[thick,red] (1.5,-0.5) to[out=180,in=0] (-0.2,0) to[out=0,in=180] (1.5,0.5);
\draw[thick,blue] (1.5,-0.8) to[out=180,in=0] (-0.2,-0.3) to[out=0,in=180] (1.5,0.2);
\end{tikzpicture}
$}}
\subfigure[Right cusp]{\makebox[0.3\textwidth]{$
\begin{tikzpicture}[baseline=-.5ex,xscale=-1,scale=0.5]
\useasboundingbox (-1.5,-1.5) -- (1.5,1.5);
\draw[dashed,gray] (0,0) circle (1.5);
\clip (0,0) circle (1.5);
\draw[thick] (1.5,-0.2) to[out=180,in=0] (-0.2,0.3) to[out=0,in=180] (1.5,0.8);
\end{tikzpicture}\mapsto
\begin{tikzpicture}[baseline=-.5ex,xscale=-1,scale=0.5]
\useasboundingbox (-1.5,-1.5) -- (1.5,1.5);
\draw[dashed,gray] (0,0) circle (1.5);
\clip (0,0) circle (1.5);
\draw[thick] (1.5,-0.2) to[out=180,in=0] (-0.2,0.3) to[out=0,in=180] (1.5,0.8);
\draw[thick,red] (1.5,-0.5) to[out=180,in=0] (-0.2,0) to[out=0,in=180] (1.5,0.5);
\draw[thick,blue] (1.5,-0.8) to[out=180,in=0] (-0.2,-0.3) to[out=0,in=180] (1.5,0.2);
\end{tikzpicture}
$}}
\subfigure[Vertex]{\makebox[0.3\textwidth]{$
\begin{tikzpicture}[baseline=-.5ex,scale=0.5]
\useasboundingbox (-1.5,-1.5) -- (1.5,1.5);
\draw[dashed,gray] (0,0) circle (1.5);
\clip (0,0) circle (1.5);
\draw[thick] (0,0.6) to[out=180,in=0] +(-1.5,0.8);
\draw[thick] (0,0.6) to[out=180,in=0] +(-1.5,-0.8);
\draw[thick] (0,0.6) to[out=0,in=180] +(1.5,0.8);
\draw[thick] (0,0.6) to[out=0,in=180] +(1.5,-0.8);
\draw[thick] (-1.5,0.6) -- (1.5,0.6);
\fill (0,0.6) circle (0.1);
\end{tikzpicture}\mapsto
\begin{tikzpicture}[baseline=-.5ex,scale=0.5]
\useasboundingbox (-1.5,-1.5) -- (1.5,1.5);
\draw[dashed,gray] (0,0) circle (1.5);
\clip (0,0) circle (1.5);
\draw[thick] (-1.5,0.6) -- (1.5,0.6);
\draw[red,thick] (-1.5,0.4) -- (1.5,0.4);
\draw[blue,thick] (-1.5, 0.2) to[out=0,in=180] (0,-0.2) to[out=0,in=180] (1.5,0.2);
\draw[thick] (0,0.6) to[out=180,in=0] +(-1.5,-0.8);
\draw[red,thick] (0,0.4) to[out=180,in=0] +(-1.5,-0.8);
\draw[blue,thick] (0,-0.2) to[out=180,in=0] +(-1.5,-0.4);
\draw[thick] (0,0.6) to[out=180,in=0] +(-1.5,0.8);
\draw[red,thick] (0,0.4) to[out=180,in=0] +(-1.5,0.8);
\draw[blue,thick] (0,-0.2) to[out=180,in=0] +(-1.5,1.2);
\draw[thick] (0,0.6) to[out=0,in=180] +(1.5,0.8);
\draw[red,thick] (0,0.4) to[out=0,in=180] +(1.5,0.8);
\draw[blue,thick] (0,-0.2) to[out=0,in=180] +(1.5,1.2);
\draw[thick] (0,0.6) to[out=0,in=180] +(1.5,-0.8);
\draw[red,thick] (0,0.4) to[out=0,in=180] +(1.5,-0.8);
\draw[blue,thick] (0,-0.2) to[out=0,in=180] +(1.5,-0.4);
\fill (0,0.6) circle (0.1);
\fill[red] (0,0.4) circle (0.1);
\fill[blue] (0,-0.2) circle (0.1);
\end{tikzpicture}$}}
\caption{Canonical front parallel copies}
\label{figure:front copy}
\end{figure}
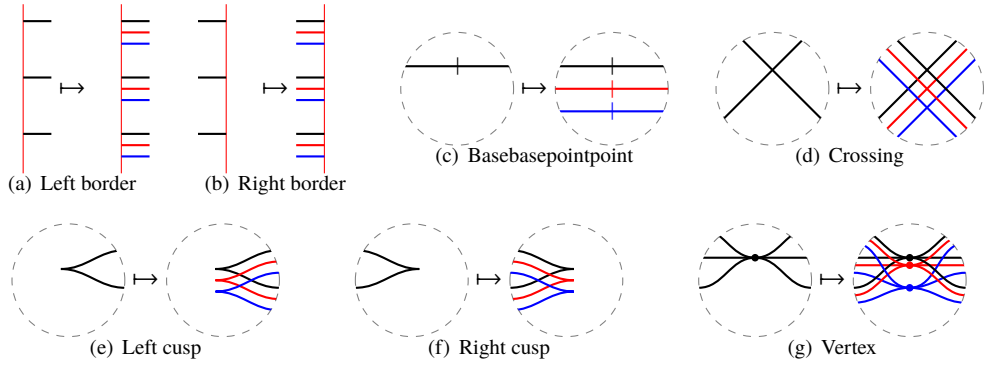

\begin{figure}[ht]
\subfigure[Left border]{\makebox[0.175\textwidth]{$
\begin{tikzpicture}[baseline=-.5ex,scale=0.75]
\useasboundingbox (-0.5,-1.5)--(0.5,1.5);
\draw[red] (0,-1.5)--(0,1.5);
\foreach \y in {1.2,0.2,-0.8} {
	\draw[thick] (0,\y) -- +(0.5,0);
}
\end{tikzpicture}\mapsto
\begin{tikzpicture}[baseline=-.5ex,scale=0.75]
\useasboundingbox (-0.5,-1.5)--(0.5,1.5);
\draw[red] (0,-1.5)--(0,1.5);
\foreach \y in {1.2,0.2,-0.8} {
	\draw[thick] (0,\y) -- +(0.5,0);
	\draw[red,thick] (0,\y-0.2) -- +(0.5,0);
	\draw[blue,thick] (0,\y-0.4) -- +(0.5,0);
}
\end{tikzpicture}$}}
\subfigure[Right border]{\makebox[0.175\textwidth]{$
\begin{tikzpicture}[baseline=-.5ex,scale=0.75]
\useasboundingbox (-0.5,-1.5)--(0.5,1.5);
\draw[red] (0,-1.5)--(0,1.5);
\foreach \y in {1.2,0.2,-0.8} {
	\draw[thick] (0,\y) -- +(-0.5,0);
}
\end{tikzpicture}\mapsto
\begin{tikzpicture}[baseline=-.5ex,scale=0.75]
\useasboundingbox (-0.5,-1.5)--(0.5,1.5);
\draw[red] (0,-1.5)--(0,1.5);
\foreach \y in {1.2,0.2,-0.8} {
	\draw[thick] (0,\y) -- +(-0.5,0);
	\draw[red,thick] (0,\y-0.2) -- +(-0.5,0);
	\draw[blue,thick] (0,\y-0.4) -- +(-0.5,0);
}
\end{tikzpicture}$}}
\subfigure[basepoint]{\makebox[0.275\textwidth]{$
\begin{tikzpicture}[baseline=-.5ex,scale=0.5]
\useasboundingbox (-1.5,-1.5) -- (1.5,1.5);
\draw[dashed,gray] (0,0) circle (1.5);
\clip (0,0) circle (1.5);
\draw[thick] (-2,.6) -- node[midway] {$\scriptstyle|$} +(4,0);
\end{tikzpicture}\mapsto
\begin{tikzpicture}[baseline=-.5ex,scale=0.5]
\useasboundingbox (-1.5,-1.5) -- (1.5,1.5);
\draw[dashed,gray] (0,0) circle (1.5);
\clip (0,0) circle (1.5);
\draw[thick] (-2,.6) -- node[midway] {$\scriptstyle|$} +(4,0);
\draw[red,thick] (-2,0) -- node[midway] {$\scriptstyle|$} +(4,0);
\draw[blue,thick] (-2,-.6) -- node[midway] {$\scriptstyle|$} +(4,0);
\end{tikzpicture}$}}
\subfigure[Crossing]{\makebox[0.275\textwidth]{$
\begin{tikzpicture}[baseline=-.5ex,scale=0.5]
\useasboundingbox (-1.5,-1.5) -- (1.5,1.5);
\draw[dashed,gray] (0,0) circle (1.5);
\clip (0,0) circle (1.5);
\draw[thick] (-2,-1.5) -- +(4,4);
\draw[white,line width=4] (-2,2.5) -- +(4,-4);
\draw[thick] (-2,2.5) -- +(4,-4);
\end{tikzpicture}\mapsto
\begin{tikzpicture}[baseline=-.5ex,scale=0.5]
\useasboundingbox (-1.5,-1.5) -- (1.5,1.5);
\draw[dashed,gray] (0,0) circle (1.5);
\clip (0,0) circle (1.5);
\draw[thick] (-2,-1.5) -- +(4,4);
\draw[red,thick] (-2,-2) -- +(4,4);
\draw[blue,thick] (-2,-2.5) -- +(4,4);
\draw[white,line width=4] (-2,2.5) -- +(4,-4);
\draw[white,line width=4] (-2,2) -- +(4,-4);
\draw[white,line width=4] (-2,1.5) -- +(4,-4);
\draw[thick] (-2,2.5) -- +(4,-4);
\draw[red,thick] (-2,2) -- +(4,-4);
\draw[blue,thick] (-2,1.5) -- +(4,-4);
\end{tikzpicture}
$}}
\subfigure[$x$-minimum]{\makebox[0.3\textwidth]{$
\begin{tikzpicture}[baseline=-.5ex,scale=0.5]
\useasboundingbox (-1.5,-1.5) -- (1.5,1.5);
\draw[dashed,gray] (0,0) circle (1.5);
\clip (0,0) circle (1.5);
\draw[thick] (1.5,0.3)+(90:1.5 and 0.5) arc (90:270:1.5 and 0.5);
\end{tikzpicture}\mapsto
\begin{tikzpicture}[baseline=-.5ex,scale=0.5]
\useasboundingbox (-1.5,-1.5) -- (1.5,1.5);
\draw[dashed,gray] (0,0) circle (1.5);
\clip (0,0) circle (1.5);
\draw[blue,thick] (1.5,-0.3)+(90:1.5 and 0.5) arc (90:270:1.5 and 0.5);
\draw[white,line width=4] (1.5,0)+(90:1.5 and 0.5) arc (90:270:1.5 and 0.5);
\draw[red,thick] (1.5,0)+(90:1.5 and 0.5) arc (90:270:1.5 and 0.5);
\draw[white,line width=4] (1.5,0.3)+(90:1.5 and 0.5) arc (90:270:1.5 and 0.5);
\draw[thick] (1.5,0.3)+(90:1.5 and 0.5) arc (90:270:1.5 and 0.5);
\end{tikzpicture}
$}}
\subfigure[$x$-maximum]{\makebox[0.3\textwidth]{$
\begin{tikzpicture}[baseline=-.5ex,scale=0.5]
\useasboundingbox (-1.5,-1.5) -- (1.5,1.5);
\draw[dashed,gray] (0,0) circle (1.5);
\clip (0,0) circle (1.5);
\draw[thick] (-1.5,0.3)+(90:1.5 and 0.5) arc (90:-90:1.5 and 0.5);
\end{tikzpicture}\mapsto
\begin{tikzpicture}[baseline=-.5ex,scale=0.5]
\useasboundingbox (-1.5,-1.5) -- (1.5,1.5);
\draw[dashed,gray] (0,0) circle (1.5);
\clip (0,0) circle (1.5);
\draw[blue,thick] (-1.5,-0.3)+(90:1.5 and 0.5) arc (90:-90:1.5 and 0.5);
\draw[white,line width=4] (-1.5,0)+(90:1.5 and 0.5) arc (90:-90:1.5 and 0.5);
\draw[red,thick] (-1.5,0)+(90:1.5 and 0.5) arc (90:-90:1.5 and 0.5);
\draw[white,line width=4] (-1.5,0.3)+(90:1.5 and 0.5) arc (90:-90:1.5 and 0.5);
\draw[thick] (-1.5,0.3)+(90:1.5 and 0.5) arc (90:-90:1.5 and 0.5);
\end{tikzpicture}
$}}
\subfigure[Vertex]{\makebox[0.3\textwidth]{$
\begin{tikzpicture}[baseline=-.5ex,scale=0.5]
\useasboundingbox (-1.5,-1.5) -- (1.5,1.5);
\draw[dashed,gray] (0,0) circle (1.5);
\clip (0,0) circle (1.5);
\draw[thick,rounded corners] (-1.5,-.8)--(0,0.2);
\draw[thick,rounded corners] (-1.5,0)--(0,0.2);
\draw[thick,rounded corners] (-1.5,.8)--(0,0.2);
\draw[thick,rounded corners] (0,0.2)--(1.5,-.8);
\draw[thick,rounded corners] (0,0.2)--(1.5,0);
\draw[thick,rounded corners] (0,0.2)--(1.5,.8);
\fill (0,0.2) circle(3pt);
\end{tikzpicture}\mapsto
\begin{tikzpicture}[baseline=-.5ex,scale=0.5]
\useasboundingbox (-1.5,-1.5) -- (1.5,1.5);
\draw[dashed,gray] (0,0) circle (1.5);
\clip (0,0) circle (1.5);
\foreach[count=\index] \c in {blue, red, black} {
\begin{scope}[yshift=(\index-2)*0.2 cm, color=\c]
\draw[white,line width=4,rounded corners] (-1.5,-.8)--(0,{-(\index-3)^2*0.25+0.2});
\draw[white,line width=4,rounded corners] (-1.5,0)--(0,{-(\index-3)^2*0.25+0.2});
\draw[white,line width=4,rounded corners] (-1.5,.8)--(0,{-(\index-3)^2*0.25+0.2});
\draw[white,line width=4,rounded corners] (0,{-(\index-3)^2*0.25+0.2})--(1.5,-.8);
\draw[white,line width=4,rounded corners] (0,{-(\index-3)^2*0.25+0.2})--(1.5,0);
\draw[white,line width=4,rounded corners] (0,{-(\index-3)^2*0.25+0.2})--(1.5,.8);
\draw[thick,rounded corners] (-1.5,-.8)--(0,{-(\index-3)^2*0.25+0.2});
\draw[thick,rounded corners] (-1.5,0)--(0,{-(\index-3)^2*0.25+0.2});
\draw[thick,rounded corners] (-1.5,.8)--(0,{-(\index-3)^2*0.25+0.2});
\draw[thick,rounded corners] (0,{-(\index-3)^2*0.25+0.2})--(1.5,-.8);
\draw[thick,rounded corners] (0,{-(\index-3)^2*0.25+0.2})--(1.5,0);
\draw[thick,rounded corners] (0,{-(\index-3)^2*0.25+0.2})--(1.5,.8);
\fill (0,{-(\index-3)^2*0.25+0.2}) circle(0.1);
\end{scope}
}
\end{tikzpicture}
$}}
\caption{Canonical Lagrangian parallel copies}
\label{figure:Lagrangian copy}
\end{figure}
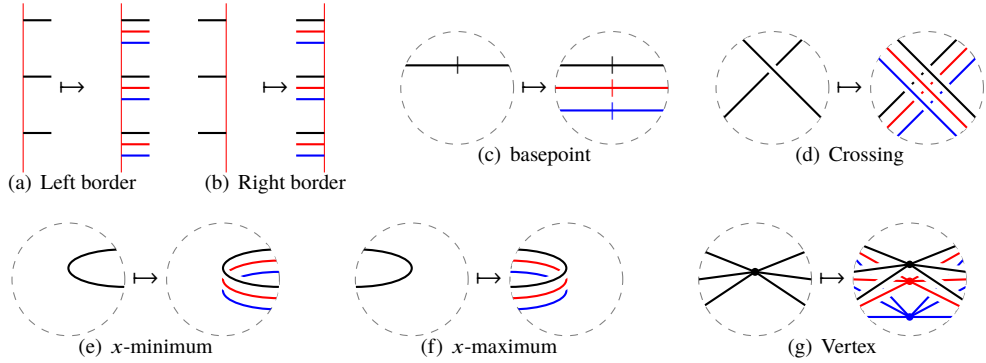

\begin{remark}
For Lagrangian projections, the $y$-translation comes from a strict contactomorphism
\[
(x,y,z) \to (x, y-\epsilon, z+x\epsilon),
\]
whose Lagrangian projection is the desired $y$-translation.
\end{remark}

Notice that it is not obvious if $\cT^{\p m}$ or $\cT^{\p m}_\lag$ is regular again. Indeed, we need to check the regularity holds near each vertex.

There are many ways to achieve the regularity but in this paper, we fix the convention as follows:
let $U_1$ be a neighborhood of the vertex $v$.
If we take a parallel copy and make $\cT^{\p 2}$ or $\cT^{\p 2}_\lag$ by the small enough translation, then all additional crossings are contained in a small neighborhood $U_2$. 
Now we take the third copy such that the newly appeared crossings \emph{avoid} the region $U_2$ and are contained in a larger neighborhood $U_3$ and so on. See Figure~\ref{figure:canonical copies near vertices}. The shaded region represents neighborhoods $U_i$.
Then it is obvious that for each $\cT$ or $\cT^\lag$, the sequence $\cT^{\p\bullet}$ or $\cT^{\p\bullet}_\lag$ of canonical projections becomes a consistent sequence of $m$-component Lagrangian projections.

\begin{remark}
In the canonical $m$-copies near a vertex $v$ of type $(\ell,r)$, there are exactly $\binom{m}2\binom\ell2+\binom{m}2\binom r2$ many additional crossings.
\end{remark}

\begin{figure}[ht]
\subfigure[Canonical Front parallel copies near a vertex]{$
\begin{aligned}
\cT&=
\begin{tikzpicture}[baseline=-.5ex,scale=0.5]
\draw[thick] (-2,1) to[out=0,in=180] (0,0) to[out=0,in=180] (2,-1);
\draw[thick] (-2,0) to[out=0,in=180] (0,0) to[out=0,in=180] (2,0);
\draw[thick] (-2,-1) to[out=0,in=180] (0,0) to[out=0,in=180] (2,1);
\draw[fill] (0,0) circle (2pt);
\draw[fill,gray,opacity=0.3] (0,0) circle (0.2);
\end{tikzpicture}&
\cT^{\p2}&=
\begin{tikzpicture}[baseline=-.5ex,scale=0.5]
\begin{scope}[yshift=-.2cm]
\draw[thick,red] (-2,1) to[out=0,in=180] (0,0) to[out=0,in=180] (2,-1);
\draw[thick,red] (-2,0) to[out=0,in=180] (0,0) to[out=0,in=180] (2,0);
\draw[thick,red] (-2,-1) to[out=0,in=180] (0,0) to[out=0,in=180] (2,1);
\draw[fill,red] (0,0) circle (2pt);
\end{scope}
\begin{scope}
\draw[thick] (-2,1) to[out=0,in=180] (0,0) to[out=0,in=180] (2,-1);
\draw[thick] (-2,0) to[out=0,in=180] (0,0) to[out=0,in=180] (2,0);
\draw[thick] (-2,-1) to[out=0,in=180] (0,0) to[out=0,in=180] (2,1);
\draw[fill] (0,0) circle (2pt);
\end{scope}
\draw[fill,gray,opacity=0.3] (0,0) circle (0.2);
\draw[fill,gray,opacity=0.3] (0,0) ellipse (0.8 and .5);
\end{tikzpicture}&
\cT^{\p3}&=
\begin{tikzpicture}[baseline=-.5ex,scale=0.5]
\begin{scope}[yshift=-.4cm]
\draw[thick,blue] (-2,1) to[out=0,in=180] (0,-0.2) to[out=0,in=180] (2,-1);
\draw[thick,blue] (-2,0) to[out=0,in=180] (0,-0.2) to[out=0,in=180] (2,0);
\draw[thick,blue] (-2,-1) to[out=0,in=180] (0,-0.2) to[out=0,in=180] (2,1);
\draw[fill,blue] (0,-0.2) circle (2pt);
\end{scope}
\begin{scope}[yshift=-.2cm]
\draw[thick,red] (-2,1) to[out=0,in=180] (0,0) to[out=0,in=180] (2,-1);
\draw[thick,red] (-2,0) to[out=0,in=180] (0,0) to[out=0,in=180] (2,0);
\draw[thick,red] (-2,-1) to[out=0,in=180] (0,0) to[out=0,in=180] (2,1);
\draw[fill,red] (0,0) circle (2pt);
\end{scope}
\begin{scope}
\draw[thick] (-2,1) to[out=0,in=180] (0,0) to[out=0,in=180] (2,-1);
\draw[thick] (-2,0) to[out=0,in=180] (0,0) to[out=0,in=180] (2,0);
\draw[thick] (-2,-1) to[out=0,in=180] (0,0) to[out=0,in=180] (2,1);
\draw[fill] (0,0) circle (2pt);
\end{scope}
\draw[fill,gray,opacity=0.3] (0,0) circle (0.2);
\draw[fill,gray,opacity=0.3] (0,0) circle (.8 and .5);
\draw[fill,gray,opacity=0.3] (0,0) circle (1.15);
\end{tikzpicture}&
\cT^{\p4}&=
\begin{tikzpicture}[baseline=-.5ex,scale=0.5]
\begin{scope}[yshift=-.6cm]
\draw[thick,green] (-2,1) to[out=0,in=180] (0,-1.3) to[out=0,in=180] (2,-1);
\draw[thick,green] (-2,0) to[out=0,in=180] (0,-1.3) to[out=0,in=180] (2,0);
\draw[thick,green] (-2,-1) to[out=0,in=180] (0,-1.3) to[out=0,in=180] (2,1);
\draw[fill,green] (0,-1.3) circle (2pt);
\end{scope}
\begin{scope}[yshift=-.4cm]
\draw[thick,blue] (-2,1) to[out=0,in=180] (0,-0.2) to[out=0,in=180] (2,-1);
\draw[thick,blue] (-2,0) to[out=0,in=180] (0,-0.2) to[out=0,in=180] (2,0);
\draw[thick,blue] (-2,-1) to[out=0,in=180] (0,-0.2) to[out=0,in=180] (2,1);
\draw[fill,blue] (0,-0.2) circle (2pt);
\end{scope}
\begin{scope}[yshift=-.2cm]
\draw[thick,red] (-2,1) to[out=0,in=180] (0,0) to[out=0,in=180] (2,-1);
\draw[thick,red] (-2,0) to[out=0,in=180] (0,0) to[out=0,in=180] (2,0);
\draw[thick,red] (-2,-1) to[out=0,in=180] (0,0) to[out=0,in=180] (2,1);
\draw[fill,red] (0,0) circle (2pt);
\end{scope}
\begin{scope}
\draw[thick] (-2,1) to[out=0,in=180] (0,0) to[out=0,in=180] (2,-1);
\draw[thick] (-2,0) to[out=0,in=180] (0,0) to[out=0,in=180] (2,0);
\draw[thick] (-2,-1) to[out=0,in=180] (0,0) to[out=0,in=180] (2,1);
\draw[fill] (0,0) circle (2pt);
\end{scope}
\draw[fill,gray,opacity=0.3] (0,0) circle (0.2);
\draw[fill,gray,opacity=0.3] (0,0) circle (.8 and .5);
\draw[fill,gray,opacity=0.3] (0,0) circle (1.15);
\draw[fill,gray,opacity=0.3] (0,-0.1) circle (1.65);
\end{tikzpicture}
\end{aligned}
$}
\subfigure[Canonical Lagrangian parallel copies near a vertex]{$
\begin{aligned}
\cT_\lag&=
\begin{tikzpicture}[baseline=-.5ex,scale=0.5]
\draw[thick] (-2,1)-- +(4,-2);
\draw[thick] (-2,0)--+(4,0);
\draw[thick] (-2,-1)--+(4,2);
\draw[fill] (0,0) circle (2pt);
\draw[fill,gray,opacity=0.3] (0,0) circle (0.2);
\end{tikzpicture}&
\cT_\lag^{\p2}&=
\begin{tikzpicture}[baseline=-.5ex,scale=0.5]
\begin{scope}[yshift=-.2cm]
\draw[thick,red] (-2,1)-- +(4,-2);
\draw[thick,red] (-2,0)--+(4,0);
\draw[thick,red] (-2,-1)--+(4,2);
\draw[fill,red] (0,0) circle (2pt);
\end{scope}
\begin{scope}
\draw[line width=3,white] (-2,1)-- +(4,-2);
\draw[line width=3,white] (-2,0)--+(4,0);
\draw[line width=3,white] (-2,-1)--+(4,2);
\end{scope}
\begin{scope}
\draw[thick] (-2,1)-- +(4,-2);
\draw[thick] (-2,0)--+(4,0);
\draw[thick] (-2,-1)--+(4,2);
\draw[fill] (0,0) circle (2pt);
\end{scope}
\draw[fill,gray,opacity=0.3] (0,0) circle (0.2);
\draw[fill,gray,opacity=0.3] (0,0) circle (.5);
\end{tikzpicture}&
\cT_\lag^{\p3}&=
\begin{tikzpicture}[baseline=-.5ex,scale=0.5]
\begin{scope}[yshift=-.4cm]
\draw[thick,blue] (-2,1) to[out=-26,in=154] (0,-0.2) to[out=-26,in=154] (2,-1);
\draw[thick,blue] (-2,0) to[out=0,in=180] (0,-0.2) to[out=0,in=180] (2,0);
\draw[thick,blue] (-2,-1) to[out=26,in=206] (0,-0.2) to[out=26,in=206] (2,1);
\draw[fill,blue] (0,-0.2) circle (2pt);
\end{scope}
\begin{scope}[yshift=-.2cm]
\draw[line width=3,white] (-2,1)-- +(4,-2);
\draw[line width=3,white] (-2,0)--+(4,0);
\draw[line width=3,white] (-2,-1)--+(4,2);
\end{scope}
\begin{scope}[yshift=-.2cm]
\draw[thick,red] (-2,1)-- +(4,-2);
\draw[thick,red] (-2,0)--+(4,0);
\draw[thick,red] (-2,-1)--+(4,2);
\draw[fill,red] (0,0) circle (2pt);
\end{scope}
\begin{scope}
\draw[line width=3,white] (-2,1)-- +(4,-2);
\draw[line width=3,white] (-2,0)--+(4,0);
\draw[line width=3,white] (-2,-1)--+(4,2);
\end{scope}
\begin{scope}
\draw[thick] (-2,1)-- +(4,-2);
\draw[thick] (-2,0)--+(4,0);
\draw[thick] (-2,-1)--+(4,2);
\draw[fill] (0,0) circle (2pt);
\end{scope}
\draw[fill,gray,opacity=0.3] (0,0) circle (0.2);
\draw[fill,gray,opacity=0.3] (0,0) circle (.5);
\draw[fill,gray,opacity=0.3] (0,0) circle (1.15);
\end{tikzpicture}&
\cT_\lag^{\p4}&=
\begin{tikzpicture}[baseline=-.5ex,scale=0.5]
\begin{scope}[yshift=-.6cm]
\draw[thick,green] (-2,1) to[out=-26,in=154] (0,-1.2) to[out=-26,in=154] (2,-1);
\draw[thick,green] (-2,0) to[out=0,in=180] (0,-1.2) to[out=0,in=180] (2,0);
\draw[thick,green] (-2,-1) to[out=26,in=206] (0,-1.2) to[out=26,in=206] (2,1);
\draw[fill,green] (0,-1.2) circle (2pt);
\end{scope}
\begin{scope}[yshift=-.4cm]
\draw[line width=3,white] (-2,1) to[out=-26,in=154] (0,-0.2) to[out=-26,in=154] (2,-1);
\draw[line width=3,white] (-2,0) to[out=0,in=180] (0,-0.2) to[out=0,in=180] (2,0);
\draw[line width=3,white] (-2,-1) to[out=26,in=206] (0,-0.2) to[out=26,in=206] (2,1);
\end{scope}
\begin{scope}[yshift=-.4cm]
\draw[thick,blue] (-2,1) to[out=-26,in=154] (0,-0.2) to[out=-26,in=154] (2,-1);
\draw[thick,blue] (-2,0) to[out=0,in=180] (0,-0.2) to[out=0,in=180] (2,0);
\draw[thick,blue] (-2,-1) to[out=26,in=206] (0,-0.2) to[out=26,in=206] (2,1);
\draw[fill,blue] (0,-0.2) circle (2pt);
\end{scope}
\begin{scope}[yshift=-.2cm]
\draw[line width=3,white] (-2,1)-- +(4,-2);
\draw[line width=3,white] (-2,0)--+(4,0);
\draw[line width=3,white] (-2,-1)--+(4,2);
\end{scope}
\begin{scope}[yshift=-.2cm]
\draw[thick,red] (-2,1)-- +(4,-2);
\draw[thick,red] (-2,0)--+(4,0);
\draw[thick,red] (-2,-1)--+(4,2);
\draw[fill,red] (0,0) circle (2pt);
\end{scope}
\begin{scope}
\draw[line width=3,white] (-2,1)-- +(4,-2);
\draw[line width=3,white] (-2,0)--+(4,0);
\draw[line width=3,white] (-2,-1)--+(4,2);
\end{scope}
\begin{scope}
\draw[thick] (-2,1)-- +(4,-2);
\draw[thick] (-2,0)--+(4,0);
\draw[thick] (-2,-1)--+(4,2);
\draw[fill] (0,0) circle (2pt);
\end{scope}
\draw[fill,gray,opacity=0.3] (0,0) circle (0.2);
\draw[fill,gray,opacity=0.3] (0,0) circle (.5);
\draw[fill,gray,opacity=0.3] (0,0) circle (1.15);
\draw[fill,gray,opacity=0.3] (0,-0.1) circle (1.65);
\end{tikzpicture}
\end{aligned}
$}
\caption{Canonical copies near a vertex}
\label{figure:canonical copies near vertices}
\end{figure}
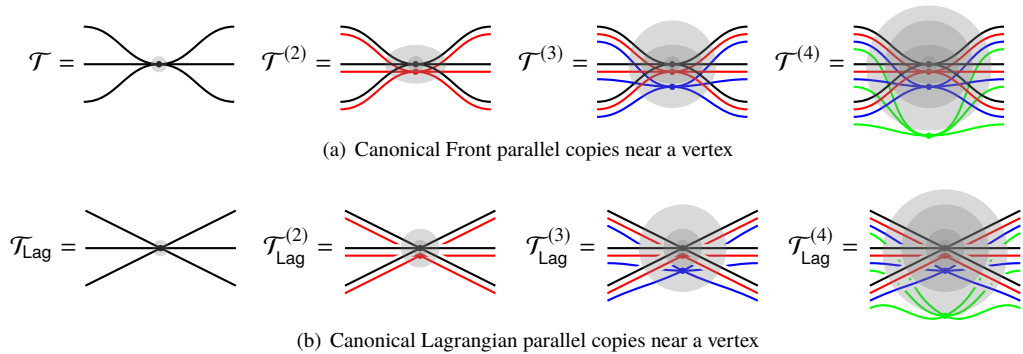

One can choose (zig-zag) sequences of consistent front and Lagrangian Reidemeister moves as follows:
\[
\begin{tikzcd}[column sep=2.5pc]
\vcenter{\hbox{\includegraphics[scale=0.7]{front_RM_I_b.pdf}}}\arrow[rrr,"\RM{I}"] \arrow[d, "(\cdot)^{\p m}"]& & &
\vcenter{\hbox{\includegraphics[scale=0.7]{front_RM_I_a.pdf}}} \arrow[d, "(\cdot)^{\p m}"]\\
\vcenter{\hbox{\includegraphics[scale=0.6]{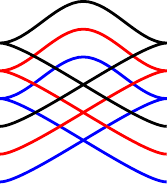}}}\arrow[r, Rightarrow, "N_1\RM{III}"] & 
\vcenter{\hbox{\includegraphics[scale=0.6]{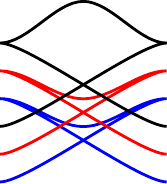}}}\arrow[r, Rightarrow, "m(m-1)\RM{II}"] &
\vcenter{\hbox{\includegraphics[scale=0.6]{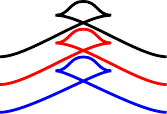}}}\arrow[r, Rightarrow, "m\RM{I}"] &
\vcenter{\hbox{\includegraphics[scale=0.6]{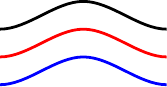}}}
\end{tikzcd}
\qquad
\begin{tikzcd}
\vcenter{\hbox{\includegraphics[scale=0.7]{front_RM_III_b.pdf}}}\arrow[r,leftrightarrow,"\RM{III}"]\arrow[d, "(\cdot)^{\p m}"]&
\vcenter{\hbox{\includegraphics[scale=0.7]{front_RM_III_a.pdf}}}\arrow[d, "(\cdot)^{\p m}"]\\
\vcenter{\hbox{\includegraphics[scale=0.6]{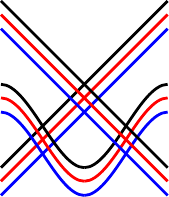}}}\arrow[r, Rightarrow, "m^3\RM{III}"] & 
\vcenter{\hbox{\includegraphics[scale=0.6]{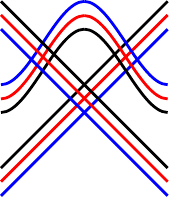}}}
\end{tikzcd}
\]
\[
\begin{tikzcd}
\vcenter{\hbox{\includegraphics[scale=0.7]{front_RM_II_a.pdf}}}\arrow[rr,"\RM{II}"]\arrow[d, "(\cdot)^{\p m}"]& &
\vcenter{\hbox{\includegraphics[scale=0.7]{front_RM_II_b.pdf}}}\arrow[d, "(\cdot)^{\p m}"]\\
\vcenter{\hbox{\includegraphics[scale=0.6]{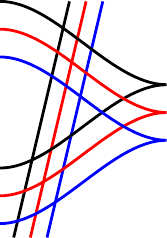}}}\arrow[r, Rightarrow, "m\binom{m}2\RM{III}"] & 
\vcenter{\hbox{\includegraphics[scale=0.6]{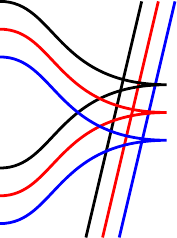}}}\arrow[r, Rightarrow, "m^2\RM{II}"] &
\vcenter{\hbox{\includegraphics[scale=0.6]{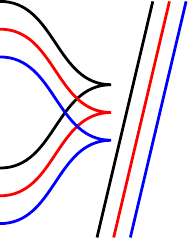}}}
\end{tikzcd}
\qquad
\begin{tikzcd}
\begin{tikzpicture}[baseline=-.5ex,xscale=-1,scale=0.6]
\draw[densely dotted](0,0) circle (1);
\clip(0,0) circle (1);
\draw[thick] (1,.6) to[out=180,in=0] (-0.5,0);
\draw[thick] (1,.2) to[out=180,in=0] (-0.5,0);
\draw[thick] (1,-.2) to[out=180,in=0] (-0.5,0);
\draw[thick] (1,-.6) to[out=180,in=0] (-0.5,0);
\draw[fill] (-0.5,0) circle (2pt);
\draw[thick] (0,1) -- (0.5,-1);
\end{tikzpicture}\arrow[rr,"\RM{V}"]\arrow[d, "(\cdot)^{\p m}"] & &
\begin{tikzpicture}[baseline=-.5ex,xscale=-1,scale=0.6]
\draw[densely dotted](0,0) circle (1);
\clip(0,0) circle (1);
\draw[thick] (1,.6) to[out=180,in=0] (0,0);
\draw[thick] (1,.2) to[out=180,in=0] (0,0);
\draw[thick] (1,-.2) to[out=180,in=0] (0,0);
\draw[thick] (1,-.6) to[out=180,in=0] (0,0);
\draw[fill] (0,0) circle (2pt);
\draw[thick] (-0.5,1) -- (0,-1);
\end{tikzpicture}\arrow[d, "(\cdot)^{\p m}"]\\
\vcenter{\hbox{\includegraphics[scale=0.5]{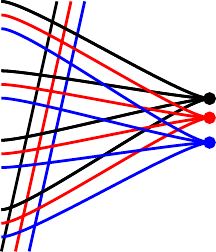}}}\arrow[r, Rightarrow, "N_2\RM{III}"] & 
\vcenter{\hbox{\includegraphics[scale=0.5]{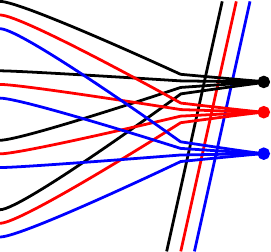}}}\arrow[r, Rightarrow, "m^2\RM{V}"] & 
\vcenter{\hbox{\includegraphics[scale=0.5]{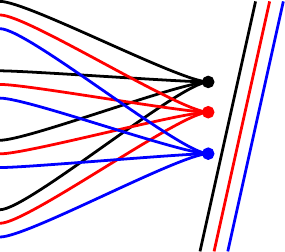}}}
\end{tikzcd}
\]
\[
\begin{tikzcd}[column sep=1.5pc]
\vcenter{\hbox{\includegraphics[scale=0.7]{front_RM_VI_b.pdf}}}\arrow[rrrr,"\RM{VI}"]\arrow[d, "(\cdot)^{\p m}"]& & & &
\vcenter{\hbox{\includegraphics[scale=0.7]{front_RM_VI_a.pdf}}}\arrow[d, "(\cdot)^{\p m}"]\\
\vcenter{\hbox{\includegraphics[scale=0.5,trim={0.5cm 0 0.5cm 0}, clip]{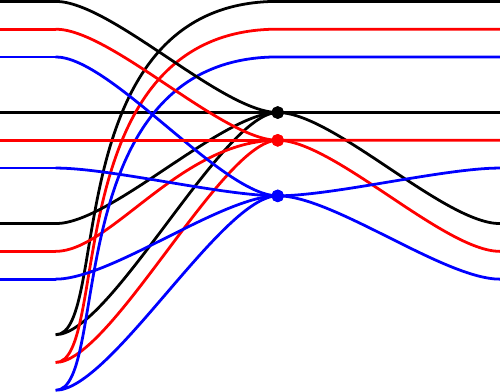}}}\arrow[r, Rightarrow, "N_2\RM{III}"] & 
\vcenter{\hbox{\includegraphics[scale=0.5,trim={0.5cm 0 0.5cm 0}, clip]{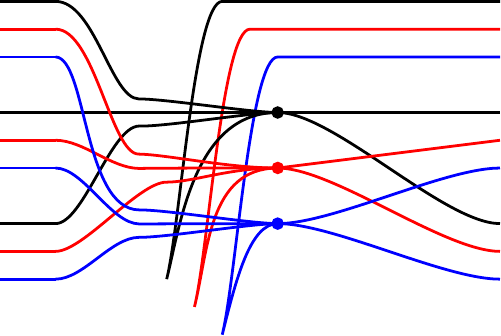}}}\arrow[r, Rightarrow, "N_3\RM{II}"] &
\vcenter{\hbox{\includegraphics[scale=0.5,trim={0.5cm 0 0.5cm 0}, clip]{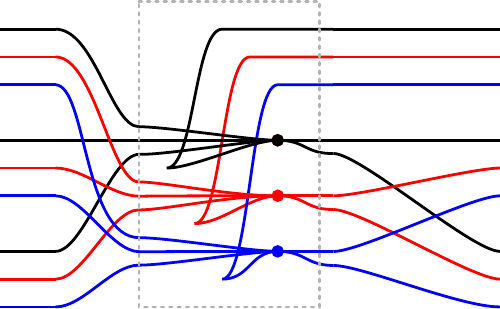}}}\arrow[r, Rightarrow, dashed, "\RM{M}"] & 
\vcenter{\hbox{\includegraphics[scale=0.5,trim={0.5cm 0 0.5cm 0}, clip]{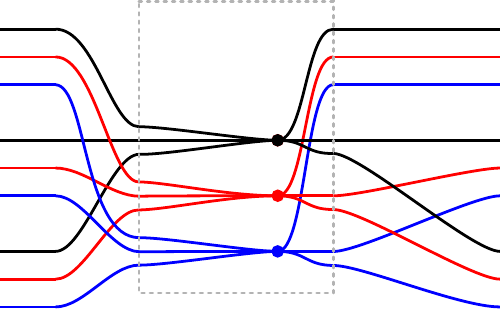}}}\arrow[r, Rightarrow, "N_4\RM{III}"] &
\vcenter{\hbox{\includegraphics[scale=0.5,trim={0.6cm 0 0.4cm 0}, clip]{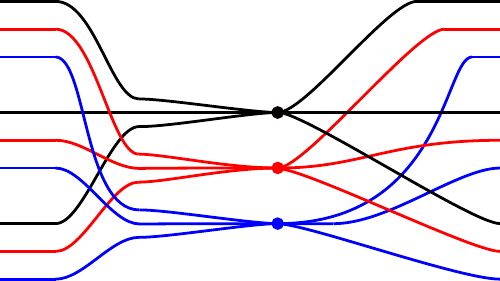}}}\\
\end{tikzcd},
\]
\[
\RM{M}=
\begin{tikzcd}
\vcenter{\hbox{\includegraphics[scale=0.7]{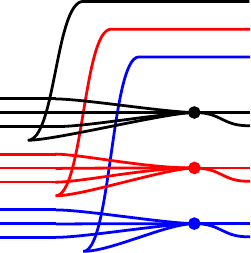}}}\arrow[from=r, Rightarrow, "rm\RM{VI}"'] \arrow[d, dashed, Rightarrow, "\RM{M}"]& 
\vcenter{\hbox{\includegraphics[scale=0.7]{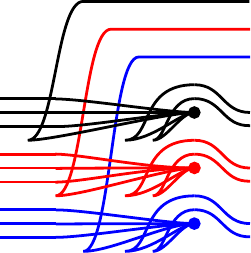}}}\arrow[from=r, Rightarrow, "r\binom{m}2\RM{II}"'] & 
\vcenter{\hbox{\includegraphics[scale=0.7]{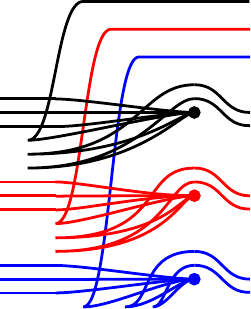}}}\arrow[d, Rightarrow, "N_5\RM{III}"] \\
\vcenter{\hbox{\includegraphics[scale=0.7]{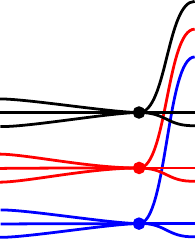}}}\arrow[from=r, Rightarrow, "rm\RM{VI}"] &
\vcenter{\hbox{\includegraphics[scale=0.7]{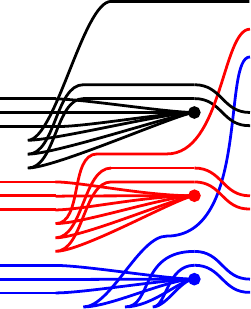}}}\arrow[from=r, Rightarrow, "\binom{m}2\RM{V}"] & 
\vcenter{\hbox{\includegraphics[scale=0.7]{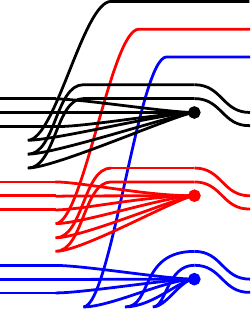}}}
\end{tikzcd}
\]
\[
\begin{tikzcd}
\vcenter{\hbox{\includegraphics[scale=0.7]{RM_0_a_1.pdf}}}\arrow[rrrr,"\RM{0_a}"]\arrow[d, "(\cdot)^{\p m}"]& & & &
\vcenter{\hbox{\includegraphics[scale=0.7]{RM_0_a_2.pdf}}}\arrow[d, "(\cdot)^{\p m}"]\\
\vcenter{\hbox{\includegraphics[scale=0.5]{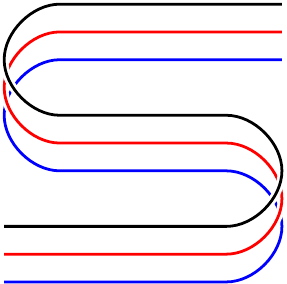}}}\arrow[r, Rightarrow, "2\binom{m}2\RM{0_c}"] &
\vcenter{\hbox{\includegraphics[scale=0.5]{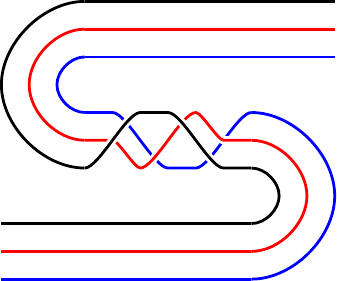}}}\arrow[r, Rightarrow, "\binom{m}3\RM{iii_a}"] &
\vcenter{\hbox{\includegraphics[scale=0.5]{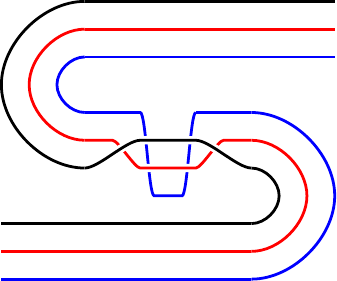}}}\arrow[r, Rightarrow, "\binom{m}2\RM{ii}"] &
\vcenter{\hbox{\includegraphics[scale=0.5]{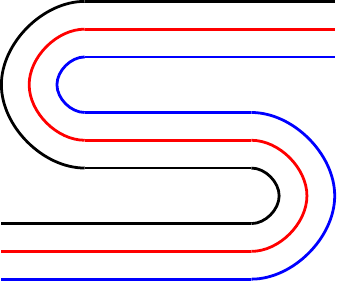}}}\arrow[r, Rightarrow, "m\RM{0_a}"] &
\vcenter{\hbox{\includegraphics[scale=0.5]{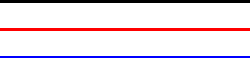}}}
\end{tikzcd}
\]
\[
\begin{tikzcd}
\vcenter{\hbox{\includegraphics[scale=0.7]{RM_0_b_1.pdf}}}\arrow[rrr,"\RM{0_b}"]\arrow[d, "(\cdot)^{\p m}"]& & & 
\vcenter{\hbox{\includegraphics[scale=0.7]{RM_0_b_2.pdf}}}\arrow[ddd, "(\cdot)^{\p m}"]\\
\vcenter{\hbox{\includegraphics[scale=0.5]{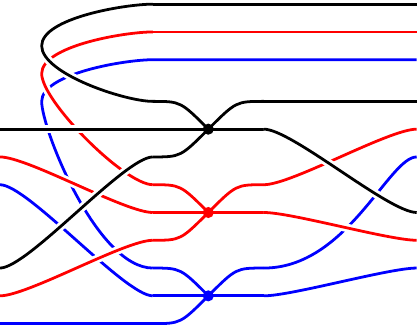}}}\arrow[r, Rightarrow, "\binom{m}2\RM{0_c}"] &
\vcenter{\hbox{\includegraphics[scale=0.5]{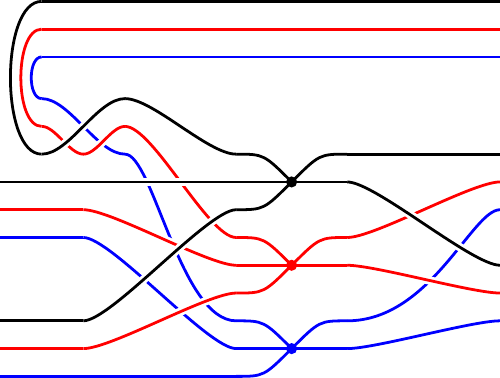}}}\arrow[r, Rightarrow, "N\RM{iii_a}"] &
\vcenter{\hbox{\includegraphics[scale=0.5]{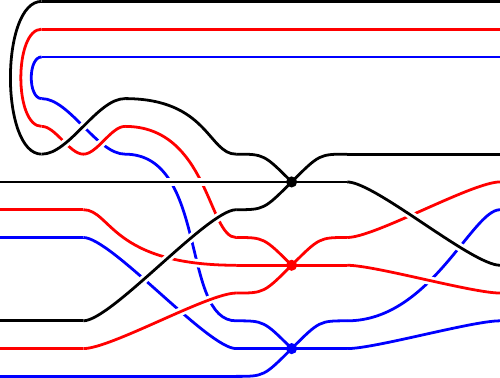}}}\arrow[d, Rightarrow, "N\RM{iii_a}"]\\
\vcenter{\hbox{\includegraphics[scale=0.5]{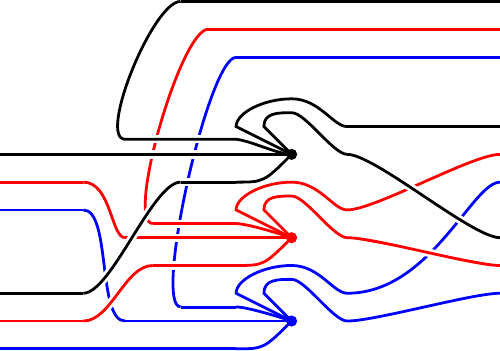}}}\arrow[r, Rightarrow, "mr\RM{iv}"]\arrow[d, Rightarrow, "r\binom{m}2\RM{ii}"] &
\vcenter{\hbox{\includegraphics[scale=0.5]{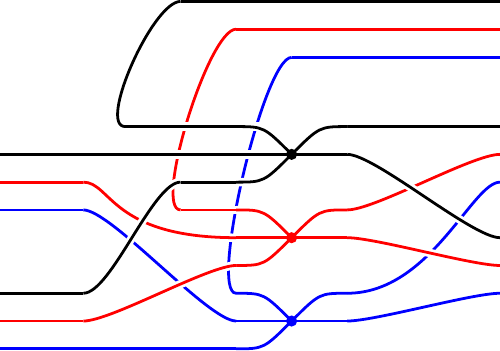}}}\arrow[from=r, Rightarrow, "\RM{0_c}"] &
\vcenter{\hbox{\includegraphics[scale=0.5]{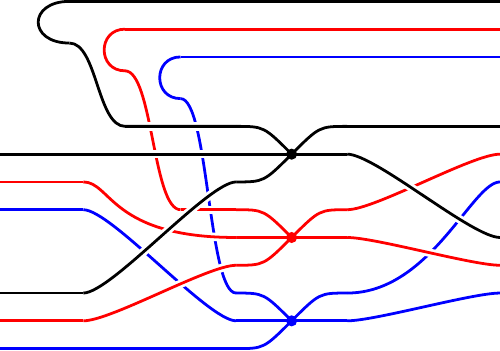}}}\\
\vcenter{\hbox{\includegraphics[scale=0.5]{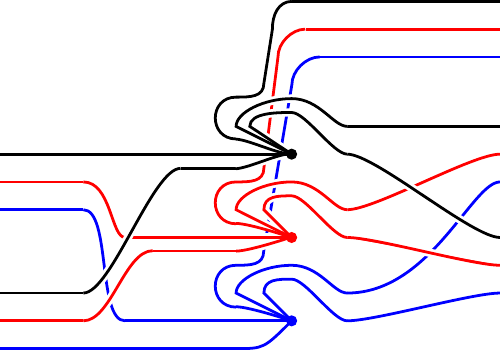}}}\arrow[r, Rightarrow, "\binom{m}2\RM{iv}"] &
\vcenter{\hbox{\includegraphics[scale=0.5]{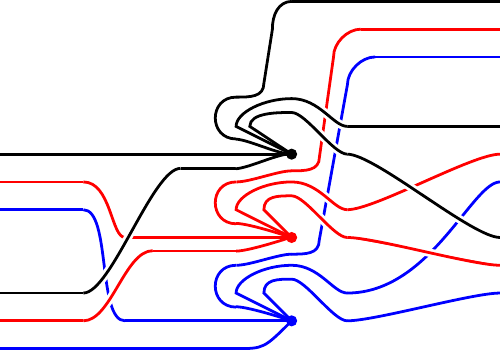}}}\arrow[r, Rightarrow, "(r+1)m\RM{0_b}"] &
\vcenter{\hbox{\includegraphics[scale=0.5]{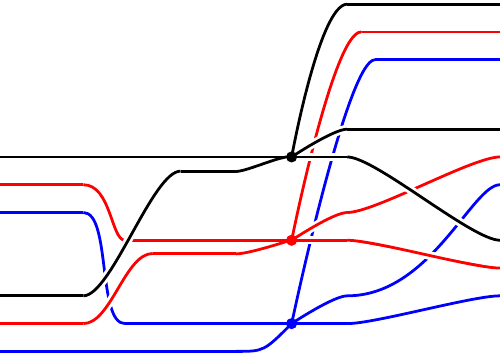}}}\arrow[r, Rightarrow, "\RM{iii_a}"] &
\vcenter{\hbox{\includegraphics[scale=0.5]{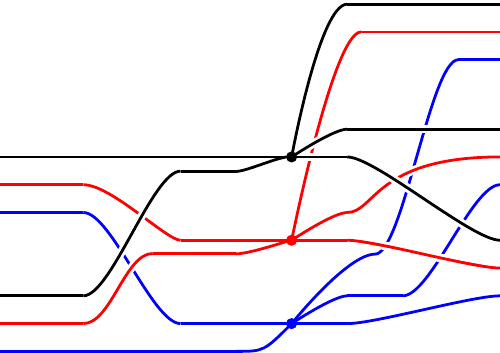}}}
\end{tikzcd}
\]
\[
\begin{tikzcd}
\vcenter{\hbox{\includegraphics[scale=0.7]{RM_0_c_1.pdf}}}\arrow[rrrr,"\RM{0_c}"]\arrow[d, "(\cdot)^{\p m}"]& & & & 
\vcenter{\hbox{\includegraphics[scale=0.7]{RM_0_c_2.pdf}}}\arrow[d, "(\cdot)^{\p m}"]\\
\vcenter{\hbox{\includegraphics[scale=0.5]{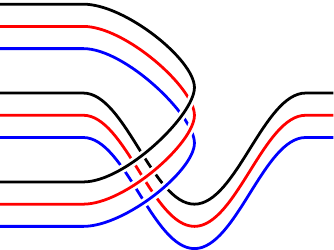}}}\arrow[r, Rightarrow, "\binom{m}2\RM{0_c}"] &
\vcenter{\hbox{\includegraphics[scale=0.5]{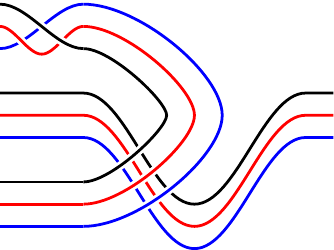}}}\arrow[r, Rightarrow, "m\RM{0_c}"] &
\vcenter{\hbox{\includegraphics[scale=0.5]{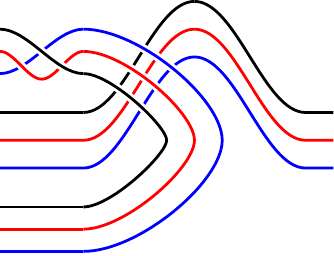}}}\arrow[r, Rightarrow, "m\binom{m}2\RM{iii_a}"] &
\vcenter{\hbox{\includegraphics[scale=0.5]{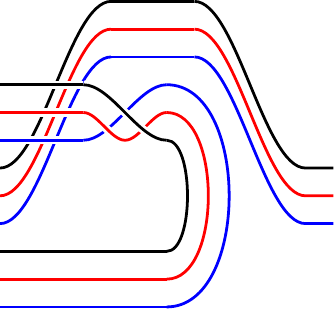}}}\arrow[r, Rightarrow, "\binom{m}2\RM{0_c}"] &
\vcenter{\hbox{\includegraphics[scale=0.5]{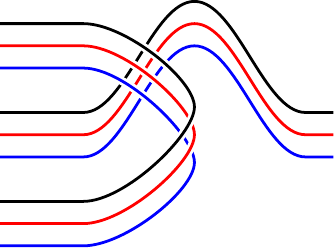}}}
\end{tikzcd}
\]
\[
\begin{tikzcd}
\begin{tikzpicture}[baseline=-.5ex,scale=0.6]
\draw[densely dotted](0,0) circle (1);
\clip(0,0) circle (1);
\draw[thick] (-0.5,-1) -- (0,1);
\draw[white,line width=5] (-1,.6) to[out=0,in=90] (0.5,0) to[out=-90,in=0] (-1,-0.6);
\draw[thick] (-1,.6) to[out=0,in=90] (0.5,0) to[out=-90,in=0] (-1,-0.6);
\end{tikzpicture}\arrow[rr,"\RM{ii}"]\arrow[d,"(\cdot)^{\p m}"] & &
\begin{tikzpicture}[baseline=-.5ex,scale=0.6]
\draw[densely dotted](0,0) circle (1);
\clip(0,0) circle (1);
\draw[thick] (0,-1) -- (0.5,1);
\draw[white,line width=5] (-1,.6) to[out=0,in=90] (0,0) to[out=-90,in=0] (-1,-0.6);
\draw[thick] (-1,.6) to[out=0,in=90] (0,0) to[out=-90,in=0] (-1,-0.6);
\end{tikzpicture}\arrow[d,"(\cdot)^{\p m}"]\\
\vcenter{\hbox{\includegraphics[scale=0.6]{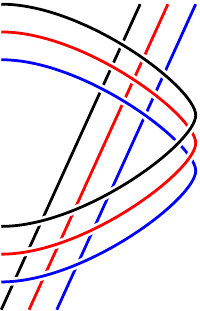}}}\arrow[r, Rightarrow, "m\binom{m}2\RM{iii_b}"] &
\vcenter{\hbox{\includegraphics[scale=0.6]{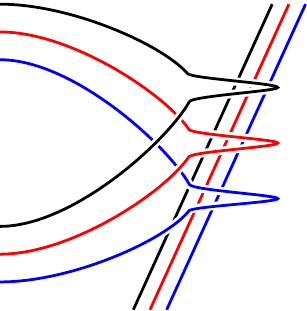}}}\arrow[r, Rightarrow, "m^2\RM{ii}"] &
\vcenter{\hbox{\includegraphics[scale=0.6]{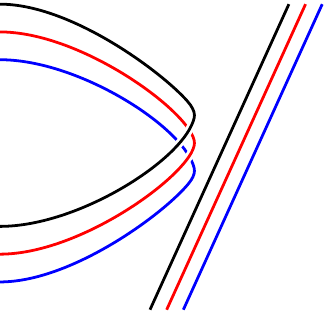}}}
\end{tikzcd}
\]
\[
\begin{tikzcd}
\vcenter{\hbox{\rotatebox[origin=c]{90}{\includegraphics[scale=0.7]{RM_III_b_1.pdf}}}}\arrow[r,"\RM{iii_a}"]\arrow[d, "(\cdot)^{\p m}"]&
\vcenter{\hbox{\rotatebox[origin=c]{90}{\includegraphics[scale=0.7]{RM_III_b_2.pdf}}}}\arrow[d, "(\cdot)^{\p m}"]\\
\vcenter{\hbox{\includegraphics[scale=0.5]{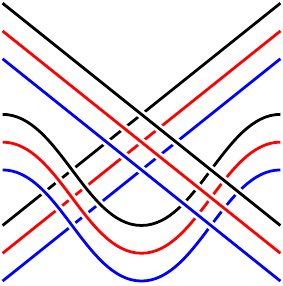}}}\arrow[r,"m^3 \RM{iii_a}"] &
\vcenter{\hbox{\includegraphics[scale=0.5]{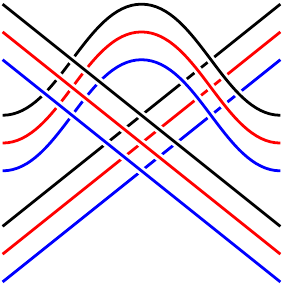}}}
\end{tikzcd}
\qquad
\begin{tikzcd}
\vcenter{\hbox{\rotatebox[origin=c]{90}{\includegraphics[scale=0.7]{RM_III_a_1.pdf}}}}\arrow[r,"\RM{iii_b}"]\arrow[d, "(\cdot)^{\p m}"]&
\vcenter{\hbox{\rotatebox[origin=c]{90}{\includegraphics[scale=0.7]{RM_III_a_2.pdf}}}}\arrow[d, "(\cdot)^{\p m}"]\\
\vcenter{\hbox{\includegraphics[scale=0.5]{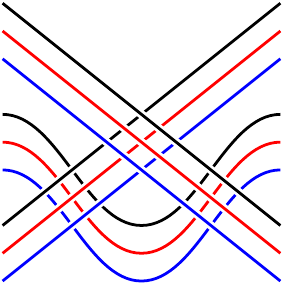}}}\arrow[r,"m^3 \RM{iii_b}"] &
\vcenter{\hbox{\includegraphics[scale=0.5]{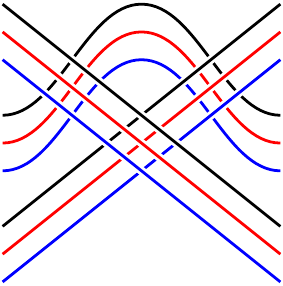}}}
\end{tikzcd}
\]
\[
\begin{tikzcd}[column sep=2.5pc]
\begin{tikzpicture}[baseline=-.5ex,scale=0.6]
\draw[densely dotted](0,0) circle (1);
\clip(0,0) circle (1);
\draw[thick] (-0.5,-1) -- (0,1);
\draw[white,line width=5] (-1,.6) -- (0.5,0) (-1,.2) -- (0.5,0) (-1,-.2) -- (0.5,0) (-1,-.6) -- (0.5,0);
\draw[thick] (-1,.6) -- (0.5,0) (-1,.2) -- (0.5,0) (-1,-.2) -- (0.5,0) (-1,-.6) -- (0.5,0);
\draw[fill] (0.5,0) circle (2pt);
\end{tikzpicture}\arrow[rr,"\RM{iv}"]\arrow[d,"(\cdot)^{\p m}"]& &
\begin{tikzpicture}[baseline=-.5ex,scale=0.6]
\draw[densely dotted](0,0) circle (1);
\clip(0,0) circle (1);
\draw[thick] (-1,.6) -- (0,0) (-1,.2) -- (0,0) (-1,-.2) -- (0,0) (-1,-.6) -- (0,0);
\draw[fill] (0,0) circle (2pt);
\draw[thick] (0,-1) -- (0.5,1);
\end{tikzpicture}\arrow[d,"(\cdot)^{\p m}"]\\
\vcenter{\hbox{\includegraphics[scale=0.5]{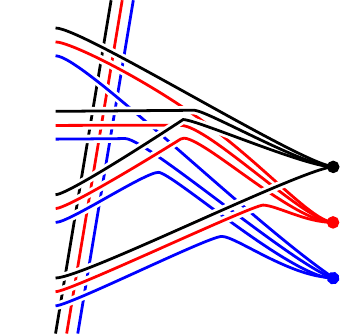}}}\arrow[r, Rightarrow, "m\binom{m}2\binom{n}2 \RM{iii_b}"] &
\vcenter{\hbox{\includegraphics[scale=0.5]{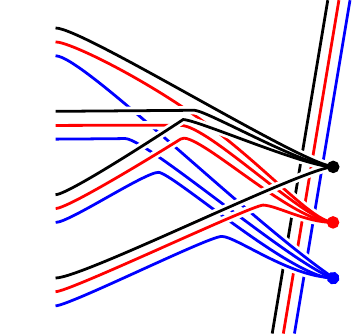}}}\arrow[r, Rightarrow, "m^2 \RM{iv_b}"] &
\vcenter{\hbox{\includegraphics[scale=0.5]{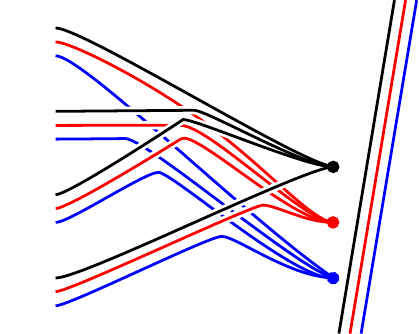}}}
\end{tikzcd}
\]

Here, the numbers $N_i$ are given as
\begin{align*}
N_1&=\frac{m(m-1)(2m-1)}6,&
N_2&=\binom m2\binom \ell 2,&
N_3&=\binom m2(\ell-1),&
N_4&=\binom{m}3\binom{r}2,&
N_5&=\binom m2(2\ell+1),
\end{align*}
where $\ell$ and $r$ are the numbers of half-edges on the left and right, respectively.

In addition, one can choose the consistent morphisms corresponding to the operations $\RM{B_*}$ on basepoints as depicted in Figure~\ref{figure:consistent basepoint moves}.

\begin{figure}[ht]
\[
\begin{tikzcd}[ampersand replacement=\&]
\begin{tikzpicture}[baseline=-.5ex,scale=0.7]
\draw[dashed](0,0) circle (1);
\clip(0,0) circle (1);
\begin{scope}[yshift=0cm]
\draw[thick] (-1,0.5)--node[midway,sloped,below=-2ex]{$|$} (0,0) -- (1,-0.5);
\draw[thick] (-1,-0.5) -- (1,0.5);
\end{scope}
\end{tikzpicture}\arrow[r, "\RM{B_1}"]\arrow[d, "(\cdot)^{\p m}"] \&
\begin{tikzpicture}[baseline=-.5ex,scale=0.7]
\draw[dashed](0,0) circle (1);
\clip(0,0) circle (1);
\begin{scope}[yshift=0cm]
\draw[thick] (-1,0.5)-- (0,0) --node[midway,sloped,below=-2ex]{$|$} (1,-0.5);
\draw[thick] (-1,-0.5) -- (1,0.5);
\end{scope}
\end{tikzpicture}\arrow[d, "(\cdot)^{\p m}"]\\
\begin{tikzpicture}[baseline=-.5ex,scale=0.7]
\begin{scope}[yshift=0.3cm]
\draw[thick] (-1,0.5)--node[pos=0.6,sloped,below=-2ex]{$|$} (0,0) -- (1,-0.5);
\draw[thick] (-1,-0.5) -- (1,0.5);
\end{scope}
\begin{scope}[yshift=0cm]
\draw[thick,red] (-1,0.5)--node[pos=0.4,sloped,below=-2ex]{$|$} (0,0) -- (1,-0.5);
\draw[thick,red] (-1,-0.5) -- (1,0.5);
\end{scope}
\begin{scope}[yshift=-0.3cm]
\draw[thick,blue] (-1,0.5)--node[pos=0.2,sloped,below=-2ex]{$|$} (0,0) -- (1,-0.5);
\draw[thick,blue] (-1,-0.5) -- (1,0.5);
\end{scope}
\end{tikzpicture}\arrow[r, Rightarrow, "m^2\RM{B_1}"] \&
\begin{tikzpicture}[baseline=-.5ex,scale=0.7]
\begin{scope}[yshift=0.3cm]
\draw[thick] (-1,0.5)-- (0,0) --node[pos=0.8,sloped,below=-2ex]{$|$} (1,-0.5);
\draw[thick] (-1,-0.5) -- (1,0.5);
\end{scope}
\begin{scope}[yshift=0cm]
\draw[thick,red] (-1,0.5)-- (0,0) --node[pos=0.6,sloped,below=-2ex]{$|$} (1,-0.5);
\draw[thick,red] (-1,-0.5) -- (1,0.5);
\end{scope}
\begin{scope}[yshift=-0.3cm]
\draw[thick,blue] (-1,0.5)-- (0,0) --node[pos=0.4,sloped,below=-2ex]{$|$} (1,-0.5);
\draw[thick,blue] (-1,-0.5) -- (1,0.5);
\end{scope}
\end{tikzpicture}
\end{tikzcd}
\qquad
\begin{tikzcd}[ampersand replacement=\&]
\begin{tikzpicture}[baseline=-.5ex,scale=0.7]
\draw[dashed](0,0) circle (1);
\clip(0,0) circle (1);
\draw[thick] (-1,0)-- (0,0) node[below=-2ex]{$|$} -- (1,0);
\end{tikzpicture}\arrow[r,"\RM{B_2}"]\arrow[d, "(\cdot)^{\p m}"] \&
\begin{tikzpicture}[baseline=-.5ex,scale=0.7]
\draw[dashed](0,0) circle (1);
\clip(0,0) circle (1);
\draw[thick] (-1,0)-- node[near start, below=-2ex]{$|$} node[near end, below=-2ex]{$|$} (1,0);
\end{tikzpicture}\arrow[d, "(\cdot)^{\p m}"]\\
\begin{tikzpicture}[baseline=-.5ex,scale=0.7]
\begin{scope}[yshift=0.6cm]
\draw[thick] (-1,0)-- (0,0) node[below=-2ex]{$|$} -- (1,0);
\end{scope}
\begin{scope}[yshift=0cm]
\draw[thick,red] (-1,0)-- (0,0) node[below=-2ex]{$|$} -- (1,0);
\end{scope}
\begin{scope}[yshift=-0.6cm]
\draw[thick,blue] (-1,0)-- (0,0) node[below=-2ex]{$|$} -- (1,0);
\end{scope}
\end{tikzpicture}\arrow[r, Rightarrow, "m\RM{B_2}"]\&
\begin{tikzpicture}[baseline=-.5ex,scale=0.7]
\begin{scope}[yshift=0.6cm]
\draw[thick] (-1,0)-- node[near start, below=-2ex]{$|$} node[near end, below=-2ex]{$|$} (1,0);
\end{scope}
\begin{scope}[yshift=0cm]
\draw[thick,red] (-1,0)-- node[near start, below=-2ex]{$|$} node[near end, below=-2ex]{$|$} (1,0);
\end{scope}
\begin{scope}[yshift=-0.6cm]
\draw[thick,blue] (-1,0)-- node[near start, below=-2ex]{$|$} node[near end, below=-2ex]{$|$} (1,0);
\end{scope}
\end{tikzpicture}
\end{tikzcd}
\]
\[
\begin{tikzcd}[ampersand replacement=\&]
\begin{tikzpicture}[baseline=-.5ex,scale=0.7]
\draw[dashed](0,0) circle (1);
\clip(0,0) circle (1);
\draw[thick] (-1,-0.5) to[out=0,in=180] (0,0) to[out=0,in=180] (1,0.5);
\draw[thick] (-1,0) to[out=0,in=180] (0,0) to[out=0,in=180] (1,0);
\draw[thick] (-1,0.5) to[out=0,in=180] (0,0) to[out=0,in=180] (1,-0.5);
\draw[fill] (0,0) circle (2pt);
\end{tikzpicture}\arrow[rr,"\RM{B_3}"]\arrow[d, "(\cdot)^{\p m}"] \& \&
\begin{tikzpicture}[baseline=-.5ex,scale=0.7]
\draw[dashed](0,0) circle (1);
\clip(0,0) circle (1);
\draw[thick] (-1,-0.5) to[out=0,in=180] (0,0) to[out=0,in=180] node[midway,sloped, below=-2ex] {$|$} (1,0.5);
\draw[thick] (-1,0) to[out=0,in=180] (0,0) to[out=0,in=180] (1,0);
\draw[thick] (-1,0.5) to[out=0,in=180] (0,0) to[out=0,in=180] (1,-0.5);
\draw[fill] (0,0) circle (2pt);
\end{tikzpicture}\arrow[d, "(\cdot)^{\p m}"]\\
\begin{tikzpicture}[baseline=-.5ex,scale=0.7]
\begin{scope}[yshift=0.25cm]
\draw[thick] (-1,-0.8) to[out=0,in=180] (0,0.5) to[out=0,in=180] (1,0.8);
\draw[thick] (-1,0) to[out=0,in=180] (0,0.5) to[out=0,in=180] (1,0);
\draw[thick] (-1,0.8) to[out=0,in=180] (0,0.5) to[out=0,in=180] (1,-0.8);
\draw[fill] (0,0.5) circle (2pt);
\end{scope}
\begin{scope}[yshift=0cm, color=red]
\draw[thick] (-1,-0.8) to[out=0,in=180] (0,0.2) to[out=0,in=180] (1,0.8);
\draw[thick] (-1,0) to[out=0,in=180] (0,0.2) to[out=0,in=180] (1,0);
\draw[thick] (-1,0.8) to[out=0,in=180] (0,0.2) to[out=0,in=180] (1,-0.8);
\draw[fill] (0,0.2) circle (2pt);
\end{scope}
\begin{scope}[yshift=-0.25cm, color=blue]
\draw[thick] (-1,-0.8) to[out=0,in=180] (0,-0.6) to[out=0,in=180] (1,0.8);
\draw[thick] (-1,0) to[out=0,in=180] (0,-0.6) to[out=0,in=180] (1,0);
\draw[thick] (-1,0.8) to[out=0,in=180] (0,-0.6) to[out=0,in=180] (1,-0.8);
\draw[fill] (0,-0.6) circle (2pt);
\end{scope}
\end{tikzpicture}\arrow[r, Rightarrow, "m\RM{B_3}"] \&
\begin{tikzpicture}[baseline=-.5ex,scale=0.7]
\begin{scope}[yshift=0.25cm]
\draw[thick] (-2,-0.8) to[out=0,in=180] (0,0.6) to[out=0,in=180] node[pos=0.4,sloped, below=-2ex] {$|$} (2,0.8);
\draw[thick] (-2,0) to[out=0,in=180] (0,0.6) to[out=0,in=180] (2,0);
\draw[thick] (-2,0.8) to[out=0,in=180] (0,0.6) to[out=0,in=180] (2,-0.8);
\draw[fill] (0,0.6) circle (2pt);
\end{scope}
\begin{scope}[yshift=0cm, color=red]
\draw[thick] (-2,-0.8) to[out=0,in=180] (0,0.2) to[out=0,in=180] node[pos=0.35,sloped, below=-2ex] {$|$} (2,0.8);
\draw[thick] (-2,0) to[out=0,in=180] (0,0.2) to[out=0,in=180] (2,0);
\draw[thick] (-2,0.8) to[out=0,in=180] (0,0.2) to[out=0,in=180] (2,-0.8);
\draw[fill] (0,0.2) circle (2pt);
\end{scope}
\begin{scope}[yshift=-0.25cm, color=blue]
\draw[thick] (-2,-0.8) to[out=0,in=180] (0,-0.6) to[out=0,in=180] node[pos=0.35,sloped, below=-2ex] {$|$} (2,0.8);
\draw[thick] (-2,0) to[out=0,in=180] (0,-0.6) to[out=0,in=180] (2,0);
\draw[thick] (-2,0.8) to[out=0,in=180] (0,-0.6) to[out=0,in=180] (2,-0.8);
\draw[fill] (0,-0.6) circle (2pt);
\end{scope}
\end{tikzpicture}\arrow[r, Rightarrow, "r\binom{m}2\RM{B_1}"] \&
\begin{tikzpicture}[baseline=-.5ex,scale=0.7]
\begin{scope}[yshift=0.25cm]
\draw[thick] (-1,-0.8) to[out=0,in=180] (0,0.5) to[out=0,in=180] node[pos=0.9,sloped, below=-2ex] {$|$} (1,0.8);
\draw[thick] (-1,0) to[out=0,in=180] (0,0.5) to[out=0,in=180] (1,0);
\draw[thick] (-1,0.8) to[out=0,in=180] (0,0.5) to[out=0,in=180] (1,-0.8);
\draw[fill] (0,0.5) circle (2pt);
\end{scope}
\begin{scope}[yshift=0cm, color=red]
\draw[thick] (-1,-0.8) to[out=0,in=180] (0,0.2) to[out=0,in=180] node[pos=0.85,sloped, below=-2ex] {$|$} (1,0.8);
\draw[thick] (-1,0) to[out=0,in=180] (0,0.2) to[out=0,in=180] (1,0);
\draw[thick] (-1,0.8) to[out=0,in=180] (0,0.2) to[out=0,in=180] (1,-0.8);
\draw[fill] (0,0.2) circle (2pt);
\end{scope}
\begin{scope}[yshift=-0.25cm, color=blue]
\draw[thick] (-1,-0.8) to[out=0,in=180] (0,-0.6) to[out=0,in=180] node[pos=0.9,sloped, below=-2ex] {$|$} (1,0.8);
\draw[thick] (-1,0) to[out=0,in=180] (0,-0.6) to[out=0,in=180] (1,0);
\draw[thick] (-1,0.8) to[out=0,in=180] (0,-0.6) to[out=0,in=180] (1,-0.8);
\draw[fill] (0,-0.6) circle (2pt);
\end{scope}
\end{tikzpicture}
\end{tikzcd}
\]
\caption{Consistent basepoint moves on the canonical front copies}
\label{figure:consistent basepoint moves}
\end{figure}
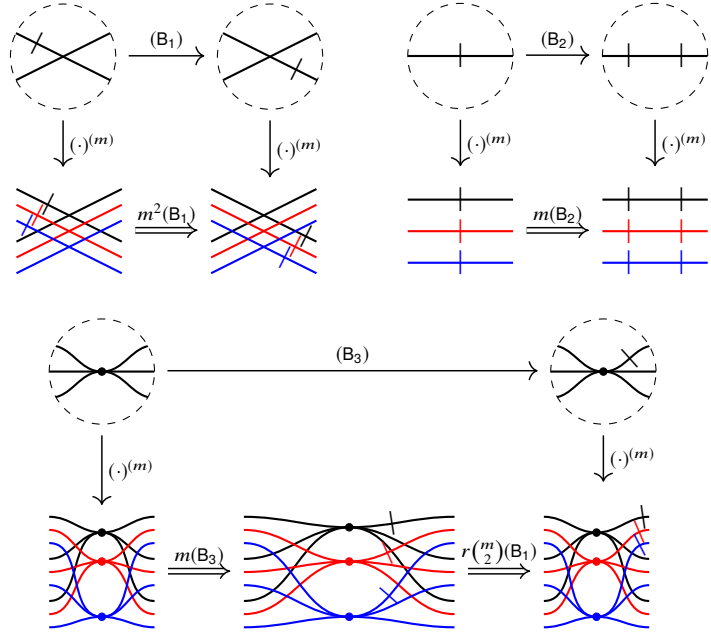

The one important observation is that all parallel copies are consistent and all arrows are elementary.
Therefore in summary, we have the following proposition.
\begin{proposition}\label{proposition:functoriality of canonical Lagrangian copies}
The canonical consistent sequences
\[
\BLT^\mu\to \BLT^{\mu,\p\bullet}\quad\text{ and }\quad
\BLT^\mu_\lag\to \BLT^{\mu,\p\bullet}_\lag
\]
are well-defined and preserves the equivalences.

In particular, each consistent Reidemeister move and basepoint move can be mapped to a zig-zag of elementary consistent Reidemeister moves or basepoint moves.
\end{proposition}

\begin{example}
Due to the definition of the bordered Legendrian graph in a normal form, the canonical consistent sequence of a bordered Legendrian graph in a normal form is a sequence of bordered Legendrian graphs in a normal form.

On the other hand, as seen in Lemma~\ref{lemma:normal form representative}, for each consistent sequence $(\cT^{\p\bullet},\mu^{\p\bullet})\in\BLT^{\mu,\p\bullet}$ of canonical front parallel copies, one can find a representative in a normal form up to consistent front Reidemeister moves. In other words, there are a consistent sequence $(\cT'^{\p\bullet},\mu'^{\p\bullet})\in\BLT^{\mu,\p\bullet}$ of canonical front parallel copies in a normal form and a sequence of consistent front Reidemeister moves between $(\cT'^{\p\bullet},\mu'^{\p\bullet})$ and $(\cT^{\p\bullet},\mu^{\p\bullet})$
\[
\begin{tikzcd}[column sep=3pc]
(\cT'^{\p\bullet},\mu'^{\p\bullet})\arrow[r,"\RM{M_1}^{\p\bullet}"]&\cdots\arrow[r,"\RM{M_k}^{\p\bullet}"]&(\cT^{\p\bullet},\mu^{\p\bullet}).
\end{tikzcd}
\]
Moreover, since the transformation from $\cT'^{\p1}$ to $\cT^{\p1}$ only needs $\RM{V},\RM{V}$ and $\RM{B_*}$, so does the sequence of consistent Reidemeister moves $\RM{M_i}^{\p\bullet}$. That is, for each $i$,
\[
\RM{M_i}^{\p\bullet}\in\left\{\RM{VI}^{\p\bullet},\RM{V}^{\p\bullet}, \RM{B_1}^{\p\bullet}, \RM{B_2}^{\p\bullet}, \RM{B_3}^{\p\bullet}\right\}.
\]
\end{example}

\begin{remark}
It is easy to see that the Ng's resolution does not map canonical parallel copies in front projections onto those in Lagrangian projections.
Indeed, if $\cT$ contains a right cusp or a vertex of type $(\ell,r)$ with $\ell\ge 2$, then the resolution of the canonical front $m$-copies is not the same as the canonical Lagrangian $m$-copies of the resolution.
\end{remark}

However, we can always find a consistent Lagrangian Reidemeister move consisting only of $\RM{ii}$'s and $\RM{iii_a}$'s between the canonical Lagrangian copy of Ng's resolution and Ng's resolutions of the canonical front copy. See Figure~\ref{figure:front and Lagrangian canonical copies}.

\begin{figure}[ht]
\[
\begin{tikzcd}[column sep=1.2pc]
\Res\left(\Rcusp\right)^{\p m}\arrow[d,equal] &
\Res\left(\Rcusp^{\p m}\right)\arrow[d,equal]\\
\begin{tikzpicture}[baseline=-.5ex]
\begin{scope}[yshift=-0.2cm]
\draw[thick] (-0.5,0) to[out=0,in=180] (0.6,0.7);
\draw[red,thick] (-0.5,-0.2) to[out=0,in=180] (0.6,0.5);
\draw[blue,thick] (-0.5,-0.4) to[out=0,in=180] (0.6,0.3);
\draw[white,line width=4] (-0.5,0.4) to[out=0,in=180] (0.6,-0.3) arc (-90:90:0.3 and 0.3);
\draw[blue,thick] (-0.5,0.4) to[out=0,in=180] (0.6,-0.3) arc (-90:90:0.3 and 0.3);
\draw[white,line width=4] (-0.5,0.6) to[out=0,in=180] (0.6,-0.1) arc (-90:90:0.3 and 0.3);
\draw[red,thick] (-0.5,0.6) to[out=0,in=180] (0.6,-0.1) arc (-90:90:0.3 and 0.3);
\draw[white,line width=4] (-0.5,0.8) to[out=0,in=180] (0.6,0.1) arc (-90:90:0.3 and 0.3);
\draw[thick] (-0.5,0.8) to[out=0,in=180] (0.6,0.1) arc (-90:90:0.3 and 0.3);
\end{scope}
\end{tikzpicture}
\arrow[r, Rightarrow, "\binom{m}2\RM{ii}"]&
\begin{tikzpicture}[baseline=-.5ex]
\begin{scope}[yshift=-0.2cm]
\draw[thick] (0.6,0.5) to[out=180,in=0] (-0.5,-0.1);
\draw[red,thick] (0.6,0.2) to[out=180,in=0] (-0.5,-0.3);
\draw[blue,thick] (0.6,-0.1) to[out=180,in=0] (-0.5,-0.5);
\draw[white,line width=4] (-0.5,0.7) to[out=0,in=180] (0.6,0.3) arc (-90:90:0.2 and 0.1);
\draw[thick] (-0.5,0.7) to[out=0,in=180] (0.6,0.3) arc (-90:90:0.2 and 0.1);
\draw[white,line width=4] (-0.5,0.5) to[out=0,in=180] (0.6,0) arc (-90:90:0.2 and 0.1);
\draw[red,thick] (-0.5,0.5) to[out=0,in=180] (0.6,0) arc (-90:90:0.2 and 0.1);
\draw[white,line width=4] (-0.5,0.3) to[out=0,in=180] (0.6,-0.3) arc (-90:90:0.2 and 0.1);
\draw[blue,thick] (-0.5,0.3) to[out=0,in=180] (0.6,-0.3) arc (-90:90:0.2 and 0.1);
\end{scope}
\end{tikzpicture}
\end{tikzcd}
\quad
\begin{tikzcd}[column sep=1.2pc]
\Res\left(\frontvertex\right)^{\p m}\arrow[d,equal] & &
\Res\left(\frontvertex^{\p m}\right)\arrow[d,equal]\\
\begin{tikzpicture}[baseline=-.5ex,scale=0.55]
\begin{scope}[yshift=-0.25cm,color=blue]
\draw[thick] (-1,-0.8) to[out=0,in=180] (1, 0.8) to[out=0,in=135] (2,0);
\end{scope}
\begin{scope}[yshift=0cm,color=red]
\draw[thick] (-1,-0.8) to[out=0,in=180] (1, 0.8) to[out=0,in=135] (2,0);
\end{scope}
\begin{scope}[yshift=0.25cm]
\draw[thick] (-1,-0.8) to[out=0,in=180] (1, 0.8) to[out=0,in=135] (2,0);
\end{scope}
\begin{scope}[yshift=0.25cm]
\draw[white,line width=4] (-1,0) to[out=0,in=180] (0, -0.8) to[out=0,in=180] (1,0) -- (1.85,0);
\draw[thick] (-1,0) to[out=0,in=180] (0, -0.8) to[out=0,in=180] (1,0) -- (2,0);
\end{scope}
\begin{scope}[yshift=0cm,red]
\draw[white,line width=4] (-1,0) to[out=0,in=180] (0, -0.8) to[out=0,in=180] (1,0) -- (1.85,0);
\draw[thick] (-1,0) to[out=0,in=180] (0, -0.8) to[out=0,in=180] (1,0) -- (2,0);
\end{scope}
\begin{scope}[yshift=-0.25cm,blue]
\draw[white,line width=4] (-1,0) to[out=0,in=180] (0, -0.8) to[out=0,in=180] (1,0) -- (1.85,0);
\draw[thick] (-1,0) to[out=0,in=180] (0, -0.8) to[out=0,in=180] (1,0) -- (2,0);
\end{scope}
\begin{scope}[yshift=-0.25cm,color=blue]
\draw[white,line width=4] (-1,0.8) to[out=0,in=180] (1, -0.8) to[out=0,in=225] (1.9,-0.1);
\draw[white,line width=4] (2.1,0.1)--(2.5,0.5) (2.1,-0.1) -- (2.5,-0.5);
\draw[thick] (2,0)--(2.5,0.5) (2,0) -- (2.5,-0.5);
\draw[thick] (-1,0.8) to[out=0,in=180] (1, -0.8) to[out=0,in=225] (2,0);
\draw[fill] (2,0) circle (2pt);
\end{scope}
\begin{scope}[yshift=0cm,color=red]
\draw[white,line width=4] (-1,0.8) to[out=0,in=180] (1, -0.8) to[out=0,in=225] (1.9,-0.1);
\draw[white,line width=4] (2.1,0.1)--(2.5,0.5) (2.1,-0.1) -- (2.5,-0.5);
\draw[thick] (2,0)--(2.5,0.5) (2,0) -- (2.5,-0.5);
\draw[thick] (-1,0.8) to[out=0,in=180] (1, -0.8) to[out=0,in=225] (2,0);
\draw[fill] (2,0) circle (2pt);
\end{scope}
\begin{scope}[yshift=0.25cm]
\draw[white,line width=4] (-1,0.8) to[out=0,in=180] (1, -0.8) to[out=0,in=225] (1.9,-0.1);
\draw[white,line width=4] (2.1,0.1)--(2.5,0.5) (2.1,-0.1) -- (2.5,-0.5);
\draw[thick] (2,0)--(2.5,0.5) (2,0) -- (2.5,-0.5);
\draw[thick] (-1,0.8) to[out=0,in=180] (1, -0.8) to[out=0,in=225] (2,0);
\draw[fill] (2,0) circle (2pt);
\end{scope}
\end{tikzpicture}
\arrow[r, Rightarrow, "N\RM{iii_a}"]&
\begin{tikzpicture}[baseline=-.5ex,scale=0.55]
\draw[thick,blue] (-1, -1.05) -- (-0.5, -1.05) to[out=0,in=180] (1, -0.55);
\draw[white,line width=4] (-1, -0.25) to[out=0,in=180] (0.5, -1.05) to[out=0,in=180] (1, -0.8);
\draw[thick,blue] (-1, -0.25) to[out=0,in=180] (0.5, -1.05) to[out=0,in=180] (1, -0.8);
\draw[white,line width=4] (-1, 0.55) to[out=0,in=180] (1, -1.05);
\draw[thick,blue] (-1, 0.55) to[out=0,in=180] (1, -1.05);
\draw[white,line width=4] (-1, -0.8) -- (-0.5, -0.8) to[out=0,in=180] (1, 0.25);
\draw[thick,red] (-1, -0.8) -- (-0.5, -0.8) to[out=0,in=180] (1, 0.25);
\draw[white,line width=4] (-1, 0) -- (-0.4,0) to[out=0,in=180] (0.5, -0.25) to[out=0,in=180] (1, 0);
\draw[thick,red] (-1, 0) -- (-0.4,0) to[out=0,in=180] (0.5, -0.25) to[out=0,in=180] (1, 0);
\draw[white,line width=4] (-1, 0.8) to[out=0,in=180] (1, -0.25);
\draw[thick,red] (-1, 0.8) to[out=0,in=180] (1, -0.25);
\draw[white,line width=4] (-1, -0.55) -- (-0.5, -0.55) to[out=0,in=180] (1, 1.05);
\draw[thick] (-1, -0.55) -- (-0.5, -0.55) to[out=0,in=180] (1, 1.05);
\draw[white,line width=4] (-1, 0.25) to[out=0,in=180] (0,0.7) to[out=0,in=180] (0.5, 0.55) to[out=0,in=180] (1, 0.8);
\draw[thick] (-1, 0.25) to[out=0,in=180] (0,0.7) to[out=0,in=180] (0.5, 0.55) to[out=0,in=180] (1, 0.8);
\draw[white,line width=4] (-1, 1.05) -- (-0.2, 1.05) to[out=0,in=180] (1, 0.55);
\draw[thick] (-1, 1.05) -- (-0.2, 1.05) to[out=0,in=180] (1, 0.55);
\draw[thick,blue] (1,-0.55) -- (1.2, -0.55) to[out=0, in=180] (1.75, 0.5) to[out=0,in=135] (2.5,-0.4);
\draw[thick,blue] (1,-0.8) -- (1.2, -0.8) to[out=0, in=180] (1.75, -0.2) to[out=0,in=180] (2.5,-0.4);
\draw[thick,blue] (1,-1.05) to[out=0,in=180] (1.85, -1) to[out=0,in=225] (2.5,-0.4);
\draw[thick,blue] (2.5,-0.4) -- (3, 0.1) (2.5,-0.4) -- (3, -0.9);
\draw[white,line width=4] (1,0) -- (2.5,0);
\draw[white,line width=4] (1,-0.25) to[out=0,in=180] (1.75,-0.75) -- (1.85,-0.75) to[out=0,in=225] (2.5,0);
\draw[white,line width=4] (2.5,0) -- (3, 0.5) (2.5,0) -- (3, -0.5);
\draw[thick,red] (1,0.25) to[out=0,in=180] (1.75,0.75) to[out=0,in=135] (2.5,0);
\draw[thick,red] (1,0) -- (2.5,0);
\draw[thick,red] (1,-0.25) to[out=0,in=180] (1.65,-0.75) -- (1.85,-0.75) to[out=0,in=225] (2.5,0);
\draw[thick,red] (2.5,0) -- (3, 0.5) (2.5,0) -- (3, -0.5);
\draw[white,line width=4] (1,0.8) to[out=0,in=180] (1.75,0.2) to[out=0,in=180] (2.5,0.4);
\draw[white,line width=4] (1,0.55) to[out=0,in=180] (1.65,-0.5) -- (1.85,-0.5) to[out=0,in=225] (2.5,0.4);
\draw[white,line width=4] (2.5,0.4) -- (3, 0.9) (2.5,0.4) -- (3, -0.1);
\draw[thick] (1,1.05) to[out=0,in=180] (1.65,1) to[out=0,in=135] (2.5,0.4);
\draw[thick] (1,0.8) to[out=0,in=180] (1.75,0.2) to[out=0,in=180] (2.5,0.4);
\draw[thick] (1,0.55) to[out=0,in=180] (1.65,-0.5) -- (1.85,-0.5) to[out=0,in=225] (2.5,0.4);
\draw[thick] (2.5,0.4) -- (3, 0.9) (2.5,0.4) -- (3, -0.1);
\draw[fill] (2.5,0.4) circle (2pt);
\draw[fill,red] (2.5,0) circle (2pt);
\draw[fill,blue] (2.5,-0.4) circle (2pt);
\end{tikzpicture}\arrow[r, Rightarrow, "N'\RM{ii}"] &
\begin{tikzpicture}[baseline=-.5ex,scale=0.55]
\draw[thick,blue] (-1, -1.05) -- (-0.5, -1.05) to[out=0,in=180] (1, -0.55);
\draw[white,line width=4] (-1, -0.25) to[out=0,in=180] (0.5, -1.05) to[out=0,in=180] (1, -0.8);
\draw[thick,blue] (-1, -0.25) to[out=0,in=180] (0.5, -1.05) to[out=0,in=180] (1, -0.8);
\draw[white,line width=4] (-1, 0.55) to[out=0,in=180] (1, -1.05);
\draw[thick,blue] (-1, 0.55) to[out=0,in=180] (1, -1.05);
\draw[white,line width=4] (-1, -0.8) -- (-0.5, -0.8) to[out=0,in=180] (1, 0.25);
\draw[thick,red] (-1, -0.8) -- (-0.5, -0.8) to[out=0,in=180] (1, 0.25);
\draw[white,line width=4] (-1, 0) -- (-0.4,0) to[out=0,in=180] (0.5, -0.25) to[out=0,in=180] (1, 0);
\draw[thick,red] (-1, 0) -- (-0.4,0) to[out=0,in=180] (0.5, -0.25) to[out=0,in=180] (1, 0);
\draw[white,line width=4] (-1, 0.8) to[out=0,in=180] (1, -0.25);
\draw[thick,red] (-1, 0.8) to[out=0,in=180] (1, -0.25);
\draw[white,line width=4] (-1, -0.55) -- (-0.5, -0.55) to[out=0,in=180] (1, 1.05);
\draw[thick] (-1, -0.55) -- (-0.5, -0.55) to[out=0,in=180] (1, 1.05);
\draw[white,line width=4] (-1, 0.25) to[out=0,in=180] (0,0.7) to[out=0,in=180] (0.5, 0.55) to[out=0,in=180] (1, 0.8);
\draw[thick] (-1, 0.25) to[out=0,in=180] (0,0.7) to[out=0,in=180] (0.5, 0.55) to[out=0,in=180] (1, 0.8);
\draw[white,line width=4] (-1, 1.05) -- (-0.2, 1.05) to[out=0,in=180] (1, 0.55);
\draw[thick] (-1, 1.05) -- (-0.2, 1.05) to[out=0,in=180] (1, 0.55);
\draw[thick,blue] (1,-0.55) to[out=0,in=135] (1.5,-0.8) -- (2,-1.05);
\draw[thick,fill,blue] (1,-0.8) to[out=0,in=0] (1.5,-0.8) circle (2pt);
\draw[thick,blue] (1,-1.05) to[out=0,in=-135] (1.5,-0.8) -- (2, 0.55);;
\draw[white,line width=4] (1.5,0) -- (2,-0.8);
\draw[thick,red] (1,0.25) to[out=0,in=135] (1.5,0) -- (2, -0.8);
\draw[thick,fill,red] (1,0) to[out=0,in=0] (1.5,0) circle (2pt);
\draw[thick,red] (1,-0.25) to[out=0,in=-135] (1.5,0) -- (2, 0.8);
\draw[white,line width=4] (1.5,0.8) -- (2,-0.55);
\draw[thick] (1,1.05) to[out=0,in=135] (1.5,0.8) -- (2, -0.55);
\draw[thick,fill] (1,0.8) to[out=0,in=0] (1.5,0.8) circle (2pt);
\draw[thick] (1,0.55) to[out=0,in=-135] (1.5,0.8) -- (2, 1.05);
\end{tikzpicture}
\end{tikzcd}
\]
\caption{Ng's resolutions of front canonical copies}
\label{figure:front and Lagrangian canonical copies}
\end{figure}
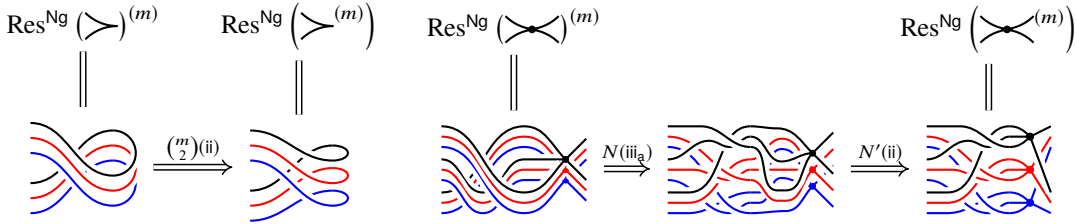

\subsection{Consistent sequences of bordered Chekanov-Eliashberg DGAs}\label{section:consistent sequence of LCH DGAs}
In this section, we will consider consistent sequences of bordered Chekanov-Eliashberg DGAs.

For each $m\ge1$, we define the ring $\ring^{\p m}$ as
\[
\ring^{\p m}\coloneqq\bigoplus_{i\in[m]} \ring^i,
\]
where $\ring^i$ is the copy of $\ring$ with the unit $\unit^i$.
Then $\ring^{\p m}$ has the unit $\unit^{\p m}\coloneqq\sum_{i\in[m]}\unit^i$ and for each subset $U\subset[m]$, there exists a unique ring homomorphism
\[
\ring^{\p m}\to \ring^U\coloneqq \bigoplus_{i\in U} \ring^i.
\]
Therefore $\ring^{\p m}$ becomes an $m$-component object and furthermore it is obviously consistent.

\subsubsection{Link-grading and composable DGAs}\label{section:edge-grading}
We first introduce an $m$-component link-grading on DGAs as described in \cite[\S3.2]{NRSSZ2015}.

\begin{definition}[Link-graded and composable DGAs]
Let $A^{\p m}=(\alg^{\p m},\differential^{\p m})$ be a DGA which is not necessarily semi-free.
We say that $A$ is \emph{$m$-component link-graded} for some $m\ge 1$ if there exist functions $r, c:\alg^{\p m}\to [m]$ such that for each $g\in \alg^{\p m}$ and any word $g_1\dots g_k$ in $\differential(g)$ different from $n\unit$ for any $n\in\ZZ$ is \emph{composable} in the sense that
\begin{align*}
r(g)&=r(g_1),&
c(g)&=c(g_k),&
r(g_i)&=c(g_{i+1})
\end{align*}
for each $1\le i<k$.
We say that a composable word $g_1\dots g_k$ is said to be \emph{of type} $(i,j)$ if $r(g_1)=i$ and $c(g_k)=j$.

An $m$-component link-graded DGA $A^{\p m}=(\alg^{\p m},\differential^{\p m})$ is \emph{composable} over $[m]$ if $\alg^{\p m}$ is decomposed as
\[
\alg^{\p m}=\bigoplus_{(i,j)\in [m]^2}\alg^{\p m}(i,j)
\]
such that it satisfies the following:
\begin{itemize}
\item for each $i,j,k,l\in [m]$, the multiplication on $\alg^{\p m}(i,j)\otimes \alg^{\p m}(k,l)$ does not vanish only if $j=k$. In this case, the result is contained in $\alg^{\p m}(i,l)$.
\item for each $(i,j)$, the restriction $A^{\p m}(i,j)=(\alg^{\p m}(i,j), \differential^{\p m}|_{\alg^{\p m}(i,j)})$ is a subchain complex. In particular, $A^{\p m}(i,i)$ becomes the DG-subalgebra.
\end{itemize}
\end{definition}

One equivalent definition of composable DGAs is as follows:
a DGA $A^{\p m}=(\alg^{\p m},\differential^{\p m})$ is composable over $[m]$ if and only if it is a subalgebra of the matrix algebra so that
\begin{align*}
\alg^{\p m}&\isomorphic\Mat_m\left(\alg^{\p m}(i,j)\right)\coloneqq\left\{
\left(a^{ij}\right)~\middle|~ a^{ij}\in \alg^{\p m}(i,j)
\right\},&
\differential^{\p m}&\isomorphic(\differential^{\p m}|_{\alg(i,j)})
\end{align*}
and also it is a $\ring^{\p m}$-module.

\begin{definition}[Semi-free composable DGAs]
We say that a composable DGA $A^{\p m}=(\alg^{\p m},\differential^{\p m})$ is \emph{semi-free} if it is tensor-algebra-like as follows: there exists a graded set $\sfS^{\p m}$
\[
\sfS^{\p m}=\coprod_{i,j\in[m]^2} \sfS^{ij}
\]
and the underlying graded algebra $\alg^{\p m}$ is the direct sum
\[
\alg^{\p m}=T_\co \sfM^{\p m}\coloneqq \bigoplus_{\ell\ge 0} \left(\sfM^{\p m}\right)^\ell,
\]
where $\sfM^{\p m}=(\sfM^{ij})\subset\alg^{\p m}$ is a $\ring^{\p m}$-module
and each $\sfM^{ij}$ is a free $(\ring^i,\ring^j)$-bimodule generated by $\sfS^{ij}$
\begin{align*}
\sfM^{ij}&\coloneqq \ring^i\otimes \ZZ\langle\sfS^{ij}\rangle\otimes \ring^j,
\end{align*}
and $\left(\sfM^{\p m}\right)^\ell$ is the $\ell$-fold product of $\sfM^{\p m}$ with respect to the matrix multiplication so that $\left(\sfM^{\p m}\right)^0\coloneqq\ring^{\p m}$.

A DGA morphism $f:A'^{\p n}\to A^{\p m}$ between semi-free composable DGAs is said to be an \emph{composable morphism} if there exists $(h:[m]\to[n])\in\AugSim$ such that for each $i',j'\in[n]$,
$f(A'^{\p n}(i',j'))$ is either contained in $A(i,j)$ for $h(i,j)=(i',j')$ or $0$ otherwise.
\end{definition}

Let $\bfi=(i_0,\dots, i_k)$ be a sequence in $[m]$.
We denote $\sfM^\bfi$ by the submodule defined as
\[
\sfM^\bfi\coloneqq
\sfM^{i_0i_1}\sfM^{i_1i_2}\dots\sfM^{i_{k-1}i_k}.
\]
Then it is easy to see that
\[
\alg^{\p m}(i,j)=\bigoplus_{\bfi}\sfM^\bfi,
\]
where $\bfi$ ranges over all sequences starting and ending at $i$ and $j$, respectively.

\begin{notation}
For simplicity, we will use the notation $\sfM^{\p m}=(\ring^{\p m}\langle\sfS^{\p m}\rangle)$ and so
\[
\alg^{\p m}=T_\co (\ring^{\p m}\langle \sfS^{\p m}\rangle).
\]
\end{notation}

\begin{definition}[Category of composable DGAs]
We denote the category of all semi-free composable DGAs and morphisms by $\DGA_\co$. In particular, the full subcategory of $m$-components semi-free composable DGAs is denoted by $\DGA_\co^{\p m}$.
\end{definition}
\begin{remark}
It is obvious that each morphism $f\in\DGA_\co^{\p m}$ preserves the link-grading. Especially, the subcategory $\DGA_\co^{\p1}$ is equivalent to $\DGA$.
\end{remark}

Let $A^{\p m}\in\DGA_\co^{\p m}$. Then for each $U\subset[m]$ with $|U|=k$, there exists a $k$-component semi-free composable DGA $A^U\in\DGA_\co^{\p k}$ which is the image of the quotient map
\[
\res^U:A^{\p m}\to A^U
\]
defined by $\unit^i\mapsto 0$ for each $i\not\in U$.
Therefore each object in $\DGA^{\p m}$ satisfies the condition for $m$-components object.

\begin{example}[Composable DGAs from link-graded semi-free DGAs]\label{example:composable from semi-free}
For each $m$-component link-graded semi-free DGA $A^{\p m}=(\alg^{\p m}=T(\ring\langle\sfS^{\p m}\rangle),\differential^{\p m})$, one can associate an $m$-component composable DGA $A_\co^{\p m}$ as described in \cite[\S3.2]{NRSSZ2015}.
We will briefly review the construction as follows:
\begin{enumerate}
\item Add idempotent elements $\unit^i$ for $i\in[m]$ and declare $\unit^i\unit_j=\delta_{ij}$.
\item For each element $g$ of type $(i,j)$, declare $\unit^i g = g = g\unit^j$.
\item For each element $g$ of type $(i,i)$ having $n\unit$ in $\differential(g)$, replace $n\unit$ with $n\unit^i$.
\item The unit $\unit$ is replaced with $\sum_{i\in[m]}\unit^i$.
\end{enumerate}
Then it is easy to check that 
\[
A^{\p m}_\co=(T_\co(\ring^{\p m}\langle\sfS^{\p m}\rangle), \differential^{\p m})
\]
and therefore it is semi-free composable.
\end{example}

\begin{definition}[Consistent sequences of composable DGAs]
The category of consistent sequences in $\DGA_\co$ will be denoted by $\DGA_\co^{\p \bullet}$.
\end{definition}

More precisely, a sequence $A^{\p\bullet}$ of $m$-component composable DGAs is consistent if and only if for each $(h:[m]\to[n])\in\AugSim$ and $U\subset[m]$, we have an isomorphism
\[
A^{h(U)}\to A^U.
\]

\begin{example/definition}[Consistent sequences of composable border DGAs]
Let  $\mu:[n]\to\grading$ be a function.
Recall the border DGA $A_{n}(\mu)\in\DGA=\DGA_\co^{\p1}$. For each $m\ge 1$, we will define the \emph{$m$-component border DGA} $A_{n}^{\p m}(\mu)\in\DGA_\co^{\p m}$ such that $A^{\p1}_{n}(\mu)=A_{n}(\mu)$.

Let $A_{n}^{\p m}(\mu)=(\alg_n^{\p m},\differential_n^{\p m})$ be an $m$-component link graded DGA defined as follows:
\begin{enumerate}
\item The algebra $\alg_{n}^{\p m}$ is the algebra $T_\co\sfM^{\p m}$ over the ring $\ring^{\p m}$
\begin{align*}
\alg_{n}^{\p m}&\coloneqq T_\co(\ring^{\p m}\langle K^{ij}\amalg Y^{ij}\rangle),&
K_n^{ij}&\coloneqq\{k_{ab}^{ij}\mid 1\le a<b\le n\},&
Y_n^{ij}&\coloneqq
\begin{cases}
\{y_a^{ij}\mid a\in[n]\} & i<j;\\
\emptyset & i\ge j.
\end{cases}
\end{align*}
\item The link-grading for each $k_{ab}^{ij}$ and $y_a^{ij}$ is defined as $(i,j)$ and the (homological) grading of each is defined as 
\[
|k_{ab}^{ij}|\coloneqq|k_{ab}|\quad\text{ and }\quad
|y_a^{ij}|\coloneqq -1.
\]
\item The differential $\differential^{\p m}_n$ is defined as
\begin{align*}
\differential^{\p m}_n(k_{ab}^{ij})&\coloneqq\sum_{\substack{a<c<b\\i<k<j}}(-1)^{|k_{ac}|-1}k_{ac}^{ik}k_{cb}^{kj}+\sum_{k<j}(-1)^{|k_{ab}|-1}k_{ab}^{ik}y_b^{kj}+\sum_{i<k}y_a^{ik}k_{ab}^{kj},\\
\differential^{\p m}_n(y_a^{ij})&\coloneqq\sum_{i<k<j} y_a^{ik}y_a^{kj}.
\end{align*}
\end{enumerate}

Geometrically, generators $k_{ab}^{ij}$ and $y_a^{ij}$ correspond to the Reeb chord from the $b$-th edge in the $j$-th copy $T_n^j$ to the $a$-th edge in the $i$-th copy $T_n^i$ and the Reeb chord from the $a$-th edge of $T_n^i$ to $a$-th edge of $T_n^j$, respectively.

For each $U\subset[m]$, the restriction $A^U_{n}(\mu)$ will be defined as the image of the quotient map sending all $k_{ab}^{ij}$'s and $y_a^{ij}$'s to zero unless $i,j\in U$.
This implies that $A^{\p m}_{n}(\mu)$ satisfies all axioms for $m$-component objects.

Moreover, the sequence $(A^{\p m}_{n}(\mu))_{m\ge1}$ satisfies the consistency as follows: for each $h:[m]\to[m']$ and $U\subset[m]$, there is an obvious isomorphism $A^{h(U)}_{n}(\mu)\to A^U_{n}(\mu)$ sending each $k_{ab}^{i'j'}$ and $y_a^{i'j'}$ to $k_{ab}^{ij}$ and $y_a^{ij}$ if $i'=h(i)$ and $j'=h(j)$.
Therefore we have a consistent sequence of composable DGAs $A^{\p\bullet}_{n}(\mu)\in\DGA_\co^{\p \bullet}$.
\end{example/definition}

\begin{corollary}\label{corollary:composable border DGAs}
The composable DGA $A_n^{\p m}(\mu)$ can be obtained from the semi-free border DGA $A_{mn}(\mu^{\p m})$ via the standard recipe described in Example~\ref{example:composable from semi-free}, where $\mu^{\p m}:[mn]\isomorphic[m]\times[n]\to\grading$ is defined to be
\[
\mu^{\p m}(i,a) = \mu(a).
\]
\end{corollary}
\begin{proof}
This follows directly from the definition and we are done.
\end{proof}

Now we consider consistent sequences of bordered DGAs.

\begin{definition}[$m$-component bordered DGAs]\label{definition:m-component borderd DGAs}
An $m$-component bordered DGA of type $(n_\Left, n_\Right)$
\[
\dga^{\p m}=\left(A_\Left^{\p m}\stackrel{\phi^{\p m}_\Left}\longrightarrow A^{\p m} \stackrel{\phi^{\p m}_\Right}\longleftarrow A^{\p m}_\Right\right)
\]
is a bordered DGA satisfying the following conditions:
\begin{enumerate}
\item All DGAs $A_*^{\p m}$ for $*=\Left,\Right$ or empty are in $\DGA_\co^{\p m}$,
\item two DGAs $A_\Left^{\p m}$ and $A_\Right^{\p m}$ are isomorphic to $A^{\p m}_{n_\Left}(\mu_\Left)$ and $A^{\p m}_{n_\Right}(\mu_\Right)$ for some $\mu_\Left$ and $\mu_\Right$,
\item two structure morphisms $\phi_\Left^{\p m}$ and $\phi_\Right^{\p m}$ are composable morphisms.
\end{enumerate}

For each $U\subset[m]$, the restriction $\dga^U$ for $U\subset[m]$ is defined by the image of the quotient map sending all elements in $A^{\p m}(i,j)$ to zero unless $i,j\in U$
\[
\dga^U\coloneqq\left(A_\Left^U\stackrel{\phi^U_\Left}\longrightarrow A^U\stackrel{\phi^U_\Right}\longleftarrow A_\Right^U\right),
\]
where $\phi^U_\Left$ and $\phi^U_\Right$ are the induced maps by $\phi^{\p m}_\Left$ and $\phi^{\p m}_\Right$, respectively.

A bordered composable morphism $\dga'^{\p n}\to \dga^{\p m}$ between $n$ and $m$-components bordered composable DGAs is a collection of bordered DGA morphisms $\left\{\dga'^{h(U)}\to \dga^U\mid U\subset[m]\right\}$ for a uniquely determined $(h:[m]\to[n])\in\AugSim$, which is compatible with the restriction map $\dga^{\p m}\to\dga^U$.

We denote the category of all bordered composable DGAs by $\BDGA_\co$ and its full subcategory of $m$-components by $\BDGA_\co^{\p m}$.
\end{definition}

By definition, the categories $\DGA_\co$ and $\DGA^{\p m}_\co$ are full subcategory of $\BDGA_\co$ and $\BDGA^{\p m}_\co$ of type $(0,0)$, respectively.

Let $\dga^{\p m}\in\BDGA^{\p m}_\co$ be a sequence of $m$-components bordered DGAs.
Then it is consistent if and only if for each $(h:[m]\to[n])\in\AugSim$ and $U\subset[m]$, we have an isomorphism $\dga(h)^U:\dga^{h(U)}\to \dga^U$ between composable bordered DGAs.

Similarly, a sequence $\bff^{\p\bullet}:\dga'^{\p \bullet}\to\dga^{\p \bullet}$ between consistent sequences of bordered DGAs is consistent if and only if for each $m\ge 1$, $\bff^{\p m}$ is a bordered morphism between $m$-component bordered DGAs $\dga'^{\p m}$ and $\dga^{\p m}$, and it satisfies the consistency, i.e., for each $h:[m]\to[n]$ and $U\subset[m]$, the following diagram is commutative:
\[
\begin{tikzcd}[column sep=4pc,row sep=2pc]
\dga'^{h(U)}\arrow[r, "\dga'(h)^U"]\arrow[d,"\bff^{h(U)}"'] & \dga'^U\arrow[d,"\bff^U"]\\
\dga^{h(U)}\arrow[r, "\dga(h)^U"] & \dga^U
\end{tikzcd}
\]

\begin{definition}[Consistent sequences of composable bordered DGAs]
We denote the category of consistent sequences of composable bordered DGAs by $\BDGA^{\p \bullet}_\co$.
\end{definition}

In terms of generators, we have the following observation. Let $A^{\p m}=\left(\alg^{\p m}=T_\co(\ring^{\p m}\langle\sfS^{\p m}\rangle),\differential^{\p m}\right)$ be a composable $m$-component DGA.
Then the restriction on $U$ is now given by the image of the map $\unit_e^i\mapsto 0$ for all $e\in E$ unless $i\in U$.
Moreover, the consistency implies that for each $(h:[m]\to[n])\in\AugSim$, the DGA $A^{\p m}$ is isomorphic to the DGA generated by $\sfS^{i'j'}$'s for $i',j'\in h([m])$.
Therefore we may identify 
\[
\sfS^{ii}\isomorphic \sfS^{11}\quad\text{ and }\quad
\sfS^{ij}\isomorphic \sfS^{12}
\]
by considering $h_i:[1]\to[m]$ and $h_{ij}:[2]\to[m]$ with $h_i(1)=i$ and $h_{ij}(1)=i, h_{ij}(2)=j$.

Finally, the differential $\differential^{\p m}$ is the same as the push-forward of $\differential^{\p n}$ on $A(h):A^{\p n}\to A^{\p m}$.
In other words, for each generator $a^{ij}\in \sfS^{ij}$,
\[
\differential^{\p m}\left(a^{ij}\right)=(A(h)\circ\differential^{\p n}\circ A(h)^{-1})(a^{ij})=A(h)\left(\differential^{\p n}\left(a^{i'j'}\right)\right),
\]
where $i'=h(i), j'=h(j)$ and $a^{i'j'}$ is regarded as an element of $\sfS^{i'j'}\isomorphic \sfS^{ij}$.
In other words, the differential in $A^{\p m}$ can be obtained from the differential in $A^{\p n}$ by removing generators not coming from $A^{\p m}$.

The following is the direct consequence of Lemma~\ref{lemma:isom} and summarizes the above discussion.
\begin{corollary}\cite[Lemma~3.15]{NRSSZ2015}
Let $\dga^{\p\bullet}\in\BDGA^{\p\bullet}_\co$ be a consistent sequence of bordered DGAs. 
Then for each $h_i:[1]\to[m]$ with $i=h_i(1)$ and $h_{ij}:[2]\to[m]$ with $i=h_{ij}(1)<j=h_{ij}(2)$, we have isomorphisms between DGAs
\[
\dga^{ii}\isomorphic \dga^{11}\quad\text{ and }\quad
\dga^{ij}\isomorphic \dga^{12}.
\]
\end{corollary}

\begin{definition}[Consistent stabilizations]
A consistent sequence of bordered DGAs $\dga'^{\p\bullet}\in\BDGA^{\p\bullet}_\co$ is a (\emph{weak}) \emph{consistent stabilization} of $\dga^{\p\bullet}\in\BDGA^{\p\bullet}_\co$ if there exists a consistent morphism $\bfpi^{\p\bullet}:\dga'^{\p\bullet}\to\dga^{\p\bullet}$ of canonical projections of stabilization, and is a \emph{strong consistent stabilization} if it is a weak stabilization and there exists a consistent morphism $\bfi^{\p\bullet}:\dga^{\p\bullet}\to\dga'^{\p\bullet}$ of canonical embeddings.
\end{definition}

Let $A'^{\p\bullet}=(\alg'^{\p\bullet}=T_\co(\ring^{\p\bullet}\langle\sfS'^{\p\bullet}\rangle),\differential'^{\p m})$ be a consistent stabilization of $A^{\p\bullet}=(\alg^{\p\bullet}=T_\co(\ring^{\p\bullet}\langle\sfS^{\p\bullet}\rangle),\differential^{\p m})$. 
Then in terms of generating sets, for each $m\ge 1$, there exists an index set $I^{\p m}=(I^{ij})$ such that
\begin{align*}
\sfS'^{ij}&=\sfS^{ij}\amalg\{e_k^{ij},\hat e_k^{ij}\mid k\in I^{ij}\},&
|\hat e_k^{ij}|&=|e_k^{ij}|+1,&
\differential'^{\p m}(\hat e_k^{ij})&=e_k^{ij},&
\differential'^{\p m}(e_k^{ij})&=0.
\end{align*}
Then for each $U\subset\AugSim$, the restriction $A'^U$ is generated by 
\[
\sfS'^U = \sfS^U\amalg\{e_k^{ij},\hat e_k^{ij}\mid k\in I^{ij}, i,j\in U\}
\]
which is the image of the quotient map sending all $e_k^{ij}$'s and $\hat e_k^{ij}$'s to zero unless both $i,j\in U$.
Finally, for each $(h:[m]\to[n])\in\AugSim$, the consistency implies that we have an isomorphism
\[
A'^{h(U)}\to A'^U
\]
sending each generator $e_k^{i'j'}$ and $\hat e_k^{i'j'}$ to $e_k^{ij}$ and $\hat e_k^{ij}$, respectively, for $h(i,j)=(i',j')$.
Therefore it is necessary and sufficient to have the following conditions
\[
I^{ij}\isomorphic I^{kl}\quad\text{ and }\quad|e_k^{ij}|= |e_k^{kl}|
\]
for $(k,l)=(1,1),(1,2)$ or $(2,1)$ according to whether $i=j$, $i<j$ or $i>j$.

As before, the one important example of a consistent stabilization is the cofibrant replacement of a consistent sequence of bordered DGAs defined as follows:
\begin{example/definition}[Consistent cofibrant replacement]\label{example:consistent cofibrant replacement}
Let $\dga^{\p \bullet}\in \BDGA^{\p\bullet}_\co$ be a consistent sequence of bordered composable DGAs.
For each $m\ge 1$, we consider the cofibrant replacement $\hat\dga^{\p m}$ of $\dga^{\p m}$ defined in Definition~\ref{definition:cofibrant replacement}.

For each $U\subset[m]$, the restriction $\hat\dga^U$ will be given as the cofibrant replacement of $\dga^U$ and therefore $\hat\dga^{\p m}$ is an $m$-component bordered DGA as desired.

Finally, the consistency of $\hat\dga^{\p m}$ comes easily from the consistency of $\dga^{\p m}$ as follows:
for each $h:[m]\to[n]$ and $U\subset[m]$, the bordered DGA $\hat\dga^{h(U)}$ is the cofibrant replacement of $\dga^{h(U)}$ as above which is isomorphic to $\dga^U$ due to the consistency of $\dga^{\p \bullet}$.
Since the cofibrant replacement remains the same under isomorphism, we have a desired isomorphism
\[
\hat\dga^{h(U)}\isomorphic \hat\dga^U.
\]
\end{example/definition}

\begin{lemma}\label{lemma:cofibrant replacements of consistent sequences of bordered DGAs}
For each $\dga^{\p \bullet}\in\BDGA^{\p \bullet}_\co$, its cofibrant replacement $\hat\dga^{\p \bullet}\in\BDGA^{\p\bullet}_\co$ is well-defined.
\end{lemma}
\begin{proof}
This is a summary of discussion above and we omit the proof.
\end{proof}

\subsubsection{Consistent sequences of LCH DGAs}\label{section:consistent sequences of LCH DGAs}
Let $(\cT^{\p\bullet},\bfmu^{\p\bullet})\in\BLT^{\mu,\p \bullet}$ be a consistent sequence of front projections of bordered Legendrian graphs with Maslov potentials.
Then for each $m\ge 1$, we have the bordered LCH DGA $A^\CE(\cT^{\p m},\mu^{\p m})$
\[
A^\CE(\cT^{\p m},\bfmu^{\p m})=
\left(
A^\CE(T^{\p m}_\Left, \mu^{\p m}_\Left) \stackrel{\phi^{\p m}_\Left}\longrightarrow
A^\CE(T^{\p m}, \mu^{\p m}) \stackrel{\phi^{\p m}_\Right}\longleftarrow
A^\CE(T^{\p m}_\Right, \mu^{\p m}_\Right)
\right).
\]

It is obvious that the DGA $A^\CE(T^{\p m}_*,\mu^{\p m}_*)$ for each $*=\Left,\Right$ or empty is $m$-component link-graded but is not composable.
However, by using the standard recipe described in Example~\ref{example:composable from semi-free}, we obtain an $m$-component composable bordered DGA $A^\CE_\co(\cT^{\p m},\bfmu^{\p m})$
\[
A^\CE_\co(\cT^{\p m},\bfmu^{\p m})=
\left(
A^\CE_\co(T^{\p m}_\Left, \mu^{\p m}_\Left) \stackrel{\phi^{\p m}_\Left}\longrightarrow
A^\CE_\co(T^{\p m}, \mu^{\p m}) \stackrel{\phi^{\p m}_\Right}\longleftarrow
A^\CE_\co(T^{\p m}_\Right, \mu^{\p m}_\Right)
\right).
\]
Then it is obvious that two induced morphisms $\phi^{\p m}_\Left$ and $\phi^{\p m}_\Right$ are composable morphisms but it is not yet known to satisfy the axiom for $m$-component objects.

\begin{proposition}\label{proposition:DGA for consistent sequence is consistent}
The bordered composable DGA $A^\CE_\co(\cT^{\p m},\bfmu^{\p m})\in\BDGA^{\p m}_\co$ is an $m$-component bordered composable DGA.
Moreover, if $(\cT^{\p\bullet}, \bfmu^{\p\bullet})\in\BLT^{\mu,\p\bullet}$ is a consistent sequence, then the corresponding sequence 
\[
A^\CE_\co(\cT^{\p \bullet},\bfmu^{\p \bullet})\coloneqq\left(A^\CE_\co(\cT^{\p m},\bfmu^{\p m})\right)_{m\ge 1}
\]
of bordered DGAs is consistent.
\end{proposition}
\begin{proof}
We know that two border DGAs $A^\CE_\co(T^{\p m}_\Left, \mu^{\p m}_\Left)$ and $A^\CE_\co(T^{\p m}_\Right, \mu^{\p m}_\Right)$ are isomorphic to the $m$-component border DGAs
\[
A^\CE_\co(T^{\p m}_\Left, \mu^{\p m}_\Left)\isomorphic A_{n_\Left}^{\p m}(\mu_\Left)\quad\text{ and }\quad
A^\CE_\co(T^{\p m}_\Right, \mu^{\p m}_\Right)\isomorphic A_{n_\Right}^{\p m}(\mu_\Right)
\]
by Corollary~\ref{corollary:composable border DGAs}.

For each $U\subset[m]$, we define the restriction $A^\CE_\co(T^{\p m},\mu^{\p m})^U$ as the image of the quotient map sending all generators of type $(i,j)$ to zero unless $i,j\in U$.
Hence $A^\CE_\co(\cT^{\p m},\bfmu^{\p m})$ admits an $m$-component object structure.

Moreover, we have an isomorphism
\begin{equation}\label{equation:restrictions of DGAs are DGAs of restrictions}
A^\CE_\co(\cT^{\p m},\bfmu^{\p m})^U\isomorphic
A^\CE_\co(\cT^U,\bfmu^U)
\end{equation}
since we only consider generators and immersed disks lying on $T^U$ in the left hand side.

Now assume that $(\cT^{\p \bullet},\bfmu^{\p\bullet})$ is consistent and let $A^\CE_\co(\cT^{\p\bullet},\bfmu^{\p\bullet})$ be the sequence of bordered LCH DGAs of $(\cT^{\p\bullet},\bfmu^{\p\bullet})\in\BLT^{\mu,\p\bullet}$.
Then for each $(h:[m]\to[n])\in\AugSim$ and $U\subset[m]$, we have isomorphisms between bordered composable DGAs
\[
A^\CE_\co(\cT^{\p n},\bfmu^{\p n})^{h(U)}\isomorphic A^\CE_\co(\cT^{h(U)},\bfmu^{h(U)})\quad\text{ and }\quad
A^\CE_\co(\cT^{\p m},\bfmu^{\p m})^U\isomorphic A^\CE_\co(\cT^U,\bfmu^U)
\]
by the observation \eqref{equation:restrictions of DGAs are DGAs of restrictions}.
Since two restrictions $(\cT^{h(U)},\bfmu^{h(U)})$ and $(\cT^U,\bfmu^U)$ are isomorphic due to the consistency of $(\cT^{\p\bullet}, \bfmu^{\p\bullet})$, we have a desired isomorphism
\[
A^\CE_\co(\cT^{\p n},\bfmu^{\p n})^{h(U)}\isomorphic
A^\CE_\co(\cT^{\p m},\bfmu^{\p m})^U
\]
which means the consistency of $A^\CE_\co(\cT^{\p\bullet},\mu^{\p\bullet})$.
\end{proof}

Let $\RM{M}^{\p\bullet}:(\cT'^{\p \bullet},\mu'^{\p \bullet})\to(\cT^{\p \bullet},\mu^{\p \bullet})$ be an elementary consistent front Reidemeister move and let us denote their bordered composable LCH DGAs by $\dga'^{\p \bullet}$ and $\dga^{\p \bullet}$
\begin{align*}
\dga'^{\p \bullet}&=\left(
A'^{\p\bullet}_{\Left}\stackrel{\phi'^{\p \bullet}_\Left}\longrightarrow A'^{\p\bullet}
\stackrel{\phi'^{\p \bullet}_\Right}\longleftarrow
A'^{\p\bullet}_{\Right}
\right)\coloneqq A^\CE_\co(\cT'^{\p\bullet},\bfmu'^{\p\bullet}),\\
\dga^{\p \bullet}&=\left(
A^{\p\bullet}_{\Left}\stackrel{\phi^{\p \bullet}_\Left}\longrightarrow A^{\p\bullet}
\stackrel{\phi^{\p \bullet}_\Right}\longleftarrow
A^{\p\bullet}_{\Right}\right)\coloneqq A^\CE_\co(\cT^{\p\bullet},\bfmu^{\p\bullet}).
\end{align*}

Then as before, we will pass through the cofibrant replacement as follows.
By Lemma~\ref{lemma:cofibrant replacements of consistent sequences of bordered DGAs}, the cofibrant replacements 
\[
\hat\dga'^{\p \bullet}=\left(
A'^{\p\bullet}_{\Left}\stackrel{\hat\phi'^{\p \bullet}_\Left}\longrightarrow \hat A'^{\p\bullet}
\stackrel{\hat\phi'^{\p \bullet}_\Right}\longleftarrow
A'^{\p\bullet}_{\Right}
\right)\quad\text{ and }\quad
\hat\dga^{\p \bullet}=\left(
A^{\p\bullet}_{\Left}\stackrel{\hat\phi^{\p \bullet}_\Left}\longrightarrow \hat A^{\p\bullet}
\stackrel{\hat\phi^{\p \bullet}_\Right}\longleftarrow
A^{\p\bullet}_{\Right}\right)
\]
are also consistent and furthermore stabilizations of $\dga'^{\p \bullet}$ and $\dga^{\p \bullet}$, respectively.
That is, there are consistent morphisms $\hat\pi'^{\p \bullet}$ and $\hat\pi^{\p \bullet}$ of canonical projections
\[
\hat\pi'^{\p \bullet}:\hat\dga'^{\p \bullet}\stackrel{\homotopic}\longrightarrow \dga'^{\p \bullet}\quad\text{ and }\quad
\hat\pi^{\p \bullet}:\hat\dga^{\p \bullet}\stackrel{\homotopic}\longrightarrow \dga^{\p \bullet}.
\]

On the other hand, $\RM{M}^{\p\bullet}$ gives us an elementary consistent Reidemeister move 
\[
\RM{M}^{\p\bullet}:(\hat\cT'^{\p\bullet},\hat\bfmu'^{\p\bullet})\to(\hat\cT^{\p\bullet},\hat\bfmu^{\p\bullet}).
\]
Then by the condition (2) of consistent Reidemeister moves in Definition~\ref{definition:consistent Reidemeister moves} and as observed in the discussion after Theorem~\ref{theorem:zig-zags of stabilizations}, for each $m\ge 1$ it induces a zig-zag of stabilizations
\[
\begin{tikzcd}[column sep=3pc]
A^\CE_\co(\cT'^{\p m}, \bfmu'^{\p\bullet})\arrow[from=r, "\hat\bfpi'^{\p m}"',->>] & 
\hat\dga'^{\p m} \arrow[r, "\isomorphic"',"\RM{\hat {III}}^{\p m}"] & 
\hat\dga^{\p m} \arrow[r, "\hat\bfpi^{\p m}",->>] & 
A^\CE_\co(\cT^{\p m}, \bfmu^{\p m}),
\end{tikzcd}
\]
\[
\begin{tikzcd}[column sep=3pc]
A^\CE_\co(\cT'^{\p m}, \bfmu'^{\p\bullet})\arrow[from=r, "\hat\bfpi'^{\p m}"',->>] & 
\hat\dga'^{\p m} \arrow[r, ->>, "\RM{\hat M}^{\p m}=\bfpi^{\p m}",yshift=.5ex]\arrow[from=r,hook, "\bfi^{\p m}",yshift=-.5ex] & 
\hat\dga^{\p m} \arrow[r, "\hat\bfpi^{\p m}",->>] & 
A^\CE_\co(\cT^{\p m}, \bfmu^{\p m})
\end{tikzcd}
\]
for $\RM{M}\in\{\RM{I}, \RM{II}, \RM{VI}\}$, or
\[
\begin{tikzcd}[column sep=3pc]
& & \tilde SA^{\p m}\arrow[dl, "\tilde\bfpi'^{\p m}"',->>, yshift=.5ex] \arrow[from=dl, hook', near end,"\tilde\bfi'^{\p m}"', yshift=-.5ex]
\arrow[dr, "\tilde\bfpi^{\p m}",->>, yshift=.5ex] \arrow[from=dr, hook, near end, "\tilde\bfi^{\p m}", yshift=-.5ex]\\
A^\CE_\co(\cT'^{\p m}, \bfmu'^{\p m})\arrow[from=r, "\hat\bfpi'^{\p m}"',->>] & 
\hat\dga'^{\p m} \arrow[rr, "\RM{\hat {VI}}^{\p m}"',->>] & &
\hat\dga^{\p m} \arrow[r, "\hat\bfpi^{\p m}",->>] & 
A^\CE_\co(\cT^{\p m}, \bfmu^{\p m})
\end{tikzcd}
\]
between $m$-component bordered composable LCH DGAs.
Moreover, the condition (3) in Definition~\ref{definition:consistent Reidemeister moves} implies the consistency of all arrows above and therefore we have a zig-zag of consistent stabilizations corresponding to $\RM{M}^{\p\bullet}$.

Similarly, one can observe the same property for elementary consistent Lagrangian Reidemeister moves as well and we omit the detail.

In summary, we have the following theorem.
\begin{theorem}\label{theorem:tangle to CE algebra}
For $(\cT^{\p\bullet},\bfmu^{\p\bullet})\in\BLT^{\mu,\p\bullet}$ and $(\cT^{\p\bullet}_\lag,\bfmu^{\p\bullet})\in\BLT^{\mu,\p\bullet}_\lag$, the assignments 
\[
(\cT^{\p\bullet},\bfmu^{\p\bullet})\mapsto A^{\CE}_\co\left(\cT^{\p\bullet},\bfmu^{\p\bullet}\right)\in\BDGA_\co^{\p\bullet}\quad\text{ and }\quad
(\cT^{\p\bullet}_\lag,\bfmu^{\p\bullet})\mapsto A^{\CE}_\co\left(\cT^{\p\bullet}_\lag,\bfmu^{\p\bullet}\right)\in\BDGA_\co^{\p\bullet}
\]
are well-defined and each elementary consistent front or Lagrangian Reidemeister move induces a zig-zag of consistent stabilizations.
\end{theorem}

Now let us consider the consistent basepoint move.
For an elementary consistent basepoint move $\RM{B_1}^{\p\bullet}$ or $\RM{b_1}^{\p\bullet}$, we already know that it induces a zig-zag of consistent stabilizations since it is indeed a sequence of elementary consistent Reidemeister moves.

For the moves $\RM{B_i}^{\p\bullet}$ or $\RM{b_i}^{\p\bullet}$ with $i=2$ or $3$, it is obvious that there is a consistent morphisms
\[
\begin{tikzcd}
A^\CE(\cT'^{\p\bullet},\bfmu')\arrow[r, "\RM{B_i}^{\p\bullet}_*", yshift=.7ex]\arrow[from=r, "\RM{B_i^{-1}}^{\p\bullet}_*", yshift=-.7ex] & A^\CE(\cT^{\p\bullet},\bfmu)&
A^\CE(\cT'^{\p\bullet}_\lag,\bfmu')\arrow[r, "\RM{b_i}^{\p\bullet}_*", yshift=.7ex]\arrow[from=r, "\RM{b_i^{-1}}^{\p\bullet}_*", yshift=-.7ex] & A^\CE(\cT^{\p\bullet}_\lag,\bfmu).
\end{tikzcd}
\]
Then as mentioned in Remark~\ref{remark:basepoint move commutativity}, both bordered consistent morphisms $\RM{B_i^{-1}}_*^{\p\bullet}$ and $\RM{b_i^{-1}}_*^{\p\bullet}$ are well-defined and left inverses of consistent morphisms $\RM{B_i}_*^{\p\bullet}$ and $\RM{b_i}_*^{\p\bullet}$, respectively.

\subsubsection{Legendrian graphs in a normal form}\label{section:augmentation category of normal form}
Recall the definition of Legendrian graphs or bordered Legendrian graphs in a normal form defined in Definition~\ref{definition:normal form}.

Let $(\cT,\bfmu)\in\BLT^\mu_\lag$ be a bordered Legendrian graph of type $(n_\Left,n_\Right)$ in a normal form and $(\cT_\lag,\bfmu)\in\BLT^\mu_\lag$ be its image of Ng's resolution.

Let us consider the canonical Lagrangian $m$-copies $\cT^{\p m}_\lag$.
One of the benefit of being in a normal form is indeed that the DGA $A^\CE_\co\left(\cT^{\p m}_\lag,\bfmu^{\p m}\right)$ can be easily described in terms of the data of $\cT_\lag$ and the DGA $A^\CE_\co(\cT_\lag,\bfmu)$.

Recall that the $x$-maximal points in $\cT$ are basepoints and their Ng's resolutions look as depicted in Definition~\ref{definition:Ng's resolution}.
Therefore the canonical Lagrangian $m$-copy near each basepoint $b\in\cT_\lag$ (or equivalently, a $x$-maximum) looks like
\[
\begin{tikzcd}
\begin{tikzpicture}[baseline=-0.5ex, scale=0.8]
\draw[thick] (-1,0.5) to[out=0,in=180] (0.5,0) node[below=-2ex] {$|$} to[out=180,in=0] (-1,-0.5);
\end{tikzpicture}\subset\cT\arrow[r,mapsto] &
\begin{tikzpicture}[baseline=-0.5ex,scale=0.8]
\draw[thick] (-1,0.5) arc(90:0:1 and 0.5) arc(0:-90:1 and 0.5) node[near end,below=-2ex,sloped] {$|$};
\end{tikzpicture}\subset\cT_\lag\arrow[r,"(\cdot)^{\p m}"] &
\begin{tikzpicture}[baseline=-0.5ex,scale=0.8]
\begin{scope}[yshift=-.25cm,color=blue]
\draw[white,line width=5] (-1,0.5) arc(90:0:1 and 0.5) arc(0:-90:1 and 0.5) node[near end,below=-2ex,sloped] {$|$};
\draw[thick] (-1,0.5) arc(90:0:1 and 0.5) arc(0:-90:1 and 0.5) node[near end,below=-2ex,sloped] {$|$};
\end{scope}
\begin{scope}[yshift=0,color=red]
\draw[white,line width=5] (-1,0.5) arc(90:0:1 and 0.5) arc(0:-90:1 and 0.5) node[near end,below=-2ex,sloped] {$|$};
\draw[thick] (-1,0.5) arc(90:0:1 and 0.5) arc(0:-90:1 and 0.5) node[near end,below=-2ex,sloped] {$|$};
\end{scope}
\begin{scope}[yshift=.25cm,color=black]
\draw[white,line width=5] (-1,0.5) arc(90:0:1 and 0.5) arc(0:-90:1 and 0.5) node[near end,below=-2ex,sloped] {$|$};
\draw[thick] (-1,0.5) arc(90:0:1 and 0.5) arc(0:-90:1 and 0.5) node[near end,below=-2ex,sloped] {$|$};
\end{scope}
\end{tikzpicture}\subset\cT_\lag^{\p m}
\end{tikzcd}
\]

Let us consider the set of $x$-minimal points in $\cT_\lag$.
Then it has one-to-one correspondence with the set of connected components of the complement of $V\amalg B$, which is the set of edges.

The set of generators for $A^{\CE}(T_\lag,\mu)$ is the union of $C$, $K_\Left$, $\tilde V$ and $\tilde B$
\begin{align*}
C&=\left\{c~\mid c\text{ is a double point}\right\},\\
K_\Left&=\left\{k_{ab}\mid 1\le a<b\le n_\Left\right\},\\
\tilde V&=\left\{v_{a,\ell}\mid v\in V, 1\le a\le \val(v), \ell\ge 1\right\},\\
\tilde B&=\left\{b_{a,\ell}\mid b\in B, a=1,2, \ell\ge 1\right\}.
\end{align*}

As seen in Definition~\ref{definition:normal form} and Example~\ref{example:DGAs for graphs in a normal form}, we define the sets $\tilde V_{\Left}$ and $\tilde V_{\Lcirclearrowright}$.
Then the generating set for $A^{\CE}\left(T_\lag^{\p m},\mu^{\p m}\right)$ is the union of the following sets:
\[
C^{\p m}\amalg K_{\Left}^{\p m}\amalg \tilde V_{\Left}^{\p m}\amalg \tilde V_{\Lcirclearrowright}^{\p m}\amalg \tilde B^{\p m}\amalg X^{\p m}\amalg Y^{\p m},
\]
where
\begin{align*}
C^\p{m}&\coloneqq\left\{c^{ij}~\middle|~ c\in C, i,j\in[m]\right\},& \left|c^{ij}\right|&\coloneqq|c|,\\
K_{\Left}^{\p m}&\coloneqq \left\{k_{ab}^{ij}~\middle|~1\le a<b\le n_\Left\right\},&\left|k_{ab}^{ij}\right|&\coloneqq\left|k_{ab}\right|,\\
\tilde V_{\Left}^{\p m}&\coloneqq\left\{v^{ij}_{a,\ell}~\middle|~ v_{a,\ell}\in \tilde V_{\Left}, i\le j\in[m]\right\},& \left|v_{a,\ell}^{ij}\right|&\coloneqq \left|v_{a,\ell}\right|,\\
\tilde V_{\Lcirclearrowright}^\p{m}&\coloneqq\left\{v_{a,\ell}^{ii}~\middle|~ v_{a,\ell}\in \tilde V_{\Lcirclearrowright},  i\in[m]\right\},& \left|v_{a,\ell}^{ii}\right|&\coloneqq \left|v_{a,\ell}\right|,\\
\tilde B^{\p m}&\coloneqq\left\{ b_{a,\ell}^{ii}~\middle|~ b_{a,\ell}\in \tilde B, i\in [m]\right\},& |b_{a,\ell}^{ii}|&\coloneqq|b_{a,\ell}|,\\
X^{\p m}&\coloneqq\left\{x_b^{ij}~\middle|~ b\in B,  i<j\in[m]\right\},& \left|x_b^{ij}\right|&\coloneqq 0,\\
Y^{\p m}&\coloneqq\left\{y_e^{ij}~\middle|~ y_e\in E, i< j\in[m]\right\}, &\left|y^{ij}\right|&\coloneqq -1.
\end{align*}

Geometrically, these generating sets corresponding to the crossings and vertex generators of $T^{\p m}$ as follows:
\begin{itemize}
\item Crossing: for each $c\in C$, there are $m^2$-many copies $c^{ij}$ in $C^{\p m}$
\begin{equation}
\begin{tikzcd}
\begin{tikzpicture}[baseline=-.5ex,scale=1]
\useasboundingbox (-1.5,-1.5) -- (1.5,1.5);
\draw[dashed,gray] (0,0) circle (1.5);
\clip (0,0) circle (1.5);
\draw[thick] (-2,-1.5) -- +(4,4);
\draw[white,line width=4] (-2,2.5) -- +(4,-4);
\draw[thick] (-2,2.5) -- +(4,-4);
\draw(0,0.7) node[above] {$c$};
\end{tikzpicture}\arrow[r,"(\cdot)^{\p m}"] &
\begin{tikzpicture}[baseline=-.5ex,scale=1]
\useasboundingbox (-1.5,-1.5) -- (1.5,1.5);
\draw[dashed,gray] (0,0) circle (1.5);
\begin{scope}
\clip (0,0) circle (1.5);
\draw[thick] (-2,-1.5) -- +(4,4);
\draw[red,thick] (-2,-2) -- +(4,4);
\draw[blue,thick] (-2,-2.5) -- +(4,4);
\draw[white,line width=4] (-2,2.5) -- +(4,-4);
\draw[white,line width=4] (-2,2) -- +(4,-4);
\draw[white,line width=4] (-2,1.5) -- +(4,-4);
\draw[thick] (-2,2.5) -- +(4,-4);
\draw[red,thick] (-2,2) -- +(4,-4);
\draw[blue,thick] (-2,1.5) -- +(4,-4);
\end{scope}
\draw(0,0.7) node[above] {$c^{11}$};
\draw(0.6,0) node[right] {$c^{1m}$};
\draw(0,-0.7) node[below] {$c^{mm}$};
\draw(-0.6,0) node[left] {$c^{m1}$};
\end{tikzpicture}
\end{tikzcd}
\end{equation}
\item Vertex: for each $v\in V$, there are two sets of generators $\tilde V_{v,\Left}^{\p m}$ and $\tilde V_{v,\Lcirclearrowright}^{\p m}$
\begin{equation}\label{eqn:m_copy_vertexlabel}
\begin{tikzcd}
\begin{tikzpicture}[baseline=-.5ex,scale=1]
\draw[dashed,gray] (0,0) circle (1.5);
\begin{scope}
\clip (0,0) circle (1.5);
\draw[thick,rounded corners] (-1.5,-.8)--(0,0.2);
\draw[thick,rounded corners] (-1.5,0)--(0,0.2);
\draw[thick,rounded corners] (-1.5,.8)--(0,0.2);
\fill (0,0.2) circle(1.5pt);
\end{scope}
\draw (210:1.5) node[left] {$1$};
\draw (180:1.5) node[left] {$2$};
\draw (150:1.5) node[left] {$3$};
\end{tikzpicture}\arrow[r,"(\cdot)^{\p m}"] &
\begin{tikzpicture}[baseline=-.5ex,scale=1]
\begin{scope}
\draw[dashed,gray] (0,0) circle (1.5);
\clip (0,0) circle (1.5);
\foreach[count=\index] \c in {blue, red, black} {
\begin{scope}[yshift=(\index-2)*0.2 cm, color=\c]
\draw[white,line width=4,rounded corners] (-1.5,-.8)--(0,{(\index-2)*0.3});
\draw[white,line width=4,rounded corners] (-1.5,0)--(0,{(\index-2)*0.3});
\draw[white,line width=4,rounded corners] (-1.5,.8)--(0,{(\index-2)*0.3});
\draw[thick,rounded corners] (-1.5,-.8)--(0,{(\index-2)*0.3});
\draw[thick,rounded corners] (-1.5,0)--(0,{(\index-2)*0.3});
\draw[thick,rounded corners] (-1.5,.8)--(0,{(\index-2)*0.3});
\fill (0,{(\index-2)*0.3}) circle(1.5pt);
\end{scope}
}
\end{scope}
\draw[-latex'] (0, 0.5) ++ (220:0.3) arc (220:155:0.3) node[above] {$v_{1,2}^{11}\in\tilde V^{\p m}_{\Left}$};
\draw[fill,opacity=0.5] (-0.68,-0.36) node (aa) { } circle (0.1);
\draw (0.5,-1.2) node (bb) {$v_{1,1}^{23}\in\tilde V^{\p m}_{\Left}$};
\draw[-latex'] (bb) -- (aa.south);
\draw (200:1.7) node {$1^1$};
\draw[red] (210:1.7) node {$1^2$};
\draw[blue] (185:1.7) node {$2^3$};
\draw (140:1.7) node {$3^1$};
\draw[red,-latex'] (150:0.2) arc (150:-150:0.2);
\draw (0.3,0) node[right] {$v_{3,1}^{22}\in\tilde V_{\Lcirclearrowright}^{\p m}$};
\end{tikzpicture}
\end{tikzcd}
\end{equation}
\item Left border: for each $k_{ab}\in K_\Left$ with $a<b$, there is the set $k_\Left^{\p m}$ of $m^2$ generators $k_{ab}^{ij}$, and for each $i<j$ and $1\le a\le \ell$, there are $\binom{m}2$ generators $y_a^{ij}$,
\begin{equation}\label{equation:m-copies of border}
\begin{tikzcd}
\begin{tikzpicture}[baseline=-.5ex]
\useasboundingbox (-1,-1.5)--(0.5,1.5);
\draw[red] (0,-1.5)--(0,1.5);
\foreach \y in {1.2,0.2,-0.8} {
	\draw[thick] (0,\y) -- +(0.5,0);
}
\draw[thick, <-] (-0.3,1.2)-- node[midway,left] {$k_{ab}$} (-0.3,-0.8);
\end{tikzpicture}\arrow[r,"(\cdot)^{\p m}"] &
\begin{tikzpicture}[baseline=-.5ex]
\useasboundingbox (-1.5,-1.5)--(0.5,1.5);
\draw[red] (0,-1.5)--(0,1.5);
\foreach \y in {1.2,0.2,-0.8} {
	\draw[thick] (0,\y) -- +(0.5,0);
	\draw[red,thick] (0,\y-0.2) -- +(0.5,0);
	\draw[blue,thick] (0,\y-0.4) -- +(0.5,0);
}
\draw[thick, <-] (-0.3,1)-- node[midway, left] {$k_{ab}^{ij}$} (-0.3,-0.8);
\draw[thick, <-] (-0.6,1.2) -- node[midway, left] {$y_a^{ii}$} (-0.6,0.8);
\end{tikzpicture}
\end{tikzcd}
\end{equation}
\item basepoint: for each $b_{a,\ell}\in\tilde B$, we have the set of $m$ generators $b_{a,\ell}^{ii}$
\begin{equation}
\begin{tikzcd}
\begin{tikzpicture}[baseline=-.5ex]
\useasboundingbox (-1.5,-1.5) -- (1.5,1.5);
\draw[dashed,gray] (0,0) circle (1.5);
\clip (0,0) circle (1.5);
\draw[thick] (-1.5,0.6) -- node[midway,below=-2ex] {$|$} node[midway,above] {$b$} ++(3,0);
\draw[thick,-latex'] (0,0.6) + (0:0.3) arc (0:-180:0.3) node[midway,below] {$b_{2,1}$};
\end{tikzpicture}\arrow[r,"(\cdot)^{\p m}"]&
\begin{tikzpicture}[baseline=-.5ex]
\useasboundingbox (-1.5,-1.5) -- (1.5,1.5);
\draw[dashed,gray] (0,0) circle (1.5);
\begin{scope}
\clip (0,0) circle (1.5);
\draw[thick] (-1.5,0.6) -- node[midway,below=-2ex] {$|$} node[midway,above] {$b^1$} ++(3,0);
\draw[red,thick] (-1.5,0) -- node[midway,below=-2ex] {$|$} node[midway,above] {$b^2$} ++(3,0);
\draw[blue,thick] (-1.5,-0.6) -- node[midway,below=-2ex] {$|$} node[midway,above] {$b^3$} ++(3,0);
\end{scope}
\draw[blue,thick,-latex'] (0,-0.6) + (0:0.3) arc (0:-180:0.3) node[midway,below] {$b_{2,1}^{33}$};
\end{tikzpicture}
\end{tikzcd}
\end{equation}
\item $x$-maximum: for each $b\in B$, we have the set of $\binom{m}2$ generators $x_b^{ij}$ for $i<j$
\begin{equation}
\begin{tikzcd}
\begin{tikzpicture}[baseline=-.5ex]
\useasboundingbox (-1.5,-1.5) -- (1.5,1.5);
\draw[dashed,gray] (0,0) circle (1.5);
\clip (0,0) circle (1.5);
\draw[thick] (-1.5,0.3)+(90:1.5 and 0.5) arc (90:-90:1.5 and 0.5) node[pos=0.9,sloped,below=-2ex] {$|$} node[pos=0.9,above] {$b$};
\draw (-1.5,0.3)+(0:1.5 and 0.5); 
\end{tikzpicture}\arrow[r,"(\cdot)^{\p m}"]&
\begin{tikzpicture}[baseline=-.5ex]
\useasboundingbox (-1.5,-1.5) -- (1.5,1.5);
\draw[dashed,gray] (0,0) circle (1.5);
\clip (0,0) circle (1.5);
\draw[blue,thick] (-1.5,-0.3)+(90:1.5 and 0.5) arc (90:-90:1.5 and 0.5) node[pos=0.9,sloped,below=-2ex] {$|$};
\draw[white,line width=4] (-1.5,0)+(90:1.5 and 0.5) arc (90:-90:1.5 and 0.5);
\draw[red,thick] (-1.5,0)+(90:1.5 and 0.5) arc (90:-90:1.5 and 0.5) node[pos=0.9,sloped,below=-2ex] {$|$};
\draw[white,line width=4] (-1.5,0.3)+(90:1.5 and 0.5) arc (90:-90:1.5 and 0.5);
\draw[thick] (-1.5,0.3)+(90:1.5 and 0.5) arc (90:-90:1.5 and 0.5) node[pos=0.9,sloped,below=-2ex] {$|$};
\draw (-1.5,0.3)+(0:1.5 and 0.5) node[right] {$x_b^{ij}$};
\end{tikzpicture}
\end{tikzcd}
\end{equation}
\item $x$-minimum: for each $e\in E$, we have the subset of $\binom{m}2$ generators $y_e^{ij}$ for $i<j$
\begin{equation}
\begin{tikzcd}
\begin{tikzpicture}[baseline=-.5ex]
\useasboundingbox (-1.5,-1.5) -- (1.5,1.5);
\draw[dashed,gray] (0,0) circle (1.5);
\clip (0,0) circle (1.5);
\draw[thick] (1.5,0.3)+(90:1.5 and 0.5) arc (90:270:1.5 and 0.5);
\draw (1.5,0.3)+(180:1.5 and 0.5) node[left] {$e$};
\end{tikzpicture}\arrow[r,"(\cdot)^{\p m}"]&
\begin{tikzpicture}[baseline=-.5ex]
\useasboundingbox (-1.5,-1.5) -- (1.5,1.5);
\draw[dashed,gray] (0,0) circle (1.5);
\clip (0,0) circle (1.5);
\draw[blue,thick] (1.5,-0.3)+(90:1.5 and 0.5) arc (90:270:1.5 and 0.5);
\draw[white,line width=4] (1.5,0)+(90:1.5 and 0.5) arc (90:270:1.5 and 0.5);
\draw[red,thick] (1.5,0)+(90:1.5 and 0.5) arc (90:270:1.5 and 0.5);
\draw[white,line width=4] (1.5,0.3)+(90:1.5 and 0.5) arc (90:270:1.5 and 0.5);
\draw[thick] (1.5,0.3)+(90:1.5 and 0.5) arc (90:270:1.5 and 0.5);
\draw (1.5,0.3)+(180:1.5 and 0.5) node[left] {$y_e^{ij}$};
\end{tikzpicture}
\end{tikzcd}
\end{equation}
\end{itemize}

Note that for each $v_{a,\ell}\in \tilde \sfV_{\Left}$, the generator $v_{a,\ell}^{ij}$ corresponds to a crossing generator between $a$-th and $(a+\ell)$-th edges in the $i$-th and $j$-th component, respectively, if $i<j$.
If $i=j$, then it is the corresponding vertex generator on the $i$-th component.

On the other hand, for $v_{a,\ell}\in\tilde \sfV_{\Lcirclearrowright}$, the generator $v_{a,\ell}^{ii}$ is the corresponding vertex generator as well but there are no generators $v_{a,\ell}^{ij}$ if $i\neq j$.

Let us consider matrices as follows: 
for each $c\in \sfC$, $k_{ab}\in K_\Left^{\p m}$, $v_{a,\ell}\in\tilde V_\Left$, $w_{a,\ell}\in \tilde V_{\Lcirclearrowright}$, $b_{a,\ell}\in \tilde B$, $b\in B$ and $e\in E$,
\begingroup
\allowdisplaybreaks
\begin{align*}
\Phi(c)&\coloneqq \left(
\begin{matrix}
c^{11}& c^{12} &\cdots & c^{1m} \\
c^{21} & c^{22} & \cdots & c^{2m} \\
\vdots & \vdots  & \ddots & \vdots\\
c^{m1} & c^{m2} & \cdots & c^{mm}
\end{matrix}\right),&
\Phi(k_{ab})&\coloneqq \left(
\begin{matrix}
k_{ab}^{11}& k_{ab}^{12} &\cdots & k_{ab}^{1m}\\
k_{ab}^{21} & k_{ab}^{22} & \cdots & k_{ab}^{2m}\\
\vdots & \vdots  & \ddots & \vdots\\
k_{ab}^{m1} & k_{ab}^{m2} & \cdots & k_{ab}^{mm}
\end{matrix}\right),\\
\Phi(v_{a,\ell})&\coloneqq \left(
\begin{matrix}
v_{a,\ell}^{11}& v_{a,\ell}^{12} &\cdots & v_{a,\ell}^{1m} \\
0 & v_{a,\ell}^{22} & \cdots & v_{a,\ell}^{2m} \\
\vdots & \vdots  & \ddots & \vdots\\
0 & 0 & \cdots & v_{a,\ell}^{mm}
\end{matrix}\right),&
\Phi(w_{a,\ell})&\coloneqq \left(
\begin{matrix}
w_{a,\ell}^{11}& 0 &\cdots & 0\\
0 &w_{a,\ell}^{22} & \cdots & 0\\
\vdots & \vdots  & \ddots & \vdots\\
0 & 0 & \cdots & w_{a,\ell}^{mm}
\end{matrix}\right),\\
\Delta(b_{a,\ell})&\coloneqq \left(
\begin{matrix}
b_{a,\ell}^{11} & 0 &\cdots & 0 \\
0 & b_{a,\ell}^{22} & \cdots & 0 \\
\vdots & \vdots  & \ddots & \vdots\\
0 & 0 & \cdots & b_{a,\ell}^{mm}
\end{matrix}\right),&
X_b&\coloneqq \left(
\begin{matrix}
1& x_b^{12} &\cdots & x_b^{1m} \\
0 & 1 & \cdots & x_b^{2m} \\
\vdots & \vdots  & \ddots & \vdots\\
0 & 0 & \cdots & 1
\end{matrix}\right),\\
Y_e&\coloneqq \left(
\begin{matrix}
0& y_e^{12} &\cdots & y_e^{1m} \\
0 &0 & \cdots & y_e^{2m} \\
\vdots & \vdots  & \ddots & \vdots\\
0 & 0 & \cdots & 0
\end{matrix}\right).
\end{align*}
\endgroup

We also define
\[
\Phi(b_{1,1})\coloneqq \Delta(b_{1,1}) X_b\quad\text{ and }\quad
\Phi(b_{2,1})\coloneqq X_b^{-1}\Delta(b_{2,1}).
\]

The differential $\differential^\p{m}$ can be given as follows:
\begin{itemize}
\item for each $c^{ij}\in C^{\p m}$ with upper and lower edges $e, e'\in E$,
\[
\differential^\p{m}\left(c^{ij}\right)\coloneqq\displaystyle \Phi(\differential(c))_{ij}+\sum_{k<j} (-1)^{|c|-1}c^{ik} y_{e'}^{kj} + \sum_{i<k} y_e^{ik} c^{kj}
\]
\item for each $k_{ab}^{ij}\in K_\Left^{\p m}$ with two edges $e$ and $e'$ containing left borders corresponding $a$ and $b$,
\[
\differential^\p{m}\left(k_{ab}^{ij}\right)\coloneqq\displaystyle \Phi(\differential(k_{ab}))_{ij}+\sum_{k<j} (-1)^{|k_{ab}|-1}k_{ab}^{ik} y_{e'}^{kj} + \sum_{i<k} y_e^{ik} k_{ab}^{kj}
\]
\item for each $v_{a,\ell}\in \tilde V_\Left$ with two half-edges $h_{v,a}\subset e\in E$ and $h_{v,a+\ell}\subset e'\in E$,
\[
\differential^\p{m}\left(v_{a,\ell}^{ij}\right)\coloneqq\displaystyle \Phi(\differential(v_{a,\ell}))_{ij}+\sum_{i<k<j} (-1)^{|v_{a,\ell}|-1}v_{a,\ell}^{ik} y_{e'}^{kj} + y_e^{ik} v_{a,\ell}^{kj}
\]
\item for each $b\in B$ with two edges $(e,e')$ adjacent to $b$,
\[
\differential^\p{m}\left(x_b^{ij}\right)\coloneqq\displaystyle \sum_{i<k<j} -x_b^{ik} y_{e'}^{kj} + y_e^{ik} x_b^{kj}
\]
\item for each $e\in E$, $b_{a,\ell}\in \tilde B$ and $w_{a,\ell}\in \tilde V_\Lcirclearrowright$,
\begin{align*}
\differential^\p{m}\left(b_{a,\ell}^{ii}\right)&\coloneqq \Phi(\differential(b_{a,\ell}))_{ii}=\displaystyle\delta_{\ell,2}+
\sum_{\ell_1+\ell_2=\ell} (-1)^{|b_{a,\ell_1}|-1}b_{a,\ell_1}^{ii}b_{a+\ell_1,\ell_2}^{ii},\\
\differential^\p{m}\left(w_{a,\ell}^{ii}\right)&\coloneqq \Phi(\differential(w_{a,\ell}))_{ii}=\displaystyle\delta_{\ell,\val(w)}+
\sum_{\ell_1+\ell_2=\ell} (-1)^{|w_{a,\ell_1}|-1}w_{a,\ell_1}^{ii}w_{a+\ell_1,\ell_2}^{ii},\\
\differential^\p{m}\left(y_e^{ij}\right)&\coloneqq (Y_e)^2_{ij}=\displaystyle\sum_{i<k<j} y_e^{ik}y_e^{kj}
\end{align*}
\end{itemize}

Now consider the structure morphisms $\phi_\Left^{\p m}:A_{n_\Left}^{\p m}\to A^{\p m}$ and $\phi_\Right^{\p m}:A_{n_\Right}^{\p m}\to A^{\p m}$.
Since $\phi_\Left^{\p m}$ is obvious, we need to consider $\phi_\Right^{\p m}$.
Then due to the geometric definition for $\phi_\Right^{\p m}$ which counts immersed polygons with the positive end at each generator for $A_{n_\Right}^{\p m}$, we have the similar description to the differential $\differential^{\p m}$ so that
\begin{align}\label{equation:structure morphism of Legendrian in a normal form}
\phi_\Right^{\p m}(k^{ij}_{a'b'})&=\Phi(\phi_\Right(k_{a'b'}))_{ij},&
\phi_\Right^{\p m}(y^{ij}_{c'})&=y^{ij}_e,
\end{align}
where $1\le a'<b'\le n_\Right$, $c'\in[n_\Right]$ and $e$ is the edge whose one end point is $c'$.

Here we are using implicitly the fact that any generator in $\tilde V_{\Lcirclearrowright}$ never appear in a differential of other types of generators or an image of $\phi_\Right$ as observed in Example~\ref{example:DGAs for graphs in a normal form}.

\section{Augmentation categories}\label{section:aug cats}
In this section, we review the construction of the augmentation categories which are $A_\infty$-categories obtained from consistent sequences of DGAs.

\subsection{Augmentation categories for consistent sequences of bordered DGAs}\label{section:augmentation category}
Let $K=(\field,\differential_\field\equiv0)$ be a DGA consisting of the base field $\field$ with the trivial differential.
An \emph{augmentation} for $A=(\alg=T\sfM,\differential)\in\DGA$ is a DGA morphism 
\[
\epsilon:A\to K.
\]

One can extend naturally the definition of augmentations to $m$-component DGAs.
Let $K^{\p m}\coloneqq\Mat_m(K)$. Then it satisfies the axiom of $m$-component DGAs.

\begin{definition}[Augmentations for $m$-component DGAs]
For $A^{\p m}=(\alg^{\p m}=T\sfM^{\p m},\differential^{\p m})\in\DGA_\co^{\p m}$ with $\sfM^{\p m}=(\sfM^{ij})$, an augmentation of $A^{\p m}$ is a morphism
\[
\epsilon^{\p m}:A^{\p m}\to K^{\p m},
\]
where
\begin{align*}
\epsilon^{\p m}&=\left(\epsilon^{\p m}(i,j)\right),&
\epsilon^{\p m}(i,j)&:A^{\p m}(i,j)\to K,&
\epsilon^{\p m}(i,i)\left(\unit^i_{E}\right)&=\unit\in\field.
\end{align*}

We denote the restriction of $\epsilon^{\p m}$ to the graded submodule $\sfM^{ij}$ by $\epsilon^{ij}\coloneqq \epsilon^{\p m}(i,j)\big|_{\sfM^{ij}}: \sfM^{ij}\to \field$, and we say that $\epsilon^{\p\bullet}$ is \emph{diagonal} if $\epsilon^{ij}\equiv 0$ for $i\neq j$.
Conversely, for a sequence $\bfe=(\epsilon_1,\dots,\epsilon_m)$ of augmentations of $A$, the diagonal augmentation $\epsilon^{\p m}_\bfe\coloneqq(\epsilon^{ij}_{\bfe})$ of $A^{\p m}$ is defined as
\begin{align*}
\epsilon^{ii}_{\bfe}&:\sfM^{ii}\stackrel{\epsilon_i}\longrightarrow \field,&
\epsilon^{ij}_{\bfe}&\equiv0,\quad \forall i\neq j.
\end{align*}

The DGA $A^{\p m}_\bfe=(\alg^{\p m}_\bfe, \differential^{\p m}_\bfe)$ is defined as
\[
\alg^{\p m}_\bfe\coloneqq\alg^{\p m}\quad\text{ and }\quad
\differential^{\p m}_\bfe\coloneqq\phi_\bfe^{\p m}\circ\differential^{\p m}\circ \left(\phi_\bfe^{\p m}\right)^{-1},
\]
where $\phi^{\p m}_\bfe:\alg^{\p m}\to\alg^{\p m}_\bfe$ is an algebra tame isomorphism such that for each generator $s\in R^{\p m}$,
\[
\phi_\bfe^{\p m}(s)\coloneqq s + \epsilon^{\p m}_\bfe(s).
\]
\end{definition}

\begin{lemma}
The $0$-th differential $\differential^{\p m}_{\bfe,0}$ of the length filtration $\differential^{\p m}_{\bfe,\ell}$ vanishes.
\end{lemma}
\begin{proof}
The proof follows from the direct computation.
\end{proof}

Let $A^{\p\bullet}\in\DGA_\co^{\p\bullet}$ and $\bfe=(\epsilon_m)$ be a sequence of augmentations of $A$.
Then by Lemma~\ref{lemma:isom}, it is obvious that for any $i<j\in[m]$, there is an isomorphism between $\left(\ring^i_{E}, \ring^j_{E}\right)$-modules
\[
\sfM^{ij}=\alg^{ij}_\bfe\isomorphic \alg^{12}_{(\epsilon_i,\epsilon_j)}.
\]

Let $\sfM_{ij}^\vee\coloneqq \left(\sfM^{ij}\right)^*[-1]$ be the dual space of $\sfM^{ij}$ with grading shift by $-1$.
That is, both $\sfM^{ij}$ and $\sfM_{ij}^\vee$ are decomposed into graded pieces
\begin{align*}
\sfM^{ij}&=\bigoplus_{d\in\ZZ} \left(\sfM^{ij}\right)^d,&
\sfM_{ij}^\vee&=\bigoplus_{d\in\ZZ} \sfM_{ij}^{\vee d},&
\sfM_{ij}^{\vee d+1}&\coloneqq\hom_\field\left(\left(\sfM^{ij}\right)^d,\field\right).
\end{align*}

For each increasing sequence $\bfi\coloneqq(i_1,\cdots, i_{k+1})$ in $[m]$ and a sequence of integers $\bfd=(d_1,\cdots,d_k)\in\ZZ^k$, let us denote the spaces
\begin{align*}
\sfM^\bfi&\coloneqq \sfM^{i_1i_2}\otimes \sfM^{i_2i_3}\otimes\cdots\otimes \sfM^{i_ki_{k+1}},&
\sfM_\bfi^\vee&\coloneqq \sfM_{i_ki_{k+1}}^\vee\otimes \sfM_{i_{k-1}i_k}^\vee\otimes\cdots\otimes \sfM_{i_1i_2}^\vee,\\
\left(\sfM^\bfi\right)^\bfd&\coloneqq \left(\sfM^{i_1i_2}\right)^{d_1}\otimes \left(\sfM^{i_2i_3}\right)^{d_2}\otimes\cdots\otimes \left(\sfM^{i_ki_{k+1}}\right)^{d_k},&
\sfM_\bfi^{\vee\bfd}&\coloneqq \sfM_{i_ki_{k+1}}^{\vee d_k}\otimes \sfM_{i_{k-1}i_k}^{\vee d_{k-1}}\otimes\cdots\otimes \sfM_{i_1i_2}^{\vee d_1},
\end{align*}
whose elements will be denoted as
\begin{align*}
a^\bfi&=a^{i_1i_2}\otimes a^{i_2i_3}\otimes\cdots\otimes a^{i_ki_{k+1}}\in \sfM^\bfi,&  a^{ij}&\in \sfM^{ij},\\
a_\bfi^\vee&=a_{i_ki_{k+1}}^\vee\otimes a_{i_{k-1}i_k}^\vee\otimes\cdots\otimes a_{i_1i_2}^\vee \in \sfM_\bfi^\vee,& a_{ij}^\vee&\in \sfM_{ij}^\vee.
\end{align*}

We equip a grading on $\sfM_\bfi^\vee\otimes \sfM^\bfi$ so that for homogeneous elements $a_\bfi^\vee\in \sfM_\bfi^\vee$ and $a'^\bfi\in \sfM^\bfi$
\[
|a_\bfi^\vee\otimes a'^\bfi|\coloneqq |a_\bfi^\vee|-|a'^\bfi|.
\]

Consider the natural pairing $\langle-,-\rangle_\bfi:\sfM_{\bfi}^\vee\otimes \sfM^{\bfi}\to\field$ between $\sfM_\bfi^\vee$ and $\sfM^\bfi$ which is nonvanishing only on 
\begin{align*}
\bigoplus_{\bfd\in\ZZ^k}& \left(\sfM_\bfi^{\vee \bfd+\mathbf{1}}\otimes \left(\sfM^\bfi\right)^\bfd\right),
\end{align*}
where $\bfd+\mathbf{1}\coloneqq(d_1+1,d_2+1,\dots,d_k+1)$.
Since each evaluation $\langle-,-\rangle_{ij}:\sfM_{ij}^{\vee d+1}\otimes \left(\sfM^{ij}\right)^d\to \field$ is of degree $-1$, the pairing $\langle-,-\rangle_\bfi$ is of grading $-k$.

Now let us define a composition map $m_\bfi:\sfM_\bfi^\vee\to \sfM_{i_1i_{k+1}}^\vee$ as follows:
for each $a_\bfi^\vee\in \sfM_\bfi^\vee$, we require that $m_\bfi\left(a_\bfi^\vee\right)$ satisfies that for each $a^{i_1i_{k+1}}\in \sfM^{i_1i_{k+1}}$,
\begin{align}\label{equation:composition}
\left\langle m_\bfi\left(a_\bfi^\vee\right), a^{i_1i_{k+1}}\right\rangle_{i_1i_{k+1}}= (-1)^\sigma
\left\langle a_\bfi^\vee, \differential^{\p m}_\bfe\left(a^{i_1i_{k+1}}\right)\right\rangle_\bfi,
\end{align}
where $\differential^{\p m}_\bfe$ is the twisted differential on $A^{\p m}_\bfe$.

\begin{definition}\label{def:sign for m_k}
For any $k\geq 1$, and any sequence $(x_1,x_2,\ldots,x_k)$ with grading $|x_i|\in\ZZ$, define
\begin{equation}\label{equation:sigma}
\sigma_k(x_1,x_2,\dots,x_k)\coloneqq k(k-1)/2+\sum_{i<j}|x_i||x_j|+\sum_{1\le j\le k} (j-1)|x_{j}|
\end{equation}
In particular, $\sigma_1=0$ and $\sigma_2(x_1,x_2)=1+|x_1||x_2|+|x_2|$.
\end{definition}

Then $\sigma$ in \eqref{equation:composition} is given as
\[
\sigma=\sigma_k\left(a_{i_ki_{k+1}}^\vee,\dots,a_{i_1i_2}^\vee\right)
\]
and $m_\bfi\left(a_\bfi^\vee\right)$ is completely determined by this equation.

The degree of $m_\bfi$ can be computed as follows: since the degrees of pairings on the left and right are $-1$ and $-k$, respectively, we have
\[
\left|m_\bfi\left(a_\bfi^\vee\right)\otimes a^{i_1i_{k+1}}\right|+\left| \langle-,-\rangle_{i_1i_{k+1}}\right|=0=
\left|a_\bfi^\vee\otimes \differential^{\p m}_\bfe\left(a^{i_1i_{k+1}}\right)\right|+\left|\langle-,-\rangle_\bfi\right|.
\]
This is equivalent to
\begin{align*}
\left|m_\bfi\left(a_\bfi^\vee\right)\right|- \left|a_\bfi^\vee\right| &= 1 - k +\left|a^{i_1i_{k+1}}\right|-\left|\differential^{\p m}_\bfe\left(a^{i_1i_{k+1}}\right)\right|=2-k,
\end{align*}
which implies that the composition map $m_\bfi$ is of degree $2-k$.

\begin{definition}
The \emph{augmentation category} $\Aug_+\left(A^\p\bullet;K\right)$ is an $A_\infty$-category defined as follows:
\begin{itemize}
\item The objects are augmentations $\epsilon:A=A^{11}\to K$;
\item The morphisms are graded vector spaces
\[
\hom_{\Aug_+}(\epsilon_1,\epsilon_2)\coloneqq\left(A^{12}_{(\epsilon_1,\epsilon_2)}\right)^\vee\isomorphic \ring\left\langle \sfS^{12}\right\rangle^\vee;
\]
\item For $k\ge 1$, the composition map
\[
m_k:\hom_{\Aug_+}(\epsilon_k,\epsilon_{k+1})\otimes\cdots\otimes\hom_{\Aug_+}(\epsilon_1,\epsilon_2)\to
\hom_{\Aug_+}(\epsilon_1,\epsilon_{k+1})
\]
is defined as $m_\bfi$ with the sequence $\bfi=(1,2,\cdots,k+1)$.
\end{itemize}
\end{definition}

\begin{proposition}[Functoriality of $\Aug_+$\cite{NRSSZ2015}]\label{proposition:functoriality}
The assignment $A^\p\bullet\mapsto \Aug_+\left(A^\p\bullet;K\right)$ defines a contravariant functor from the category of consistent sequences of DGAs onto the category $\Alg_\infty$ of $A_\infty$-categories.
\end{proposition}
\begin{proof}
This is just a combination of Propositions~3.17 and 3.20 in \cite{NRSSZ2015} and so we omit the proof.
\end{proof}

We briefly review the construction of the $A_\infty$-functor 
\[
\funF^{\p\bullet}=\left(\funF^{\p {k}}\right)_{k\ge1}:\Aug_+\left(A^{\p\bullet};K\right)\to\Aug_+\left(A'^{\p\bullet};K\right)
\]
induced from the consistent morphism $f^\p\bullet:A'^{\p\bullet}\to A^{\p\bullet}$
\begin{align*}
\funF^{\p0}&:\Ob\left(\Aug_+\right)\to\Ob\left(\Aug_+'\right),\\
\funF^{\p k}&:\hom_{\Aug_+}\left(\epsilon_k,\epsilon_{k+1}\right)\otimes\cdots\otimes\hom_{\Aug_+}\left(\epsilon_1,\epsilon_2\right)\to
\hom_{\Aug_+'}\left(\funF^{\p0}\left(\epsilon_1\right), \funF^{\p0}\left(\epsilon_{k+1}\right)\right),
\end{align*}
where $\bfe=\left(\epsilon_1,\cdots,\epsilon_{k+1}\right)$ is a sequence of augmentations for $A$, and 
\begin{align*}
\Aug_+&\coloneqq \Aug_+\left(A^{\p\bullet};K\right),&
\Aug_+'&\coloneqq \Aug_+\left(A'^{\p\bullet};K\right).
\end{align*}

Let $\funF^{\p0}\left(\epsilon_i\right)\coloneqq \epsilon_i^f$ be the pull-back for each $i$ by $\epsilon_i$, which is an augmentation for $A$.
\begin{align*}
\bfe^f&\coloneqq\left(\epsilon_1^f,\cdots,\epsilon_{k+1}^f\right),&
\epsilon_i^f(a)&\coloneqq \epsilon_i(f(a)).
\end{align*}

Recall the DGAs $A^{\p{k+1}}_\bfe=\left(\alg^{\p{k+1}}_\bfe, \differential^{\p{k+1}}_\bfe\right)$ and $A'^{\p{k+1}}_\bfe=\left(\alg'^{\p{k+1}}_{\bfe^f}, \differential'^{\p{k+1}}_{\bfe^f}\right)$ twisted by the algebra automorphisms $\phi_\bfe^{\p{k+1}}$ and $\phi_{\bfe^f}^{\p{k+1}}$ on $A^\p{k+1}$ and $A'^{\p{k+1}}$, respectively.
Then the composition 
\[
f^{\p{k+1}}_{\bfe}\coloneqq\phi^\p{k+1}_{\bfe}\circ f^\p{k+1}\circ \left(\phi^\p{k+1}_{\bfe^f}\right)^{-1}:A'^{\p{k+1}}_{\bfe^f} \to A_{\bfe}^{\p{k+1}}
\]
becomes a DGA morphism since
\begin{align*}
f^\p{k+1}_{\bfe}\circ\differential'^{\p{k+1}}_{\bfe^f}&=
\left(\phi^{\p{k+1}}_{\bfe}\circ f^{\p{k+1}}\circ \left(\phi^{\p{k+1}}_{\bfe^f}\right)^{-1}\right)\circ\differential'^{\p{k+1}}_{\bfe^f}\\
&=\phi^{\p{k+1}}_{\bfe}\circ f^\p{k+1}\circ \differential'^{\p{k+1}} \circ \left(\phi^{\p{k+1}}_{\bfe^f}\right)^{-1}\\
&=\phi^\p{k+1}_{\bfe}\circ \differential^\p{k+1} \circ  f^\p{k+1}\circ \left(\phi^{\p{k+1}}_{\bfe^f}\right)^{-1}\\
&=\differential_\bfe^\p{k+1} \circ \left(\phi^\p{k+1}_{\bfe}\circ f^\p{k+1}\circ \left(\phi^{\p{k+1}}_{\bfe^f}\right)^{-1}\right)\\
&=\differential_\bfe^\p{k+1} \circ f^\p{k+1}_{\bfe}.
\end{align*}
\begin{notation}
The $\ell$-th length filtration of $f^{\p {k+1}}_{\bfe}$ will be denoted by $f^{\p{k+1}}_{\bfe,\ell}$.
\end{notation}

The $A_\infty$-functor $\funF^{\p\bullet}$ will be defined by dualizing the composition $f^\p{k+1}_{\bfe}$.
More precisely,
\[
\funF^{\p k}:\hom_{\Aug_+}(\epsilon_k,\epsilon_{k+1})\otimes\cdots\otimes\hom_{\Aug_+}(\epsilon_1,\epsilon_2)\to \hom_{\Aug_+}\left(\epsilon_1^f,\epsilon_{k+1}^f\right)
\]
is defined as follows: for each $a_\bfi^\vee\coloneqq a_{k,k+1}^\vee\otimes\cdots\otimes a_{1,2}^\vee$ with $a_{i,i+1}^\vee\in \left(A^{12}_{(\epsilon_i,\epsilon_{i+1})}\right)^\vee$, 
\[
\left\langle\funF^{\p k}\left(a_\bfi^\vee\right), a'^{1,k+1}\right\rangle_{1,k+1}=
(-1)^\sigma\left\langle a_\bfi^\vee, f^\p{k+1}_{\bfe}\left(a'^{1,k+1}\right)\right\rangle'_\bfi,
\]
where $\sigma$ is the same as \eqref{equation:sigma} in Definition~\ref{def:sign for m_k}.

\begin{example}[Augmentation category for border DGAs]\label{example:augmentation category for border DGAs}
Let $\mu:[n]\to\grading$ be a function and $A_n^{\p \bullet}(\mu)$ be the consistent sequence of border DGAs defined in Section~\ref{section:consistent sequence of LCH DGAs}.
For simplicity, we denote $\Aug_+(A_n^{\p\bullet}(\mu);\field)$ by $\Aug_+$.

For each $m\ge 1$, the algebra $\alg_n^{\p\bullet}$ has the generating sets
\[
K_n\coloneqq\left\{k_{ab}^{ij}~\middle|~ a<b, 1\le i,j\le m\right\}\quad\text{ and }\quad
Y_n\coloneqq\left\{y_a^{ij}~\middle|~1\le a\le n, 1\le i<j\le m\right\},
\]
where the grading is given as
\begin{align*}
|k_{ab}^{ij}|&\coloneqq \mu(a)-\mu(b)-1,&
|y_a^{ij}|&\coloneqq -1.
\end{align*}
The differential for each generator is given as follows:
\begin{align*}
\differential^{\p m}_n(k_{ab}^{ij})&\coloneqq\sum_{\substack{a<c<b\\i<k<j}}(-1)^{|k_{ac}|-1}k_{ac}^{ik}k_{cb}^{kj}+\sum_{k<j}(-1)^{|k_{ab}|-1}k_{ab}^{ik}y_b^{kj}+\sum_{i<k}y_a^{ik}k_{ab}^{kj},\\
\differential^{\p m}_n(y_a^{ij})&\coloneqq\sum_{i<k<j} y_a^{ik}y_a^{kj}.
\end{align*}

The set of objects of $\Aug_+$ is the augmentation variety for $A_n(\mu)$:
\begin{equation}
\Ob\left(\Aug_+\right)=\aug(A_n(\mu);\field)\coloneqq\{\epsilon:A_n(\mu)\to \field\},
\end{equation}
and for any $\epsilon_1,\epsilon_2\in\Aug_+$, the set $\hom_{\Aug_+}(\epsilon_1,\epsilon_2)$ of morphisms is
\begin{equation}
\hom_{\Aug_+}(\epsilon_1,\epsilon_2)=\sfM^{12\vee}=\field\left\langle k_{ab}^{12\vee}, y_c^{12\vee}~\middle|~1\leq a< b\leq n, c\in[n]\right\rangle.
\end{equation}

The composition map $m_k$ is given as follows:
\begin{itemize} 
\item For $\epsilon_1,\epsilon_2\in\Aug_+$, the map 
\[
m_1:\hom_{\Aug_+}(\epsilon_1,\epsilon_2)\to\hom_{\Aug_+}(\epsilon_1,\epsilon_2)
\]
is defined as
\begin{align*}
m_1\left(k_{ab}^{12\vee}\right)&=-\sum_{c<a}\epsilon_1(k_{ca})k_{cb}^{12\vee} +\sum_{b<d}(-1)^{|k_{ab}^{12\vee}|}k_{ad}^{12\vee}\epsilon_2(k_{bd}),\\
m_1\left(y_c^{12\vee}\right)&=-\sum_{a<c}\epsilon_1(k_{ac})k_{ac}^{12\vee} +\sum_{c<d}k_{cd}^{12\vee}\epsilon_2(k_{cd}).
\end{align*}
\item For $\epsilon_1,\epsilon_2,\epsilon_3\in\Aug_+$, the map 
\[
m_2:\hom_{\Aug_+}(\epsilon_2,\epsilon_3)\otimes \hom_{\Aug_+}(\epsilon_1,\epsilon_2)\to \hom_{\Aug_+}(\epsilon_1,\epsilon_3)
\]
is defined as
\begin{align*}
m_2(k_{cd}^{12\vee}\otimes k_{ab}^{12\vee})&=\delta_{bc}(-1)^{|k_{ab}||k_{cd}|+|k_{ab}|+|k_{cd}|}k_{ad}^{12\vee},\\
m_2(y_{c}^{12\vee}\otimes k_{ab}^{12\vee})&=-\delta_{bc} k_{ab}^{12\vee},\\
m_2(k_{cd}^{12\vee}\otimes y_{a}^{12\vee})&=-\delta_{ac} k_{cd}^{12\vee},\\
m_2(y_{c}^{12\vee}\otimes y_{a}^{12\vee})&=-\delta_{ac} y_{a}^{12\vee}.
\end{align*}
\item For $m_k$ with $k\geq 3$, the higher composition $m_k$ vanishes.
\end{itemize}

In particular, the $A_{\infty}$-category $\Aug_+(A_n^{\p m}(\mu);\field)$ is in fact a DG category.
\end{example}

Notice that for each $\epsilon\in \Aug_+(A_n^{\p \bullet}(\mu);K)$, the element defined as 
\begin{equation}\label{equation:unit}
-y\coloneqq -\sum_{c\in[n]} y_c^{12\vee} \in \hom_{\Aug_+}(\epsilon,\epsilon)
\end{equation}
becomes a cocycle since
\[
m_1(-y)=\sum_{c\in[n]}\left(-\sum_{a<c}\epsilon(k_{ac})k_{ac}^{12\vee} +\sum_{c<d}k_{cd}^{12\vee}\epsilon(k_{cd})\right)=0.
\]

\begin{corollary}\label{corollary:unitality of border DGAs}
For each $\epsilon\in\Aug_+(A_n^{\p \bullet}(\mu);K)$, the element $-y \in \hom_{\Aug_+}(\epsilon,\epsilon)$ defined in \eqref{equation:unit} becomes the unit of the augmentation category $\Aug_+(A_n^{\p \bullet}(\mu);K)$.
\end{corollary}
\begin{proof}
This is a direct consequence of the above computation.
\end{proof}

\begin{proposition}\label{proposition:stabilization}
The augmentation category $\Aug_+\left(A'^{\p\bullet};K\right)$ of a stabilization $A'^{\p\bullet}$ of $A^{\p\bullet}$ is $A_\infty$-quasi-equivalent to the augmentation category $\Aug_+\left(A^{\p\bullet};K\right)$ of $A^{\p\bullet}$.
\end{proposition}
\begin{proof}
We first identify $A'^{\p m}$ with a stabilization of $A^{\p m}$ for each $m\ge 1$ and denote generators for stabilizations by $e$'s and $\hat e$'s.
For simplicity, we will denote augmentation categories of $A^{\p\bullet}$ and $A'^{\p\bullet}$ by $\Aug_+$ and $\Aug_+'$, respectively.

Let us consider two consistent morphisms $\iota^{\p\bullet}$ and $\pi^{\p\bullet}$
\begin{align*}
\iota^{\p\bullet}&:A^{\p\bullet}\to A'^{\p\bullet},&
\pi^{\p\bullet}&:A'^{\p\bullet}\to A^{\p\bullet},
\end{align*}
where $\iota^{\p m}$ is the canonical inclusion of $A^{\p m}\to A'^{\p m}$ and $\pi^{\p m}$ sends each $e$ and $\hat e$ in $A'^{\p m}$ to zero.
Then it is obvious that they are homotopy-inverses to each other so that
\begin{align*}
\pi^{\p\bullet}\circ \iota^{\p\bullet} &= \identity^{\p\bullet}_A,&
\iota^{\p\bullet}\circ \pi^{\p\bullet} &\stackrel{H^{\p\bullet}}\homotopic \identity^{\p\bullet}_{A'},
\end{align*}
where the sequence $H^{\p\bullet}$ of homotopies is given by
\begin{align}\label{equation:consistent homotopy}
H^{\p m}(s)\coloneqq
\begin{cases}
s & s\in A\subset A';\\
\hat e & s = e\in A';\\
0 & s=\hat e\in A'.
\end{cases}
\end{align}

It is well-known that an $A_\infty$-functor is an $A_\infty$-(quasi-)equivalence if it satisfies two condition, (i) essentially (quasi-)surjective, and (ii) (quasi-)fully faithful.
In other words, we need to show the following: let $\funI^{\p\bullet}$ be the $A_\infty$-functor induced from $\iota^{\p\bullet}$.
\begin{itemize}
\item for each $\epsilon\in\Aug_+$, there exists $\epsilon'\in\Aug_+'$ such that $\funI^{\p0}(\epsilon')\coloneqq \iota^*\epsilon'$ and $\epsilon$ are isomorphic up to homotopy,
\item for $\epsilon_1',\epsilon_2'\in\Aug_+'$, the induced chain map
\[
\funI^{\p1}:\hom_{\Aug_+'}(\epsilon_1',\epsilon_2')\to\hom_{\Aug_+}\left(\funI^{\p0}(\epsilon_1'),\funI^{\p0}(\epsilon_2')\right)
\]
is a quasi-isomorphism, i.e., an isomorphism between their cohomology groups.
\end{itemize}

The essential quasi-surjectivity is obvious. Indeed, for any $\epsilon\in\Aug_+$, there is an augmentation $\epsilon'\in\Aug_+$ extended from $\epsilon$ as
\begin{align*}
\epsilon'(s)\coloneqq\begin{cases}
\epsilon(s) & s\in A\subset A';\\
0 & \text{otherwise},
\end{cases}
\end{align*}
such that the pull-back of $\epsilon'$ is precisely $\epsilon$
\[
\funI^{\p 0}(\epsilon') = \epsilon'\circ \iota = \epsilon,
\]
and therefore it is indeed surjective.

Let $\epsilon_1',\epsilon_2'\in\Aug_+'$ be two augmentations.
Then by the identification of $A'$, we have 
\begin{equation}\label{equation:direct sum of hom-set of stabilization}
\hom_{\Aug_+'}(\epsilon_1',\epsilon_2')\isomorphic
\hom_{\Aug_+}\left(\funI^{\p0}(\epsilon_1'), \funI^{\p0}(\epsilon_2')\right)\oplus\left(\bigoplus_{i\in I^{12}}
\ring_\alg\left\langle e_{12}^{i\vee}, \hat e_{12}^{i\vee}\right\rangle, m_1'
\right),
\end{equation}
for some index set $I^{12}$, where
\begin{equation}\label{equation:acyclic for stabilizations}
m_1'(e_{12}^{i\vee})=\hat e_{12}^{i\vee}\quad\text{ and }\quad
m_1'(\hat e_{12}^{i\vee})=0.
\end{equation}
Then the induced chain map $\funI^{\p1}$ is the projection onto $\hom_{\Aug_+}\left(\funI^{\p0}(\epsilon_1'), \funI^{\p0}(\epsilon_2')\right)$, which is surjective with the kernel
\[
\ker\left(\funI^{\p1}\right)\isomorphic\left(\bigoplus_{i\in I^{12}}
\ring_\alg\left\langle e_{12}^{i\vee}, \hat e_{12}^{i\vee}\right\rangle,m_1'\right).
\]
Since the kernel is acyclic as seen in \eqref{equation:acyclic for stabilizations}, $\funI^{\p1}$ is a quasi-isomorphism, which implies the quasi-fully faithfulness, and we are done.
\end{proof}

\begin{corollary}\label{corollary:unitality along inclusion}
If $\Aug_+'$ is homologically unital, then so is $\Aug_+$.
\end{corollary}
\begin{proof}
Since the induced $A_\infty$ functor $\funI^{\p \bullet}$ is surjective, $\Aug_+$ is homologically unital if so is $\Aug_+'$.
\end{proof}

\begin{remark}
One may expect that the $A_\infty$ functor $\Pi^{\p\bullet}:\Aug_+\to\Aug_+'$ induced from $\pi^{\p\bullet}$ is again an $A_\infty$-quasi-equivalence but the essential quasi-surjectivity is not obvious unless $\Aug_+$ is unital.
\end{remark}

\begin{corollary}\label{corollary:unitality along projection}
If $\Aug_+$ is homologically unital, $\Pi^{\p\bullet}$ is an $A_\infty$-quasi-equivalence and so $\Aug_+'$ is homologically unital.
\end{corollary}
\begin{proof}
Let us consider the essential quasi-surjectivity first.
For each augmentation $\epsilon'\in\Aug_+'$, one can find an augmentation $\epsilon\in\Aug_+$ such that $\epsilon'$ and $\epsilon^\pi$ agree with each other on $A$ but may have different values for some $\hat e$.

Similar to \eqref{equation:direct sum of hom-set of stabilization}, the chain complex $\hom_{\Aug_+'}\left(\epsilon',\epsilon^\pi\right)$ is obviously a direct sum
\[
\hom_{\Aug_+'}\left(\epsilon',\epsilon^\pi\right)\isomorphic
\hom_{\Aug_+}\left(\epsilon,\epsilon\right)
\oplus
\left(\bigoplus_{i\in I^{12}} \ring_\alg\left\langle e_{12}^{i\vee}, \hat e_{12}^{i\vee}\right\rangle, m_1'\right).
\]
As before, the second summand is acyclic and therefore the existence of the unit in $\Aug_+$ implies the existence of an isomorphism in $\hom_{\Aug_+'}\left(\epsilon',\epsilon^\pi\right)$.

For two augmentations $\epsilon_1, \epsilon_2\in\Aug_+$, the induced chain map $\Pi^{\p 1}$ sends all $a_{12}^\vee$ to themselves
\begin{align*}
\Pi^{\p1}&:\hom_{\Aug_+}(\epsilon_1,\epsilon_2)\to\hom_{\Aug_+'}\left(\epsilon_1^\pi,\epsilon_2^\pi\right),&
\Pi^{\p1}\left(a_{12}^\vee\right)&=a_{12}^\vee.
\end{align*}
Therefore it is injective and its cokernel $\coker\Pi^{\p1}$ is isomorphic to the chain complex
\[
\coker\left(\Pi^{\p1}\right)\isomorphic\left(\bigoplus_{i\in I^{12}} \ring_\alg\left\langle e_{12}^{i\vee}, \hat e_{12}^{i\vee}\right\rangle, m_1'\right)
\]
which is acyclic and so it is quasi-fully faithful as desired.
\end{proof}

In summary, we have the following proposition:
\begin{proposition}\label{proposition:zig-zags of stabilizations}
Suppose that there is a zig-zag of stabilizations
\[
\begin{tikzcd}[column sep=3pc]
A^{\p\bullet}_0 \arrow[r, "\iota_1", hook, yshift=.7ex] \arrow[from=r, "\pi_1", ->>, yshift=-.7ex] &
A^{\p\bullet}_1 \arrow[from=r, "\iota_1'"', hook', yshift=.7ex] \arrow[r, "\pi_1'"', ->>, yshift=-.7ex] & 
\cdots \arrow[from=r, "\iota_{n-1}'"', hook', yshift=.7ex] \arrow[r, "\pi_{n-1}'"', ->>, yshift=-.7ex] &
A^{\p\bullet}_{n-1} \arrow[from=r, "\iota_{n-1}'"', hook', yshift=.7ex] \arrow[r, "\pi_{n-1}'"', ->>, yshift=-.7ex] & 
A^{\p\bullet}_n
\end{tikzcd}
\]
and $\Aug_+\left(A_0^{\p\bullet};K\right)$ is homologically unital.
Then $\Aug_+\left(A^{\p\bullet}_i;K\right)$ for each $i$ is homologically unital and $A_\infty$-quasi-equivalent to $\Aug_+\left(A_0^{\p\bullet};K\right)$.
\end{proposition}

Now we consider the bordered version of augmentation categories.
\begin{definition}[Augmentations for bordered DGAs]
Let $K=(\field,\differential_\field\equiv0)$ be a DGA consisting of the base field $\field$ with the trivial differential.
An augmentation for $\dga=(A_\Left\to A \leftarrow A_\Right)$ is a bordered DGA morphism 
\[
\varepsilon=(\epsilon_\Left, \epsilon, \epsilon_\Right):\dga\to\cK\coloneqq(K = K = K),
\]
which makes the following diagram commutative:
\begin{align*}
\begin{tikzcd}[column sep=4pc, row sep=2pc,ampersand replacement=\&]
\dga\arrow[d,"\varepsilon"']\\
\cK
\end{tikzcd}
&=\ \left(
\begin{tikzcd}[column sep=4pc, row sep=2pc,ampersand replacement=\&]
A_\Left \arrow[r]\arrow[d, "\epsilon_\Left"']  \& A\arrow[d, "\epsilon"] \& A_\Right\arrow[l]\arrow[d, "\epsilon_\Right"]\\
K\arrow[r,equal] \& K \& K\arrow[l,equal]
\end{tikzcd}
\right)
\end{align*}
\end{definition}

\begin{definition}
A \emph{bordered} $A_\infty$-category $\left(\scrA_\Left\leftarrow \scrA \rightarrow \scrA_\Right\right)$ of type $(n_\Left,n_\Right)$ consists of $A_\infty$-categories, $\scrA_\Left, \scrA$ and $\scrA_\Right$ and two $A_\infty$-functors $\scrA\to\scrA_\Left$ and $\scrA\to\scrA_\Right$ such that both $\scrA_\Left$ and $\scrA_\Right$ are $A_\infty$-equivalent to augmentation categories 
\[
\scrA_\Left\isomorphic\Aug(A_{n_\Left}^{\p \bullet}(\mu_\Left);K)\quad\text{ and }\quad
\scrA_\Right\isomorphic\Aug(A_{n_\Right}^{\p \bullet}(\mu_\Right);K)
\]
for some $\mu_\Left:[n_\Left]\to\grading$ and $\mu_\Right:[n_\Right]\to\grading$, respectively.

The morphism $\cF^{\p\bullet}$ between two bordered $A_\infty$-categories is a triple $\left(\funF^{\p\bullet}_\Left, \funF^{\p\bullet}, \funF^{\p\bullet}_\Right\right)$ of $A_\infty$-functors making the following diagram commutative:
\[
\cF^{\p\bullet}=\left(
\begin{tikzcd}[column sep=4pc, row sep=2pc]
\scrA_\Left\arrow[from=r]\arrow[d,"\funF^{\p\bullet}_\Left"'] & \scrA\arrow[d,"\funF^{\p\bullet}"] & \scrA_\Right\arrow[from=l]\arrow[d,"\funF^{\p\bullet}_\Right"]
\\
\scrA'_\Left\arrow[from=r] & \scrA' & \scrA'_\Right\arrow[from=l]
\end{tikzcd}\right).
\]
We say that $\cF^{\p\bullet}$ is an $A_\infty$-(quasi)-equivalence if so are $\funF^{\p\bullet}_\Left$, $\funF^{\p\bullet}$ and $\funF^{\p\bullet}_\Right$.

We denote the category of bordered $A_\infty$-categories by $\BAlg_\infty$.
\end{definition}

Indeed, for each consistent sequence $\dga^{\p\bullet}=\left(A_\Left^{\p\bullet}\stackrel{\phi_\Left^{\p\bullet}}\longrightarrow A^{\p\bullet}\stackrel{\phi_\Right^{\p\bullet}}\longleftarrow A_\Right^{\p\bullet}\right)$ of bordered DGAs, we have an associated bordered augmentation $A_\infty$-category
\[
\Aug_+(\dga^{\p\bullet};\cK)\coloneqq
\left(
\begin{tikzcd}[column sep=3pc]
\Aug_+(A_\Left^{\p\bullet};K)\arrow[from=r,"\Aug_+(\phi_\Left^{\p\bullet})"'] &
\Aug_+(A^{\p\bullet};K) \arrow[r, "\Aug_+(\phi_\Right^{\p\bullet})"] & \Aug_+(A_\Right^{\p\bullet};K)
\end{tikzcd}
\right).
\]

\begin{corollary}\label{corollary:well-definedness of bordered augmentation categories}
The contravariant functor $\Aug_+(-;\cK):\BDGA^{\p\bullet}_\co\to\BAlg_\infty$ is well-defined.
\end{corollary}
\begin{proof}
This is a corollary of Proposition~\ref{proposition:functoriality}.
\end{proof}

\begin{proposition}\label{proposition:natural equivalence of mapping cylinder}
Let $\dga'^{\p \bullet}$ be a stabilization of $\dga^{\p\bullet}$.
Then two bordered augmentation categories $\Aug_+(\dga'^{\p\bullet};\cK)$ and $\Aug_+(\dga^{\p\bullet};\cK)$ are $A_\infty$-quasi-equivalent.
\end{proposition}
\begin{proof}
Let $\Pi^{\p\bullet}\coloneqq\Aug_+(\bfpi^{\p\bullet}):\Aug_+(\dga^{\p\bullet};\cK)\to\Aug_+(\dga'^{\p\bullet};\cK)$ be the induced $A_\infty$ functor from the canonical projection $\bfpi:\dga'^{\p\bullet}\to\dga^{\p\bullet}$.
Since $\dga'^{\p\bullet}$ is a stabilization, two induced functors $\Aug_+(\pi_\Left)$ and $\Aug_+(\pi_\Right)$ on borders are equivalences. Hence it suffices to prove the $A_\infty$-quasi-equivalence for $\Aug_+(\pi^{\p\bullet})$.

Due to Proposition~\ref{proposition:zig-zags of stabilizations}, it is obvious that the $A_\infty$ functor $\funI^{\p\bullet}:\Aug_+(A'^{\p\bullet};K)\to\Aug_+(A^{\p\bullet};K)$ induced from $\iota^{\p\bullet}:A^{\p\bullet}\to A'^{\p\bullet}$ is an $A_\infty$-quasi-equivalence whose quasi-inverse is precisely $\Aug_+(\pi^{\p\bullet}):\Aug_+(A^{\p\bullet};K)\to \Aug_+(A'^{\p\bullet};K)$ and is also an $A_\infty$-quasi-equivalence as desired.
\end{proof}

\subsection{Augmentation categories for bordered Legendrian graphs}

Let $(\cT^{\p\bullet},\bfmu^{\p\bullet})\in\BLT^{\p\bullet}$ be a consistent sequence of bordered Legendrian graphs.
Then by taking $A^{\CE}_\co$, we have a consistent sequence of DGAs $A^{\CE}_\co\left(\cT^{\p\bullet},\bfmu^{\p\bullet}\right)$ as seen in Theorem~\ref{theorem:tangle to CE algebra}.

\begin{definition}[Augmentation categories for consistent sequences of bordered Legendrian graphs]
Let $(\cT^{\p\bullet},\bfmu^{\p\bullet})\in\BDGA^{\mu,\p\bullet}_\co$.
The \emph{bordered augmentation category} for $(\cT^{\p\bullet},\bfmu^{\p\bullet})$ is the composition
\begin{align*}
\Aug_+\left(\cT^{\p\bullet},\bfmu^{\p\bullet};\cK\right)&=
\left(
\Aug_+\left(T^{\p\bullet}_\Left,\mu^{\p\bullet}_\Left;K\right) \leftarrow
\Aug_+\left(T^{\p\bullet},\mu^{\p\bullet};K\right) \rightarrow
\Aug_+\left(T^{\p\bullet}_\Right,\mu^{\p\bullet}_\Right;K\right)
\right)\\
&\coloneqq \Aug_+(A^{\CE}_\co(\cT^{\p\bullet},\bfmu^{\p\bullet});\cK),
\end{align*}
where for $*=\Left, \Right$ or empty,
\[
\Aug_+\left(T^{\p\bullet}_*,\mu^{\p\bullet}_*;K\right)\coloneqq
\Aug_+(A^\CE_\co(T^{\p\bullet}_*,\mu^{\p\bullet}_*);K).
\]

In particular, if $\cT^{\p\bullet}$ is the consistent sequence of canonical front copies of $\cT$, then we denote simply by
\[
\Aug_+\left(\cT,\bfmu;\cK\right)\coloneqq
\left(
\Aug_+\left(T_\Left,\mu_\Left;K\right) \leftarrow
\Aug_+\left(T,\mu;K\right) \rightarrow
\Aug_+\left(T_\Right,\mu_\Right;K\right)
\right).
\]
\end{definition}

For an example of computing the augmentation category, see Section \ref{sec:example augmentation category}.

\begin{theorem}[Invariance theorem]\label{thm:invariance aug}
The assignment $(\cT,\bfmu)\to \Aug_+\left(\cT,\bfmu;\cK\right)$ is well-defined and invariant under Legendrian isotopy and basepoint moves up to zig-zags of $A_\infty$-quasi-equivalences.
\end{theorem}
\begin{proof}
The well-definedness is obvious.
Indeed, the canonical front copy is well-defined and every morphism is a zig-zag of elementary morphisms by Proposition~\ref{proposition:functoriality of canonical Lagrangian copies}, it induces a well-defined consistent sequence of bordered LCH DGAs and elementary morphisms correspond to stabilizations by Theorem~\ref{theorem:tangle to CE algebra}.

Now as always, we pass through the bordered augmentation categories of cofibrant replacements. Then we have a zig-zag of stabilizations of consistent bordered DGAs
\[
\begin{tikzcd}[column sep=3pc]
\dga'\arrow[from=r,"\hat\bfpi'^{\p\bullet}"',->>]& \hat \dga'=\dga_{0}^{\p\bullet}
\arrow[from=r, "\bfpi_1^{\p\bullet}",->>,yshift=-0.7ex]
\arrow[r,"\bfi_1^{\p\bullet}",yshift=0.7ex,hook]
& \cdots 
\arrow[r, "\bfpi'^{\p\bullet}_{n-1}"',->>,yshift=-0.7ex]
\arrow[from=r,"\bfi'^{\p\bullet}_{n-1}"',yshift=0.7ex,hook']& \dga^{\p\bullet}_{n}=\hat \dga\arrow[r,"\hat\bfpi'^{\p\bullet}",->>]& \dga,
\end{tikzcd}
\]
where $\dga'=A^\CE_\co(\cT'^{\p\bullet},\bfmu'^{\p\bullet})$ and $\dga=A^\CE_\co(\cT^{\p\bullet},\bfmu^{\p\bullet})$ and $\hat\dga'$ and $\hat\dga$ are cofibrant replacements, respectively.

Due to Corollary~\ref{corollary:well-definedness of bordered augmentation categories}, the bordered augmentation categories for these zig-zags are well-defined and by Proposition~\ref{proposition:natural equivalence of mapping cylinder}, every stabilization gives us an $A_\infty$-quasi-equivalences.

We assume that $(\cT',\bfmu')$ and $(\cT,\bfmu)$ are related by a basepoint move $\RM{B_i}$.
For $\RM{B_1}$, we are done since $\RM{B_1}^{\p\bullet}$ induces a zig-zag of stabilizations as seen already in Section~\ref{section:consistent sequences of LCH DGAs}.

For $\RM{B_i}$ with $i=2$ or $3$, the induced consistent basepoint move on the canonical front copies is a sequence of a elementary basepoint move and a Reidemeister move as seen in Figure~\ref{figure:consistent basepoint moves}.
Now suppose that two consistent sequences $(\cT'^{\p\bullet},\bfmu'^{\p\bullet})$ and $(\cT^{\p\bullet},\bfmu^{\p\bullet})$ are related with an elementary consistent basepoint change move $\RM{B_i}^{\p\bullet}$.

Due to the discussion after Theorem~\ref{theorem:tangle to CE algebra}, we have the induced consistent morphism $\RM{B_i}_*^{\p\bullet}$ and its left inverse $\RM{B_i^{-1}}_*^{\p\bullet}$.
Therefore we have a surjective $A_\infty$-functor
\[
\Aug_+(\RM{B_i}^{\p\bullet}):\Aug_+(\cT^{\p\bullet},\bfmu^{\p\bullet};\cK)\to \Aug_+(\cT'^{\p\bullet},\bfmu'^{\p\bullet};\cK).
\]

For any pair of augmentation $\epsilon_1, \epsilon_2\in\Aug_+(\cT^{\p\bullet},\bfmu^{\p\bullet};\cK)$, we have a chain map
\[
\Aug_+(\RM{B_i}^{\p2}):\hom_+(\epsilon_1,\epsilon_2)\to\hom_+(\epsilon_1',\epsilon_2'),
\]
where $\epsilon_j'\in\Aug_+(\cT'^{\p\bullet},\bfmu'^{\p\bullet};\cK)$ is an induced augmentation via $\RM{B_i}_*$.
Notice that the basepoint move $\RM{B_i}$ for $i=2$ or $3$ does not alter the crossings, especially crossings between the 1st and 2nd copies of $\cT'^{\p2}$. Therefore the above chain map is an isomorphism between graded vector spaces which implies the fully faithfulness of $\Aug_+(\RM{B_i}^{\p\bullet})$ and we are done.
\end{proof}

The following theorem is a generalization of Proposition~3.25 in \cite{NRSSZ2015}.
\begin{theorem}\label{theorem:unitality of augmentation category of LG in a normal form}
Let $(\cT,\bfmu)\in\BLT^\mu$ be a bordered Legendrian graph in a normal form.
Then its augmentation category $\Aug_+(\cT,\bfmu;\cK)$ is strictly unital. That is,
the $A_\infty$ category $\Aug_+(T_*,\mu_*;K)$ for each $*=\Left,\Right$ and empty is strictly unital and two $A_\infty$-functors are unit-preserving.
\end{theorem}
\begin{proof}
Let $\epsilon\in\Aug_+(\cT,\bfmu;\cK)$. We need to show that there exists an element in $\hom_{\Aug_+}(\epsilon,\epsilon)$ which plays the role of the identity.

Recall the generating sets for $A^{\CE}_\co(\cT^{\p\bullet},\bfmu^{\p\bullet})$ described in Section~\ref{section:augmentation category of normal form}.
We define an element $-y^\vee_{12}\in\hom_{\Aug_+}(\epsilon,\epsilon)$ as follows:
\[
-y^\vee_{12}\coloneqq \sum_{e\in \bar E} -y^{e\vee}_{12} \in \hom_{\Aug_+}(\epsilon,\epsilon),
\]
where $y^{e\vee}_{12}$ is the dual of $y_{e}^{12}$.

Due to the formula of the differential $\differential^{\p m}$ given in Section~\ref{section:augmentation category of normal form}, the only possibility for $s^{13}$ containing $y_e^{12}$ or $y_e'^{23}$ in its differential is either
\[
(-1)^{|s|-1}s^{12}y^{23}_{e'}\quad \text{ or }\quad
y^{12}_e s^{23},
\]
respectively. Note that when $s^{13}=y^{13}_e$ for some $e$, then these two terms coincide.
Moreover, the situation is essentially the same as the augmentation category for border DGAs as computed in Example~\ref{example:augmentation category for border DGAs}.

Therefore, one and only one generator appears in both $m_2\left(-y^\vee_{12}\otimes s^{12}\right)$ and $m_2\left(s^{12}\otimes -y^\vee_{12}\right)$ which is precisely $s^{12}$ itself
\[
m_2\left(-y^\vee_{12}\otimes s^{12}\right)=s^{12}=m_2\left(s^{12}\otimes -y^\vee_{12}\right).
\]

Finally, the absence of terms of length at least 3 in any differential containing $y$ implies that the higher composition $m_k$ will vanish whenever it contains $-y^\vee_{12}$.

Now we prove that two $A_\infty$ functors $\Aug_+\left(\phi_*^{\p\bullet}\right):\Aug_+(T,\mu;K)\to\Aug_+(T_*,\mu_*;K)$ for $*=\Left$ and $\Right$ are unit-preserving.
Due to the definition of the induced $A_\infty$-functor briefly reviewed in Section~\ref{section:aug cats} after Proposition~\ref{proposition:functoriality}, we have for each augmentation $\epsilon\in\Aug_+(T,\mu;K)$
\begin{align*}
\Aug_+\left(\phi_*^{\p\bullet}\right)(-y^\vee_{12})&=
\sum_{s^{12}\in R^{12}} (-1)^{\sigma_1}\left\langle -y^\vee_{12}, \phi_{*,\epsilon}^{\p2}(s^{12})\right\rangle s^\vee_{12},
\end{align*}
where $\sigma_1=0$ as seen in Definition~\ref{def:sign for m_k} and $\phi_{*,\bfe}^{\p2}$ is a twisted DGA morphism by using the diagonal augmentation $\bfe=(\epsilon,\epsilon)$.

More precisely, since $\phi_\Left^{\p 2}$ identifies each generators in $A_{\co,n_\Left}^{\p 2}\isomorphic A^\CE_\co(T_\Left^{\p2},\mu_\Left^{\p2})$ with the corresponding generator in $A^\CE_\co(T^{\p2},\mu^{\p2})$, 
\[
\phi_{\Left,\bfe}^{\p2}(k_{ab}^{12})=k_{ab}^{12}\quad\text{ and }\quad
\phi_{\Left,\bfe}^{\p2}(y_{c}^{12})=y_c^{12}
\]
and therefore the image of $\phi_\Left^{\p2}$ has a nontrivial pairing only for $y_c^{12}$.
In other words, 
\begin{align*}
\Aug_+\left(\phi_\Left^{\p\bullet}\right)(-y^\vee_{12})&=
\sum_{c\in[n_\Left]} \left\langle -y^\vee_{12}, \phi_{\Left,\epsilon}^{\p2}(y_c^{12})\right\rangle y^{c\vee}_{12}=\sum_{c\in[n_\Left]}(-y^{c\vee}_{12})\in\hom_{\Aug_{+,\Left}}(\phi_\Left^*(\epsilon),\phi_\Left^*(\epsilon)),
\end{align*}
which is the unit as desired.

On the other hand, for the right border, we recall the map $\phi_\Right^{\p m}$ described in \eqref{equation:structure morphism of Legendrian in a normal form}. That is, for $m=2$, we have
\begin{align}
\phi_\Right^{\p 2}(k^{12}_{a'b'})&=\Phi(\phi_\Right(k_{a'b'}))_{12},&
\phi_\Right^{\p 2}(y^{12}_{c'})&=y^{12}_e.
\end{align}
Then $\phi_\Right^{\p2}(k^{12}_{a'b'})$ does not involve any $y^{12}_e$'s as observed before and therefore the pairing survives only for $y^{12}_{c'}$'s. 
In that case $\phi_{\Right,\bfe}^{\p2}(y^{12}_{c'})=y^{12}_{c'}$ and we have
\begin{align*}
\Aug_+\left(\phi_\Right^{\p\bullet}\right)(-y^\vee_{12})&=
\sum_{c'\in[n_\Right]} \left\langle -y^\vee_{12}, \phi_{\Right,\bfe}^{\p2}(y_{c'}^{12})\right\rangle y^{c'\vee}_{12}=\sum_{c'\in[n_\Right]}(-y^{c'\vee}_{12})\in\hom_{\Aug_{+,\Right}}(\phi_\Right^*(\epsilon),\phi_\Right^*(\epsilon)),
\end{align*}
which is the unit as well. This completes the proof.
\end{proof}

\begin{theorem}[Unitality]\label{theorem:unitality}
For any bordered Legendrian graph $(\cT,\bfmu)\in \BLT^\mu$, the augmentation category $\Aug_+(\cT,\bfmu;\cK)$ is homologically unital.
\end{theorem}
\begin{proof}
Due to Lemma~\ref{lemma:normal form representative}, one can obtain $\cT$ from a Legendrian graph $\cT'$ in a normal form up to Legendrian isotopy and basepoint moves.
Then by Theorem~\ref{theorem:unitality of augmentation category of LG in a normal form}, the augmentation category $\Aug_+(\cT',\bfmu';\cK)$ is strictly unital.

As seen in the proof of Theorem~\ref{thm:invariance aug}, for each Legendrian isotopy, we have a zig-zag of stabilizations between consistent sequences of DGAs for $(\cT',\bfmu')$ and $(\cT,\bfmu)$.
By Proposition~\ref{proposition:zig-zags of stabilizations}, we have the homological unitality for $\Aug_+(T,\mu;K)$ which implies that so is $\Aug_+(\cT,\bfmu;\cK)$ as desired.

For each basepoint move, we have either a zig-zag of stabilizations for $\RM{B_1}$ where the above argument is applicable, or an $A_\infty$-equivalence for $\RM{B_i}$ with $i=2$ or $3$ by the proof of Theorem~\ref{thm:invariance aug} again, which preserves of course the (homological) unitality.
\end{proof}

\begin{proposition}
Let $(\cT,\bfmu)\in\BLT$ and $(\cT_\lag,\bfmu)\coloneqq\Res(\cT)$.
Suppose that $(\cT^{\p \bullet},\bfmu^{\p\bullet})$ and $(\cT_\lag^{\p\bullet},\bfmu^{\p\bullet})$ are consistent sequences of canonical front and Lagrangian copies, respectively.
Then two augmentation categories $\Aug_+(\cT^{\p \bullet},\bfmu^{\p\bullet};\cK)$ and $\Aug_+(\cT_\lag^{\p \bullet},\bfmu^{\p\bullet};\cK)$ are $A_\infty$-quasi-equivalent.
\end{proposition}
\begin{proof}
As seen in Section~\ref{section:canonical copies} and Figure~\ref{figure:front and Lagrangian canonical copies}, there exists a zig-zag of elementary Lagrangian Reidemeister moves between $\Res(\cT^{\p\bullet})$ and $\cT_\lag^{\p\bullet}$.
By Theorem~\ref{theorem:tangle to CE algebra}, we have a zig-zag of consistent stabilizations between
\[
A^\CE(\cT^{\p \bullet},\bfmu^{\p\bullet})\coloneqq
A^\CE(\Res(\cT^{\p \bullet}),\bfmu^{\p\bullet})\quad\text{ and }\quad
A^\CE(\cT_\lag^{\p \bullet},\bfmu^{\p\bullet}).
\]
Finally, Proposition~\ref{proposition:natural equivalence of mapping cylinder} completes the proof.
\end{proof}

\section{Sheaf categories for Legendrian graphs}\label{section:sheaf cats}

In this section we give the preliminaries on the microlocal theory of sheaves, the main reference is \cite{KS1994}. We also establish the necessary combinatorial tools for constructible sheaves, which will be used in the proof of the augmentation-sheaf correspondence in the next section. 

\subsection{Micro-support and constructible sheaves}
For the moment, let $M$ be any smooth manifold, and $\field$ be a base field. 
We use the same notations as in \cite{STZ2017}. Let $\Sh(M;\field)$ to be the abelian category of sheaves of $\field$-modules. Let $\Sh_{\text{naive}}(M;\field)$ to be the triangulated DG category of complexes of sheaves of $\field$-modules on $M$ whose cohomology sheaves are constructible (that is, locally constant with perfect stalks on each stratum) with respect to some nice stratification (e.g. Whitney stratification). Let $\Sh(M;\field)$ be the DG quotient \cite{Drinfeld2004} of the DG category $\Sh_{\text{naive}}(M;\field)$ with respect to acyclic complexes. Given a Whitney stratification $\cS$ of $M$, \emph{define} $\Sh_{\cS}(M;\field)$ to be the full subcategory of $\Sh(M;\field)$ consisting of objects whose cohomology sheaves are constructible with respect to $\cS$.

We firstly recall the notion of micro-support introduced by Kashiwara and Schapira:

\begin{proposition/definition}[Micro-support]\cite[Prop.5.1.1, Def.5.1.2]{KS1994}\label{def/prop:Micro-support}
Let $\cF$ be a sheaf
\footnote{Whenever we say a sheaf, we mean a complex of sheaves of $\field$-modules.} on $M$, and $p=(x_0,\xi_0)\in T^*M$. 
We say $p\notin SS(\cF)$ if one of the following equivalent conditions holds:

\begin{enumerate}
\item
There exists a neighborhood $U$ of $p$, such that for any $x_1\in M$ and any $C^1$-function $\varphi$ on a neighborhood of $x_1$ satisfying $(x_1,d\varphi(x_1))\in U$ and $\varphi(x_1)=0$,  we have
\[
R\Gamma_{\{\varphi(x)\geq 0\}}(\cF)_{x_1}\simeq 0.
\]
Equivalently, by the distinguished triangle $R\Gamma_{\{\varphi(x)\geq 0\}}(\cF)\rightarrow\cF\rightarrow R\Gamma_{\{\varphi(x)<0\}}(\cF)\xrightarrow[]{+1}$, we get a quasi-isomorphism 
\[
\cF_{x_1}\xrightarrow[]{\sim} R\Gamma_{\{\varphi(x)<0\}}(\cF)_{x_1}.
\]

\item
Up to taking an open chart near $x_0$, we can assume $M$ is an open subset in a vector space $E$. Then there exists a neighborhood $U$ of $x_0$, an $\epsilon>0$, and a proper closed convex cone $\gamma$ in $E$ with $0\in\gamma$, satisfying $\gamma\setminus\{0\}\subset \{v\mid\langle v,\xi_0\rangle<0\}$, such that if we set
\[
H\coloneqq \{x\mid\langle x-x_0,\xi_0\rangle\geq -\epsilon\}\quad\text{ and }\quad
L\coloneqq \{x\mid\langle x-x_0,\xi_0\rangle= -\epsilon\},
\]
then $H\cap (U+\gamma)\subset M$ and we have the natural isomorphism:
\[
R\Gamma(H\cap(x+\gamma);\cF)\xrightarrow[]{\sim} R\Gamma(L\cap(x+\gamma);\cF)
\]
for all $x\in U$. Recall that a cone is called proper if it contains no lines.
\end{enumerate}

The set $SS(\cF)$ is called the \emph{micro-support} (or \emph{singular support}) of $\cF$. 
\end{proposition/definition}

The micro-support satisfies the following properties:
\begin{enumerate}
\item $SS(\cF)\cap 0_M=\Supp(\cF)$ is the support of $\cF$, where $0_M$ the zero section of $T^*M$.

\item For any sheaf $\cF$, $SS(\cF)$ is a conical (i.e. invariant under the action of $\RR_+$ which scales the cotangent fibers) and closed co-isotropic subset of $T^*M$.

\item If $\cF$ is a constructible sheaf with respect to a Whitney stratification $\cS$, then $SS(\cF)$ is a conical Lagrangian (i.e. Lagrangian wherever it is smooth) subset of $T^*M$ with $SS(\cF)\subset \bigcup_{i\in\cS}T^*_{S_i}M$.

\item(Triangular inequality)
If $\cF_1\rightarrow \cF_2\rightarrow \cF_3\xrightarrow[]{+1}$ is an exact triangle in $\Sh(M;\field)$, then $SS(\cF_i)\subset SS(\cF_j)\cup SS(\cF_k)$ for all distinct $i,j,k\in\{1,2,3\}$.

\item(Microlocal Morse lemma)
If $f:M\rightarrow \RR$ is a smooth function such that $(x,df(x))\notin SS(\cF)$ for all $x\in f^{-1}([a,b])$, and $f$ is proper on the support of $\cF$. Then the restriction map is a quasi-isomorphism:
\[
R\Gamma(f^{-1}(-\infty,b);\cF)\xrightarrow[]{\sim} R\Gamma(f^{-1}(-\infty,a);\cF)
\]  
\end{enumerate}

From now on, let $M\coloneqq I_x\times \RR_z$ be the base manifold, where $I_x=(x_\Left,x_\Right)$ with $-\infty\leq x_\Left<x_\Right\leq \infty$ is an open interval in $\RR_x$. Let $\cT=(T_\Left\leftarrow T\rightarrow T_\Right)$ be a bordered Legendrian graph in $J^1I_x=T^{\infty,-}M$.

\begin{definition}
Given a possibly singular Legendrian $T\subset T^{\infty,-}M$, we define $\Sh(T;\field)=\Sh_T(M;\field)$ to be the full subcategory of $\Sh(M;\field)$ consisting of those objects $\cF$, whose micro-support at infinity is contained in $T$ (i.e. $SS(\cF)\subset 0_M\cup \RR_{>0}T$). Furthermore, let $\Sh(T;\field)_0=\Sh_T(M:\field)_0$ be the full subcategory of $\Sh_T(M;\field)$ whose objects are those $\cF$ with acyclic stalks for $z\ll 0$. 

In particular, given a bordered Legendrian graph $\cT=(T_\Left\rightarrow T\leftarrow T_\Right)$ in $J^1I_x$, by the obvious restriction of sheaves, we obtain a diagram of constructible sheaf categories 
\[
\Sh(\cT;\field)\coloneqq (\Sh(T_\Left;\field)\leftarrow \Sh(T;\field)\rightarrow \Sh(T_\Right;\field))
\]
and define $\Sh(\cT;\field)_0$ similarly.

As in \cite[\S2.2.1]{STZ2017}, let $\cT^+$ be the extended Legendrian in $T^{\infty,-}M$ as follows: For each crossing $c$ of $\cT$, $\cT$ meets the semicircle $T_c^{\infty,-}M$ in exactly two points, connected by a unique arc in $T_c^{\infty,-}M$. Then $\cT^+$ is the union of $\cT$ with all these arcs induced by the crossings. One can similarly define the sheaf categories $\Sh(\cT^+;\field)$ and $\Sh(\cT^+;\field)_0$.
\end{definition}

\begin{theorem}\label{thm:inv of sh cat}
The diagram of DG categories 
\[
\Sh(\cT;\field)=(\Sh(T_\Left;\field)\leftarrow \Sh(T;\field)\to\Sh(T_\Right;\field))
\]
is a Legendrian isotopy invariant, up to DG equivalence. That is, for any Legendrian isotopy between two bordered Legendrian graphs $T,T'$, there is an equivalence between two diagrams of DG categories taking the form:
\[
\begin{tikzcd}[column sep=4pc, row sep=2pc]
\Sh(T_\Left;\field)\arrow[equal,d,"\identity"] & \Sh(T;\field)\arrow[l,"\mathrm{r}_\Left"']\arrow[d,"\simeq","\frK"']\arrow[r,"\mathrm{r}_\Right"]\arrow[equal,dl,"\sim"',"\frh_\Left"]\arrow[equal,dr,"\sim","\frh_\Right"'] &\Sh(T_\Right;\field)\arrow[equal,d,"\identity"]\\
\Sh(T_\Left';\field) & \Sh(T';\field)\arrow[l,"\mathrm{r}_\Left'"]\arrow[r,"\mathrm{r}_\Right'"'] &\Sh(T_\Right';\field)
\end{tikzcd}
\]
where the arrow $\frK$ in the middle is an equivalence of DG categories which, under restrictions, commutes with the identity functor $\identity:\Sh(T_\Left;\field)\xrightarrow[]{\sim}\Sh(T_\Left';\field)$ (resp.  $\identity:\Sh(T_\Right;\field)\xrightarrow[]{\sim}\Sh(T_\Right';\field)$) up to specified natural isomorphism $\frh_\Left: \mathrm{r}_\Left\xRightarrow[]{\sim}\mathrm{r}_\Left'\circ \frK$ (resp. $\frh_\Right: \mathrm{r}_\Right\xRightarrow[]{\sim}\mathrm{r}_\Right'\circ \frK$). Also, the same holds for $\Sh(\cT;\field)_0$.  
\end{theorem}

\begin{proof}
The first proof is a direct consequence of the results of Guillermou-Kashiwara-Schapira \cite{GKS2012}. For the more related details, see \cite[Thm.4.1, Rmk.4.2,4.3]{STZ2017}.
\end{proof}

In the rest of this section, we will give a combinatorial description of the sheaf categories. In particular, we will use this description to give an alternative proof of the invariance theorem. The combinatorial description will also be needed in proving our main result ``augmentations are sheaves".

\subsection{Combinatorial description for constructible sheaves}\label{subsubsec:comb model}
As before, let $M=I_x\times \RR_z$ be the base manifold, and $\cT$ be a bordered Legendrian graph in $J^1I_x=T^{\infty,-}M$.

We can always assume the front projection $T$ is regular so that the only \emph{singularities} of $T$ are crossings, cusps, and vertices. This induces a \emph{Whitney stratification} $\cS_{T}$ of $M$ whose 0-dimensional strata are the singularities, 1-dimensional strata are the \emph{arcs}---the connected components of $T\setminus\{\text{singularities}\}$, and 2-dimensional strata are the \emph{regions}---connected components of $M\setminus T$. By definition, $\Sh(T;\field)\subset\Sh(T^+;\field)$ are full subcategories of $\Sh_{\cS_T}(M;\field)$. 

Given a stratification $\cS$, the \emph{star} of a stratum $S\in\cS$ is the union of strata whose closure contains $S$, denoted by $\Star(S)$. Given 2 strata $S$ and $S'$, we denote by $S\leq S'$ and define an arrow $S\rightarrow S'$ if $S\subset \overline{S'}$, or equivalently, $\Star(S)\supset \Star(S')$. 
This defines $\cS$ as a poset category. We say that $\cS$ is a \emph{regular cell complex} if every stratum is contractible as well as the star of each stratum is contractible. As in \cite{STZ2017}, we can choose a regular cell complex $\cS$ refining the stratification $\cS_T$ with the additional 1-dimensional strata away from the crossings of $T$ if necessary.

\begin{assumption}\label{ass:regular cell complex}
For the regular cell complex $\cS$ refining the stratification $\cS_T$ induced by $T$, we assume that each additional 1-dimensional strata contains no vertical tangent and 
\begin{enumerate}
\item it is tangent to the existing arcs at the singularity if it has an end at a cusp or a vertex, or 
\item it is transverse to the boundary if it has an end at the boundary $\boundary M=\{(x,z)\mid x=x_\Left, x_\Right\}$.
\end{enumerate} 
\end{assumption}
The assumption can always be satisfied by choosing $\cS$ appropriately.

\begin{definition}\label{def:costandard objects}
For any regular cell complex $\cS$ of an oriented manifold $M$ and any stratum $S\in\cS$, let us define a \emph{co-standard object} $\omega_{S}$ of $\Sh_{\cS}(M;\field)$ as
\[
\omega_{S}\coloneqq \field_S[\dim S]=R(j_S)_!(\field[\dim S])\in\Sh_{\cS}(M;\field),
\]
where $j_S:S\rightarrow M$ is the inclusion and $\field$ is regarded as the constant sheaf on $S$.
\end{definition}

\begin{lemma}[{\cite[Lem.2.3.2]{Nadler2009}}]\label{lem:envelope}
The triangulated DG category $\Sh_{\cS}(M;\field)$ is the triangulated envelope of the co-standard objects. 
\end{lemma}

An immediate corollary is as follows:
\begin{corollary}\label{cor:quasi-isomorphism}
For any $\cF\in\Sh_{\cS}(M;\field)$ and $x\in w\in\cS$, we have natural quasi-isomorphisms 
\[
\begin{tikzcd}
\cF_x\arrow[r,"\homotopic"] & (R\Gamma(\cF))_x\arrow[from=r, "\homotopic"']& R\Gamma(\Star(w);\cF).
\end{tikzcd}
\]
\end{corollary}
\begin{proof}
 A direct computation shows that this holds for all the co-standard objects. Moreover, the above property of $\cF$ is preserved under taking quasi-isomorphisms, shifts, and cones. Hence by the lemma above, we are done.
\end{proof}

Now we come back to our setting where $M=I_x\times\RR_z$.
By definition, $\Sh_{\cS}(M;\field)$ contains $\Sh_{\cS_T}(M;\field)$, hence $\Sh(T^+;\field)$ and $\Sh(T;\field)$ as full subcategories as well.
For an open subinterval $J_x$ of $I_x$, we denote $M|_{J_x}\coloneqq J_x\times \RR_z$ and define $\cS|_{J_x}$ to be the stratification of $M|_{J_x}$ induced by $\cS$, whose strata are the connected components of $S\cap (M|_{J_x})$ for all $S\in\cS$.
Then $\cS|_{J_x}$ is a Whitney stratification of $M|_{J_x}$ refining the stratification induced by $T|_{J_x}$ and we obtain a natural dg functor $\mathrm{r}:\Sh_{\cS}(M;\field)\rightarrow \Sh_{\cS|_{J_x}}(M|_{J_x};\field)$ coming from the restriction.

Given a regular cell complex $\cS$ of $M$, there is a combinatorial description of $\Sh_{S}(M;\field)$ as follows: Denote the induced poset category by $\cS$ again and denote the category of cochain complexes of $\field$-modules with cochain maps by $\Ch(\field)$. 

\begin{definition}\label{def:rep cat}
For any poset category $\cS$, let $\Fun_{\text{naive}}(\cS,\field)$ be the DG category of functors from $\cS$ to $\Ch(\field)$, which are valued on perfect complexes, that is, complexes which are quasi-isomorphic to a bounded complex of finite projective $\field$-modules.
We define $\Fun(\cS,\field)$ to be the DG quotient (see \cite{Drinfeld2004}) of $\Fun_{\text{naive}}(\cS,\field)$ with respect to the thick subcategory of objects taking values in acyclic complexes. 
\end{definition}

\begin{notation}\label{rem:from abelian to dg category}
We denote the abelian category of functors from $\cS$ to the abelian category $\field-\Mod$ of $\field$-modules by $Fun(\cS,\field)$.

We denote by $\Ch_{\text{naive}}(\cS,\field)$ and $\Ch_{\text{dg}}(\cS,\field)$ the DG category of cochain complexes of objects in the abelian category $Fun(\cS,\field)$ with morphisms the usual complexes of maps between complexes and its DG quotient of $\Ch_{\text{naive}}(\cS,\field)$ by the full subcategory of acyclic objects, respectively.
Then $\Fun_{\text{naive}}(\cS,\field)$ and $\Fun(\cS,\field)$ are full subcategories of $\Ch_{\text{naive}}(\cS,\field)$ and $\Ch_{\text{dg}}(\cS,\field)$ respectively.  
\end{notation}

Observe that we obtain a functor of poset categories $i:\cS|_{J_x}\rightarrow \cS$ by inclusion of strata, which then induces a natural DG functor $i^*:\Fun(\cS,\field)\rightarrow \Fun(\cS|_{J_x};\field)$.

\begin{proposition/definition}\label{def:comb model to sheaf}
There is a functor $i_{\cS}:Fun(\cS,\field)\rightarrow \Sh(M;\field)$ defined as follows: let $F\in Fun(\cS,\field)$ and $w\in\cS$.
\begin{enumerate}
\item The stalk $i_{\cS}(F)|_w$ is $F(w)$ viewed as a constant sheaf. In particular, $i_{\cS}(F)$ is constructible with respect to $\cS$.
\item Let $U_w$ be any contractible open subset of $\Star(w)$ such that $\cS|_{U_w}$ is a regular cell complex and the map $\cS|_{U_w}\xrightarrow[]{\sim}\cS|_{\Star(w)}$ of partially ordered sets of strata is a bijection.
Then $\Gamma(U_w;i_{\cS}(F))=F(w)$ and the restriction map $\Gamma(\Star(w);i_{\cS}(F))\rightarrow \Gamma(U_w;i_{\cS}(F))$ is the identity. In particular, $(i_{\cS}(F))_x=F(w)$ for all $x\in w$ and it follows that $i_{\cS}$ is exact. 
\item The restriction map $\Gamma(\Star(w_1);i_{\cS}(F))\rightarrow \Gamma(\Star(w_2);i_{\cS}(F))$ is $F(w_1\rightarrow w_2)$ for all arrows $w_1\rightarrow w_2$ in $\cS$. 
\end{enumerate}

Let $\gamma_{\cS}: Sh(M;\field)\rightarrow Fun(\cS,\field)$ be a functor defined as
\[
\cF\mapsto [S\mapsto \Gamma(\Star(S);\cF)].
\]
Then $\gamma_{\cS}\circ i_{\cS}=\identity$ and $(i_{\cS},\gamma_{S})$ is an adjoint pair and so $\gamma_{\cS}$ is left exact.
As a consequence, we obtain an adjoint pair $(i_{\cS},\Gamma_{\cS})$ in the DG lifting:
\begin{align}
i_{\cS}:\Fun(\cS,\field)\rightleftarrows \Sh(M;\field):\Gamma_{\cS}=R\gamma_{\cS}
\end{align}
 and in fact the essential image of $i_{\cS}$ is contained in $\Sh_{\cS}(M;\field)$. 
More explicitly, $\Gamma_{\cS}=R\gamma_{\cS}$ is given by $\cF^{\bullet}\mapsto [w\mapsto R\Gamma(\Star(w);\cF^{\bullet})]$, and for any $F^{\bullet}\in\Fun(\cS,\field)$ and $\cG^{\bullet}\in\Sh(M;\field)$, we have a natural quasi-isomorphism:
\begin{equation}\label{eqn:adjoint pair for comb model}
\rhom^{\bullet}(i_{\cS}(F^{\bullet}),\cG^{\bullet})\simeq \rhom^{\bullet}(F^{\bullet},\Gamma_{\cS}(\cG^{\bullet}))
\end{equation}
Moreover, we get a natural isomorphism $\beta:\identity\overset{\sim}\Rightarrow \Gamma_{\cS}\circ i_{\cS}$. 
\end{proposition/definition}
\begin{proof}
Firstly, let us show that for any $F\in Fun(\cS,\field)$, $i_{\cS}(F)$ indeed defines a sheaf on $M$. By \cite[Thm.2.7.1]{Vakil2013}, it suffices to define $i_{\cS}(F)$ as a sheaf on a base of $M$. 

We take a base $\cB=\{B_i\}$ for the topology of $M$ such that each $B_i$ is of the form $U_w$ as in (2). 
Let $\cF^{\pre}$ be a presheaf on $\cB$ defined as follows: for each $B_i=U_w$ with $w\in\cS$, $\cF^{\pre}(B_i)\coloneqq F(w)$ and for each inclusion $B_j=V_{w_2}\hookrightarrow B_i=U_{w_1}$, we require that $\Star(w_2)\subset \Star(w_1)$. 
In other words, if $w_1\rightarrow w_2$ is an arrow in $\cS$, then the restriction map $\Gamma(U_{w_1};\cF^{\pre})\rightarrow \Gamma(V_{w_2};\cF^{\pre})$ is defined to be $F(w_1\rightarrow w_2)$. Clearly, this is a well-defined presheaf on $\cB$. 
Then $i_{\cS}(F)\coloneqq \cF$ is defined to be the sheafification of $\cF^{\pre}$. 
Recall that for any open subset $U$ in $M$, we have 
\begin{align*}
\Gamma(U;\cF)=\{(f_p\in \cF^{\pre}_p)_{p\in U}\mid~&\text{For any $p$, there exists a pair $(V_S,s)$ such that $V_S\subset U$ is a neighborhood of $p$}\\ 
&\text{and $V_S\in\cB$ for some $S\in\cS$, and $s\in\cF^{\pre}(V_S)=F(S)$ with $s_q=f_q$ on $V_S$}.
\end{align*}
Also, for all $S\in\cS$ with $S\cap U\neq\emptyset$, each $p\in S\cap U$ has a system of neighborhoods in $\cB$ of the form $U_S$. Then by definition, we have $\Gamma(U_S;\cF^{\pre})=F(S)\isomorphic\cF^{\pre}_p$. Hence, it follows that any section of $\Gamma(U;\cF)$ is locally constant on $S\cap U$ with values in $F(S)$. In particular, for any connected component $w$ in $S\cap U$, we have $\Gamma(w;\cF)=F(S)$. Thus $\cF|_S$ is $F(S)$ viewed as a (locally) constant sheaf. This shows (1), hence $\cF\in\Sh_{\cS}(M;\field)$.

Moreover, we can rewrite the definition as follows. Let us denote by $\cS|_U$ the stratification of $U$ consisting of the connected components of $S\cap U$ for all $S\in\cS$ and let $\tau_U:\cS|_U\rightarrow \cS$ be the map of partially ordered sets induced by inclusion of strata. 
Then the previous definition can be translated into 
\begin{align}\label{eqn:section formula}
\Gamma(U;\cF)=\varprojlim_{w\in\cS|_U} (F\circ\tau_U)(w).
\end{align} 
In other words, 
\[
\Gamma(U;\cF)=\left\{(f(w))_w\in\prod_{w\in\cS|_U}(F\circ\tau_U(w))~\middle|~w_1\leq w_2\Rightarrow (F\circ\tau_U)(w_1\rightarrow w_2)(f(w_1))=f(w_2)\right\}.
\]
Besides, for any inclusion $V\hookrightarrow U$, there is an induced map of partially ordered sets $\tau_{U,V}:\cS|_{V}\rightarrow \cS|_U$ via inclusions of strata. Then $\tau_V=\tau_U\circ \tau_{U,V}$, and the restriction map is just $\tau_{U,V}^*:\Gamma(U;\cF)\rightarrow \Gamma(V;\cF)$ via the pullback of functions.
That is, for any $f\in\Gamma(U;\cF)\subset\prod_{w\in\cS|_U}(F\circ\tau_U)(w)$, we have $(\tau_{U,V}^*f)(w)=f(\tau_{U,V}(w))$ for all $w\in\cS|_V$. 

Now, let $U=U_w$ (including $\Star(w)$) as in (2). We obtain $\Gamma(U;\cF)\isomorphic F(\tau(w\cap U))=F(w)$ as $w\cap U$ is the unique minimum in $\cS|_U$.
For any inclusion $V_{w_2}\hookrightarrow U_{w_1}$ of two open subsets as in (2) with $w_1,w_2\in\cS$, we also know $w_2\leq w_1$ and $\tau_{U_{w_1},V_{w_2}}(w_2\cap V_{w_2})=w_2\cap U_{w_1}\leq w_1\cap U_{w_1}$ in $\cS|_{U_{w_1}}$.
Then the restriction map $\Gamma(U_{w_1};\cF)\isomorphic F(w_1)\rightarrow \Gamma(V_{w_2};\cF)\isomorphic F(w_2)$ is identified with $F(w_1\rightarrow w_2)$ as follows:
\[
f\isomorphic f(w_1\cap U)\mapsto \tau_{U_{w_1},V_{w_2}}^*f\isomorphic(\tau_{U_{w_1},V_{w_2}}^*f)(w_2\cap V)=f(w_2\cap U)=F(w_1\rightarrow w_2)(f(w_1\cap U)).
\]
This shows (2) and (3) and in fact $\cF^{\pre}\isomorphic\cF$ is already a sheaf on $\cB$. 

The functoriality of $i_{\cS}$ follows directly from the definition of $\cF$ above. By (1), the essential image of $i_{\cS}$ is contained in $\Sh_{\cS}(M;\field)$ and it follows immediately from (2) that $\gamma_{\cS}\circ i_\cS=\identity$.

By \cite[Ex.2.7.C]{Vakil2013}, morphisms of sheaves correspond to morphisms of sheaves on a base.
By a direct computation, we have an adjunction of $(i_{\cS},\gamma_{\cS})$ which yields the following: In \eqref{eqn:adjoint pair for comb model}, taken an injective resolution of $\cG^{\bullet}$, say $\cG^{\bullet}\xrightarrow[]{\sim}\cI^{\bullet}$, then by adjunction of $(i_{\cS},\gamma_{\cS})$, we have an injective object $\gamma_{\cS}(\cI^{\bullet})$ and it follows that 
\begin{align*}
\rhom^{\bullet}(i_{\cS}(F^{\bullet}),\cG^{\bullet})&\simeq \hom^{\bullet}(i_{\cS}(F^{\bullet}),\cI^{\bullet})\\
&\simeq \hom^{\bullet}(F^{\bullet},\gamma_{\cS}(\cI^{\bullet}))\\
&\simeq \rhom^{\bullet}(F^{\bullet},\Gamma_{\cS}(\cG^{\bullet})).
\end{align*}
where the last quasi-isomorphism follows from the fact that $\Gamma_{\cS}(\cG^{\bullet})\xrightarrow[]{\sim}\Gamma_{\cS}(\cI^{\bullet})=\gamma_{\cS}(\cI^{\bullet})$ is a quasi-isomorphism. 

Finally, for any $F^{\bullet}\in\Fun(\cS,\field)$ and $x\in w\in\cS$, by (2) and Corollary~\ref{cor:quasi-isomorphism}, we have a natural quasi-isomorphism 
\[
F^{\bullet}(w)=\Gamma(\Star(w);i_{\cS}(F^{\bullet}))\xrightarrow[]{\sim}R\Gamma(\Star(w);i_{\cS}(F^{\bullet}))=R\gamma_{\cS}(i_{\cS}(F^{\bullet}))(w)\simeq (i_{\cS}(F^{\bullet}))_x.
\]
Hence, we get a natural isomorphism $\beta:\identity\overset{\sim}\Rightarrow \Gamma_{\cS}\circ i_{\cS}$. This completes the proof.
\end{proof}

We recall the following lemma:
\begin{lemma}\cite[Prop.3.9]{STZ2017}, \cite[Lem.2.3.2]{Nadler2009}\label{lem:Gamma_S}
Let $\cS$ be a regular cell complex for $M$. Then the functor
\[
\Gamma_{\cS}: \Sh_{\cS}(M;\field)\rightarrow \Fun(\cS,\field), 
\quad\cF\mapsto [S\mapsto R\Gamma(\Star(S);\cF)]
\]
is a quasi-equivalence with a quasi-inverse $i_{\cS}$.
Moreover,  $i_{\cS}$ commutes with the functors induced by restriction to $J_x$.
\end{lemma}

\begin{proof}
The last statement follows directly from the definition of $i_{\cS}$ and it suffices to show the quasi-equivalence.

For any $S\in\cS$, we define a functor $\delta_S\in\Fun(\cS,\field)$ by $\delta_S(w)\coloneqq 0$ if $w\neq S$, and $\delta_S(S)\coloneqq \field[\dim S]$, and sends all arrows $w_1\rightarrow w_2$ to zero. Clearly, $\Fun(\cS,\field)$ is the triangulated envelope of $\delta_S$'s. 
Moreover by Proposition/Definition \ref{def:comb model to sheaf}, we have $i_{\cS}(\delta_S)\simeq \field_S[\dim S]=\omega_S$. 
Hence, $i_{\cS}$ is essentially surjective by Lemma \ref{lem:envelope} and so is $\Gamma_{\cS}$ by the natural isomorphism $\beta:\identity\xRightarrow[]{\sim}\Gamma_{\cS}\circ i_{\cS}$ in Proposition/Definition \ref{def:comb model to sheaf}.

It suffices to show that $i_{\cS}$ is fully faithful. In fact, for any $F,G\in\Fun(\cS,\field)$, by the adjunction $(i_{\cS},\Gamma_{\cS})$ in Proposition/Definition \ref{def:comb model to sheaf}, we have 
\[
\rhom^{\bullet}(i_{\cS}F,i_{\cS}G)\simeq \rhom^{\bullet}(F,\Gamma_{\cS}(i_{\cS}G))\simeq \rhom^{\bullet}(F,G)
\]
where the last quasi-isomorphism follows from the natural isomorphism $\beta:\identity\overset{\sim}\Rightarrow \Gamma_{\cS}\circ i_{\cS}$, which then implies the quasi-isomorphism $G\simeq \Gamma_{\cS}\circ i_{\cS}(G)$. This finishes the proof.
\end{proof}

\begin{remark}
For the last part of the proof above, we can also finish the argument by showing that $\Gamma_{\cS}$ is fully faithful in a different way as follows:
For all $S, W\in\cS$, a direct computation shows
\[
\Gamma_{\cS}(\field_S[\dim S])(W)\coloneqq R\Gamma(\Star(W);\field_S[\dim S])\simeq (\field_S)_{p}[\dim S]
\]
for any $p\in W$. That is, $\Gamma_{\cS}(\field_S[\dim S])(W)$ is acyclic unless $W=S$ when $\Gamma_{\cS}(\field_S[\dim S])(S)\simeq \field[\dim S]$. In other words, $\Gamma_{\cS}(\field_S[\dim S])\simeq\delta_S$. 
Now, for any $S\in\cS$, we take $F^{\bullet}=\Gamma_{\cS}(\field_S[\dim S])\simeq \delta_S$. Then $i_{\cS}(F^{\bullet})\simeq \field_S[\dim S]$. We apply the adjunction \eqref{eqn:adjoint pair for comb model} to obtain
\[
\rhom^{\bullet}(\field_S[\dim S],\cG^{\bullet})\simeq \rhom^{\bullet}(\Gamma_{\cS}(\field_{S}[\dim S]),\Gamma_{\cS}(\cG^{\bullet})).
\]

By an argument of taking shifts, exact triangles, and quasi-isomorphisms in the place of $\field_S[\dim S]$, we then obtain 
\[
\rhom^{\bullet}(\cF^{\bullet},\cG^{\bullet})\simeq \rhom^{\bullet}(\Gamma_{\cS}(\cF^{\bullet}),\Gamma_{\cS}(\cG^{\bullet}))
\]
for any $\cF^{\bullet},\cG^{\bullet}\in\Sh_{\cS}(M;\field)$. 
This shows the fully faithfulness of $\Gamma_{\cS}$ by a different argument.
\end{remark}

\begin{remark}
By \cite[Lem.2.3.2]{Nadler2009}, we also know that for all $S,W\in\cS$, the morphism complex is given as
\[
\rhom^{\bullet}_{\Sh_{\cS}(M;\field)}(\field_S[\dim S], \field_W[\dim W])\simeq
\begin{cases}
\field & S\leq W,\\
0 & \text{otherwise}.
\end{cases}
\]
Hence so is $\rhom^{\bullet}_{\Fun(\cS,\field)}(\delta_S,\delta_W)$.
\end{remark}

\begin{definition}\cite[Def.3.11]{STZ2017}\label{def:comb model}
Let $\cT=(T_\Left\to T\leftarrow T_\Right)\in\BLT$ be a bordered Legendrian graph in $J^1I_x\isomorphic T^{\infty,-}M$ and $\cS$ be a regular cell complex refining the stratification $\cS_T$ induced by $T$. Let $\Fun_{T^+}(\cS,\field)$ (resp. $\Fun_{T^+}(\cS,\field)_0$) be the full subcategory of $\Fun(\cS,\field)$ of objects $F$ satisfying $(1)$ and $(2)$ (resp. $(1)$--$(3)$) as follows:
\begin{enumerate}
\item
Every arrow from a 0-dimensional stratum which is not a crossing, a cusp, or a vertex, or from a 1-dimensional stratum which is not contained in an arc of $T$, is sent to a quasi-isomorphism. In other words, every arrow from a zero- or one-dimensional stratum contained in $\cS$ but not in $\cS_T$, is sent to a quasi-isomorphism.
\item
If $S,S'\in\cS$, with $S'$ bounds $S$ from above, then $S'\rightarrow S$ is sent to a quasi-isomorphism. In other words, every downward arrow is sent to a quasi-isomorphism.
\item
If $S\in\cS$ is contained in the bottom region of $T$ (i.e. the region contains the points with $z\ll 0$), then $F(S)$ is acyclic.  
\item
(Crossing condition) At each crossing $c$ of $\cT$, which is also a 0-dimensional stratum in $\cS$, there is an induced subcategory of $\cS$ as follows:
\[
\begin{tikzcd}
& & N & &\\
& nw \ar[ur] \ar[dl] & & ne \ar[ul] \ar[dr] & \\
W  & & c \ar[ul] \ar[ur] \ar[dl] \ar[dr] \ar[uu] \ar[dd] \ar[ll] \ar[rr] & & E\\
& sw \ar[ul] \ar[dr] & & se \ar[ur]   \ar[dl] & \\
& & S & &
\end{tikzcd}
\]
All triangles in the diagram are commutative and the total complex of the bicomplex $F(c)\rightarrow F(nw)\oplus F(ne)\rightarrow F(N)$ is acyclic.
\end{enumerate}

We also define $\Fun_T(\cS;\field)$ and $\Fun_T(\cS;\field)_0$ to be the full subcategories of $\Fun_{T^+}(\cS;\field)$ and $\Fun_{T^+}(\cS;\field)_0$, respectively, consisting of functors $F$ satisfying the extra crossing condition $(4)$.
\end{definition}

Then we have the following result similar to \cite[Thm.3.12]{STZ2017}.
\begin{lemma}[Combinatorial model]\label{lem:comb model}
Let $\cT\in\BLT$. For a regular cell complex $\cS$ for $M=I_x\times \RR_z$ obtained by refining the stratification $\cS_T$, the functor $\Gamma_{\cS}$ induces a quasi-equivalence $\Gamma_{\cS}:\Sh(T;\field)\xrightarrow[]{\sim} \Fun_T(\cS;\field)$ with quasi-inverse $i_{\cS}$, and $i_{\cS}$ commutes with restriction to $J_x$. So is $\Gamma_{\cS}:\Sh(T;\field)_0\xrightarrow[]{\sim} \Fun_T(\cS;\field)_0$. The results also hold when $T$ is replaced by $T^+$. 
\end{lemma}
\begin{proof}
The proof is entirely similar to that of \cite[Thm.3.12]{STZ2017}. Both the micro-support condition of a sheaf and the properties in Definition \ref{def:comb model} of functors in $\Fun(\cS,\field)$ can be checked locally. This reduces the proof to match the micro-support condition with the corresponding property in Definition \ref{def:comb model} near an arc, a cusp, a crossing, and a vertex. 
The first three cases have been already covered by \cite[Thm.3.12]{STZ2017} and the only new ingredient is what happens at a vertex. 

\noindent{}\textbf{Local combinatorial model near a vertex}.
Let $v\in T$ be a vertex of type $(\ell,r)$. At first we assume that there are no additional 1-dimensional strata in $\cS$ ending at $v$ and that $y(v)=0$ and $T_v^*M\cap(\RR_{>0}\cdot T)=\RR_{>0}\cdot(-dz)$.

Let $p=\alpha dx+\beta dz\in T_v^*M\setminus\RR_{>0}\cdot T$. Then either $\alpha\neq 0$ or $\alpha=0$ and $\beta>0$. 
Near $v$, we label the region below the arc $i$ by $I_i$ for $1\leq i\leq \ell+r$, and label the regions above $v$ and below $v$ by $N$ and $S$, respectively. 
Then $I_{\ell}=S=I_{\ell+r}$ and the poset subcategory of $\cS$ near $v$ looks as
\[
\begin{tikzpicture}[baseline=-.5ex,execute at begin node=$\scriptstyle, execute at end node=$]
\begin{scope}
\clip (-3,-3) rectangle (3,3);
\fill (0,0) circle (2pt);
\foreach \a [count=\index] in {0.2,0,-0.2} {
\draw[thick,smooth,domain=0:-3,->] plot ({\x}, {\a*\x*\x});
\draw[thick,smooth,domain=0:3,->] plot ({\x}, {\a*\x*\x});
}
\foreach \a [count=\index] in {-1,1} {
\draw[thick,smooth,domain=0:-1.7,->] plot ({\x}, {\a*\x*\x});
\draw[thick,smooth,domain=0:1.7,->] plot ({\x}, {\a*\x*\x});
}
\end{scope}
\draw (0,0) node[above] (v) {v} (-1.8,3.2) node (e1) {1} (-3.2,1.8) node (e2) {2} (-3.2,0) node (e3){3} (-3.2,-1.8) node (e4){\cdots} (-1.8,-3.2) node (e5){\ell};
\draw (1.8,3.2) node (e6){\ell+1} (3.4,1.8) node (e7){\ell+2} (3.4,0) node (e8){\ell+3} (3.4,-1.8) node (e9) {\cdots} (1.8,-3.2) node (e10) {\ell+r};
\draw (0,3.2) node (N) {N} (-3.2,3.2) node (r1) {I_1} (-3.2,0.8) node (r2) {I_2} (-3.2,-0.8) node (r3) {I_3} (-3.2,-3.2) node (r4) {I_{\ell-1}};
\draw (0,-3.2) node (S) {S} (3.4,3.2) node (r5) {I_{\ell+1}} (3.4,0.8) node (r6) {I_{\ell+2}} (3.4,-0.8) node (r7) {I_{\ell+3}} (3.4,-3.2) node (r8) {I_{\ell+r-1}};
\begin{scope}
\clip (3.5,3.5) -- (-3.5,3.5) -- (-3.5,-3.5) -- (3.5,-3.5) -- (3.5,0) -- (0.5,0) arc (0:-359:0.5) -- (3.5,0) -- (3.5,3.5);
\draw[thin, red, ->] (0,0) -- (0,3);
\draw[thin, red, ->] (0,0) -- (0,-3);
\foreach \a [count=\index] in {0.1,-0.1} {
\draw[thin,red,smooth,domain=0:-3,->] plot ({\x}, {\a*\x*\x});
\draw[thin,red,smooth,domain=0:3,->] plot ({\x}, {\a*\x*\x});
}
\foreach \a [count=\index] in {-0.33,0.33} {
\draw[thin,red,smooth,domain=0:-3,->] plot ({\x}, {\a*\x*\x});
\draw[thin,red,smooth,domain=0:3,->] plot ({\x}, {\a*\x*\x});
}
\end{scope}
\draw[->] (e1) -- (N);
\draw[dashed,->] (e1) -- (r1);
\draw[->] (e2) -- (r1);
\draw[dashed,->] (e2) -- (r2);
\draw[->] (e3) -- (r2);
\draw[dashed,->] (e3) -- (r3);
\draw[->] (e4) -- (r3);
\draw[dashed,->] (e4) -- (r4);
\draw[->] (e5) -- (r4);
\draw[dashed,->] (e5) -- (S);
\draw[->] (e6) -- (N);
\draw[dashed,->] (e6) -- (r5);
\draw[->] (e7) -- (r5);
\draw[dashed,->] (e7) -- (r6);
\draw[->] (e8) -- (r6);
\draw[dashed,->] (e8) -- (r7);
\draw[->] (e9) -- (r7);
\draw[dashed,->] (e9) -- (r8);
\draw[->] (e10) -- (r8);
\draw[dashed,->] (e10) -- (S);
\end{tikzpicture}
\]

Observe that the downward arrow $v\rightarrow S$ is the same as the compositions $v\rightarrow \ell+r\rightarrow S$ and $v\rightarrow \ell\rightarrow S$. Thus the lemma in this case is equivalent to the following claim.
\begin{claim}
Let $B_r(v)$ be an $r$-neighborhood of $v$ for small enough $r\ll 1$.
Then for any $\cF\in \Sh_{\cS}(M;\field)$, the restriction $\cF|_{B_r(v)}$ is contained in $\Sh_{T}(B_r(v);\field)$ if and only if the arrows
\begin{align*}
\{i\rightarrow I_i\mid 1\leq i\leq \ell+r\}\amalg\{v\rightarrow \ell, v\rightarrow \ell+r\}
\end{align*}
which are the dotted arrows in the diagram above, are sent to quasi-isomorphisms under $\Gamma_{\cS}(\cF)$.
\end{claim}
\begin{proof}
For simplicity, we may assume that $v$ is at the origin $(0,0)$ and up to a local $C^1$-diffeomorphism near $v$, we can assume the half-edges $1,2,\ldots, \ell$ are modelled on the graphs of the decreasing functions $z=a_i x^2$ on $(-1,0]$ with $a_i=1-\frac i\ell$, and the half-edges $\ell+1,\ldots,\ell+r$ are modelled on the graphs of the increasing functions $z=b_ix^2$ on $[0,1)$ with $b_i=1-\frac{\ell+r-i}r$.

If $\cF|_{B_r(v)}\in \Sh_{T}(B_r(v);\field)$, then by the local model near an arc in the smooth case, we have already known that $\Gamma_{\cS}(\cF)(i\rightarrow I_i)$ is a quasi-isomorphism for all $1\leq i\leq \ell+r$. 

We take a linear Morse function $\varphi=\alpha x+\beta z$ with $\alpha\neq0$, and then $\varphi(v)=0, d\varphi(v)=\alpha dx+\beta dz\notin SS(\cF)$.
By definition, $\cF_v\xrightarrow[]{\sim}R\Gamma_{\{\varphi(x)<0\}}(\cF)_v$ is a quasi-isomorphism, which is equivalent to 
\[
R\Gamma(B_r(v)\cap\varphi^{-1}(-\infty,\epsilon);\cF)\xrightarrow[]{\sim}R\Gamma(B_r(v)\cap \varphi^{-1}(-\infty,-\epsilon);\cF)
\]
for a small enough $0<\epsilon$. We set $Y_{\pm}\coloneqq B_r(v)\cap\varphi^{-1}(-\infty,\pm\epsilon)$.  

Then $Y_+$ and $Y_-$ are the regions which are to the left of the red and blue lines or to the right of those lines according to the sign of $\alpha$:
\[
\vcenter{\hbox{
\begingroup%
  \makeatletter%
  \providecommand\color[2][]{%
    \errmessage{(Inkscape) Color is used for the text in Inkscape, but the package 'color.sty' is not loaded}%
    \renewcommand\color[2][]{}%
  }%
  \providecommand\transparent[1]{%
    \errmessage{(Inkscape) Transparency is used (non-zero) for the text in Inkscape, but the package 'transparent.sty' is not loaded}%
    \renewcommand\transparent[1]{}%
  }%
  \providecommand\rotatebox[2]{#2}%
  \ifx\svgwidth\undefined%
    \setlength{\unitlength}{138.46465164bp}%
    \ifx\svgscale\undefined%
      \relax%
    \else%
      \setlength{\unitlength}{\unitlength * \real{\svgscale}}%
    \fi%
  \else%
    \setlength{\unitlength}{\svgwidth}%
  \fi%
  \global\let\svgwidth\undefined%
  \global\let\svgscale\undefined%
  \makeatother%
  \begin{picture}(1,0.88279992)%
    \put(0,0){\includegraphics[width=\unitlength,page=1]{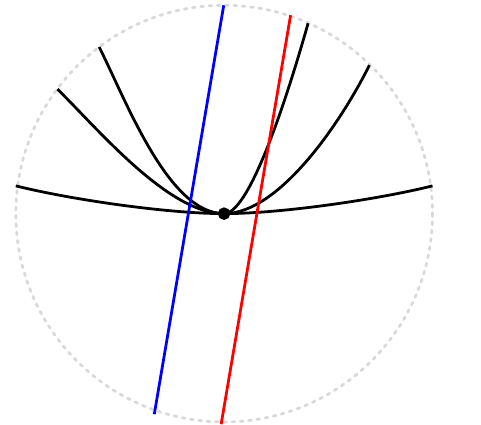}}%
    \put(0.49567874,0.74177679){\color[rgb]{0,0,0}\makebox(0,0)[lt]{\begin{minipage}{0.11555296\unitlength}\raggedright $_{N}$\end{minipage}}}%
    \put(0.37960992,0.29978651){\color[rgb]{0,0,0}\makebox(0,0)[lt]{\begin{minipage}{0.46221183\unitlength}\raggedright $_{I_\ell=S=I_{\ell+r}}$\end{minipage}}}%
    \put(0.21752627,0.65500888){\color[rgb]{0,0,0}\makebox(0,0)[lt]{\begin{minipage}{0.11555296\unitlength}\raggedright $_{I_1}$\end{minipage}}}%
    \put(0,0){\includegraphics[width=\unitlength,page=2]{morse_group_vertex_1.pdf}}%
    \put(0.46101286,0.40667322){\color[rgb]{0,0,0}\makebox(0,0)[lt]{\begin{minipage}{0.11555296\unitlength}\raggedright $_{v}$\end{minipage}}}%
    \put(0.59389876,0.67822267){\color[rgb]{0,0,0}\makebox(0,0)[lt]{\begin{minipage}{0.13814768\unitlength}\raggedright $_{I_{\ell+1}}$\end{minipage}}}%
    \put(0.18368576,0.83421917){\color[rgb]{0,0,0}\makebox(0,0)[lt]{\begin{minipage}{0.11555296\unitlength}\raggedright $_1$\end{minipage}}}%
    \put(0.09124339,0.74177681){\color[rgb]{0,0,0}\makebox(0,0)[lt]{\begin{minipage}{0.11555296\unitlength}\raggedright $_2$\end{minipage}}}%
    \put(-0.00119897,0.52800383){\color[rgb]{0,0,0}\makebox(0,0)[lt]{\begin{minipage}{0.11555296\unitlength}\raggedright $_\ell$\end{minipage}}}%
    \put(0.622787,0.886218){\color[rgb]{0,0,0}\makebox(0,0)[lt]{\begin{minipage}{0.11555296\unitlength}\raggedright $_{\ell+1}$\end{minipage}}}%
    \put(0.7441176,0.81688623){\color[rgb]{0,0,0}\makebox(0,0)[lt]{\begin{minipage}{0.11555296\unitlength}\raggedright $_{\ell+2}$\end{minipage}}}%
    \put(0.91744706,0.52222619){\color[rgb]{0,0,0}\makebox(0,0)[lt]{\begin{minipage}{0.11555296\unitlength}\raggedright $_{\ell+r}$\end{minipage}}}%
  \end{picture}%
\endgroup%
}}
\]

Hence, we have quasi-isomorphisms 
\[
R\Gamma(Y_+;\cF)\simeq\Gamma_{\cS}(\cF)(v)\quad\text{ and }\quad R\Gamma(Y_-;\cF)\simeq \cone\left(\bigoplus_{i=1}^\ell\Gamma_{\cS}(\cF)(i)\xrightarrow[]{d}\bigoplus_{i=1}^{\ell-1}\Gamma_{\cS}(\cF)(I_i)\right)[-1],
\]
where the $i$-th component of $d$ is the difference of the restriction maps induced by $i\rightarrow I_i$ and $i+1\rightarrow I_i$, for $1\leq i\leq l-1$. Under this identification, the restriction map $R\Gamma(Y_+;\cF)\rightarrow R\Gamma(Y_-;\cF)$ is induced by the morphisms $\{v\rightarrow i\mid 1\leq i\leq l\}$ and is a quasi-isomorphism if and only if the total complex of 
\[
\Gamma_{\cS}(\cF)(v)\rightarrow\bigoplus_{i=1}^{\ell}\Gamma_{\cS}(\cF)(i)\rightarrow\bigoplus_{i=1}^{\ell-1}\Gamma_{\cS}(\cF)
\]
is acyclic. As $\Gamma_{\cS}(\cF)(i\rightarrow I_i)$ is a quasi-isomorphism for all $1\leq i\leq \ell-1$, this happens if and only if $\Gamma_{\cS}(\cF)(v\rightarrow \ell)$ is a quasi-isomorphism as well.

Conversely, let $\cF\in\Sh_{\cS}(M;\field)$ be a sheaf such that $\Gamma_{\cS}(\cF)(i\rightarrow I_i)$ for $1\leq i\leq \ell$ and $\Gamma_{\cS}(\cF)(v\rightarrow \ell), \Gamma_{\cS}(\cF)(v\rightarrow \ell+r)$ are all quasi-isomorphisms. 
In order to prove that $\cF|_{B_r(v)}\in\Sh_T(B_r(v);\field)$, it suffices to show that $SS(\cF)\cap (T_v^*M-\{0\})\subset \RR_{>0}\cdot (-dz)$ similar to the smooth case.

For any $p=\alpha dx+\beta dz$ in $T_v^*M$ with $\alpha\neq 0$, the same argument as above applies to any $C^1$-function $\varphi$ with $d\varphi(v)$ sufficiently close to $p$, rather than $\alpha x+\beta z$. Hence, $p\notin SS(\cF)$.

Suppose that $p=(v,\beta dz)\in T_v^*M$ with $\beta>0$. We use $(2)$ in Proposition/Definition \ref{def/prop:Micro-support} to show that $p\notin SS(\cF)$. 
As seen in the picture below
\[
\vcenter{\hbox{\input{morse_group_vertex_2_input.tex}}}
\]
we can take a smaller open ball $U$ around $v$ and a small $\epsilon>0$ defining a line $L\coloneqq \{x\mid \langle x-v,\beta dz\rangle=-\epsilon\}$, and take a proper closed convex cone $\gamma$ in $E=\RR_{xz}^2$ with $0\in\gamma$ so that 
\[
\gamma\setminus\{0\}\subset\{w\mid\langle w,\beta dz\rangle<0\},
\]
which is just the lower half-plane. For example, let $v_1,v_2$ be the two downward unit vectors generating the two green rays as shown above then $\gamma=\{tv_1+sv_2:t,s\geq 0\}$. 
Let $H\coloneqq \{x:\langle x-v,\beta dz\rangle\geq -\epsilon\}$ be the region above $L$ and then we clearly have $H\cap(U+\gamma)\subset B_r(v)$. 
It suffices to show that for all $x\in U$, we have the natural quasi-isomorphism 
\begin{equation}\label{eqn:propagation}
\mathrm{r}:R\Gamma(H\cap(x+\gamma);\cF)\xrightarrow[]{\sim}R\Gamma(L\cap(x+\gamma);\cF)
\end{equation} 
Recall \cite[Rmk.2.6.9]{KS1994} that for any compact subset $Z$ of $B_r(v)$, there exista a quasi-isomorphism 
\[
\varinjlim_{\cU\supset Z}R\Gamma(\cU;\cF)\xrightarrow[]{\sim}R\Gamma(Z;\cF),
\]
where $\cU$ runs over the open neighborhoods of $Z$.

As illustrated above, $H\cap(x+\gamma)$ is the region bounded by a triangle $\Delta_x$ whose sides are parallel to 3 lines $\RR\cdot v_1, \RR\cdot v_2$ and $L$, where $L\cap(x+\gamma)$ is the bottom edge of $\Delta_x$ contained in the region $S$. 
Hence, $R\Gamma(L\cap(x+\gamma);\cF)\simeq\Gamma_{\cS}(\cF)(S)=R\Gamma(S;\cF)$.
Referring the picture above and according to the choice of $x$, we have the following four cases:
\begin{enumerate}
\item If $x=x_1$ so that $H\cap(x_1+\gamma)$ has empty intersection with $T$, then $R\Gamma(H\cap(x+\gamma);\cF)\simeq\Gamma_{\cS}(\cF)(S)$ and so we have the natural isomorphism $\mathrm{r}$ in \eqref{eqn:propagation}.
\item If $x=x_2$ so that $v\in H\cap(x_2+\gamma)$, then $R\Gamma(H\cap(x+\gamma);\cF)\simeq\Gamma_{\cS}(\cF)(v)$ and the map $\mathrm{r}$ is the composition $\Gamma_{\cS}(\cF)(\ell+r\rightarrow S)\circ\Gamma_{\cS}(\cF)(v\rightarrow \ell+r)$. By the hypothesis of $\cF$, each component in the composition is a quasi-isomorphism, hence so is the composition.
\item If $x=x_3$ so that $H\cap(x_3+\gamma)$ has non-empty intersection with $T$ exactly along half-edges $i,i+1,\ldots,\ell$ for some $1\leq i\leq \ell$, then we have
\[
R\Gamma(H\cap(x+\gamma);\cF)\simeq \cone\left(\bigoplus_{k=i}^\ell\Gamma_{\cS}(\cF)(k)\xrightarrow[]{d}\oplus_{k=i}^{\ell-1}\Gamma_{\cS}(\cF)(I_k)\right)[-1],
\]
where the $k$-th component of $d$ is the difference of the restriction maps induced by $k\rightarrow I_k$ and $k+1\rightarrow I_k$ as before. 
Under this identification, the map \eqref{eqn:propagation} is induced by the morphism $\ell\rightarrow S$. Then by the nine lemma (or the $3\times 3$-lemma) for triangulated categories \cite[Lem.2.6]{May2001}, we obtain a diagram in which all squares commute except for the non-displayed one on the bottom right that anti-commutes, and all rows and columns are exact triangles:
\[
\begin{tikzcd}[column sep=2pc, row sep=2pc]
R\Gamma(H\cap(x+\gamma);\cF) \arrow[r] \arrow[d,"\mathrm{r}"] & \bigoplus_{k=i}^\ell\Gamma_{\cS}(\cF)(k) \arrow[r, "d"] \arrow[d,"\mathrm{r}'"] & \bigoplus_{k=i}^{\ell-1}\Gamma_{\cS}(\cF)(I_k)\arrow[r,"{+1}"] \arrow[d] &\ \\
\Gamma_{\cS}(\cF)(S) \arrow[r,"\identity"] \arrow[d] & \Gamma_{\cS}(\cF)(S)\arrow[r] \arrow[d] & 0 \arrow[r,"+1"] \arrow[d] &\ \\
\cone(\mathrm{r})\arrow[r] \arrow[d,"{+1}"] & \bigoplus_{k=i}^{\ell-1}\Gamma_{\cS}(\cF)(k)[1]\arrow[r,"d'"] \arrow[d,"+1"] & \bigoplus_{k=i}^{\ell-1}\Gamma_{\cS}(\cF)(I_k)[1] \arrow[r,"+1"] \arrow[d,"+1"] &\ \\
\ &\ &\ & \ 
\end{tikzcd}
\]
where the map $\mathrm{r}'$ is induced by $\ell\rightarrow S$ and the $k$-th component of $d'$ is the difference of restriction maps induced by $k\rightarrow I_k$ and $k+1\rightarrow I_k$ if $i\leq k<\ell-1$ and is the restriction map induced by $\ell-1\rightarrow I_{\ell-1}$ if $k=\ell-1$. 
By the hypothesis of $\cF$, the map $\Gamma_{\cS}(\cF)(k\rightarrow I_k)$ is a quasi-isomorphism for $k=\ell-1,\ell-2,\ldots, i$, and so is $d'$. The exactness of the third row implies that $\cone(\mathrm{r})$ is acyclic and therefore$\mathrm{r}$ is a quasi-isomorphism.
\item If $x=x_4$ so that $H\cap(x+\gamma)$ has non-empty intersection with $T$ exactly along half-edges $j,j+1,\ldots,\ell+r$, for some $\ell+1\leq j\leq \ell+r$, then the same argument as in the third case holds and we conclude that the map $\mathrm{r}$ in \eqref{eqn:propagation} is a quasi-isomorphism.
\end{enumerate}
This completes the proof of the claim.
\end{proof}

Finally, suppose that there are some additional 1-dimensional strata ending at $v$. 
By our assumption on $\cS$ at the beginning of Section \ref{subsubsec:comb model}, we can regard the extra 1-dimensional strata near $v$ as additional half-edges at $v$ and obtain a new vertex $v'$ with $\ell'$ left half-edges and $r'$ right half-edges for some $\ell'\geq \ell$ and $r'\geq r$. 

We label the half-edges and regions near $v'$ as before.
Then exactly the same argument as above holds and proves the lemma for this case. Equivalently, for all $\cF\in\Sh_{\cS}(M;\field)$, we have $\cF\in\Sh_T(B_r(v);\field)$ if and only if all arrows $\{i\rightarrow I_i\mid 1\leq i\leq \ell'+r'\}\amalg\{v'\rightarrow \ell', v'\rightarrow \ell'+r'\}$ are sent to quasi-isomorphisms under $\Gamma_{\cS}(\cF)$. 
\end{proof}

\subsection{A legible model for constructible sheaves}\label{subsubsec:legible model}
Let $\cT=(T_\Left\to T\leftarrow T_\Right)\in\BLT$ be a bordered Legendrian graph in $J^1I_x=T^{\infty,-}M$ as before and $\cS$ be a stratification refining $\cS_T$.
We can simplify the combinatorial model further under the following stronger assumption:
\begin{assumption}\label{ass:legible model}
The stratification $\cS$ is the induced stratification $\cS_{\tilde T}$ of a bordered Legendrian graph $\tilde \cT=(T_\Left\to\tilde T\leftarrow T_\Right)$ which extends $\cT$ so that $\tilde \cT$ contains no cusps and no vertices with only left or right half-edges.

In particular, this implies that $\cS$ is a regular cell complex and satisfies Assumption \ref{ass:regular cell complex}.

\end{assumption}

\begin{figure}[ht]
\begin{align*}
T&=
\begin{tikzpicture}[baseline=-.5ex,scale=0.8]
\draw[thick] (-2.5,0) to[out=0,in=180] (0,1.5)  to[out=0,in=180] (2.5,0);
\draw[thick] (-2.5,0) to[out=0,in=180] (-1,-1)  to[out=0,in=180] (0,-0.5) to[out=0,in=180] (1,-1) to[out=0,in=180] (2.5,0);
\draw[thick] (-1,0) to[out=0,in=180] (0,0.5)  to[out=0,in=180] (1,0) ;
\draw[thick] (-1,0) to[out=0,in=180] (0,-0.5) to[out=0,in=180] (1,0);
\draw[fill] (0,-0.5) node[below] {$_v$} circle (3pt);
\end{tikzpicture}
&
{\tilde T}&=
\begin{tikzpicture}[baseline=-.5ex,scale=0.8]
\draw[thick] (-2.5,0) to[out=0,in=180] (0,1.5)  to[out=0,in=180] (2.5,0);
\draw[thick] (-2.5,0) to[out=0,in=180] (-1,-1) to[out=0,in=180] (0,-0.5) to[out=0,in=180] (1,-1) to[out=0,in=180] (2.5,0);
\draw[thick] (-1,0) to[out=0,in=180] (0,0.5) to[out=0,in=180] (1,0);
\draw[thick] (-1,0) to[out=0,in=180] (0,-0.5) to[out=0,in=180] (1,0);
\draw[dashed] (-3,0) to (-1,0);
\draw[dashed] (1,0) to (3,0);
\draw[fill] (0,-0.5) node[below] {$_v$}  circle (3pt);
\end{tikzpicture}
\end{align*}
\caption{An example of stratifications satisfying Assumption~\ref{ass:legible model}}
\end{figure}
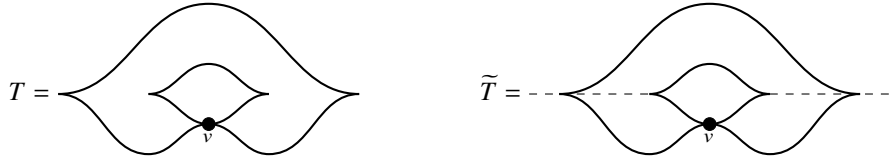

For a stratification $\cS$ satisfying the above assumption, we denote by $G_{\cS}$ the finite set of functions $f$ on $(x_\Left,x_\Right)$ such that $f$ is either $\pm\infty$ or a continuous function whose graph is contained in a union of zero- and one-dimensional strata in $\cS$. 

We observe the following: 
suppose that $\cS$ satisfies Assumption \ref{ass:legible model}. 
Then any region $S$ in $\cS$ is an open disk in $M$ bounded by the graphs of 2 functions $f_L\leq f_U$ in $G_{\cS}$. That is, $S=\{(x,z)|x\in(x_\Left,x_\Right), f_L(x)<z<f_U(x)\}$.\footnote{A function $f_L$ or $f_U$ may be constant at $\pm\infty$ and in this case the region is unbounded.}
We call $(f_L,f_U)$ a \emph{bounding pair} for $S$. The choice of the pair $(f_L,f_U)$ might not be unique. However, since $S$ is contractible, there exist $x_\Left\leq x_{1,S}<x_{2,S}\leq x_\Right$ depending only on $S$ such that $f_L<f_U$ on $(x_{1,S},x_{2,S})$ but $f_L=f_U$ on $(x_\Left,x_{1,S}]\amalg[x_{2,S},x_\Right)$. 
In other words, the graphs of $f_L$ and $f_U$ on $(x_{1,S},x_{2,S})$ are the lower and upper boundary of $S$, which will be denoted by $B_S$ and $A_S$, respectively.

For any two bounding pairs $(f_L,f_U)$ and $(f_L',f_U')$ for $S$, we denote by $(f_L,f_U)\leq (f_L',f_U')$ if $f_L\leq f_L'$ and $f_U\leq f_U'$. In addition, the pair $(\max\{f_L,f_L'\},\max\{f_U,f_U'\})$ becomes also a bounding pair for $S$ which is a common upper bound. It follows that there is a \emph{unique maximal bounding pair} for $S$, which will be denoted by $(l_S,u_S)$.

\begin{definition}\label{def:poset for legible model}
Let $\cS=\cS(\tilde T)$ be a regular cell complex as above. A \emph{poset category} or a \emph{finite (acyclic) quiver with relations}, denoted by $\cR(\cS)$, is defined as follows:
\begin{itemize}
\item
The objects are the 2-dimensional cells of $\cS$;
\item
For any two objects $R$ and $S$ separated by the 1-dimensional stratum $s$ with $S$ below and $R$ above, we assign an arrow $e_s:S\rightarrow R$ that impose $S\leq R$ and generates the partial order on $\cR(\cS)$. 
\item
For any crossing or vertex $v$ of $\tilde T$, there are unique regions $N$ and $S$ immediately above and below $v$ and exactly two directed \emph{paths} $\gamma_\Left(v):S\rightarrow N$ and $\gamma_\Right(v):S\rightarrow N$ which are compositions of arrows, such that $\gamma_\Left(v)$ and $\gamma_\Right(v)$ pass through objects corresponding to the regions in $\Star_{\cS}(v)$ and go around $v$ from the left and right hand side, respectively.
Then we impose the relation $\gamma_\Left(v)\relation\gamma_\Right(v)$. 
\end{itemize}

Moreover, we say that an arrow $e:S\to R$ is \emph{simple} if there are no regions between them. We say that a directed path $\gamma:S\to R$ is simple if it is a composition of simple arrows.

We denote by $\mathrm{A}(\cR(\cS))$ the quotient of the path algebra $\field\langle \cR(\cS)\rangle$ of the acyclic quiver $\cR(\cS)$ by the ideal generated by $(\gamma_\Left(v)-\gamma_\Right(v))$'s
\[
\mathrm{A}(\cR(\cS))\coloneqq \field\langle\cR(\cS)\rangle/(\gamma_\Left(v)-\gamma_\Right(v):v\in V(\cS)).
\]
Here, $V(\cS)$ is the set of 0-dimensional strata of $\cS$. 
\end{definition}

We remark that while the path algebra of an acyclic quiver is of global cohomological dimension $1$, our algebra $\mathrm{A}(\cR(\cS))$ is of global cohomological dimension $2$ in general. 
\begin{proposition}\label{prop:poset for legible model}
The poset category $\cR(\cS)$ is well-defined.
\end{proposition}
\begin{proof}
To show that the definition indeed gives a partial order, we need to check the antisymmetric property, that is, for any two 2-dimensional cells $R,S$, both $R\leq S$ and $S\leq R$ implies that $R=S$.
This follows immediately from the last statement in the lemma below.
\end{proof}

\begin{lemma}\label{lem:poset for legible model}
For any two 2-dimensional cells $R, S$ in $\cS$, the following are all equivalent:
\begin{enumerate}
\item
$u_R\leq l_S$.
\item
$S$ is above the graph of $u_R$.
\item
$R<S$.
\end{enumerate}
Moreover, $R\leq S$ if and only if $l_R\leq l_S$.
\end{lemma}
\begin{proof}$(1)\Rightarrow (2)$ is clear. 

\noindent{}$(2)\Rightarrow (1)$: If $S$ is above the graph of $u_R$, then $l_S\geq u_R$ on $[x_{1,S},x_{2,S}]$. Suppose that $l_S(x)<u_R(x)$ for some $x$, say $x>x_{2,S}$. Then $I_W\coloneqq \{x\in (x_{2,S},x_\Right)|l_S(x)<u_R(x)\}$ is non-empty. Let $x_0\coloneqq \inf(I_W)\geq x_{2,S}$, then $l_S\geq u_R$ on $[x_{1,S},x_0]$, $u_S(x_0)=l_S(x_0)=u_R(x_0)$, and the graph of $u_R$ is strictly above that of $l_S$ on an small open interval $(x_0,x_0+\epsilon)$. Define a pair $(l_S',u_S')$, which coincides with $(l_S,u_S)$ on $(x_\Left,x_0]$, and $l_S'=u_S'=\max\{l_S=u_S,u_R\}$ on $[x_0,x_\Right)$. This defines a bounding pair for $S$, and $(l_S,u_S)<(l_S',u_S')$, contradicting to the maximality of $(l_S,u_S)$. Hence, $l_S(s)\geq u_R(x)$ for all $x\in (x_\Left,x_\Right)$.

\noindent{}$(3)\Rightarrow (2):$ It suffices to show:
If $R$ and $S$ are separated by an 1-dimensional stratum $s$, with $R$ below and $S$ above. Then $S$ is above the graph of $u_R$.

In fact, $s$ is the graph of the function $u_R=l_S$ on an open interval $(x_1,x_2)$. Then, $S$ contains points which are above the graph of $u_R$, hence, the whole disk $S$ is above the graph of $u_R$.

\noindent{}$(2)\Rightarrow (3):$ $S\neq R$ is clear. It suffices to show $R\leq S$.
If the lower boundary of $S_0=S$ is not contained in the graph of $u_R$, then it contains a 1-dimensional stratum $s_1$, which is strictly above the graph of $u_R$. Let $S_1$ be the 2-dimensional cell below $s_1$, then $S_1$ is above the graph of $u_R$ as well. Also by $(3)\Rightarrow (2)$, $S_0$ is above the graph of $u_{S_1}$. By $(2)\Rightarrow (1)$, we then have $u_R\leq l_{S_1}\leq u_{S_1}\leq l_{S_0}\leq u_{S_0}$ in $G_{\cS}$, with $l_{S_0}\neq u_{S_0}, l_{S_1}\neq u_{S_1}$. If the lower boundary of $S_1$ is not contained in the graph of $u_R$, we repeat the procedure above to obtain $s_2, S_2$. In particular, $S_2$ is above the graph of $u_R$, $S_1$ is above the graph of $u_{S_2}$, and $u_R\leq l_{S_2}\leq u_{S_2}\leq l_{S_1}\leq u_{S_1}\leq l_{S_0}\leq u_{S_0}$ in $G_{\cS}$, with $l_{S_i}\neq u_{S_i}$ for $i=0,1,2$. Since $G_{\cS}$ is finite, after repeat the previous procedure finitely many times, we obtain a finite sequence $s_i, S_i$ for $1\leq i\leq N$ for some $N\geq 0$, such that, $s_i$ separates the 2-dimensional cells $S_{i-1}, S_i$ with $S_{i-1}$ above and $S_i$ below, $u_{S_i}\leq l_{S_{i-1}}$, $u_R\leq l_{S_N}$, and the lower boundary of $S_N$ is contained in the graph of $u_R$. The last condition just says that $l_{S_N}=u_R$ on $[x_{1,S_N},x_{2,S_N}]$.

Let us show that there is a 1-dimensional cell $s_{N+1}$ separating $R, S_N$ with $R$ below and $S_N$ above. Otherwise, the open intervals $(x_{1,R},x_{2,R})$ and $(x_{1,S_N},x_{2,S_N})$ have empty intersection. Say, $x_{2,R}\leq x_{1,S_N}$. Then by definition $x_{2,R}$, have $l_R=u_R=l_{S_N}$ on $[x_{1,S_N},x_{2,S_N}]$. Define a pair $(l_R',u_R')$ such that it coincides with $(l_R,u_R)$ outside $[x_{1,S_N},x_{2,S_N}]$, and $l_R'=u_R'=u_{S_N}$ on $[x_{1,S_N},x_{2,S_N}]$. This defines a new bounding pair for $R$ with $(l_R,u_R)<(l_R',u_R')$, contradiction. 
Now, we have shown that $S=S_0\geq S_1\geq\ldots\geq S_N\geq R$, hence $S\geq R$. 

\noindent{}The last statement of the lemma: `$\Rightarrow$' follows from $(3)\Rightarrow (1)$.
\\
$\Leftarrow$: If $R=S$, we are done. Otherwise, assume $l_R\leq l_S$ and $R\neq S$. In particular, $S$ is above the graph of $l_R$. Then $S$ is above the graph of $u_R$ as well, as $R\neq S$ is the only 2-dimensional cell bounded by the graphs of $l_R\leq u_R$. Now $(2)\Rightarrow (3)$ implies $R<S$.

This finishes the proof of the lemma.
\end{proof}

For a region $R\in\cR(\cS)$, we define an open subset of $M$ to be the upper half-space of $l_R(x)$
\[
M_R\coloneqq\{(x,z)|x_\Left<x<x_\Right, z>l_R(x)\}.
\]
Then $M_{R'}\subset M_R$ if and only if $R'\geq R$.

\begin{lemma}\label{lem:lower boundary of a region}
Let $R\in\cR(\cS)$ be fixed. Then for any $R'>R$ in $\cR(\cS)$ with lower boundary $B_{R'}$, the intersection $B_{R'}(R)\coloneqq B_{R'}\cap M_R$ is non-empty and contractible.
\end{lemma}
\begin{proof}
\emph{$B_{R'}(R)$ is non-empty:} Since $R'>R$, by definition, there is a 1-dimensional cell $s$ separating $R'$ and another region $S$, with $R'$ above and $S$ below, such that $S\geq R$. In particular, $S$ is above the graph of $l_R$. It follows that $s$ is contained in $M_R$, hence contained in $B_{R'}(R)=B_{R'}\cap M_R$. 

\noindent{}\emph{$B_{R'}(R)$ is contractible.} We use the notations at the beginning of Section \ref{subsubsec:legible model}. Recall that $B_{R'}$ is simply the graph of $l_{R'}$ over $(x_{1,R'},x_{2,R}')$. Moreover, $l_R<u_R$ on $(x_{1,R}, x_{2,R})$ and $l_R=u_R$ outside $(x_{1,R}, x_{2,R})$. Also, by Lemma \ref{lem:poset for legible model}, we have $u_R\leq l_{R'}$.

If $B_{R'}(R)$ is not contractible, then $B_{R'}$ has a 1-dimensional cell $s$ contained in the graph of $l_R$, such that, $B_{R'}$ contains points both on the left and right of $s$. As $R'$ is above the graph of $u_R$, $s$ has no points in the lower boundary of $R$. In other words, $s$ is contained in the graph of $l_R=u_R$ over $(x_\Left, x_{1,R}]$ or $[x_{2,R}, x_\Right)$, say, the latter. Assume $s$ lives over the interval $(x_1,x_2)$, that is, $s$ is the graph of $l_{R'}=u_R=l_R$ over $(x_1,x_2)$. Then $x_{2,R}\leq x_1$. Moreover, $l_{R'}\neq l_R$ as functions on $[x_2, x_\Right)$, as $B_{R'}$ contains points living over $(x_2,x_\Right)$. 

Now, can define a pair $(\tilde{l}_R,\tilde{u}_R)$ of continuous functions in $G_{\cS}$ such that, it coincides with $(l_R,u_R)$ on $(x_\Left,x_2]$, and $\tilde{l}_R=\tilde{u}_R\coloneqq \max\{l_{R'}, l_R=u_R\}=l_{R'}$ on $[x_2,x_\Right)$. Then $(\tilde{l}_R,\tilde{u}_R)$ is a new bounding pair for $R$, with $(l_R,u_R)<(\tilde{l}_R,\tilde{u}_R)$, contradicting to the maximality of $(l_R,u_R)$. This finishes the proof.
\end{proof}

Now there is an induced functor of poset categories $\rho:\cS\rightarrow \cR(\cS)$, which sends $w$ to $\rho(w)$, the unique 2-dimensional cell `below' $w$. More precisely, we define $\rho(w)$ to be $w$ if $w$ is a 2-dimensional stratum, or to be the 2-dimensional cell immediately below $w$ otherwise.
Then for each region $R\in\cR(\cS)$, the upper half-space $M_R$ is the union of the strata
\begin{equation*}
M_R=\bigcup_{\rho(w)\ge R} w
\end{equation*} 
by Lemma \ref{lem:poset for legible model}.

As usual, we define the abelian category $Fun(\cR(\cS),\field)$, the DG category $\Fun(\cR(\cS);\field)$, and the restriction functors $\Fun(\cR;\field)\rightarrow \Fun(\cR(\cS|_{J_x});\field)$ for an open sub-interval $J_x$ of $I_x$. By pre-composition, we get a functor 
\[
\rho^*:Fun(\cR(\cS);\field)\rightarrow Fun(\cS;\field),
\]
which is clearly exact. We can then use the same letter $\rho^*$ for its DG lifting
\[
\rho^*:\Fun(\cR(\cS);\field)\rightarrow \Fun(\cS;\field).
\]

\begin{proposition/definition}\label{def:adjoint pair for legible model}
Let us define a functor between two abelian categories
\[
\rho_*:Fun(\cS,\field)\rightarrow Fun(\cR(\cS),\field),\quad F\mapsto [R\mapsto \Gamma(M_R;i_{\cS}(F))].
\]
Then $\rho_*\circ \rho^*=\identity$ and $(\rho^*,\rho_*)$ is an adjoint pair
\[
\rho^*:Fun(\cR(\cS),\field)\rightleftarrows Fun(\cS,\field):\rho_*.
\]

As a consequence, we obtain an adjoint pair $(\rho^*,R\rho_*)$ in the DG liftings
\[
\rho^*:\Fun(\cR(\cS),\field)\rightleftarrows \Fun(\cS,\field):R\rho_*
\]
where $R\rho_*$ is given by $F\mapsto [R\mapsto R\Gamma(M_R;i_{\cS}(F))]$.
Moreover, we have a natural isomorphism $\beta:\identity\overset{\sim}\Rightarrow R\rho_*\circ\rho^*$.
\end{proposition/definition}
\begin{proof}
We firstly show that $\rho_*\circ\rho^*=\identity$.
By definition and formula \eqref{eqn:section formula}, for any $F\in Fun(\cR(\cS),\field)$ and $R\in\cR(\cS)$, we have 
\[
(\rho_*\circ\rho^*)(F)(R)=\Gamma(M_R;i_{\cS}(\rho^*F))=\varprojlim_{w\in\cS|_{M_R}}(\rho^*F)(\tau_{M_R}(w)).
\]

Since $\cS|_{M_R}=\{w\in\cS|\rho(w)\geq R\}$ and $\tau_{M_R}$ is just the inclusion map, $(\rho^*F)(\tau_{M_R}(w))=F(\rho(w))$ for all $w\in\cS|_{M_R}$ and it follows that 
\[
(\rho_*\circ\rho^*)(F)(R)=\varprojlim_{\rho(w)\geq R}F(\rho(w))\isomorphic F(R).
\]
Therefore $\rho_*\circ\rho^*=\identity$.

Next, we show that $(\rho^*,\rho_*)$ is an adjoint pair. For any $F\in Fun(\cR(\cS),\field)$ and $G\in Fun(\cS,\field)$, a morphism $f\in\hom(\rho^*F,G)$ is a collection of maps $\{f_w\mid{w\in\cS}\}$ with $f_w:(\rho^*F)(w)=F(\rho(w))\rightarrow G(w)$ such that for any arrow $w_1\rightarrow w_2$, we get a commutative diagram
\[
\begin{tikzcd}
F(\rho(w_1))\arrow[r,"f_{w_1}"]\arrow[d]& G(w_1)\arrow[d]\\
F(\rho(w_2))\arrow[r,"f_{w_2}"] & G(w_2)
\end{tikzcd}
\]

Equivalently, we get a collection of maps $\{f_{R,w}\mid{\rho(w)\geq R}\}$
\[
f_{R,w}: F(R)\rightarrow F(\rho(w))=(\rho^*F)(w)\xrightarrow[]{f_w}G(w)
\]
for any fixed region $R$ in $\cR(\cS)$ and $w\in\cS|_{M_R}$, i.e., $\rho(w)\geq R$ such that we get a commutative diagram
\[
\begin{tikzcd}
F(R_1)\arrow[r,"f_{R_1,w_1}"]\arrow[d] & G(w_1)\arrow[d]\\
F(R_2)\arrow[r,"f_{R_2,w_2}"] & G(w_2)
\end{tikzcd}
\]
for any $R_1\leq R_2$ in $\cR(\cS)$ and $w_1\leq w_2$ in $\cS$ with $\rho(w_i)\geq R_i$.
This is again equivalent to have a collection of maps $\{\tilde{f}_R\mid{R\in\cR(\cS)}\}$ with 
\[
\tilde{f}_R=\varprojlim_{\rho(w)\geq R}f_{R,w}:F(R)\rightarrow \varprojlim_{\rho(w)\geq R}G(w),
\]
such that for any arrow $R_1\rightarrow R_2$ in $\cR(\cS)$, there is a commutative diagram
\[
\begin{tikzcd}
F(R_1)\arrow[r,"\tilde{f}_{R_1}"]\arrow[d] & \displaystyle{\varprojlim_{\rho(w_1)\geq R_1}}G(w_1)\arrow[d]\\
F(R_2)\arrow[r,"\tilde{f}_{R_2}"] &  \displaystyle{\varprojlim_{\rho(w_2)\geq R_2}}G(w_2).
\end{tikzcd}
\]

Since $(\rho_*G)(R)=\Gamma(M_R;i_{\cS}(G))=\varprojlim_{\rho(w)\geq R}G(w)$ by formula \eqref{eqn:section formula}, the above data is identical to a morphism $\tilde{f}$ in $\hom(F,\rho_*G)$. That is, we have a natural isomorphism $\hom(\rho^*F,G)\simeq \hom(F,\rho_*G)$.

Finally, we show the natural isomorphism $\beta:\identity\overset{\sim}\Rightarrow R\rho_*\circ\rho^*$. Indeed, for any $F^{\bullet}\in\Fun(\cR(\cS),\field)$ and any region $R$ in $\cR(\cS)$, we have a natural morphism 
\begin{equation}\label{eqn:quasi-isomorphism beta for legible model}
\beta_{F^{\bullet}}(R):F^{\bullet}(R)\rightarrow (R\rho_*\circ\rho^*F^{\bullet})(R)
\end{equation}
defined by 
\[
F^{\bullet}(R)=(\rho_*\circ\rho^*F^{\bullet})(R)=\Gamma(M_R;i_{\cS}\rho^*F^{\bullet})\rightarrow R\Gamma(M_R;i_{\cS}\rho^*F^{\bullet}).
\]
We want to show that $\beta_{F^{\bullet}}(R)$ is a quasi-isomorphism.

For any region $R'$ in $\cR(\cS)$, let $\delta'_{R'}\in\Fun(\cR(\cS),\field)$ be a functor defined by $\delta'_{R'}(W)=0$ if $W\neq R$, and $\delta'_{R'}(R')=\field$ such that $\delta'_{R'}$ sends all the arrows in $\cR(\cS)$ to zero. 
As seen in the proof of Lemma~\ref{lem:Gamma_S}, $\Fun(\cR(\cS),\field)$ is the triangulated envelope of the objects of the form $\delta'_{R'}$. 
Observe that for fixed $R$, the property of $\beta_{F^{\bullet}}(R)$ in \eqref{eqn:quasi-isomorphism beta for legible model} being a quasi-isomorphism is preserved under taking quasi-isomorphisms, shifts, and cones in the place of $F^{\bullet}$. Thus, it suffices to show that $\beta_{F^{\bullet}}(R)$ is a quasi-isomorphism only when $F^{\bullet}=\delta'_{R'}$ for $R'\in\cR(\cS)$. 

By definition, we see that $i_{\cS}\rho^*\delta'_{R'}=\field_{E_{R'}}$, where $E_{R'}$ is the union of $R'$ and its upper boundary $A_{R'}$. It suffices to show that
\begin{align}\label{eqn:sheaf cohomology of M_R}
 R\Gamma(M_R;\field_{E_{R'}})\simeq
\begin{cases}
\field & R'=R,\\
0 & \textrm{otherwise}.
\end{cases}
\end{align}

If $R'$ is not contained in $M_R$, then neither is $E_{R'}$ ans we have $R\Gamma(M_R;\field_{E_{R'}})\simeq 0$. Otherwise, if $R'\subset M_R$, then so is $E_{R'}$. If $R'\neq R$, by Lemma \ref{lem:lower boundary of a region}, we have $\bar{R'}(R)\coloneqq \bar{R'}\cap M_R=E_{R'}\amalg B_{R'}(R)$, where both of $\bar{R'}(R)$ and $B_{R'}(R)$ are non-empty contractible closed subsets of $M_R$. We then obtain the following exact triangle:
\[
R\Gamma(M_R;\field_{E_{R'}})\rightarrow R\Gamma(M_R;\field_{\overline{R'}(R)})\rightarrow R\Gamma(M_R;\field_{B_{R'}(R)})\xrightarrow[]{+1},
\]
where the second arrow is the same as the quasi-isomorphism $R\Gamma(\overline{R'}(R);\field)\simeq \field\rightarrow R\Gamma(B_{R'}(R);\field)\simeq\field$. Therefore, $R\Gamma(M_R;\field_{E_{R'}})$ is acyclic. Finally, if $R'=R$, then $E_R$ is a contractible closed subset of $M_R$ and it follows that $R\Gamma(M_R;\field_{E_R})\simeq R\Gamma(E_R;\field)\simeq \field$ as desired.
\end{proof}

\begin{definition}\label{def:legible model}
Let $\Fun_{T}(\cR(\cS),\field)$ and $\Fun_{T}(\cR(\cS),\field)_0$ be full subcategories of $\Fun(\cR(\cS),\field)$ consisting of functors $F$ satisfying the conditions (1)--(2) and (1)--(3) described below:
\begin{enumerate} 
\item
For any additional 1-dimensional stratum $s$ of $\cS$ separating two 2-dimensional cells $R_1$ below and $R_2$ above, the morpgism $F(e_s):F(R_1)\xrightarrow[]{\sim}F(R_2)$ is a quasi-isomorphism. 
\item
Around any crossing $c$ of $T$, we label the 4 regions by $N,S,W$ and $E$ as in Definition \ref{def:comb model} and there exists the following commutative diagram
\[
\begin{tikzcd}
 & F(N)&\\
F(W)\arrow[ur] & & F(E)\arrow[ul]\\
 &F(S)\arrow[ul]\arrow[ur]&
\end{tikzcd}
\]
such that the total complex of $F(S)\rightarrow F(W)\oplus F(E)\rightarrow F(N)$ is acyclic.
\item
If $R$ is the bottom region in $\cS$, then $F(R)$ is acyclic.
\end{enumerate}
\end{definition}

\begin{proposition}[Legible model]\label{prop:legible model}
Let $\cS$ be a regular cell complex refining the stratification induced by $\cT$ and satisfying Assumption \ref{ass:legible model}. 
Then
\begin{equation}\label{eqn:legible model}
\rho^*:\Fun(\cR(\cS),\field)\xrightarrow[]{\sim}\Fun_{\tilde T^+}(\cS,\field)
\end{equation}
is a quasi-equivalence with a quasi-inverse $R\rho_*$, and $\rho^*$ is compatible with restriction to an open sub-interval. 
Moreover, the similar statement holds for $\rho^*:\Fun_{T}(\cR(\cS),\field)_{(0)}\xrightarrow[]{\sim}\Fun_{T}(\cS,\field)_{(0)}$.
\end{proposition}

For the proof of the proposition, we need the following lemma.
\begin{lemma}\label{lem:sheaf filtration}
For any $\cF\in\Sh_{\tilde T^+}(M;\field)$, there exists a filtration 
\[
0=\cF_m\rightarrow\cF_{m-1}\rightarrow\cdots\rightarrow \cF_0=\cF
\]
such that each of the associated graded pieces $\Gr_i\cF_{\bullet}$ is contained in $\Sh_{\tilde T^+}(M;\field)$ and supported on a single region $R_i$ of $\cS$. 

More precisely, for each $E_i$ which is the union of $R_i$ and its upper boundary, we have $\Gr_i\cF_{\bullet}\homotopic (A_i)_{E_i}$ for some perfect complex of $\field$-modules $A_i$ regarded as a constant sheaf on $M$.  
\end{lemma}
\begin{proof}
We use the notations at the beginning of Section \ref{subsubsec:legible model}. For each function $f$ in $G_\cS$, we define an open subset of $M$
\[
M_f\coloneqq \{(x,z)\in M\mid x\in(x_\Left,x_\Right),z>f(x)\}.
\]

Consider the maximal bounding pair $(l_{R_1},u_{R_1})$ for the bottom region $R_1$ in $\cS$. Then $f_0\coloneqq l_{R_1}\equiv-\infty$ and $f_1\coloneqq u_{R_1}$ is the unique minimum in $G_\cS\setminus \{f_0\}$. In particular, $M_{f_0}=M$. 

If $M_{f_1}$ is non-empty, then $\cR(\cS|_{M_{f_1}})$ is a finite non-empty partially ordered subset of $\cR(\cS)$ which contains a minimum, say $R_2$. In other words, $R_2$ is a minimum among those regions of $\cS$ which are above the graph of $f_1$. 
Then the lower boundary of $R_2$ is contained in the graph of $f_1$. Otherwise, the lower boundary contains a 1-dimensional stratum $s$ which is strictly above the graph of $f_1$. The region $S=\rho(s)$ below $s$ is less than $R_2$ and this is a contradiction to the minimality of $R_2$. 

Now we take $f_2\in G_{\cS}$ so that $(f_1,f_2)$ is a bounding pair of $R_2$.
Then by induction, we can repeat the procedure above to obtain a sequence $(f_i)$ in $G_{\cS}$ such that for each $i$, the graphs of $f_{i-1}$ and $f_i$ bound a single region $R_i$. Since $G_{\cS}$ is finite, there exists a $m\in\NN$ such that $f_m=\infty$ is the unique maximum in $G_{\cS}$. Hence, $M_{f_m}=\emptyset$ and the procedure stops. In fact, $m$ is the number of regions in $\cS$.

Let $M_i\coloneqq M_{f_i}$. We obtain a sequence of open inclusions 
\[
\emptyset=M_m\subset M_{m-1}\subset\dots\subset M_0=M
\]
and $\cF_i\coloneqq \cF_{M_i}=R(j_{M_i})_{!}j_{M_i}^{-1}\cF$, where $j_{M_i}:M_i\rightarrow M$ is the open inclusion. 
By \cite[Prop.5.4.8.(ii)]{KS1994}, we have $\cF_i\in\Sh_{T^+(\cS)}(M;\field)$ and obtain a filtration
\[
0=\cF_m\rightarrow\cF_{m-1}\rightarrow\dots\rightarrow \cF_0=\cF,
\]
whose associated $i$-th graded piece is $\Gr_i\cF_{\bullet}=\cF_{M_{i-1}\setminus M_i}$ and induced by the exact triangle
\[
\cF_{M_i}\rightarrow\cF_{M_{i-1}}\rightarrow\cF_{M_{i-1}\setminus M_i}\xrightarrow[]{+1}
\]

By definition, $\Gr_i\cF_{\bullet}$ is supported in a single region $R_i$. More precisely, it has possibly non-zero stalks only at points in the region $R_i$ and its upper boundary. Moreover, by the triangular inequality for micro-supports, we have $\Gr_i\cF_{\bullet}\in\Sh_{\tilde\cT^+}(M;\field)$. In fact, 
\[
\Gr_i\cF_{\bullet}\simeq \cF_{M_{i-1}\setminus M_i}\simeq (F(R_i))_{M_{i-1}\setminus M_i},
\]
where the complex $F(R_i)=R\Gamma(R_i;\cF)$ is regarded as a constant sheaf on $M$.
\end{proof}

\begin{proof}[Proof of Proposition \ref{prop:legible model}.]
The proof is similar to that of \cite[Prop.3.22]{STZ2017}, the only nontrivial part is to show the quasi-equivalence \eqref{eqn:legible model}. 

At first, let us show that $\rho^*$ is fully faithful. For any $F,G\in\Fun(\cR(\cS),\field)$, by the adjunction $(\rho^*,R\rho_*)$ in Proposition/Definition \ref{def:adjoint pair for legible model}, we have 
\[
\rhom^{\bullet}(\rho^*F,\rho^*G)\homotopic \rhom^{\bullet}(F,R\rho_*\rho^*G)\homotopic \rhom^{\bullet}(F,G),
\]
where the last quasi-isomorphism follows from the natural isomorphism $\beta:\identity\overset{\sim}\Rightarrow R\rho_*\circ\rho^*$ in Proposition/Definition \ref{def:adjoint pair for legible model}, which then implies the quasi-isomorphism $G\simeq R\rho_*\circ\rho^*(G)$. Thus, $\rho^*$ is fully faithful.
The natural isomorphism also shows that $R\rho_*$ is essentially surjective.
 
Now, it suffices to show that $\rho^*$ is essentially surjective.
For any functor $F\in\Fun(\cS;\field)$, let $\cF\coloneqq i_{\cS}(F)$. Then by Proposition/Definition \ref{def:comb model to sheaf} and Corollary \ref{cor:quasi-isomorphism}, we have 
\[
R\Gamma(\Star(w);\cF)\simeq \cF_x=F(w)=\Gamma(\Star(w);\cF)
\]
for all $x\in w\in\cS$.
The adjunction $(\rho^*,R\rho_*)$ gives us an adjunction map $\epsilon_F:\rho^*R\rho_*F\rightarrow F$, where for any $w\in\cS$, 
\[
\rho^*R\rho_*F(w)=R\rho_*F(\rho(w))=R\Gamma(M_{\rho(w)};\cF),
\]
and by definition of $M_{\rho(w)}$, we have $\Star(w)\subset M_{\rho(w)}$ and therefore
\[
\epsilon_F(w):\rho^*R\rho_*F(w)=R\Gamma(M_{\rho(w)};\cF)\rightarrow R\Gamma(\Star(w);\cF)\simeq F(w)
\]
defined by the restriction of sections. 

Hence it suffices to show that $\epsilon_F:\rho^*R\rho_*F\xrightarrow[]{\sim} F$ is a quasi-isomorphism for all $F\in\Fun_{T^+}(\cS;\field)$, or equivalently, the restriction 
\begin{equation}\label{eqn:restriction}
R\Gamma(M_{\rho(w)};\cF)\xrightarrow[]{\sim} R\Gamma(\Star(w);\cF)
\end{equation}
is a quasi-isomorphism for any $w\in\cS$.

Now we apply Lemma \ref{lem:sheaf filtration} to $\cF$ to reduce the proof of the quasi-isomorphism \eqref{eqn:restriction} to the case when $m=1$, that is, when the sheaf $\cF$ is supported on a single region $R'$ and $\cF\homotopic A_{E_{R'}}$ for some perfect complex of $\field$-modules $A$ with $E_{R'}$ the union of $R'$ and its upper boundary.

However, for $\cF\simeq A_{E_{R'}}$, we see that $R\Gamma(\Star(w);\cF)\simeq \cF_x\simeq\cF_y$ for all $w\in\cS$ and $x\in w, y\in\rho(w)$ by Corollary \ref{cor:quasi-isomorphism}.
Therefore
\[
R\Gamma(\Star(w);A_{E_{R'}})\homotopic
\begin{cases}
A & R'=\rho(w),\\
0 & \textrm{otherwise}.
\end{cases}
\]

Finally, the same holds for $R\Gamma(M_{\rho(w)};A_{E_{R'}})$ as seen in \eqref{eqn:sheaf cohomology of M_R} after replacing $\field$ with $A$ in the same argument. This completes the proof of Proposition \ref{prop:legible model}.
\end{proof}

We can further simplify the legible model for sheaf categories by taking resolutions from the knowledge of quiver representations \cite{ASS2006}. 
Indeed, the category $\Fun(\cR(\cS),\field)$ is nothing but a representation category of the quiver with relations $\cR(\cS)$, i.e., the representation category of $\mathrm{A}(\cR(\cS))$ with values on perfect complexes (i.e. complexes which are quasi-isomorphic to a bounded complex of finite projective $\field$-modules).
Now, for any $F\in\Fun(\cR(\cS),\field)$, we can replace $F$ by either a projective or injective resolution.

Let us review a key fact about $\mathrm{A}(\cR(\cS))$ which we will use later on. 
Recall some terminology in quiver representation theory.
Let us denote by $\cR(\cS)_0$ and $\cR(\cS)_1$ the sets of objects and arrows respectively. 
For any object $a\in \cR(\cS)_0$, the path of length $0$ at $a$ will be denote by $\lambda_a$. Then $\mathrm{A}(\cR(\cS))$ is a free $\field$-module with the basis the set of all paths (of lengths $\geq 0$) in $\cR(\cS)$ modulo the relation in Definition \ref{def:poset for legible model}. 

\begin{proposition}\label{prop:(co)fibrant replacement}
Let $\cS$ be a regular cell complex refining the stratification induced by $\cT$ and satisfying Assumption \ref{ass:legible model}.
Then for $\mathrm{A}\coloneqq \mathrm{A}(\cR(\cS))$, we have the following characterzations for indecomposable (left) projective and injective modules:
\begin{enumerate}
\item The indecomposable (left) projective modules of $\mathrm{A}$ are $P_a\coloneqq\mathrm{A}\lambda_a$ with $a\in\cR(\cS)_0$. In particular, $P_a(s)$ is injective for all arrows $s$ in $\cR(\cS)_1$. 
\item The indecomposable (left) injective modules of $\mathrm{A}$ are $I_a\coloneqq\hom_\field(\lambda_a\mathrm{A},\field)$ with $a\in\cR(\cS)_0$. In particular, $I_a(s)$ is surjective for all arrows $s$ in $\cR(\cS)_1$. 
\item The (left) simple modules of $\mathrm{A}$ are $S_a$ with $a\in\cR(\cS)_0$, where $S_a$ consists of $\field$ at the object $a$ and $0$ otherwsie, and $S_a$ sends all the arrows in $\cR(\cS)_1$ to zero.
\end{enumerate}
\end{proposition}
\begin{proof}
This is a standard fact in quiver representation theory. See \cite[II.2.Lem.2.4, III.2.Lem.2.1]{ASS2006}.
\end{proof}

\subsection{A combinatorial proof of invariance of sheaf categories}
We can now give an alternative proof of Theorem \ref{thm:inv of sh cat} via the combinatorial descriptions of sheaf categories, i.e. Lemma \ref{lem:Gamma_S}, Lemma \ref{lem:comb model}, Proposition \ref{prop:legible model}, and Proposition \ref{prop:(co)fibrant replacement}.

\begin{proof}[Proof of Theorem \ref{thm:inv of sh cat}]
It suffices to show the invariance of $\Sh(\cT;\field)$ under the 6 types of Legendrian Reidemeister moves in the front projection. First observe that the diagram $\Sh(\cT;\field)$ is unchanged if we regard each cusp in $\pi_{xz}(T)$ as a 2-valency vertex. Then the Legendrian Reidemeister moves $\RM{I}$ and $\RM{IV}$ are special cases of move $\RM{VI}$, and move $\RM{II}$ is a special case of move $\RM{V}$. So it suffices to show the invariance under moves $\RM{III}$, $\RM{V}$, and $\RM{VI}$.
Let $T,T'$ be any pair of bordered Legendrian graphs in $T^{\infty,-}M$, which differ by any one of the above three moves. We can take a pair of regular cell complexes $\cS,\cS'$ refining the stratifications of $M$ induced by $T,T'$ respectively, such that:
\begin{enumerate}
\item
$\cS,\cS'$ satisfy Assumption \ref{ass:legible model} for $T,T'$ respectively.
\item
$\cS,\cS'$ coincide outside the local bordered Legendrian graphs involving the Legendrian Reidemeister move. 
\end{enumerate}
Denote by $\cS_\Left=\cS_\Left'$, $\cS_\Right=\cS_\Right'$ the induced stratifications of $\cS$ (equivalently, $\cS'$) near the left and right boundary of $M=I_x\times\RR_z$ respectively. In other words, say $T_\Left$ (resp. $T_\Right$) lives in over the interval $I_\Left=(x_\Left,x_\Left+\epsilon)\subset I_x$ (resp. $I_\Right=(x_\Right-\epsilon,x_\Right)\subset I_x$), then $\cS_\Left=\cS|_{I_\Left}$ (resp. $\cS_\Right=\cS|_{I_\Right}$) is the stratification with strata the connected components of $S\cap I_\Left\times\RR_z$ (resp. $S\cap I_\Right\times\RR_z$) for all $S\in\cS$, as defined at the beginning of Section \ref{subsubsec:comb model}.

Apply Lemma \ref{lem:Gamma_S}, Lemma \ref{lem:comb model} and Proposition \ref{prop:legible model}, it suffices to show, for each of three moves above, the equivalence between the 2 diagrams of DG categories: $\Fun_{\cT}(\cR(\cS),\field)\coloneqq (\Fun_{T_\Left}(\cR(\cS_\Left),\field)\leftarrow \Fun_T(\cR(\cS),\field)\rightarrow \Fun_{T_\Right}(\cR(\cS_\Right),\field))$, and 
$\Fun_{\cT'}(\cR(\cS'),\field)=(\Fun_{T_\Left'}(\cR(\cS_\Left'),\field)\leftarrow \Fun_{T'}(\cR(\cS),\field)\rightarrow \Fun_{T_\Right'}(\cR(\cS_\Right'),\field))$.
This can essentially be shown by the diagrams in Figures \ref{fig:Sheaf invariance_move III}, \ref{fig:Sheaf invariance_move V}, \ref{fig:Sheaf invariance_move VI}.

\subsubsection{Move $\RM{III}$}

If $T,T'$ differ by a move $\RM{III}$, can take a pair of regular cell complexes $\cS,\cS'$ as above so that, the local bordered Legendrian graphs involving the move $\RM{III}$ are as in Figure \ref{fig:Sheaf invariance_move III}.(1),(4) respectively. In the figure, denote by $T_i, \cS_i$ the induced bordered Legendrian graph and  regular cell complexes refining $\cS_{T_i}$ in (i), for $1\leq i\leq 4$. In particular, we have $T_1\coloneqq T, T_4\coloneqq T', T_2=T_3$, and $\cS_1=\cS, \cS_4=\cS', \cS_2=\cS_3$, and all the $(T_i,\cS_i)$'s coincide outside the local bordered Legendrian graphs in the figure. The letters in each picture label the regions.

\begin{figure}[ht]
\[
\begin{tikzcd}
(1)=\vcenter{\hbox{\scriptsize\input{invariance_sheaf_R3_1_input.tex}}}\ar[r,"\RM{III}",leftrightarrow] \ar[d, "\sim",leftrightarrow]& 
\vcenter{\hbox{\scriptsize\input{invariance_sheaf_R3_4_input.tex}}}=(4)\\
(2)=\vcenter{\hbox{\scriptsize\input{invariance_sheaf_R3_2_input.tex}}}\ar[r,"\sim",leftrightarrow]& 
\vcenter{\hbox{\scriptsize\input{invariance_sheaf_R3_3_input.tex}}}=(3)\ar[u,"\sim",leftrightarrow]
\end{tikzcd}
\]
\caption{Invariance of sheaf categories: Legendrian Reidemeister move $\RM{III}$.}
\label{fig:Sheaf invariance_move III}
\end{figure}

\begin{proposition/definition}\label{def:adjunction 1 for move III}
There is an adjunction $(i,\pi)$ of functors between two abelian categories
\[
i:Fun(\cR(\cS_1),\field)\rightleftarrows Fun(\cR(\cS_2),\field):\pi
\]
defined as follows:
\begin{enumerate}
\item
$i(F_1)$ forgets $F_1(C)$ in (1);
\item
$\pi(F_2)$ keeps the same data as $F_2$ on other regions, and $\pi(F_2)(C)\coloneqq F_2(X_1)\times_{F_2(Z)}F_2(X_2)=\ker(F_2(X_1)\oplus F_2(X_2)\xrightarrow[]{(+,-)}F_2(Z))$. Then the map $\pi(F_2)(A\rightarrow C)$ is uniquely determined by the universal property of $\pi(F_2)(C)$ as a kernel. 
\end{enumerate}
\noindent{}Moreover, both of $i$ and $\pi$ are exact, and $i\circ\pi=id$. As a consequence, we obtain an adjunction $(i,\pi)$ of DG functors in the DG lifting:
\[
i:\Fun(\cR(\cS_1),\field)\rightleftarrows \Fun(\cR(\cS_2),\field):\pi
\]
Moreover, we get a natural isomorphism $\beta:i\circ\pi\overset{\sim}\Rightarrow \identity$, and $\pi:\Fun(\cR(\cS_2),\field)\rightarrow \Fun(\cR(\cS_1),\field)$ is fully faithful.
\end{proposition/definition}

\begin{proof}
Everything is done by a direct check except the last statement, that is, $\pi$ is fully faithful, which can be shown as follows. For any $F,G\in\Fun(\cR(\cS_2),\field)$, by the adjunction $(i,\pi)$, we have
\[
\rhom^{\bullet}(\pi(F),\pi(G))\simeq \rhom^{\bullet}(i\circ\pi(F),G)\simeq \rhom^{\bullet}(F,G)
\]
where the last quasi-isomorphism follows from the natural isomorphism $\beta:i\circ\pi\overset{\sim}\Rightarrow \identity$, which gives the quasi-isomorphism $\beta_F:i\circ\pi(F)\xrightarrow[]{\sim}F$.
\end{proof}

\begin{lemma}\label{lem:equiv 1 for move III}
Let $\cD_1\coloneqq \Fun_{T_1}(\cR(\cS_1),\field)$ be the DG category in (1), and $\cD_2$ be the DG full subcategory of $\Fun_{T_2}(\cR(\cS_2),\field)$ whose objects are functors $F_2$ such that two additional crossing conditions induced by the two red squares in (2) hold, i.e. both of the total complexes $\Tot(F_2(A)\rightarrow F_2(B_1)\oplus F_2(X_1)\xrightarrow[]{(+,-)}F_2(Z))$ and $\Tot(F_2(A)\rightarrow F_2(B_2)\oplus F_2(X_2)\xrightarrow[]{(+,-)}F_2(Z))$ are acyclic. Then the adjunction $(i,\pi)$ of DG functors in Proposition/Definition \ref{def:adjunction 1 for move III} induces equivalences
\[
i:\cD_1\overset{\sim}\rightleftarrows \cD_2:\pi
\]
which are quasi-inverses to each other.
\end{lemma}

\begin{proof}
Firstly, the adjunction of DG functors $(i,\pi)$ in Proposition/Definition \ref{def:adjunction 1 for move III} induces an adjunction of DG functors 
$
i:\cD_1\rightarrow \cD_2:\pi
$.

\noindent{}To show this, it suffices to show that the essential image of $\cD_1$ under $i$ is contained in $\cD_2$, and the essential image of $\cD_2$ under $\pi$ is contained in $\cD_1$. The former is clear, as for example, for any $F_1\in\cD_1$, the crossing conditions for $F_1$ at the crossing above and the crossing to the left of $C$ in (1) implies the first crossing condition for $i(F_2)$, i.e. the total complex $\Tot(F_2(A)\rightarrow F_2(B_1)\oplus F_2(X_1)\xrightarrow[]{(+,-)}F_2(Z))$ is acyclic. The latter essentially follows from Proposition \ref{prop:(co)fibrant replacement}. More precisely, for any $F_2\in\cD_2\subset\Fun_{T_2}(\cR(\cS),\field)$, by the proposition, $F_2$ is quasi-isomorphic to an object in which all the arrows are sent to surjections. So, we can assume $F_2$ itself has this property. Now, the crossing condition for $\pi(F_2)$ at the crossing above $C$ is automatic, and
the crossing conditions for $F_2$ is equivalent to the natural cochain maps $F_2(A)\xrightarrow[]{\sim}\ker(F_2(B_1)\oplus F_2(X_1)\xrightarrow[]{(+,-)}F_2(Z))$, and $F_2(A)\xrightarrow[]{\sim}\ker(F(B_2)\oplus F_2(X_2)\xrightarrow[]{(+,-)}F_2(Z))$ being quasi-isomorphisms, which are equivalent to the crossing conditions for $\pi(F_2)$ at the crossings to the left and to the right of $C$. Hence, $\pi(F_2)\in\cD_2$.

Now, the natural isomorphism $\beta:i\circ\pi\overset{\sim}\Rightarrow \identity: \cD_2\rightarrow \cD_2$ implies that $i:\cD_1\rightarrow\cD_2$ is essentially surjective.  
By Proposition/Definition \ref{def:adjunction 1 for move III}, $\pi:\cD_2\rightarrow\cD_1$   is fully faithful. It then suffices to show $\pi$ is essentially surjective.

For any $F_1\in\Fun(\cR(\cS_1),\field)$, by Proposition \ref{prop:(co)fibrant replacement}, $F_1$ is quasi-isomorphic to an object in which all arrows are sent to surjections. We can then assume $F_1$ itself has this property. Now, the crossing conditions for $F_1$ in the local bordered Legendrian graph are equivalent to the natural quasi-isomorphisms $F_1(C)\xrightarrow[]{\sim}\ker(F_1(X_1)\oplus F_1(X_2)\xrightarrow[]{(+,-)}F_1(Z))$, $F_1(A)\xrightarrow[]{\sim}\ker(F_1(B_1)\oplus F_1(C)\xrightarrow[]{(+,-)}F_1(X_2))$, and $F_1(A)\xrightarrow[]{\sim}\ker(F_1(B_2)\oplus F_1(C)\xrightarrow[]{(+,-)}F_1(X_1))$, equivalently, $F_1(C)\xrightarrow[]{\sim}\pi\circ i(F_1)(C)$, and the two crossing conditions induced by the two red squares in (2) defining $i(F_1)\in\cD_2$. Hence, we get a natural quasi-isomorphism $\alpha_{F_1}:F_1\xrightarrow[]{\sim}\pi\circ i(F_1)$. In particular, $\pi:\cD_2\rightarrow\cD_1$ is essentially surjective. 
\end{proof}

Similar to $\cD_2$, let $\cD_3$ be the corresponding DG category defined by (3). Then (2) and (3) are identical, that is, $\cD_3=\cD_2$. We just have changed the letters labelling the regions from `$X$' to `$X'$', to make it convenient for us to compare with (4). 

Let $\cD_4\coloneqq \Fun_{T_4}(\cR(\cS_4),\field)$ be the DG category in (4). By a dual argument to that in proving the equivalence between $\cD_1$ and $\cD_2$, we immediately obtain equivalences
\[
p:\cD_3\overset{\sim}\rightleftarrows\cD_4:j
\]
which are quasi-inverses to each other. Here $j$ is induced from the forgetful functor $i:Fun(\cR(\cS_4),\field)\rightarrow Fun(\cR(\cS_3),\field)$ with $i(F_4)$ forgets $F_4(Y')$, and $p$ is induced from the functor $p:Fun(\cR(\cS_3),\field)\rightarrow Fun(\cR(\cS_4),\field)$ with $p(F_3)(Y')=\coker(F_3(A')\xrightarrow[]{(+,-)^t}F_3(B_1')\oplus F_3(B_2'))$.

Now, by composition, we then get an equivalence $\Fun_{T}(\cR(\cS),\field)\simeq \Fun_{T'}(\cR(\cS),\field)$. Notice that the composition sends the object $F|_{T_\Left}$, $F$ restricted to $T_\Left$, to $F'|_{T_\Left'}$ which is quasi-isomorphic to $F_{T_\Left'=T_\Left}$. Hence, the equivalence commutes with $\identity:\Fun_{T_\Left}(\cR(\cS_\Left),\field)\xrightarrow[]{\sim}\Fun_{T_\Left'}(\cR(\cS_\Left'),\field)$ up to a specified natural isomorphism.
The same holds for $T_\Right=T_\Right'$. Therefore, we get an equivalence of diagrams of DG categories: $\Fun_{\cT}(\cR(\cS),\field)\simeq \Fun_{\cT'}(\cR(\cS),\field)$, as desired.

\subsubsection{Move $\RM{V}$}

The proof proceeds similarly as above. If $T,T'$ differ by a move $\RM{V}$, can take the pair of regular cell complexes $\cS,\cS'$ so that the local bordered Legendrian graphs involving the move are as in Figure \ref{fig:Sheaf invariance_move V}.(1),(4) respectively. In the picture, $T$ and $T'$ both have $l$ left and $r$ right half-edges at the vertex, labelled by $1,\ldots,l$ and $l+1,\ldots,l+r$ from top to bottom respectively. In addition, the bordered Legendrian graph $T(\cS)$ (resp. $T(\cS')$) underlying $\cS$ (resp. $\cS'$) is $T$ (resp. $T'$) plus the additional dashed arcs. We have added the dashed arcs to ensure that $S,S'$ satisfy Assumption \ref{ass:legible model}. As before, let $T_i,\cS_i$ be the induced bordered Legendrian graphs and corresponding regular cell complexes in (i), for $1\leq i\leq 4$. In particular, $(T_1,\cS_1)=(T,\cS), (T_4,\cS_4)=(T',\cS_4)$, and $(T_2,\cS_2)=(T_3,\cS_3)$. 

\begin{figure}[ht]
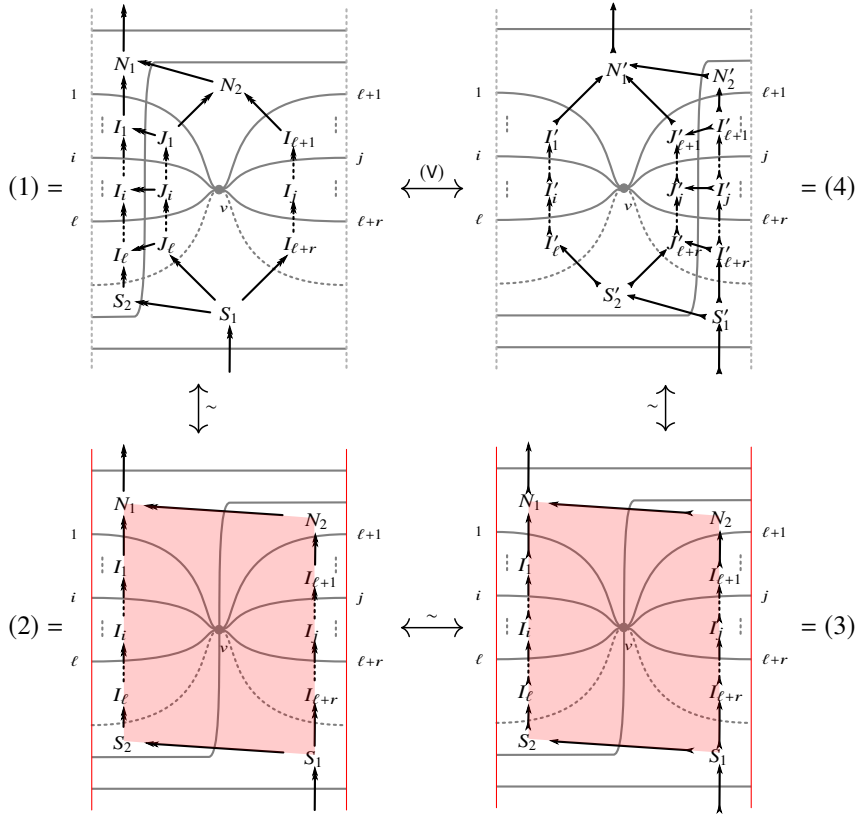

\[
\begin{tikzcd}
(1)=\vcenter{\hbox{\scriptsize\input{invariance_sheaf_R5_1_input.tex}}}\ar[r,"\RM{V}",leftrightarrow] \ar[d, "\sim",leftrightarrow]& 
\vcenter{\hbox{\scriptsize\input{invariance_sheaf_R5_4_input.tex}}}=(4)\\
(2)=\vcenter{\hbox{\scriptsize\input{invariance_sheaf_R5_2_input.tex}}}\ar[r,"\sim",leftrightarrow]& 
\vcenter{\hbox{\scriptsize\input{invariance_sheaf_R5_3_input.tex}}}=(3)\ar[u,"\sim",leftrightarrow]
\end{tikzcd}
\]
\caption{Invariance of sheaf categories: Legendrian Reidemeister move $\RM{V}$.}
\label{fig:Sheaf invariance_move V}
\end{figure}

Similar to Proposition/Definition \ref{def:adjunction 1 for move III}, we have
\begin{proposition/definition}\label{def:adjunction 1 for move V}
There is an adjunction $(i,\pi)$ of functors between two abelian categories
\[
i:Fun(\cR(\cS_1),\field)\rightleftarrows Fun(\cR(\cS_2),\field):\pi
\]
defined as follows:
\begin{enumerate}
\item
$i(F_1)$ forgets $F_1(J_k)$'s in (1) for $1\leq k\leq l$;
\item
$\pi(F_2)$ keeps the same data as $F_2$ on other regions, and $\pi(F_2)(J_k)\coloneqq F_2(I_k)\times_{F_2(N_1)}F_2(N_2)=\ker(F_2(I_k)\oplus F_2(N_2)\xrightarrow[]{(+,-)}F_2(N_1))$ for $1\leq k\leq l$. Then all the additional maps for $\pi(F_2)$ are uniquely determined by the universal properties of $\pi(F_2)(J_k)$'s as kernels. 
\end{enumerate}
\noindent{}Moreover, both of $i$ and $\pi$ are exact, and $i\circ\pi=id$. As a consequence, we obtain an adjunction $(i,\pi)$ of DG functors in the DG lifting:
\[
i:\Fun(\cR(\cS_1),\field)\rightleftarrows \Fun(\cR(\cS_2),\field):\pi
\]
Moreover, we get a natural isomorphism $\beta:i\circ\pi\overset{\sim}\Rightarrow \identity$, and $\pi:\Fun(\cR(\cS_2),\field)\rightarrow \Fun(\cR(\cS_1),\field)$ is fully faithful.
\end{proposition/definition}

\begin{proof}
The proof is identical to that of Proposition/Definition \ref{def:adjunction 1 for move III}.
\end{proof}

Similar to Lemma \ref{lem:equiv 1 for move III}, we have
\begin{lemma}\label{lem:equiv 1 for move V}
Let $\cD_1\coloneqq \Fun_{T_1}(\cR(\cS_1),\field)$ be the DG category in (1), and $\cD_2$ be the DG full subcategory of $\Fun_{T_2}(\cR(\cS_2),\field)$ whose objects are functors $F_2$ such that the additional crossing condition induced by the red square in (2) holds, i.e. the total complex $\Tot(F_2(S_1)\rightarrow F_2(S_2)\oplus F_2(N_2)\xrightarrow[]{(+,-)}F_2(N_1))$ is acyclic. Then the adjunction $(i,\pi)$ of DG functors in Proposition/Definition \ref{def:adjunction 1 for move V} induces equivalences
\[
i:\cD_1\overset{\sim}\rightleftarrows \cD_2:\pi
\]
which are quasi-inverses to each other.
\end{lemma}

\begin{proof}
Similar to that of Lemma \ref{lem:equiv 1 for move III}.
\end{proof}

Similar to the case of move III, let $\cD_3$ be the DG category defined by (3). Then $\cD_3=\cD_2$. Let $\cD_4\coloneqq \Fun_{T_4}(\cR(\cS_4),\field)$ be the DG category in (4). By a dual argument to that in proving the equivalence between $\cD_1$ and $\cD_2$, we immediately obtain equivalences
\[
p:\cD_3\overset{\sim}\rightleftarrows\cD_4:j
\]
which are quasi-inverses to each other. Here $j$ is induced from the forgetful functor $i:Fun(\cR(\cS_4),\field)\rightarrow Fun(\cR(\cS_3),\field)$ with $i(F_4)$ forgets $F_4(J_k')$'s for $l+1\leq k\leq l+r$, and $p$ is induced from the functor $p:Fun(\cR(\cS_3),\field)\rightarrow Fun(\cR(\cS_4),\field)$ with $p(F_3)(J_k')=\coker(F_3(S_1')\xrightarrow[]{(+,-)^t}F_3(I_k')\oplus F_3(S_2'))$.

Again by composition, we get an equivalence of diagrams of DG categories: $\Fun_{\cT}(\cR(\cS),\field)\simeq \Fun_{\cT'}(\cR(\cS),\field)$, as desired.

\subsubsection{Move $\RM{VI}$}

If $T,T'$ differ by a move $\RM{VI}$, can take the pair of regular cell complexes $\cS,\cS'$ so that the local bordered Legendrian graphs involving the move are as in Figure \ref{fig:Sheaf invariance_move VI}.(1),(2) respectively. 
In the picture, $T$ has $l$ left and $r\geq 1$ right half-edges, labelled by $1,\ldots, l$ and $l+1,\ldots,l+r$ from top to bottom respectively. 
$T'$ has $l+1$ left and $r-1$ right half-edges at the vertex, with the labelling inherited from that of $T$. In addition, the bordered Legendrian graph $T(\cS)$ (resp. $T(\cS')$) underlying $\cS$ (resp. $\cS'$) is $T$ (resp. $T'$) plus the additional dashed arcs. We have added the dashed arcs to ensure that $S,S'$ satisfy Assumption \ref{ass:legible model}. 

\begin{figure}[ht]
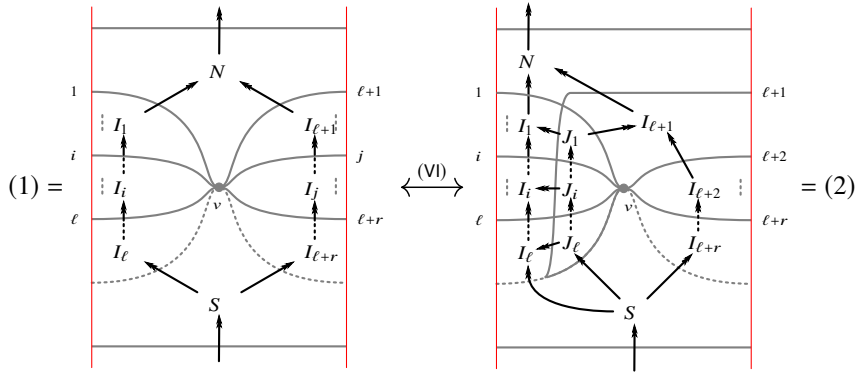

\[
\begin{tikzcd}
(1)=\vcenter{\hbox{\scriptsize\input{invariance_sheaf_R6_1_input.tex}}}\ar[r,"\RM{VI}",leftrightarrow] &
\vcenter{\hbox{\scriptsize\input{invariance_sheaf_R6_2_input.tex}}}=(2)
\end{tikzcd}
\]
\caption{Invariance of sheaf categories: Legendrian Reidemeister move $\RM{VI}$.}
\label{fig:Sheaf invariance_move VI}
\end{figure}

Exactly the same as before, we have
\begin{proposition/definition}\label{def:adjunction for move VI}
There is an adjunction $(i,\pi)$ of functors between two abelian categories
\[
i:Fun(\cR(\cS'),\field)\rightleftarrows Fun(\cR(\cS),\field):\pi
\]
defined as follows:
\begin{enumerate}
\item
$i(F')$ forgets $F'(J_k)$'s in (2) for $1\leq k\leq l$;
\item
$\pi(F)$ keeps the same data as $F$ on other regions, and $\pi(F)(J_k)\coloneqq F(I_k)\times_{F(N)}F(I_{l+1})=\ker(F(I_k)\oplus F(I_{l+1})\xrightarrow[]{(+,-)}F(N))$ for $1\leq k\leq l$. Then all the additional maps for $\pi(F)$ are uniquely determined by the universal properties of $\pi(F)(J_k)$'s as kernels. 
\end{enumerate}
\noindent{}Moreover, both of $i$ and $\pi$ are exact, and $i\circ\pi=id$. As a consequence, we obtain an adjunction $(i,\pi)$ of DG functors in the DG lifting:
\[
i:\Fun(\cR(\cS'),\field)\rightleftarrows \Fun(\cR(\cS),\field):\pi
\]
Moreover, we get a natural isomorphism $\beta:i\circ\pi\overset{\sim}\Rightarrow \identity$, and $\pi:\Fun(\cR(\cS),\field)\rightarrow \Fun(\cR(\cS'),\field)$ is fully faithful.
\end{proposition/definition}

\begin{proof}
The proof is identical to that of Proposition/Definition \ref{def:adjunction 1 for move III}.
\end{proof}

Also, we have
\begin{lemma}\label{lem:equiv for move VI}
Let $\cD\coloneqq \Fun_{T}(\cR(\cS),\field)$ and $\cD'\coloneqq \Fun_{T'}(\cR(\cS'),\field)$ be the DG categories in (1) and (2) respectively. Then the adjunction $(i,\pi)$ of DG functors in Proposition/Definition \ref{def:adjunction 1 for move V} induces equivalences
\[
i:\cD'\overset{\sim}\rightleftarrows \cD:\pi
\]
which are quasi-inverses to each other.
\end{lemma}

\begin{proof}
Similar to that of Lemma \ref{lem:equiv 1 for move III}.
\end{proof}
Again as before, we get an equivalence of diagrams of DG categories: $\Fun_{\cT}(\cR(\cS),\field)\simeq \Fun_{\cT'}(\cR(\cS),\field)$, as desired.
Now, We've finished the proof of Theorem \ref{thm:inv of sh cat}.
\end{proof}

\subsection{Microlocal monodromy}
Given a bordered Legendrian graph $\cT=(T_\Left\rightarrow T\leftarrow T_\Right)$, equipped with a $\ZZ$-valued Maslov potential $\mu$, let $\cS$ be a regular cell complex refining the stratification $\cS_\cT$ induced by $\cT$. Denote by $\Loc(T\setminus V_T)$ the category of local systems of cochain complexes of $\field$-modules on the complement $T\setminus V_T$ of vertices.

As in \cite[Def.5.4]{STZ2017}, we define the \emph{microlocal monodromy}.
\begin{definition}[Microlocal monodromy]\label{def:microlocal monodromy}
There is a natural functor $\muhom:\Sh(T;\field)\rightarrow \Loc(T\setminus V_T)$, called \emph{microlocal monodromy}, such that for each arc $a$ of $\cT$, the region above $a$, denote by $N$, and the arrow $a\rightarrow N$ in $\cS$, we define
\[
\muhom(\cF)(a)\coloneqq \cone(\cF(\Star(a))\rightarrow \cF(\Star(N)))[-\mu(a)].
\]
\end{definition}

\begin{proposition}\label{prop:microlocal monodromy}
Then the microlocal monodromy $\muhom$ is invariant under Legendrian isotopy of $\cT$.
\end{proposition}
\begin{proof}
The proof is entirely the same to that in \cite[\S5.1]{STZ2017}. In another perspective, for any sheaf $F\in\Sh(\cT;\field)$ and any point $p\in \cT\setminus\cV_\cT$ which is a smooth Legendrian point of $\cT$, by Proposition 7.5.3 in \cite{KS1994}, the microlocal monodromy $\muhom(\cF)(p)$ is nothing but the microlocal stalk of $\cF$ at $p$, which is intrinsic.
\end{proof}

\begin{definition}[Subcategory of microlocal rank 1]
We define $\cC_1(T,\mu;\field)$ to be the full DG subcategory of $\Sh(T;\field)$ whose objects are $\cF$ such that $\muhom(\cF)$ is a local system of rank $1$ $\field$-modules in cohomological degree $0$, and define the induced diagram of constructible sheaf categories
\[
\cC_1(\cT,\bfmu;\field)\coloneqq (\cC_1(T_\Left,\mu_\Left;\field)\leftarrow\cC_1(T,\mu;\field)\rightarrow\cC_1(T_\Right,\mu_\Right;\field)).
\]
\end{definition}
This DG category $\cC_1(\cT,\bfmu;\field)$ will be the sheaf side of the augmentation-sheaf correspondence in the next section.
As a consequence of Theorem \ref{thm:inv of sh cat}, we obtain
\begin{corollary}\label{cor:invariance of C_1}
The category $\cC_1(\cT,\bfmu;\field)$ is a Legendrian isotopy invariant up to DG equivalence.
\end{corollary}
\begin{proof}
This follows from Theorem \ref{thm:inv of sh cat} and the fact that the notion of the microlocal monodromy is intrinsic as seen in Proposition \ref{prop:microlocal monodromy}.
\end{proof}

Let $\cS$ be a regular cell complex which refines $\cS_T$ and satisfies Assumption \ref{ass:legible model}. We use the notations in Definitions \ref{def:poset for legible model} and \ref{def:legible model}.
\begin{definition}\label{def:legible model for C_1}
We define $\Fun_{(T,\mu),1}(\cR(\cS),\field)$ to be the full DG subcategory of $\Fun_T(\cR(\cS),\field)_0$ consisting of functors $F$ such that $\cone(F(e_s))[-\mu(s)]\homotopic \field$ for all arcs $s$ contained in $T$. 

By restriction, we then obtain a diagram of DG categories 
\[
\Fun_{(\cT,\bfmu),1}(\cR(\cS),\field)\coloneqq(\Fun_{(T_\Left,\mu_\Left),1}(\cR(\cS|_{T_\Left}),\field)\leftarrow\Fun_{(T,\mu),1}(\cR(\cS),\field)\rightarrow\Fun_{(T_\Right,\mu_\Right),1}(\cR(\cS|_{T_\Right}),\field))
\]
\end{definition}

\begin{corollary}\label{cor:legible model for C_1}
There is an $A_{\infty}$-equivalence:
\begin{eqnarray*}
\Fun_{(\cT,\bfmu),1}(\cR(\cS),\field)\homotopic\cC_1(\cT,\bfmu;\field).
\end{eqnarray*}
\end{corollary}
\begin{proof}
This is a direct corollary of Lemma \ref{lem:comb model} and Proposition \ref{prop:legible model}.
\end{proof}

\section{Augmentations are sheaves for Legendrian graphs}\label{section:augs are sheaves}

\subsection{Local calculation for augmentation categories}
In this section, we will compute the $A_\infty$-structures completely for the trivial bordered Legendrian graphs and the bordered Legendrian graphs containing a vertex of type $(0,n_\Right)$. 

\subsubsection{Augmentation category for a trivial bordered Legendrian graph}
\label{sec:aug var for trivial tangle}
Let $(I_n,\mu)$ be a trivial bordered Legendrian graph of $n$ parallel strands, equipped with a $\ZZ$-valued Maslov potential $\mu$.
We will describe the augmentation category $\Aug_+(I_n;\field)$, which has been already seen in Corollary~\ref{corollary:unitality of border DGAs} for the unitality.

\begin{notation}
In this section, we denote $\Aug_+(I_n;\field)$ by $\Aug_+$ for simplicity.
\end{notation}

As seen in Example~\ref{example:augmentation category for border DGAs}, the Chekanov-Eliashberg DGA $A^\CE\left(I_n^{\p m}\right)$ is generated by the set
\begin{align*}
\left\{k_{ab}^{ij}~\middle|~ a<b, 1\le i,j\le m\right\}\coprod
\left\{y_a^{ij}~\middle|~1\le a\le n, 1\le i<j\le m\right\},
\end{align*}
where the grading is given as
\begin{align*}
|k_{ab}^{ij}|&\coloneqq \mu(a)-\mu(b)-1,&
|y_a^{ij}|&\coloneqq -1.
\end{align*}

\begin{assumption}\label{assumption:indices ranges}
From now on, we denote $y_a^{ij}$ by $k_{aa}^{ij}$.
We regard $k_{ab}^{ij}$ as zero unless it is well-defined.
\end{assumption}

Then under the above assumption, the differential $\differential^{\p m}$ is simply given as
\begin{align*}
\differential^{\p m}k_{ab}^{ij}&=\sum_{\substack{a\leq c\leq b\\1\le k\le m}} (-1)^{|k_{ac}^{ik}|-1}k_{ac}^{ik}k_{cb}^{kj}.
\end{align*}

Recall from Example~\ref{example:augmentation category for border DGAs}. $\Aug_+$ is a DG category such that:
\begin{enumerate}
\item The objects are the augmentations for $A_n(\mu)$:
\begin{equation}
\Ob\left(\Aug_+\right)=\aug(A_n(\mu);\field).
\end{equation}

\item
In $A^{\p2}(I_n)$, we have 
\[
\sfM^{12}\coloneqq \field\left\langle k_{ab}^{12}~\middle|~1\leq a\le b\leq n \right\rangle.
\]
Then, for any two objects $\epsilon_1,\epsilon_2$, the set of morphisms is 
\begin{equation}
\hom_{\Aug_+}(\epsilon_1,\epsilon_2)=\sfM^{12\vee}=\field\left\langle k_{ab}^{12\vee}~\middle|~1\leq a\le b\leq n\right\rangle.
\end{equation}

\item For $\epsilon_1,\epsilon_2\in\Aug_+$, the map 
\[
m_1:\hom_{\Aug_+}(\epsilon_1,\epsilon_2)\to\hom_{\Aug_+}(\epsilon_1,\epsilon_2)
\]
is defined as
\begin{align}\label{eqn:m_1 for trivial tangle}
m_1\left(k_{ab}^{12\vee}\right)&=-\sum_{c<a}\epsilon_1(k_{ca})k_{cb}^{12\vee} +\sum_{b<d}(-1)^{|k_{ab}^{12\vee}|}k_{ad}^{12\vee}\epsilon_2(k_{bd})
\end{align}
\item For $\epsilon_1,\epsilon_2,\epsilon_3\in\Aug_+$, the map 
\[
m_2:\hom_{\Aug_+}(\epsilon_2,\epsilon_3)\otimes \hom_{\Aug_+}(\epsilon_1,\epsilon_2)\to \hom_{\Aug_+}(\epsilon_1,\epsilon_3)
\]
is defined as
\begin{align}\label{eqn:m_2 for trivial tangle}
m_2\left(k_{cd}^{12\vee}\otimes k_{ab}^{12\vee}\right)&=\delta_{bc}(-1)^{\sigma_2}(-1)^{|k_{ab}^{12\vee}|}k_{ad}^{12\vee}
\nonumber\\
&=\delta_{bc}(-1)^{|k_{ab}^{12\vee}||k_{cd}^{12\vee}|+1}k_{ad}^{12\vee},
\end{align}
where 
\[
\sigma_2=1+|k_{ab}^{12\vee}||k_{cd}^{12\vee}|+|k_{ab}^{12\vee}|.
\]
\end{enumerate}

\begin{definition}[Morse complex]\label{def:Morse complex}
Consider a free graded $\field$-module $C=C(I_n)$, 
\begin{align*}
C&\coloneqq\bigoplus_{1\le a\le n}\field e_a,&
|e_a|&\coloneqq-\mu(a).
\end{align*}
equipped with a decreasing filtration $F^{\bullet}$ via 
\[
F^iC \coloneqq \bigoplus_{k>i}\field \cdot e_k,
\]
for $0\leq i\leq n$. 
\begin{itemize}
\item A \emph{Morse complex} for $(C,F^{\bullet})$ is a complex $(C,d)$ with a $\field$-linear filtration-preserving differential $d$ of degree $1$.
\item Let $\MC=\MC(C;\field)$ be the set of Morse complexes of $(C,F^{\bullet})$.
\end{itemize}
\end{definition}

\begin{lemma}[{\cite[Def.4.4]{Su2017}}]\label{lem:aug and Morse complex for trivial tangle}
There is a canonical identification 
\begin{align*}
\Xi:\aug(A_n(\mu);\field) &\stackrel{\isomorphic}{\to} \MC\\
\epsilon &\mapsto d=d(\epsilon),
\end{align*}
where
\[
de_i\coloneqq\sum_{j>i}(-1)^{\mu(i)}\epsilon(a_{ij})e_j.
\]
\end{lemma}

The set of Morse complexes $\MC$ can be lifted to an obvious DG category $\cMC=\cMC(I_n;\field)$: 
\begin{enumerate}
\item The set of objects is $\MC$.
\item The morphism space $\left(\hom_{\cMC}(d_1,d_2),D\right)$ is 
\[
\hom_{\cMC}(d_1,d_2)\coloneqq\End(C,F^\bullet),
\]
the complex of endomorphisms of $(C,F^{\bullet})$ whose differential is given by 
\[
Df\coloneqq d_2\circ f-(-1)^{|f|}f\circ d_1.
\]
\item The composition $\cdot$ is the usual composition of endomorphisms of $C$.
\end{enumerate}

\begin{notation}
We use the identification $\End(C)\isomorphic C\otimes C^*$ and so for $(e\otimes f^*), (e'\otimes f'^*)\in \End(C)$,
\begin{align*}
(e\otimes f^*)(g)&\coloneqq \langle f^*,g\rangle e\in C,\\
(e\otimes f^*)\cdot(e'\otimes f'^*)&\coloneqq \langle f^*,e'\rangle(e\otimes f'^*)\in\End(C).
\end{align*}
\end{notation}

\begin{lemma}[{\cite[Theorem~7.25]{NRSSZ2015}}]\label{lem:aug cat for trivial tangle}
There is a (strict) isomorphism of DG categories
\[
\frh:\Aug_+(I_n;\field)\longrightarrow\cMC
\]
which is given on objects by
\[
\epsilon\mapsto d(\epsilon)\coloneqq \sum_{a<b} \left((-1)^{\mu(a)}\epsilon(k_{ab})\right)(e_b\otimes e_a^*),
\]
as in Lemma \ref{lem:aug and Morse complex for trivial tangle}, and on morphisms 
$\hom_{\Aug_+}(\epsilon_1,\epsilon_2)\rightarrow\hom_{\cMC}(d(\epsilon_1),d(\epsilon_2))$ by 
\begin{align*}
k_{ab}^{12\vee}&\mapsto (-1)^{s(a,b)}(e_b\otimes e_a^*),
\end{align*}
where the sign $s(a,b)$ is given by:
\[
s(a,b)\coloneqq \mu(a)(\mu(b)+1)+1.
\]
\end{lemma} 

\begin{remark}\label{rem:compare sign convention with NRSSZ}
Notice that the sign convention for $s(a,b)$ here is different from the original one, due to the fact that our definition for $d(\epsilon)$ differs by a sign from $\frh(\epsilon)$ in \cite[Theorem~7.25]{NRSSZ2015}. 
\end{remark}

\begin{proof}[Proof of Lemma~\ref{lem:aug cat for trivial tangle}]
For completeness, we give the proof.
By Lemma \ref{lem:aug and Morse complex for trivial tangle}, it suffices to show that $\frh$ commutes with the differential and composition, that is, $\frh\circ m_1=D\circ\frh$, and $\frh\circ m_2=\frh\cdot\frh$. 

Firstly, let us show $\frh\circ m_1=D\circ\frh$. 
For two objects $\epsilon_1,\epsilon_2\in\Aug_+$, let $x\in\hom_+(\epsilon_1,\epsilon_2)$ be a homogeneous element of the form
\[
x=\sum_{1\le a\le b\le n}x_{ab}k_{ab}^{12\vee}.
\]
In particular, $x_{ab}\neq 0$ implies that $|x|=|k_{ab}^{12\vee}|=\mu(a)-\mu(b)$ since $x$ is homogeneous.

By definition of $\frh$, we have
\[
\frh(x)=\sum_{1\le a\le b\le n}(-1)^{s(a,b)}x_{ab}(e_b\otimes e_a^*).
\]

On the other hand, apply the formula \eqref{eqn:m_1 for trivial tangle} for $m_1$, we have
\begin{align}\label{eqn:sign manipulation}
\frh\circ m_1(x)
&=-\sum_{c<a\leq b}\epsilon_1(k_{ca})x_{ab}\frh\left(k_{cb}^{12\vee}\right)+
\sum_{a\leq b<d}(-1)^{|k_{ab}^{12\vee}|}x_{ab}\epsilon_2(k_{bd})\frh\left(k_{ad}^{12\vee}\right)\nonumber\\ 
&=-\sum_{c<a\leq b}(-1)^{s(c,b)}\epsilon_1(k_{ca})x_{ab}(e_b\otimes e_c^*)\\
&\mathrel{\hphantom{=}}+\sum_{a\leq b<d}(-1)^{\mu(a)-\mu(b)+s(a,d)}x_{ab}\epsilon_2(k_{bd})(e_c\otimes e_a^*)\nonumber\\
&=-\sum_{c<a\leq b}(-1)^{|x|}\left((-1)^{\mu(c)}\epsilon_1(k_{ca})\right)\left((-1)^{s(a,b)}x_{ab}\right)(e_b\otimes e_a^*)\cdot(e_a\otimes e_c^*)\nonumber\\
&\mathrel{\hphantom{=}}+\sum_{a\leq b<d}\left((-1)^{s(a,b)}x_{ab}\right)\left((-1)^{\mu(b)}\epsilon_2(k_{bd})\right)(e_d\otimes e_b^*)\cdot(e_b\otimes e_a^*)\nonumber\\
&=-(-1)^{|x|}\frh(x)\circ d(\epsilon_1)+d(\epsilon_2)\circ\frh(x)\nonumber\\
&=D\circ\frh(x)
\end{align}
as desired. Here, the fourth equality follows directly from the formulas for $\frh(x)$ and $d(\epsilon_i)$ above. 

The only nontrivial part is the third equality involving two sign manipulations, which we check as follows: 
for the first sign manipulation, it suffices to show that, if $\epsilon_1(k_{ca})x_{ab}\neq 0$, then 
\[
s(c,b)\equiv |x|+\mu(c)+s(ab)\mod{2}.
\]
In fact, the condition implies that 
\begin{align*}
|k_{ca}|&=\mu(c)-\mu(a)-1=0,&
|x|&=|k_{ab}^{12\vee}|=\mu(a)-\mu(b)=\mu(c)-1-\mu(b)
\end{align*}
and therefore
\[
s(c,b)-s(a,b)-\mu(c)-|x|=\mu(b)+1-\mu(c)-(\mu(c)-1-\mu(b))\equiv 0\mod{2}
\]
as desired.

Similarly, for the second one, it suffices to show that, 
\begin{align*}
x_{ab}\epsilon_2(k_{bd})\neq 0\Longrightarrow \mu(a)-\mu(b)+s(a,d)\equiv s(a,b)+\mu(b)\mod{2},
\end{align*}
or equivalently, 
\[
s(a,d)-s(a,b)+\mu(a)\equiv 0\mod{2}.
\]
In fact, the condition implies that 
\begin{align*}
|k_{bd}|&=\mu(b)-\mu(d)-1=0,&
|x|&=|k_{ab}^{12\vee}|=\mu(a)-\mu(b)=\mu(a)-\mu(d)-1,
\end{align*}
and therefore we have the desired identity
\[
s(a,d)-s(a,b)+\mu(a)=-\mu(a)+\mu(a)\equiv 0\mod{2}.
\]

It remains to show $\frh\circ m_2=\frh\cdot\frh$. For any objects $\epsilon_1,\epsilon_2,\epsilon_3$ in $\Aug_+(T_n;\field)$, let 
\begin{align*}
x&\coloneqq\sum_{a\leq b}x_{ab}k_{ab}^{12\vee}\in\hom_+(\epsilon_2,\epsilon_3),&
y&\coloneqq\sum_{a\leq b}y_{ab}k_{ab}^{12\vee}\in\hom_+(\epsilon_1,\epsilon_2)
\end{align*}
be two homogeneous elements.
We want to show $\frh\circ m_2(x,y)=\frh(x)\cdot\frh(y)$. By definition of $\frh$, we have
\begin{align*}
\frh(x)\cdot\frh(y)
&=\sum_{a\leq c\leq b}\left((-1)^{s(c,b)}x_{cb}(e_q\otimes e_r^*)\right)\cdot\left((-1)^{s(a,c)}y_{ac}(e_c\otimes e_a^*)\right)\\
&=\sum_{a\leq c\leq b}(-1)^{s(c,b)+s(a,c)}y_{a,c}x_{c,b}(e_b\otimes e_a^*).
\end{align*}

On the other hand, apply the formula (\ref{eqn:m_2 for trivial tangle}), we have:
\begin{align*}
\frh\circ m_2(x,y)
&=\frh\left(\sum_{a\leq c\leq b}x_{cb}y_{ac}m_2\left(k_{cb}^{12\vee}\otimes k_{ac}^{12\vee}\right)\right)\\
&=\sum_{a\leq c\leq b}y_{ac}x_{cb}(-1)^{1+|k_{ac}^{12\vee}||k_{cb}^{12\vee}|}(-1)^{s(a,b)}(e_b\otimes e_a^*)\\
&=\frh(x)\circ\frh(y)
\end{align*}
as desired. Here, the last equality follows from the previous formula for $\frh(x)\circ\frh(y)$, and the following sign manipulation: 
\begin{align*}
&\mathrel{\hphantom{=}}1+|k_{ac}^{12\vee}||k_{cb}^{12\vee}|+s(a,b)\\
&=1+(\mu(a)-\mu(c))(\mu(c)-\mu(b))+\mu(a)(\mu(b)+1)+1\\
&\equiv\mu(a)\mu(c)+\mu(c)\mu(b)+\mu(c)+\mu(a)+2\\
&\equiv s(a,c)+s(c,b)\mod{2}
\end{align*}
This finishes the proof of the lemma.
\end{proof}

\subsubsection{Augmentation category for a vertex}
\label{ex:aug cat for a vertex}
Let $\cV_\front\in\BLT_\front$ be the bordered Legendrian graph in $J^1\bfU$ of type $(n_\Left,n_\Right)$, whose front projection is as the left picture in Figure \ref{fig:vertex_front projection m-copy}. In particular, $\cV_\front$ contains a vertex $v$ with $r$ right half-edges, labelled from top to bottom by $1,2,\ldots,r$. 
Then $n_\Right=n_\Left+r$.

As usual, we label the left (resp. right) ends from top to bottom by $1,2,\ldots, n_\Left$ (resp. $1,2,\ldots,n_\Right$). 
We assume the right half-edges of $v$ connect the right ends $k,k+1,\ldots,k+r-1$ from top to bottom. 
In other words, the right half-edge $p$ is connected to the right end $k+p-1$. For simplicity, we denote $a'\coloneqq a$ for $1\leq a<k$ and $b'\coloneqq b+r$ for $k\leq b\leq n_\Left$. 
We would like to compute the augmentation category
\[
\Aug_+\left(\cV_\front;\cK\right)=\left(\Aug_+(V_{\front,\Left};\field)\leftarrow\Aug_+(V_\front;\field)\rightarrow\Aug_+(V_{\front,\Right};\field)\right).
\]

\begin{figure}[ht]
\subfigure[$\cV$]{\makebox[0.25\textwidth]{$
\begin{tikzpicture}[baseline=-.5ex]
\begin{scope}[xshift=-3cm]
\begin{scope}[thick]
\draw (-1,1.8) node[left] {$\scriptscriptstyle i$} -- +(2,0) node[right] {$\scriptscriptstyle a=a'$};
\draw (-1,-2.2) node[left] {$\scriptscriptstyle j$} -- +(2,0) node[right] {$\scriptscriptstyle b+r=b'$};
\draw (-.5,0.1) to[out=0,in=180] (1,1) node[right] {$\scriptscriptstyle k$};
\draw (-.5,0.1) to[out=0,in=180] (1,.2);
\draw (-.5,0.1) to[out=0,in=180] (1,-.6);
\draw (-.5,0.1) to[out=0,in=180] (1,-1.4) node[right] {$\scriptscriptstyle k+r-1$};
\draw[fill] (-.5,0.1) circle (0.05);
\end{scope}
\draw[red] (-1,-2.8)--(-1,2) (1,-2.8)--(1,2);
\end{scope}
\end{tikzpicture}
$}}
\subfigure[$\cV^{\p2}$]{\makebox[0.25\textwidth]{$
\begin{tikzpicture}[baseline=-.5ex]
\begin{scope}
\begin{scope}[yshift=-.2cm,thick,red]
\draw (-1,1.8) -- +(2,0);
\draw (-1,-2.2) -- +(2,0);
\draw (-.5,0.1) to[out=0,in=180] (1,1);
\draw (-.5,0.1) to[out=0,in=180] (1,.2);
\draw (-.5,0.1) to[out=0,in=180] (1,-.6);
\draw (-.5,0.1) to[out=0,in=180] (1,-1.4);
\draw[fill] (-.5,0.1) circle (0.05);
\end{scope}
\begin{scope}[thick]
\draw (-1,1.8) node[left] {$\scriptscriptstyle i$} -- +(2,0) node[right] {$\scriptscriptstyle a=a'$};
\draw (-1,-2.2) node[left] {$\scriptscriptstyle j$} -- +(2,0) node[right] {$\scriptscriptstyle b+r=b'$};
\draw (-.5,0.1) to[out=0,in=180] (1,1) node[right] {$\scriptscriptstyle k$};
\draw (-.5,0.1) to[out=0,in=180] (1,.2);
\draw (-.5,0.1) to[out=0,in=180] (1,-.6);
\draw (-.5,0.1) to[out=0,in=180] (1,-1.4) node[right] {$\scriptscriptstyle k+r-1$};
\draw[fill] (-.5,0.1) circle (0.05);
\end{scope}
\draw[red] (-1,-2.8)--(-1,2) (1,-2.8)--(1,2);
\end{scope}
\end{tikzpicture}
$}}
\subfigure[$\cV^{\p3}$]{\makebox[0.25\textwidth]{$
\begin{tikzpicture}[baseline=-.5ex]
\begin{scope}[xshift=3cm]
\begin{scope}[yshift=-.4cm,thick,blue]
\draw (-1,1.8) -- +(2,0);
\draw (-1,-2.2) -- +(2,0);
\draw (-.5,-.9) to[out=0,in=180] (1,1);
\draw (-.5,-.9) to[out=0,in=180] (1,.2);
\draw (-.5,-.9) to[out=0,in=180] (1,-.6);
\draw (-.5,-.9) to[out=0,in=180] (1,-1.4);
\draw[fill] (-.5,-.9) circle (0.05);
\end{scope}
\begin{scope}[yshift=-.2cm,thick,red]
\draw (-1,1.8) -- +(2,0);
\draw (-1,-2.2) -- +(2,0);
\draw (-.5,0.1) to[out=0,in=180] (1,1);
\draw (-.5,0.1) to[out=0,in=180] (1,.2);
\draw (-.5,0.1) to[out=0,in=180] (1,-.6);
\draw (-.5,0.1) to[out=0,in=180] (1,-1.4);
\draw[fill] (-.5,0.1) circle (0.05);
\end{scope}
\begin{scope}[thick]
\draw (-1,1.8) node[left] {$\scriptscriptstyle i$} -- +(2,0) node[right] {$\scriptscriptstyle a=a'$};
\draw (-1,-2.2) node[left] {$\scriptscriptstyle j$} -- +(2,0) node[right] {$\scriptscriptstyle b+r=b'$};
\draw (-.5,0.1) to[out=0,in=180] (1,1) node[right] {$\scriptscriptstyle k$};
\draw (-.5,0.1) to[out=0,in=180] (1,.2);
\draw (-.5,0.1) to[out=0,in=180] (1,-.6);
\draw (-.5,0.1) to[out=0,in=180] (1,-1.4) node[right] {$\scriptscriptstyle k+r-1$};
\draw[fill] (-.5,0.1) circle (0.05);
\end{scope}
\draw[red] (-1,-2.8)--(-1,2) (1,-2.8)--(1,2);
\end{scope}\end{tikzpicture}
$}}

\caption{Front projection $m$-copy near a vertex.}
\label{fig:vertex_front projection m-copy}
\end{figure}
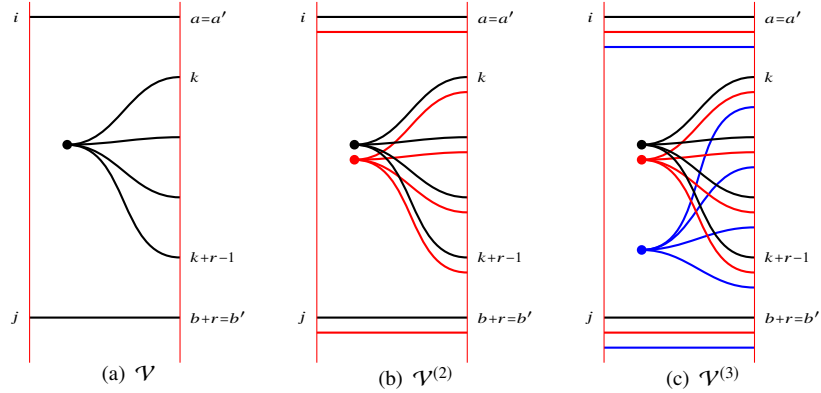

For $m\geq 1$, let $\cV^{\p m}_\front=\left(V^{\p m}_{\front,\Left}\to V^{\p m}_\front \leftarrow V^{\p m}_{\front,\Right}\right)\in\BLT_\front^{\p m}$ be the canonical front projection $m$-copy of $\cV_\front$. 
For examples, $\cV^{\p m}_\front$'s for $m=1,2,3$ are shown as in Figure~\ref{fig:vertex_front projection m-copy} from left to right.
\begin{remark}
Here we are using the convention of the canonical front parallel copies described in Section~\ref{section:canonical copies}.
\end{remark}

\begin{notation}
For simplicity, let us denote 
\begin{align*}
A^{\p m}&\coloneqq A^\CE\left(V_\front^{\p m}\right),&
A^{\p m}_{*}&\coloneqq A^\CE\left(V_{\front,*}^{\p m}\right),\\
\Aug_+&\coloneqq\Aug_+(V_\front;\field),&
\Aug_{+,*}&\coloneqq\Aug_+(V_{\front,*};\field),
\end{align*}
where $*=\Left,\Right$.
\end{notation}

As usual, for each copy $\cV^i_\front$ in $\cV^{\p m}_\front$, label the left ends and right ends from top to bottom by $1,2,\ldots,n_\Left$ and $1,2,\ldots,n_\Right$ respectively, label the vertex by $v^i$ and the right half-edges from top to bottom by $1,2,\ldots,r\in\ZZ/r$. 

In the bordered Legendrian graph $\cV^{\p m}_\front$, label the Reeb chords of $V^{\p m}_{\front,\Left}$ (resp. $V_{\front,\Right}^{\p m}$) corresponding to the pairs of strands by $k_{ab}^{ij}$ (resp. $k'^{ij}_{ab}$). 
That is, for $1\leq a<b\leq n_\Left, 1\leq i,j\leq m$ or $1\leq a=b\leq n_\Left, 1\leq i<j\leq m$ (resp. $1\leq a<b\leq n_\Right, 1\leq i,j\leq m$ or $1\leq a=b\leq n_\Right, 1\leq i<j\leq m$), we obtain a Reeb chord (or line segment) of $V_{\front,\Left}^{\p m}$ (resp. $V_{\front,\Right}^{\p m}$) going from the $b$-th strand of $V^j_{\front,\Left}$ (resp. $V^j_{\front,\Right}$) to the $a$-th strand of $V^i_{\front,\Left}$ (resp. $V^i_{\front,\Right}$).

Label the Reeb chords of $V^i_{\front}$ at the vertex $v^i$ by either $c_{ab}^{ii}$ if $1\le a<b\le r$ or $v_{a,\ell}^{ii}$ otherwise.
i.e. $c_{ab}^{ii}$ or $v_{a,\ell}^{ii}$ is the ``Reeb chord'' starting from the initial half-edge $a+\ell\in\Zmod{r}$, traveling around the vertex $v^i$ counterclockwise and covering $\ell$ minimal sectors, hence ending at the half-edge $a\in\Zmod{r}$ of $v^i$. 
\begin{remark}
This is the same as the construction for Legendrian graphs in a normal form described in Section~\ref{section:augmentation category of normal form} but the reflected manner. In other words, the generators $c_{ab}^{ii}$ correspond to vertex generators lying on the \emph{right side} of $v$.
\begin{align*}
c_{13}^{ii}&=\begin{tikzpicture}[baseline=-.5ex,xscale=-1, yscale=-1]
\draw[thick,fill] (-1,0) node[right] {$2$} -- (0,0) node[above] {$v^i$} circle (2pt);
\draw[thick] (-1,0.5) node[right] {$3$} -- (0,0);
\draw[thick] (-1,-0.5) node[right] {$1$} -- (0,0);
\draw[green,-latex',thick] (180+26.56:0.7) arc (180+26.56:180-26.56:0.7);
\end{tikzpicture}&
v_{2,2}^{ii}&=\begin{tikzpicture}[baseline=-.5ex,xscale=-1, yscale=-1]
\draw[thick,fill] (-1,0) node[right] {$2$} -- (0,0) node[above] {$v^i$} circle (2pt);
\draw[thick] (-1,0.5) node[right] {$3$} -- (0,0);
\draw[thick] (-1,-0.5) node[right] {$1$} -- (0,0);
\draw[green,-latex',thick] (180:0.7) arc (180:-180+26.56:0.7);
\end{tikzpicture}
\end{align*}
\end{remark}

There are additional Reeb chords of $V^{\p m}_\front$ corresponding to the crossings of the right half-edges of the vertices. 
Label the Reeb chord going from the $b$-th right half-edge of $v^j$ in $V^j_{\front}$ to the $a$-th right half-edge of $v^i$ in $V^i_{\front}$ by $c_{ab}^{ij}$ for  $1\leq b<a\leq r$ and $1\leq i<j\leq m$.
See Figure \ref{fig:vertex_front projection m-copy_label} for an example when $r=3, m=4$.

\begin{figure}[!htbp]
\[
\vcenter{\hbox{\def\svgscale{0.7}\input{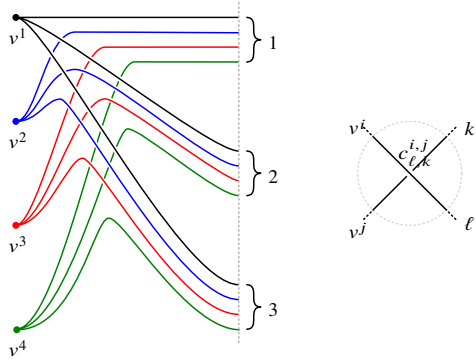}}}
\]
\caption{The crossings in the front projection $m$-copy near a vertex: $r=3, m=4$.}
\label{fig:vertex_front projection m-copy_label}
\end{figure}

We denote the set of Reeb chords by $R^{ij}$ starting at the $j$-th copy $V^j_\front$ and ending at the $i$-th copy $V^i_\front$
\[
R^{ij}\coloneqq \left\{k_{pq}^{ij}, v_{a,\ell}^{ij}, c_{ab}^{ij}\right\}.
\]
Then the LCH DGA $A^{\p m}=(\alg^{\p m},\differential^{\p m})\coloneqq A^\CE(V^{\p m}_\front)$ can be described as follows.
\[
\alg^{\p m}\coloneqq \ZZ\left\langle R^{ij}~\middle|~ 1\le i\le j\le m\right\rangle,
\]
the free associative algebra generated by the Reeb chord in $V^{\p m}_\front$.

The grading for each generator is defined as follows: let $\mu_\Left\coloneqq\mu|_{V_{\front,\Left}}$.
\begin{align*}
|k_{ab}^{ij}|&=\mu_\Left(a)-\mu_\Left(b)-1,\\
|v_{a,\ell}^{ii}|&=|v_{a,\ell}|=\mu(a)-\mu(a+\ell)+N(n,a,\ell)-1,\\
|c_{ab}^{ij}|&=
\begin{cases}
\mu(a)-\mu(b)-1& i=j;\\
\mu(a)-\mu(b) & i<j,
\end{cases}
\end{align*}
where $N(n,a,\ell)$ is the same as described in \eqref{equation:N}.

The differentials of $k_{ab}^{ij}$ and $v_{a,\ell}^{ii}$ are given by the differentials of border DGAs and internal DGAs defined in Example/Definition~\ref{example:composable border DGAs}. Indeed, under Assumption~\ref{assumption:indices ranges}, we have
\begin{align}\label{eqn:differential for m-copy of V_L}
\differential^{\p m}k_{ab}^{ij}&=\sum_{\substack{a\leq c\leq b\\1\le k\le m}} (-1)^{|k_{ac}^{ik}|+1}k_{ac}^{ik}k_{cb}^{kj},\\
\differential^{\p m} v_{a,\ell}^{ii}&=\delta_{\ell,r}+\sum_{\ell_1+\ell_2=\ell}(-1)^{|v_{a,\ell_1}^{ii}|-1}v_{a,\ell_1}^{ii}v_{a+\ell_1,\ell_2}^{ii}.
\end{align}

The differential of $c_{ab}^{ij}$ is obtained by counting certain admissible disks of degree 1 in $(J^1U, V^{\p m}_\front)$.

For simplicity, we denote $\tilde{a}\coloneqq (-1)^{|a|-1}a$ for any Reeb chord $a$. Then for $i<j$,
\begin{align}\label{eqn:differential_front projection m-copy at a vertex}
\differential^{\p m}c_{ab}^{ij}=\sum_{\bfa, \bfi} (-1)^\epsilon c_{\bfa}^{\bfi},
\end{align}
where for sequences $\bfa=(a_1,\dots, a_{\ell+1})$ and $\bfi=(i_1,\dots,i_{\ell+1})$ with $\ell\ge1$
\[
c_{\bfa}^{\bfi}\coloneqq c_{a_1a_2}^{i_1i_2}c_{a_2a_3}^{i_2i_3}\dots c_{a_{\ell}a_{\ell+1}}^{i_{\ell}i_{\ell+1}}
\]
and the summation is over all possible sequences $\bfa$ and $\bfi$ in $[r]$ and $[m]$ such that
\begin{align*}
a_1&=a,&
a_\ell&=b,&
i_1&=i,&
i_\ell&=j,
\end{align*}
and either 
\begin{align}\label{equation:vertex differential of type 1}
a_1&<a_2>\dots>a_\ell,& i_1&=i_2<\dots<i_\ell,&
\epsilon&=0,
\end{align}
or
\begin{align}\label{equation:vertex differential of type 2}
a_1&>a_2<a_3>\dots>a_\ell,& i_1&<i_2=i_3<\dots<i_\ell,&
\epsilon&=|c_{a_1a_2}^{i_1i_2}|-1.
\end{align}

In Figure~\ref{fig:vertex_front projection m-copy_label}, two shaded regions labelled by $A$ and $B$ correspond to differentials of two types in \eqref{equation:vertex differential of type 1} and \eqref{equation:vertex differential of type 2}, respectively.

In fact, we have a bordered DGAs 
\[
\dga^{\p m}=\left(A_\Left^{\p m}\stackrel{\phi_\Left^{\p m}}\longrightarrow A^{\p m}\stackrel{\phi_\Right^{\p m}}\longleftarrow A_\Right^{\p m}\right),
\] 
where $\phi_\Left^{\p m}$ is the inclusion of the sub-DGA $A_\Left^{\p m}$ generated by $\{k_{ab}^{ij}\}$, and 
\[
A_\Right^{\p m}=\ZZ\langle k'^{ij}_{ab}:1\leq p<q\leq n_\Right, 1\leq t,s\leq m \text{ or } 1\leq p=q\leq n_\Right, 1\leq t<s\leq m\rangle
\]
with the differential similar to $\differential_\Left^{\p m}$ as in (\ref{eqn:differential for m-copy of V_L}), that is:
\[
\differential^{\p m}b_{pq}^{ts}=\sum_{p\leq o\leq q, u} (-1)^{|b_{po}^{tu}|+1}b_{po}^{tu}b_{oq}^{us}
\]
where we sum over all $p\leq o\leq q, 1\leq u\leq m$ such that $b_{po}^{tu}$ and $b_{oq}^{us}$ are well-defined. 

Notice that in the case when $m=1$, we have the identification $\iota_\Left^{(1)}=\iota_\Left:\dga_\Left^{(1)}=\dga_\Left\rightarrow \dga^{(1)}=\dga(V)\leftarrow \dga(V_\Right)=\dga_\Right^{(1)}:\iota_\Right=\iota_\Right^{(1)}$, with $a_{pq}\coloneqq a_{pq}^{11}, v_{pq}\coloneqq v_{pq}^{11}$ and $b_{pq}\coloneqq b_{pq}^{11}$.

In addition, the DGA map $\iota_\Right^{\p m}$ is defined via counting certain admissible disks of index 0 in $(I_x\times \RR_z, V^{\p m})$. More precisely, we have
\begin{enumerate}
\item
If $1\leq p\leq q\leq r$ and $1\leq t, s\leq m$ such that $b_{k+p-1,k+q-1}^{ts}$ is well-defined, then:
\begin{align}\label{eqn:iota_R_front projection m-copy at a vertex}
\iota_\Right^{\p m}(b_{k+p-1,k+q-1}^{ts})=\sum v_{p,a-p}^{tt}c_{ab}^{tu}\ldots c_{cq}^{vs}+\sum \tilde{c_{pa}^{tu}}v_{a,b-a}^{uu}c_{bc}^{uv}\ldots c_{dq}^{ws}
\end{align}
where the summation is over all possible such \emph{composable} words so that the summand is well-defined. In particular, this implies that $\iota_\Right^{\p m}(b_{k+p-1,k+q-1}^{ts})=0$ if $t>s$, and if $t=s$, then $\iota_\Right^{\p m}(b_{k+p-1,k+q-1}^{tt})=v_{p,q-p}^{tt}$ with $p<q$.

\item
If $1\leq p\leq q\leq n_\Left, 1\leq t,s\leq m$ such that $b_{p'q'}^{ts}$ is well-defined, in particular, $1\leq p'\leq q'\leq n_\Right$ and $p',q'\notin\{k,k+1,\ldots,k+r-1\}$, then
\begin{align}
\iota_\Right^{\p m}(b_{p'q'}^{ts})=a_{pq}^{ts}
\end{align}

\item
Otherwise, that is, if $1\leq p\leq q\leq n_\Right, 1\leq t, s\leq m$ such that, exactly one of $p,q$'s belongs to $\{k,k+1,\ldots,k+r-1\}$ and $b_{pq}^{ts}$ is well-defined, then 
\begin{align}
\iota_\Right^{\p m}(b_{pq}^{ts})=0
\end{align}
\end{enumerate}
For example, in the case $r=3, m=4$ as in Figure \ref{fig:vertex_front projection m-copy_label}, we have
\[
\iota_\Right^{\p4}(b_{k+1,k+1}^{1,4})=v_{2,1}^{1,1}c_{3,2}^{1,4}+\tilde{c_{2,1}^{1,2}}v_{1,2}^{2,2}c_{3,2}^{2,4}+\tilde{c_{2,1}^{1,3}}v_{1,2}^{3,3}c_{3,2}^{3,4}+\tilde{c_{2,1}^{1,4}}v_{1,1}^{4,4}
\]
Notice that when $m=1$, we have 
\begin{align}\label{eqn:iota_R}
\begin{cases}
\iota_\Right(b_{k+p-1,k+q-1})=v_{p,q-p} & \textrm{if $1\leq p<q\leq r$};\\
\iota_\Right(b_{p'q'})=a_{pq} & \textrm{if $1\leq p<q\leq n_\Left$};\\
\iota_\Right(b_{pq})=0 & \textrm{else}.
\end{cases}
\end{align}

To describe the augmentation category $\Aug_+(V;\field)$, firstly observe that $\dga(V^{\p m})$ is the push-out of $\dga(V_\Left^{\p m})$ and $\cI_v^{\p m}$, where $\cI_v^{\p m}$ is the sub-DGA of $\dga^{\p m}$ generated by $\{v_{pq}^{ii}, c_{pq}^{ts}\}$. 
That is, for $m\geq 1$, we have the following push-out diagram of DGAs:
\[
\begin{tikzcd}
\ZZ\arrow[r]\arrow[d] & \cI_v^{\p m}\arrow[d]\\
\dga(V_\Left^{\p m})\arrow[r,"\iota_\Left"] & \dga^{\p m}
\end{tikzcd}
\]
Also, $\cI_v^{\p m}=\dga(v^{\p m})$ is nothing but the LCH DGA for the front projection $m$-copy of $v$, viewed as a bordered Legendrian graph, obtained from $V$ by removing the extra parallel strands.  
It follows that
\[
\Aug_+(V;\field)\isomorphic\Aug_+(V_\Left)\times\Aug_+(v;\field)
\]
is a strict product of two $A_{\infty}$-categories. By Section~\ref{sec:aug var for trivial tangle}, we've already seen the identification $\frh_\Left:\Aug_+(V_\Left;\field)\isomorphic\cMC(V_\Left;\field)$. It suffices to describe $\Aug_+(v;\field)$. 

Firstly, we introduce a Morse complex for a vertex as follows:
\begin{definition}\cite[Def.4.3.1]{ABS2019count}\label{def:Morse complex at a vertex}
Let $\field[Z]$ be the graded polynomial ring in one variable $Z$ with $|Z|=1$. 
We define a free graded left $\field[Z]$-module 
\begin{align*}
C(v)&\coloneqq \field[Z]\langle e_1,\dots,e_r\rangle,& |e_a|&\coloneqq -\mu(a)
\end{align*}
and a decreasing filtration 
\[
C(v)\supset F^1 C(v)\supset\cdots\supset F^r C(v)\supset Z^2\cdot F^{r+1} C(v) = Z^2\cdot F^1 C(v)
\]
of $C(v)$ by free graded left $\field$-submodules such that for each $a\in\Zmod{r}$,
\begin{align*}
F^aC(v)&\coloneqq \field\langle e_a\rangle\oplus\bigoplus_{i>0} \field\langle Z^{N(v,a,i)}e_{a+i} \rangle.
\end{align*}

We define $\MC(v;\field)$ to be the set of all $\field[Z]$-superlinear endomorphisms $d$ of $C(v)$ of degree $1$ which preserves $F^{\bullet}$ and satisfies $d^2+Z^2=0$.
\end{definition}

\begin{lemma}\cite[Lem.4.3.3]{ABS2019count}\label{lem:aug and Morse complex for a vertex}
There is a canonical identification
\begin{align*}
\Xi_v:\aug(v;\field)&\stackrel{\isomorphic}{\to} \MC(v;\field);\\
\epsilon &\mapsto d=d(\epsilon),
\end{align*}
where
\[
(-1)^{\mu(i)}de_i=\sum_{j>0}\epsilon(v_{i,j})Z^{N(v,i,j)}e_{i+j}.
\]
\end{lemma}

From now on, we will always use the identification above. 

\begin{lemma}\cite[Lem.~4.3.6]{ABS2019count}\label{lem:stratification of aug var at a vertex}
Let $v$ be a vertex as above.
There are decompositions of $\aug(v;\field)$ and $\MC(v;\field)$ over the finite set $\NR(v)$
\begin{align*}
\MC(v;\field)&=\coprod_{\ruling\in \NR(v)}B(v;\field)\cdot d_{\ruling}
\end{align*}
In particular, $\MC(v;\field)=\emptyset$ if $\val(v)$ is odd.
\end{lemma}

Now we can describe $\Aug_+(v;\field)$.

\noindent{}\emph{Objects.} Notice that $\dga(v^{\p1})\isomorphic\cI_v$ with $v_{p,q}=v_{p,q}^{11}$, where $\cI_v$ is the internal DGA of $v$ defined in Example/Definition~\ref{example:composable border DGAs}.
Hence, the set of objects of $\Aug_+(v;\field)$ is:
\[
\Ob~\Aug_+(v;\field)= \aug(v;\field),
\]
the variety of augmentations for $\cI_v$. Notice also that, by Lemma \ref{lem:stratification of aug var at a vertex}, $\aug(v;\field)=\emptyset$ if $r$ is odd. So, from now on, we will assume $r$ is even.

\noindent{}\emph{Morphisms.} As in Section~\ref{sec:aug var for trivial tangle}, define $C_{12}$ to be the free $\field$-module generated by $\{c_{pq}^{12}\}$ in $\cI_v^{(2)}$. Then, for any two objects $\epsilon_1,\epsilon_2$, the set of morphisms in $\Aug_+(v;\field)$ is 
\[
\hom_+(\epsilon_1,\epsilon_2)=C_{12}^{\vee}=\field\langle c_{pq}^+, 1\leq q<p\leq r\rangle
\]
as a free $\field$-module, where $C_{12}^{\vee}\coloneqq C_{12}^*[-1]$, and $x^+\coloneqq (x^{12})^{\vee}=(x^{12})^*[-1]$. In particular, $|c_{pq}^+|=|c_{pq}^{1,2}|+1$.

\noindent{}\emph{Compositions $m_K$.} For any $K\geq 1$, and objects $\epsilon_1,\epsilon_2,\ldots, \epsilon_{K+1}$, the composition map
\[
m_K:\hom_+(\epsilon_K,\epsilon_{K+1})\otimes\ldots\otimes\hom_+(\epsilon_2,\epsilon_3)\otimes\hom_+(\epsilon_1,\epsilon_2)\rightarrow \hom_+(\epsilon_1,\epsilon_{K+1})
\]
is defined as before.

\noindent{}\emph{$m_1$.} For $K=1$ and $\epsilon=(\epsilon_1,\epsilon_2)$, as in Section~\ref{sec:aug var for trivial tangle}, we compute
\[
\differential_{\epsilon}^{(2)}c_{pq}^{12}=\sum_{A>p}\phi_{\epsilon}(v_{p,A-p}^{11})c_{Aq}^{12}+\sum_{B<q}\tilde{c_{p,B}^{12}}\phi_{\epsilon}(v_{B,q-B}^{22}) 
\]
where $1\leq q<p\leq r$.
It follows that
\begin{align}\label{eqn:m_1 at a vertex}
m_1(c_{AB}^+)=\sum_{B<p<A}\epsilon_1(v_{p,A-p})c_{pB}^+ + (-1)^{|c_{AB}^+|}\sum_{B<q<A}\epsilon_2(v_{B,q-B})c_{Aq}^+ 
\end{align}
where $1\leq B<A\leq r$.

\noindent{}\emph{$m_K$ for $K\geq 2$.} In general, for $K\geq 2$ and given $\epsilon=(\epsilon_1,\epsilon_2,\ldots,\epsilon_{K+1})$ defining a diagonal augmentation for $\cI_v^{(K+1)}$, we need to compute $m_K(c_{i_K}^+,\ldots,c_{i_2}^+,c_{i_1}^+)$, where $c_{i_j}^+\in\{c_{pq}^+\}$ is a generator for $\hom_+(\epsilon_j,\epsilon_{j+1})$. To do that,  
we need to calculate $\differential_{\epsilon}^{(K+1)}c_{pq}^{1,K+1}$ and look at the the $\field$-coefficient of $c_{i_1}^{12}c_{i_2}^{23}\ldots c_{i_k}^{KK+1}$ in the result. By the formula (\ref{eqn:differential_front projection m-copy at a vertex}) for $\differential^{(K+1)}c_{pq}^{1K+1}$, with $1\leq q<p\leq r$, we have 
\begin{eqnarray*}
\differential_{\epsilon}^{(K+1)}(c_{pq}^{1K+1})&=&\sum\phi_{\epsilon}(v_{p,a_1-p}^{11})c_{a_1a_2}^{12}\ldots c_{a_Kq}^{KK+1}\\
&+&\sum\tilde{c_{pa_1}^{12}}\phi_{\epsilon}(v_{a_1,a_2-a_1}^{22})c_{a_2a_3}^{23}\ldots c_{a_Kq}^{KK+1}+\textrm{else} 
\end{eqnarray*}
in which only the first two terms contribute to $m_K$. In addition, the first summation is over all $1\leq a_1,\ldots,a_K\leq r$ such that $p<a_1>a_2>\ldots>a_K>q$; the second summation is over all $1\leq a_1,\ldots,a_K\leq r$ such that $p>a_1<a_2>a_3>\ldots>a_K>q$.

We can then divide the computation of $m_K(c_{i_K}^+,\ldots,c_{i_2}^+,c_{i_1}^+)$ into 4 cases:
\begin{enumerate}
\item
If $K\geq r+1$, then both of the two summations above are empty. Hence, we have
\begin{equation}
m_K=0, \textrm{ for all } K\geq r+1.
\end{equation}

\item
For $2\leq K\leq r$, and any $1\leq p=a_0, a_1, a_2,\ldots, a_K, q=a_{K+1}\leq r$ such that $a_0>a_1<a_2>a_3>\ldots>a_K>a_{K+1}$ and $a_0>a_{K+1}$,
have
\begin{eqnarray}
m_K(c_{a_Ka_{K+1}}^+,c_{a_{K-1}a_K}^+,\ldots,c_{a_2a_3}^+,c_{a_0a_1}^+)
=(-1)^{\sigma_K+|c_{a_0a_1}^+|}\epsilon_2(v_{a_1,a_2-a_1})c_{a_0a_{K+1}}^+
\end{eqnarray}  
where $\sigma_K=\sigma_K(c_{a_0a_1}^+,c_{a_2a_3}^+,\ldots, c_{a_Ka_{K+1}}^+)$ as in Definition \ref{def:sign for m_k}.

\item
For $2\leq K\leq r-1$, and any $1\leq a_1,a_2,\ldots,q=a_{K+1}\leq r$ such that $a_1>a_2>\ldots>a_{K+1}$, have
\begin{eqnarray}
m_K(c_{a_Ka_{K+1}}^+,\ldots, c_{a_2a_3}^+,c_{a_1a_2}^+)
=(-1)^{\sigma_K}\sum_{a_{K+1}<p<a_1}\epsilon_1(v_{p,a_1-p})c_{pa_{K+1}}^+
\end{eqnarray}
where $\sigma_K=\sigma_K(c_{a_1a_2}^+,c_{a_2a_3}^+,\ldots,c_{a_Ka_{K+1}}^+)$ as in Definition \ref{def:sign for m_k}.

\item
Else, $2\leq K\leq r$, and $(c_{i_1}^+,c_{i_2}^+,\ldots,c_{i_K}^+)$ is not of the form $(c_{a_0a_1}^+,c_{a_2a_3}^+,\ldots, c_{a_Ka_{K+1}}^+)$ in (2) or $(c_{a_1a_2}^+,c_{a_2a_3}^+,\ldots,c_{a_Ka_{K+1}}^+)$ in (3) above, then 
\begin{equation}
m_K(c_{i_K}^+,\ldots,c_{i_2}^+,c_{i_1}^+)=0.
\end{equation}
\end{enumerate}
This gives the description of $\Aug_+(v;\field)$.

Now we want to compute the whole diagram of augmentation categories: 
\begin{equation}
\Aug_+(\cV;\field)\coloneqq (\Aug_+(V_\Left;\field)\xleftarrow[]{\fr_\Left}\Aug_+(V;\field)\xrightarrow[]{\fr_\Right}\Aug_+(V_\Right;\field))
\end{equation}
The functor $\fr_\Left$ is clear, under the identification $\Aug_+(V;\field)\isomorphic\Aug_+(V_\Left;\field)\times\Aug_+(v;\field)$, it is the obvious restriction functor $\Aug_+(V_\Left;\field)\times\Aug_+(v;\field)\rightarrow \Aug_+(V_\Left;\field)$. Recall that by Lemma \ref{lem:aug cat for trivial tangle}, we have the identification $\frh_\Right:\Aug_+(V_\Right;\field)\isomorphic\cMC(V_\Right;\field)$. 
It suffices to describe $\fr_\Right$.

\noindent{}\emph{On objects.} 
Given any object $\epsilon$ in $\Aug_+(V;\field)$, denote $\epsilon_\Left\coloneqq \epsilon\circ \iota_\Left$, $\epsilon_\Right\coloneqq \epsilon\circ\iota_\Right$, and $\epsilon_v\coloneqq \epsilon\circ\iota_v$, where $\iota_v:\cI_v\hookrightarrow \dga(V)$ is the natural inclusion of DGAs. Then by definition, $\fr_\Right(\epsilon)=\epsilon_\Right$. 
We will use the identification $\Aug_+(V;\field)\isomorphic\Aug_+(V_\Left)\times\Aug_+(v;\field)$ and $\frh_\Right:\Aug_+(V_\Right;\field)\isomorphic\cMC(V_\Right;\field)$. Then $\epsilon=(d_\Left,d_v)$, where $(C(V_\Left),d_\Left\coloneqq d(\epsilon_\Left))$ and $(C(v),d_v\coloneqq d(\epsilon_v))$ are Morse complexes defined as in Lemma \ref{lem:aug and Morse complex for trivial tangle} and Lemma \ref{lem:aug and Morse complex for a vertex} respectively. Also, we have identification $\epsilon_\Right= d_\Right$, i.e.  $(C(V_\Right),d(\epsilon_\Right))$.
 
In fact,  $\fr_\Right$ is a quotient of cochain complexes $\fr_\Right:(C(V_\Left),d_\Left)\oplus (C(v),d_v)\twoheadrightarrow (C(V_\Left),d_\Left)\oplus (C(v_\Right),d_{v,R})\isomorphic (C(V_\Right),d_\Right)$ as follows. By definition, $C(V_\Left)=\oplus_{p=1}^{n_\Left}\field e_\Left^p$ with $e_\Left^p$ the basis element corresponding to the strand $p$ of $V_\Left$, $C(v)=\oplus_{p=1}^{r}\field[Z]e_v^p$ with $e_v^p$ corresponding to the right half-edge $p$ of $v$, $C(V_\Right)=\oplus_{p=1}^{n_\Right}\field e_\Right^p$ with $e_\Right^p$ corresponding to the strand $p$ of $V_\Right$.
Now, \emph{define} $\varphi_V:C(T_\Left)\oplus C(v)\twoheadrightarrow C(V_\Right)$ by $\varphi_V(e_p^L)=e_\Right^{p'}$, $\varphi_V(e_v^p)=e_\Right^{k+p-1}$, and $\varphi_V(Z^ne_v^p)=0$ for $n>0$. Then by the formula (\ref{eqn:iota_R}) for $\iota_\Right$, $\fr_\Right=\varphi_V:(C(V_\Left),d_\Left)\oplus (C(v),d_v)\twoheadrightarrow (C(V_\Left),d_\Left)\oplus (C(v_\Right),d_{v,R})\isomorphic (C(V_\Right),d_\Right)$ is indeed a quotient of complexes, where $(C(v_\Right),d_{v,R})$ is the subcomplex of $(C(V_\Right),d_\Right)$ with $C(v_\Right)=\oplus_{1\leq p\leq r}\field e_\Right^{k+p-1}$. 
Observe that $\varphi_V$ in fact defines naturally a DG functor $\cMC(V_\Left;\field)\times\cMC(v;\field)\rightarrow\cMC(V_\Left;\field)\times\cMC(v_\Right;\field)\hookrightarrow \cMC(V_\Right;\field)$ of DG categories.

\noindent{}\emph{On morphisms.} By definition, the $A_{\infty}$-functor $\fr_\Right$ on morphisms are given by a collection of assignments $\{\fr_{R,K}=(\fr_\Right)_K\}_{K\geq 1}$. That is, 
for any $K\geq 1$ and $K+1$ objects $\epsilon_1,\epsilon_2,\ldots,\epsilon_{K+1}$ in $\Aug_+(V;\field)$, we have a map
\[
\fr_{R,K}:\hom_+(\epsilon_K,\epsilon_{K+1})\otimes\ldots\hom_+(\epsilon_2,\epsilon_3)\otimes\hom_+(\epsilon_1,\epsilon_2)\rightarrow \hom_+(\epsilon_{1,R},\epsilon_{K+1,R})
\]
of degree $1-K$, where $\epsilon_{i,R}\coloneqq \fr_\Right(\epsilon_i)$. It suffices to determine $\fr_{R,K}(c_{i_K}^+,\ldots,c_{i_2}^+,c_{i_1}^+)$ for any collection of generators, where $c_{i_j}^+\in\{a_{pq}^+,c_{pq}^+\}$ is a generator for $\hom_+(\epsilon_j,\epsilon_{j+1})=\field\langle a_{pq}^+,c_{pq}^+\rangle$ as a free $\field$-module.

As in Section~\ref{sec:aug var for trivial tangle}, $\epsilon\coloneqq (\epsilon_1,\epsilon_2,\ldots,\epsilon_K)$ (resp. $\epsilon_\Right\coloneqq (\epsilon_{1,R},\epsilon_{2,R},\ldots,\epsilon_{K+1,R})$) defines a diagonal augmentation for $\dga^{(K+1)}$ (resp. $\dga_\Right^{(K+1)}$). \emph{Define} $\iota_{R,\epsilon}^{(K+1)}\coloneqq \phi_{\epsilon}\circ\iota_\Right^{(K+1)}\circ\phi_{\epsilon_\Right}^{-1}:\dga_\Right^{(K+1)}\otimes\field\rightarrow \dga^{(K+1)}\otimes\field$. Recall that we have
\begin{equation}\label{eqn:r_R,K}
\fr_{R,K}(c_{i_K}^+,\ldots,c_{i_2}^+,c_{i_1}^+)=(-1)^{\sigma_K}\sum b_{pq}^+\langle\iota_{R,\epsilon}^{(K+1)}b_{pq}^{1,K+1},c_{i_1}^{1,2}c_{i_2}^{2,3}\ldots c_{i_K}^{K,K+1} \rangle
\end{equation}
where $\langle\iota_{R,\epsilon}^{(K+1)}b_{pq}^{1,K+1},c_{i_1}^{1,2}c_{i_2}^{2,3}\ldots c_{i_K}^{K,K+1} \rangle$ denotes the $\field$-coefficient of $c_{i_1}^{1,2}c_{i_2}^{2,3}\ldots c_{i_K}^{K,K+1}$ in $\iota_{R,\epsilon}^{(K+1)}b_{pq}^+$, and $\sigma_K\coloneqq \sigma_K(c_{i_1}^+,c_{i_2}^+,\ldots,c_{i_K}^+)$ as in Definition \ref{def:sign for m_k}. Observe that, by the formula for $\iota_\Right^{(K+1)}$, the $A_{\infty}$-functor $\fr_\Right$ is simply the composition 
\[
\fr_\Right:\Aug_+(V_\Left;\field)\times\Aug_+(v;\field)\xrightarrow[]{\identity\times\fr_{v,R}}\Aug_+(V_\Left;\field)\times\Aug_+(v_\Right;\field)\hookrightarrow \Aug_+(V_\Right;\field)
\] 
Here, $\fr_{v,R}:\Aug_+(v;\field)\rightarrow\Aug_+(v_\Right;\field)$ is the right restriction functor for the bordered Legendrian graph $v$, by removing the extra parallel strands of $V$. And, the second arrow in the composition is the natural inclusion of DG categories $\cMC(V_\Left;\field)\times\cMC(v_\Right;\field)\hookrightarrow\cMC(V_\Right;\field)$, under the obvious identification. It suffices to describe $\fr_{v,R}=\fr_\Right|_{\Aug_+(v;\field)}$.

\noindent{}\emph{$(\fr_{v,R})_1$.} 
Let $K=1$ and $\epsilon=(\epsilon_1,\epsilon_2)$. Apply formula (\ref{eqn:iota_R_front projection m-copy at a vertex}), for $1\leq p\leq q\leq r$, we compute:
\[
\iota_{R,\epsilon}^{(2)}(b_{k+p-1,k+q-1}^{1,2})=\sum_{A>q} \phi_{\epsilon}(v_{p,A-p}^{1,1})c_{A,q}^{1,2}+\sum_{B<p}\tilde{c_{p,B}^{1,2}}\phi_{\epsilon}(v_{B,q-B}^{2,2})
\]
Then by (\ref{eqn:r_R,K}), for $1\leq B<A\leq r$, we have
\begin{eqnarray}\label{eqn:r_R,1 at a vertex}
(\fr_{v,R})_1(c_{A,B}^+)&=&\sum_{p\leq B}\epsilon_1(v_{p,A-p})b_{k+p-1,k+B-1}^+\nonumber\\ 
&+&(-1)^{|c_{A,B}^+|}\sum_{q\geq A}\epsilon_2(v_{B,q-B})b_{k+A-1,k+q-1}^+
\end{eqnarray}

\noindent{}\emph{$(\fr_{v,R})_K$ for $K\geq 2$.}
In general, let $K\geq 2$ and $\epsilon=(\epsilon_1,\epsilon_2,\ldots,\epsilon_{K+1})$.
Apply formula (\ref{eqn:iota_R_front projection m-copy at a vertex}), for $1\leq p\leq q\leq r$, we have
\begin{eqnarray*}
\iota_{R,\epsilon}^{(K+1)}(b_{k+p-1,K+q-1}^{1,K+1})
&=&\sum\phi_{\epsilon}(v_{p,a_1-p}^{1,1})c_{a_1a_2}^{1,2}\ldots c_{a_Kq}^{K,K+1}\\
&+&\sum\tilde{c_{pa_1}^{1,2}}\phi_{\epsilon}(v_{a_1,a_2-a_1}^{2,2})c_{a_2a_3}^{2,3}\ldots c_{a_Kq}^{K,K+1}+\textrm{else} 
\end{eqnarray*}
in which only the first two terms contribute to $(\fr_{v,R})_K$. In addition, the first summation is over all $1\leq a_1,\ldots,a_K\leq r$ such that $p<a_1>a_2>\ldots>a_K>q$; the second summation is over all $1\leq a_1,\ldots,a_K\leq r$ such that $p>a_1<a_2>a_3>\ldots>a_K>q$.

Similar to $m_K$, we can divide the computation of $(\fr_{v,R})_K(c_{i_K}^+,\ldots,c_{i_2}^+,c_{i_1}^+)$ into 4 cases:
\begin{enumerate}
\item
If $K\geq r+1$, then both of the two summations above are empty. Hence, we have
\begin{equation}
(\fr_{v,R})_K=0, \textrm{ for all } K\geq r+1.
\end{equation}

\item
For $2\leq K\leq r$, and any $1\leq p=a_0, a_1, a_2,\ldots, a_K, q=a_{K+1}\leq r$ such that $a_0>a_1<a_2>a_3>\ldots>a_K>a_{K+1}$ and $a_0\leq a_{K+1}$,
have
\begin{equation}
(\fr_{v,R})_K(c_{a_Ka_{K+1}}^+,\ldots,c_{a_2a_3}^+,c_{a_0a_1}^+)=(-1)^{\sigma_K+|c_{a_0a_1}^+|}\epsilon_2(v_{a_1,a_2-a_1})b_{k+a_0-1,k+a_{K+1}-1}^+
\end{equation}  
where $\sigma_K=\sigma_K(c_{a_0a_1}^+,c_{a_2a_3}^+,\ldots, c_{a_Ka_{K+1}}^+)$ as in Definition \ref{def:sign for m_k}.

\item
For $2\leq K\leq r-1$, and any $1\leq a_1,a_2,\ldots,q=a_{K+1}\leq r$ such that $a_1>a_2>\ldots>a_{K+1}$, have
\begin{equation}
(\fr_{v,R})_K(c_{a_Ka_{K+1}}^+,\ldots,c_{a_1a_2}^+)=(-1)^{\sigma_K}\sum_{p\leq a_{K+1}}\epsilon_1(v_{p,a_1-p})b_{k+p-1,k+a_{K+1}-1}^+
\end{equation}
where $\sigma_K=\sigma_K(c_{a_1a_2}^+,c_{a_2a_3}^+,\ldots,c_{a_Ka_{K+1}}^+)$ as in Definition \ref{def:sign for m_k}.

\item
Else, $2\leq K\leq r$, and $(c_{i_1}^+,c_{i_2}^+,\ldots,c_{i_K}^+)$ is not of the form $(c_{a_0a_1}^+,c_{a_2a_3}^+,\ldots, c_{a_Ka_{K+1}}^+)$ in (2) or $(c_{a_1a_2}^+,c_{a_2a_3}^+,\ldots,c_{a_Ka_{K+1}}^+)$ in (3) above, then 
\begin{equation}
(\fr_{v,R})_K(c_{i_K}^+,\ldots,c_{i_2}^+,c_{i_1}^+)=0.
\end{equation}
\end{enumerate}
This finishes the description of $\fr_{v,R}:\Aug_+(v;\field)\rightarrow \Aug_+(v_\Right;\field)$, hence that of $\fr_\Right$.

Furthermore, the following gives a simpler description of $\Aug_+(\cV;\field)$:
\begin{lemma}\label{lem:aug cat for a vertex} 
The $A_{\infty}$-functor
\[
\fr_{v,R}:\Aug_+(v;\field)\to\Aug_+(v_\Right;\field).
\]
is an $A_{\infty}$-equivalence. As a consequence, we obtain a commutative diagram of $A_{\infty}$-categories:
\[
\begin{tikzcd}[column sep=4pc, row sep=2pc]
\Aug_+(V_\Left;\field)\arrow[equal, d]& \Aug_+(V;\field)\arrow[l,"\fr_\Left"']\arrow[r,"\fr_\Right"]\arrow[d,"\simeq"',"\identity\times \fr_{v,R}"] & \Aug_+(V_\Right;\field)\arrow[equal,d]\\
\Aug_+(V_\Left;\field)\arrow[d,"\isomorphic"] & \Aug_+(V_\Left;\field)\times\Aug_+(v_\Right;\field)\arrow[l,"\fp_1"']\arrow[hook,r,"\fri"]\arrow[d,"\isomorphic"] & \Aug_+(V_\Right;\field)\arrow[d,"\isomorphic"]\\
\cMC(V_\Left;\field) & \cMC(V_\Left;\field)\times\cMC(v_\Right;\field)\arrow[l,"\fp_1"']\arrow[hook,r,"\fri"] &\cMC(V_\Right;\field)
\end{tikzcd}
\]
which induces equivalences between all the three rows, with each row viewed as a diagram of $A_{\infty}$-categories. 
Here, $\fp_1$'s are the projections to the first factors, and $\fri$'s are the natural inclusions, and the functor $\identity\times\fr_\Right$ is defined via 
\[
\Aug_+(V;\field)\isomorphic\Aug_+(V_\Left;\field)\times\Aug_+(v;\field)\xrightarrow[]{\identity\times\fr_{v,R}} \Aug_+(V_\Left;\field)\times\Aug_+(v_\Right;\field).
\]
\end{lemma}
In particular, the result implies that $\Aug_+(\cV;\field)$ is equivalent to the third row of Morse complex categories, which involves only DG categories and DG functors. This will be one key step in showing our main result ``augmentations are sheave''.

\begin{proof}[An algebraic proof of Lemma \ref{lem:aug cat for a vertex}]
Notice that the equivalence between the second and third row is just a direct consequence of Lemma \ref{lem:aug cat for trivial tangle}. The only nontrivial part is to show $\fr_{v,R}$ is an $A_{\infty}$-equivalence. For simplicity, we can now \emph{assume} $V=v$, that is, $n_\Left=0$, $k=1$, and $\fr_{v,R}=\fr_\Right$.

Remember that on objects $\fr_\Right$ is just the map of augmentation varieties $r_\Right:\Aug(V;\field)\rightarrow\Aug(V_\Right;\field)$, which is surjective. In particular, $\fr_\Right$ is essentially surjective. It suffices to show that $\fr_\Right$ is fully faithful, i.e. for any two objects $\epsilon_1,\epsilon_2$ in $\Aug_+(V:\field)$, the co-chain map 
\begin{equation}
\fr_{R,1}=(\fr_\Right)_1:\hom_V(\epsilon_1,\epsilon_2)\rightarrow\hom_{V_\Right}(\epsilon_{1,R},\epsilon_{2,R})
\end{equation}
is a quasi-isomorphism.
Here $\hom_T(-,-)\coloneqq \hom_{\Aug_+(T;\field)}(-,-)$ for any bordered Legendrian graph $T$. 
Equivalently, it suffices to show that the cone of $-\fr_{R,1}$ is acyclic. 
By definition, the cone of $-\fr_{R,1}$ is the cochain complex
\[
\cone(-\fr_{R,1})\coloneqq (\hom_{V_\Right}(\epsilon_{1,R},\epsilon_{2,R})\oplus\hom_{V}(\epsilon_1,\epsilon_2)[1],\delta)
\]
with differential 
\[
\delta\coloneqq 
\left(\begin{matrix}
m_{1,R} & -\fr_{R,1}\\
0 & -m_1
\end{matrix}\right)
\]

Let $\End^{*}(C(V_\Right))$ be the free $\ZZ$-graded $\field$-module of $\field$-linear endomorphisms of $C(V_\Right)=\oplus_{1\leq p\leq r}\field\cdot e_p$ without filtration. For $i=1,2$, remember that $d(\epsilon_{i,R})$ is the image of $\epsilon_{i,R}$ under the DG equivalence $\frh_\Right:\Aug_+(V_\Right;\field)\isomorphic \cMC(V_\Right;\field)$ in Lemma \ref{lem:aug cat for trivial tangle} for $V_\Right$.
Then \emph{define} a cochain complex $(\cE,D)\coloneqq (\End^{*}(C(V_\Right)),D)$ with differential: 
\[
D(f)\coloneqq d(\epsilon_{2,R})\circ f-(-1)^{|f|}f\circ d(\epsilon_{1,R}).
\]

Notice that $\hom_V(\epsilon_1,\epsilon_2)[1]=\oplus_{1\leq q<p\leq r}\field\cdot c_{p,q}^+[1]$, and $\hom_{V_\Right}(\epsilon_{1,R},\epsilon_{2,R})=\oplus_{1\leq p\leq q\leq r}\field\cdot b_{p,q}^+$. In particular, $|c_{p,q}^+[1]|=|c_{p,q}^+|-1=\mu(p)-\mu(q)$. \emph{Define} a $\field$-linear map $\alpha:\cone(-\fr_{R,1})\rightarrow \cE$ with:
\begin{align*}
\alpha(b_{p,q}^+)&\coloneqq (-1)^{s(p,q)}|q\rangle\langle p|, & 1\leq p&\leq q\leq r;\\
\alpha(c_{p,q}^+[1])&\coloneqq (-1)^{s(p,q)}|q\rangle\langle p|, & 1\leq q&<p\leq r.
\end{align*}
where $s(p,q)=\mu(p)(\mu(q)+1)+1$ and $|q\rangle\langle p|$ are defined as in Lemma \ref{lem:aug cat for trivial tangle}. \newline

\noindent{}\emph{Claim 1:} The map $\alpha:\cone(-\fr_{R,1})\rightarrow \cE$ is an isomorphism of cochain complexes. 
\begin{proof}[Proof of Claim 1]
Clearly, $\alpha$ is a $\field$-linear isomorphism of degree 0. It suffices to show that $\alpha$ is a cochain map.
Firstly observe that, $\alpha|_{\hom_{V_\Right}(\epsilon_{1,R},\epsilon_{2,R})}$ is just the composition 
\[
\hom_{V_\Right}(\epsilon_{1,R},\epsilon_{2,R})\xrightarrow[]{(\frh_\Right)_1}\hom_{\cMC(V_\Right;\field)}(d(\epsilon_{1,R}),d(\epsilon_{2,R}))\hookrightarrow \cE,
\]
where $(\frh_{R})_1$ is induced from the DG equivalence $\frh_{R}:\Aug_+(V_\Right;\field)\isomorphic \cMC(V_\Right;\field)$, and the second arrow is the inclusion of complexes. Hence, it follows by Lemma \ref{lem:aug cat for trivial tangle} that $\alpha|_{\hom_{V_\Right}(\epsilon_{1,R},\epsilon_{2,R})}$
 is a cochain map. It remains to show that, for any homogeneous element $x\in\hom_V(\epsilon_1,\epsilon_2)[1]$, we have $\alpha\circ\delta(x)=D\circ\alpha (x)$.

Say, $x=\sum_{1\leq B<A\leq r}x_{A,B}c_{A,B}^+[1]$, and denote $x^+\coloneqq x[-1]$. In particular, $x_{A,B}\neq 0$ implies that $|x|=|c_{A,B}^+[1]|=\mu(A)-\mu(B)$. Firstly, we have
\begin{align*}
\alpha(x)&=\sum_{B<A}(-1)^{s(A,B)}x_{A,B}|B\rangle\langle A|\\
d(\epsilon_{1,R})&=\sum_{p<A}(-1)^{\mu(p)}\epsilon_{1,R}(b_{p,A})|A\rangle\langle p|\\
d(\epsilon_{2,R})&=\sum_{B<q}(-1)^{\mu(B)}\epsilon_{2,R}(b_{B,q})|q\rangle\langle B|
\end{align*}

On the other hand, by definition, $\delta(x)=-\fr_{R,1}(x^+)-m_1(x^+)[1]$.
Apply formula (\ref{eqn:m_1 at a vertex}), we obtain:
\begin{align*}
&\alpha(m_1(x^+)[1])\\
=&\sum_{B<p<A}\epsilon_1(v_{p,A-p})x_{A,B}\alpha(c_{p,B}^+[1]) +\sum_{B<q<A}(-1)^{|c_{A,B}^+|}x_{A,B}\epsilon_2(v_{B,q-B})\alpha(c_{A,q}^+[1])\\
=&\sum_{q<p<A}\epsilon_{1,R}(b_{p,A})x_{A,q}\alpha(c_{p,q}^+[1]) +\sum_{B<q<p}(-1)^{|c_{p,B}^+|}x_{p,B}\epsilon_{2,R}(b_{B,q})\alpha(c_{p,q}^+[1])\\
=&\sum_{q<p<A}(-1)^{s(p,q)}\epsilon_{1,R}(b_{p,A})x_{A,q}|q\rangle\langle p| +\sum_{B<q<p}(-1)^{|c_{p,B}^+|+s(p,q)}x_{p,B}\epsilon_{2,R}(b_{B,q})|q\rangle\langle p|
\end{align*}
Similarly, by formula (\ref{eqn:r_R,1 at a vertex}), we have:
\begin{align*}
&\alpha(\fr_{R,1}(x^+))\\
=&\sum_{p\leq B<A}\epsilon_1(v_{p,A-p})x_{A,B}\alpha(b_{p,B}^+) +\sum_{B<A\leq q}(-1)^{|c_{A,B}^+|}x_{A,B}\epsilon_2(v_{B,q-B})\alpha(b_{A,q}^+)\\
=&\sum_{p\leq q<A}\epsilon_{1,R}(b_{p,A})x_{A,q}\alpha(b_{p,q}^+) +\sum_{B<p\leq q}(-1)^{|c_{p,B}^+|}x_{p,B}\epsilon_{2,R}(b_{B,q})\alpha(b_{p,q}^+)\\
=&\sum_{p\leq q<A}(-1)^{s(p,q)}\epsilon_{1,R}(b_{p,A})x_{A,q}|q\rangle\langle p|+\sum_{B<p\leq q}(-1)^{|c_{p,B}^+|+s(p,q)}x_{p,B}\epsilon_{2,R}(b_{B,q})|q\rangle\langle p|
\end{align*}
Combine the two computations above, we then obtain:
\begin{align*}
&\alpha\circ\delta(x)\\
=&-\sum_{p,q<A}(-1)^{s(p,q)}\epsilon_{1,R}(b_{p,A})x_{A,q}|q\rangle\langle p| -\sum_{B<p,q}(-1)^{|c_{p,B}^+|+s(p,q)}x_{p,B}\epsilon_{2,R}(b_{B,q})|q\rangle\langle p|\\
=&-\sum_{p,q<A}(-1)^{s(p,q)}\epsilon_{1,R}(b_{p,A})x_{A,q}|q\rangle\langle p| +\sum_{B<p,q}(-1)^{\mu(p)-\mu(B)+s(p,q)}x_{p,B}\epsilon_{2,R}(b_{B,q})|q\rangle\langle p|\\
=&-\sum_{p,q<A}(-1)^{|x|}((-1)^{\mu(p)}\epsilon_{1,R}(b_{p,A}))((-1)^{s(A,q)}x_{A,q})|q\rangle\langle A|\circ|A\rangle\langle p|\\ 
\mathrel{\hphantom{=}}&+\sum_{B<p,q}((-1)^{s(p,B)}x_{p,B})((-1)^{\mu(B)}\epsilon_{2,R}(b_{B,q}))|q\rangle\langle B|\circ|B\rangle\langle p|\\
=&-(-1)^{|x|}\alpha(x)\circ d(\epsilon_{1,R})+d(\epsilon_{2,R})\circ\alpha(x)\\
=&D\circ\alpha(x)
\end{align*}
as desired.
In the computation above, the third equality follows directly from the formulas for $\alpha(x)$ and $d(\epsilon_{i,R})$ above. The only nontrivial part is the second equality involving two sign manipulations, which are of the same form as those in (\ref{eqn:sign manipulation}), hence follow by exactly the same argument.
This finishes the proof of Claim 1.
\end{proof}

Now, by Claim 1, it suffices to show that $(\cE,D)$ is acyclic, which is purely a homological algebra problem. In fact, we can prove a stronger statement:\newline

\noindent{}\emph{Claim 2:} For any free $\ZZ$-graded $\field$-module $C\coloneqq \oplus_{p=1}^r\field\cdot e_p$, let $(C,d_1),(C,d_2)\in\MC(C;\field)$ be two Morse complexes as in Definition \ref{def:Morse complex}, such that either $d_1$ or $d_2$ is acyclic.
Let $(\cE,D)\coloneqq (\End^{\bullet}(C),D)$ be the complex of $\field$-linear endomorphisms of $C$, with differential given by $D(f)\coloneqq d_2\circ f-(-1)^{|f|}f\circ d_1$. Then $(\cE,D)$ is chain homotopic to zero.

To prove the claim, we firstly make some preparation. 

\begin{definition}
Define $\GNR(C)$ to be the set of all involutions $\rho$ {\em possibly with fixed points} on $I(C)\coloneqq\{1,2,\ldots,n\}$ such that
\[
|e_i|=|e_j|-1, \qquad\text{ if }\rho(i)=j \text{ and }i<j.
\]
\end{definition}

Let $B(C;\field)$ be the automorphism group of $(C, F^{\bullet})$ in Definition \ref{def:Morse complex}. Then $B(C;\field)$ acts on $\MC(C;\field)$ via conjugation:  For any $g\in B(C;\field)$ and $d\in\MC(C;\field)$,
\[
g\cdot d=g\circ d\circ g^{-1}.
\]

\begin{definition}[Canonical differentials]\label{def:canonical differential}
For each involution $\rho\in \GNR(C)$, define a \emph{canonical differential} $d_{\rho}\in\MC(C;\field)$ as follows:
\begin{align*}
(-1)^{\mu(i)}d_{\rho}e_i=
\begin{cases}
e_j &\text{ if }\rho(i)=j, \text{ and } i<j;\\
0&\text{ otherwise.}
\end{cases}
\end{align*}
\end{definition}

Recall that, we have
\begin{lemma}[{\cite[Lem.4.5]{Su2017}}]\label{lem:ruling stratification for trivial tangle}
The group action of $B(C;\field)$ induces a decomposition of $\MC(C;\field)$ into finitely many orbits 
\begin{align*}
\MC(C;\field)&=\coprod_{\rho\in \GNR(C)}B(C;\field)\cdot d_{\rho}.
\end{align*}
\end{lemma}

Now, we can prove Claim 2.

\begin{proof}[Proof of Claim 2]
By Lemma \ref{lem:ruling stratification for trivial tangle}, if $d_i$ is acyclic, then $d_i=g_i\cdot d_{k_i}=g_i\circ d_{k_i}\circ g_i^{-1}$ for some $g_i\in B(C;\field)$ and some involution $\ruling_i\in \NR(C)$. Here, $d_{\ruling_i}$ is the canonical differential associated to $\ruling_i$ as in Definition \ref{def:canonical differential}. By definition, $d_{{\ruling_i}}e_p=e_q$ if $p<q=\ruling_i(p)$, and $d_{\ruling_i}e_p=0$ otherwise. \emph{Define} $\delta_{\ruling_i}\in\cE^{-1}$ by: $\delta_{\ruling_i}e_q=e_p$ if $p<q=\ruling_i(p)$, and $\delta_{\ruling_i}e_q=0$ otherwise. Then clearly we have $\delta_{\ruling_i}^2=0$, and $d_{\ruling_i}\circ\delta_{\ruling_i}+\delta_{\ruling_i}\circ d_{\ruling_i}=id$. \emph{Define} $\delta_i\in\cE^{-1}$ as $\delta_i\coloneqq g_i\cdot \delta_{\ruling_i}=g_i\circ\delta_{\ruling_i}\circ g_i^{-1}$. It follows immediately that $\delta_i^2=0$ and $d_i\circ\delta_i+\delta_i\circ d_i=id$. 

Now, if $d_2$ is acyclic, \emph{define} a map $H_2:\cE\rightarrow \cE$ of degree $-1$ by $H_2(f)\coloneqq \delta_2\circ f$. We then compute:
\begin{align*}
[D,H_2](f)&=D\circ H_2(f)+H_2\circ D(f)\\
&=D(\delta_2\circ f)+\delta_2\circ D(f)\\
&=d_2\circ(\delta_2\circ f)-(-1)^{|f-1|}(\delta_2\circ f)\circ d_1
+\delta_2\circ (d_2\circ f-(-1)^{|f|}f\circ d_1)\\
&=f
\end{align*}
That is, $\identity=[D,H_2]:\cE\rightarrow\cE$ is chain homotopic to zero, as desired. If $d_1$ is acyclic, by a similar argument, one can show that, $\identity=[D,H_1]$ with $H_1(f)=(-1)^{|f|}f\circ\delta_1$, as desired. 
\end{proof}

This finishes the proof of the lemma.
\end{proof}

Now we give an alternative geometric proof of Lemma~\ref{lem:aug cat for a vertex}.

\begin{proof}[A geometric proof of Lemma~\ref{lem:aug cat for a vertex}]
Let us start by introducing an alternative parallel copies, let us say $v'^{\p\bullet}$, of the vertex $v$ of type $(0,r)$ as follows:
\begin{figure}[!htbp]
\[
\vcenter{\hbox{\def\svgscale{0.7}\input{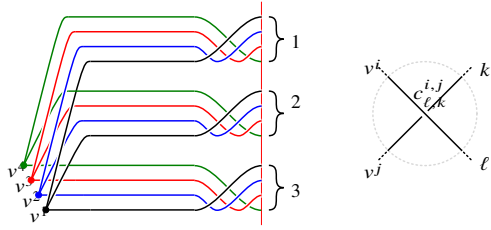}}}
\]
\caption{An alternative parallel $m$-copy near a vertex: $r=3, m=4$.}
\label{fig:vertex_front projection m-copy_alternative}
\end{figure}
Clearly $v'^{\p\bullet}$ is in $\BLT^{\p\bullet}_\lag$, and let us consider its consistent DGA $A^{\p\bullet}(v')$.
Let us use the same labeling convention of the crossing as in Figure~\ref{fig:vertex_front projection m-copy_label}, then $A^{\p\bullet}(v')$ is generated by the following sets of generators:
\begin{align*}
C_1 &\coloneqq \{c^{i,j}_{a,a} \mid 1\leq i < j \leq m,\  1\leq a \leq r\};\\
C_2 &\coloneqq \{c^{i,j}_{a,b} \mid 1\leq i < j \leq m,\  1\leq a < b \leq r\};\\
\widetilde V &\coloneqq \{v^{i}_{a,\ell} \mid 1\leq i \leq m,\  1\leq a \leq r, \ \ell\in \NN \}.
\end{align*}
Note that the generators in $C_1$ are the crossings near the right border while the ones in $C_2$ are crossings near the vertices $v^i$ in Figure~\ref{fig:vertex_front projection m-copy_alternative}.
The grading is given by
\[
|c^{i,j}_{a,b}|=\mu(a)-\mu(b)-1,\qquad |v^i_{a,\ell}|=\mu(a)-\mu(a+\ell)+(N-1),
\]
where $N=N(n,a,\ell)$ is the same as in \eqref{equation:N}.
The differential for the vertex generators $\widetilde V$ is the same as before, see \eqref{equation:differential for vertices} in Example/Definition~\ref{example:composable border DGAs}.
For the generators in $C_1$ and $C_2$, the differential is given by
\begin{align*}
\partial'^{\p m} c^{i,j}_{a,a}&=\sum_{i<k<j}c^{i,k}_{a,a}c^{k,j}_{a,a};\\
\partial'^{\p m} c^{i,j}_{a,b}&=(-1)^{|v^i_{a,b-a}|+1}v^i_{a,b-a}c^{i,j}_{b,b}+\sum_{a\leq c<b}(-1)^{|c^{i,j}_{a,c}|+1} c^{i,j}_{a,c} v^j_{c,b-c}\\
&+\sum_{\substack{i<k<j \\ a<c<b}}(-1)^{|c^{i,k}_{a,c}|+1} c^{i,k}_{a,c} c^{k,j}_{c,b}
\end{align*}

Now describe the corresponding augmentation category $\Aug_+(v')$.
The objects of $\Aug_+(v')$ is the augmentations of $A(v;\field)$:
\[
\Ob( \Aug_+(v') )=\aug(v;\field).
\]
For any two objects $\epsilon_1, \epsilon_2$, the set of morphism becomes
\begin{align*}
\hom_{\Aug_+(v')}(\epsilon_1,\epsilon_2)=\field\left\langle c_{a,b}^{12\vee} \mid  1\leq a\leq b\leq r \right\rangle.
\end{align*}
By dualizing the differential $\partial^{\p m}$, we have the following $A_\infty$ structure $\{m'_k\}_{k\geq 1}$ as follows:
\begin{itemize} 
\item For $\epsilon_1,\epsilon_2\in\aug(v;\field)$, the map 
\[
m'_1:\hom_{\Aug_+(v')}(\epsilon_1,\epsilon_2)\to\hom_{\Aug_+(v')}(\epsilon_1,\epsilon_2)
\]
is defined as
\begin{align}\label{eqn:m_1 in augmentation of v'}
m'_1\left(c_{ab}^{12\vee}\right)&=-\sum_{c<a}\epsilon_1(v_{c,a-c}) c_{cb}^{12\vee} +\sum_{b<d}(-1)^{|c_{ab}^{12\vee}|} c_{ad}^{12\vee}\epsilon_2(v_{b, d-b})
\end{align}
\item For $\epsilon_1,\epsilon_2,\epsilon_3\in\aug(v;\field)$, the map 
\[
m'_2:\hom_{\Aug_+(v')}(\epsilon_2,\epsilon_3)\otimes \hom_{\Aug_+(v')}(\epsilon_1,\epsilon_2)\to \hom_{\Aug_+(v')}(\epsilon_1,\epsilon_3)
\]
is defined as
\begin{align*}
m'_2\left(c_{cd}^{12\vee}\otimes c_{ab}^{12\vee}\right)
=\delta_{bc}(-1)^{|c_{ab}^{12\vee}|}c_{ad}^{12\vee}
\end{align*}
\item For $m'_k$ with $k\geq 3$, the higher composition $m'_k$ vanishes.
\end{itemize}
Note that 
\[
\unit' \coloneqq \sum_{1 \leq a \leq r}-c^{12\vee}_{a,a} \in \hom_{\Aug_+(v')}(\epsilon, \epsilon)
\] 
is the unit for all $\epsilon\in \aug(v')$, and hence $\Aug_+(v')$ is a (strictly) unital $A_\infty$-category. 

We show that $\fr'_{\Right}:\Aug_+(v')\to \Aug_+(v'_\Right)$ is an $A_\infty$-equivalence.
Similar as before, it suffices to check that the $A_\infty$-functor $\fr'_{\Right}$ is essentially surjective and fully-faithful. Indeed, $\fr'_{\Right}$ is surjective (and hence essentially surjective). The main issue again is to show the following
\begin{lemma}\label{lem:quasi isom geometric proof} As in the above setup,
\[
\fr'_{R,1}=(\fr'_\Right)_1:\hom_{\Aug_+(v')}(\epsilon_1,\epsilon_2)\rightarrow\hom_{\Aug_+(v'_\Right)}(\epsilon_{1,R},\epsilon_{2,R})
\]
is a quasi-isomorphism.
\end{lemma}

\begin{proof}[Proof of Lemma~\ref{lem:quasi isom geometric proof}]
The DGA $A^{\p m}(v'_\Right)$ is nothing but a trivial bordered Legendrian described in Section~\ref{sec:aug var for trivial tangle}. The generators are
\begin{align*}
\left\{k_{aa}^{ij}~\middle|~1\leq a\leq r,\  1\le i<j\le m\right\}\coprod
\left\{k_{ab}^{ij}~\middle|~ 1\leq a<b\leq r,\  1\le i,j\le m\right\},
\end{align*}
with grading $|k_{ab}^{ij}| \coloneqq \mu(a)-\mu(b)-1$, 
and the differential
\begin{align*}
\differential_\Right^{\p m}k_{ab}^{ij}&=\sum_{\substack{a\leq c\leq b\\1\le k\le m}} (-1)^{|k_{ac}^{ik}|-1}k_{ac}^{ik}k_{cb}^{kj}.
\end{align*}
Here we use the convention on generators in Assumption~\ref{assumption:indices ranges}.

Now consider the DGA map
\begin{align*}
\phi_\Right^{\p m} : A^{\p m}(v'_\Right)\to A^{\p m}(v').
\end{align*}
By a direct counting of polygons in Figure~\ref{fig:vertex_front projection m-copy_alternative}, we have
\begin{align*}
\phi_\Right^{\p m}(k^{ii}_{ab})&=v^i_{a,b-a},&
\phi_\Right^{\p m}(k^{ij}_{ab})&=c^{i,j}_{a,b}.
\end{align*}
Then by the construction below Proposition~\ref{proposition:functoriality},
we obtain
\begin{align*}
\fr'_{R,1}=(\fr'_\Right)_1:\hom_{\Aug_+(v')}(\epsilon_1,\epsilon_2)&\rightarrow\hom_{\Aug_+(v'_\Right)}(\epsilon_{1,R},\epsilon_{2,R});\\
c^{12\vee}_{a,b}&\mapsto k^{12\vee}_{a,b}
\end{align*}
for $1\leq a \leq b \leq r$. By comparing (\ref{eqn:m_1 for trivial tangle}) and (\ref{eqn:m_1 in augmentation of v'}), we conclude that $\fr'_{R,1}$ is a (strict) isomorphism between two chain complexes, in particular, a quasi-isomorphism.
\end{proof}

As a second step, let us show
\begin{lemma}\label{lem:quasi equiv_geometric proof}
There is a $A_\infty$ quasi equivalence between $\Aug_+(v;\field)$ and $\Aug_+(v';\field)$.
\end{lemma}

\begin{figure}[!htbp]
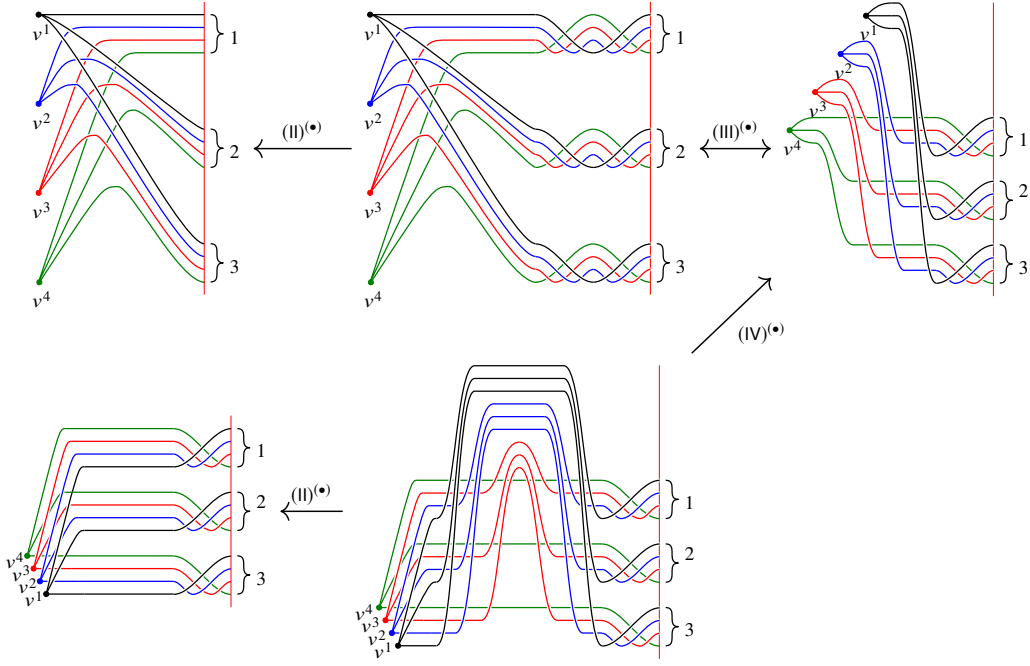

\begin{tikzcd}
\vcenter{\hbox{\def\svgscale{0.6}\input{vertex_r3_m4_v0_input.tex}}}&
\vcenter{\hbox{\def\svgscale{0.6}\input{vertex_r3_m4_v1_input.tex}}}
\ar[l,"\RM{II}^{\p\bullet}"'] \ar[r,"\RM{III}^{\p\bullet}",leftrightarrow]&
\vcenter{\hbox{\def\svgscale{0.6}\input{vertex_r3_m4_v2_input.tex}}}\\
\vcenter{\hbox{\def\svgscale{0.6}\input{vertex_r3_m4_v4_input.tex}}} &
\vcenter{\hbox{\def\svgscale{0.6}\input{vertex_r3_m4_v3_input.tex}}} \ar[l, "\RM{II}^{\p\bullet}"'] \ar[ur, "\RM{IV}^{\p\bullet}"']
\end{tikzcd}
\caption{A sequence of consistent Reidemeister moves between $v^{(m)}$ and $v'^{(m)}$.}
\label{fig:a sequence of consistent Reidemeister moves_example}
\end{figure}

\begin{proof}[Proof of Lemma~\ref{lem:quasi equiv_geometric proof}]
Let us denote the intermediate Lagrangian projections between $v^{\p m}$ and $v'^{\p m}$ as in Figure~\ref{fig:a sequence of consistent Reidemeister moves_example}. Even though the illustration is given when $r=3$ and $m=4$, the similar works for arbitrary $r$ and $m$. Note that each of them is in $\BLT^{\p\bullet}_\lag$, so it makes sense to consider the corresponding consistent DGAs.

Each pair of consecutive consistent Lagrangian projections is related by a sequence of consistent moves as depicted in Figure~\ref{fig:a sequence of consistent Reidemeister moves_example}.
This implies that there is a zig-zag of stabilizations between $A^{\p\bullet}(v)$ and $A^{\p\bullet}(v')$.
Moreover, we already argued that $\Aug_+(v')$ is strictly unital.
Then by Proposition~\ref{proposition:zig-zags of stabilizations}, we conclude that $\Aug_+(v;\field)$ and $\Aug_+(v';\field)$ are $A_\infty$ quasi equivalent to each other which proves Lemma~\ref{lem:quasi equiv_geometric proof}.
\end{proof}

Note that the consistent Reidemeister moves in Figure~\ref{fig:a sequence of consistent Reidemeister moves_example} fix the right border. So Lemma~\ref{lem:quasi equiv_geometric proof} and the fact that $\fr'_{\Right}:\Aug_+(v')\to \Aug_+(v'_\Right)$ is an $A_\infty$ equivalence imply that 
\[\fr_{\Right}:\Aug_+(v)\to \Aug_+(v_\Right)\] 
is also a $A_\infty$ quasi equivalence.
This gives a geometric proof of Lemma~\ref{lem:aug cat for a vertex}.
\end{proof}

\subsection{Sheaf property of augmentation categories}
Similar to \cite{NRSSZ2015}, for the purpose of proving ``augmentations are sheaves" for Legendrian graphs, or more generally, bordered Legendrian graphs, we will need a sheaf property of augmentation categories, which we now explain.
 
Let $\cC$ be a constructible sheaf of (homotopically) unital $A_{\infty}$-categories on an interval $(x_\Left,x_\Right)$, with respect to a stratification $\cX$ consisting of zero dimensional strata $x_i$
and one-dimensional strata $u_{i, i+1} = (x_i, x_{i+1})$.  Equivalently, the generization maps induce, for each $i$, a diagram of unital $A_{\infty}$-categories and (homotopically) unital $A_{\infty}$-functors:
\[
\cC(u_{i-1, i}) \xleftarrow{\fg_{L}} \cC_{x_{i}} \xrightarrow{\fg_{R}} \cC(u_{i, i+1}).
\]
It then follows from the sheaf axiom that, if $x_i <  x_{i+1} < \ldots < x_j$ are the zero dimensional strata in the interval $(a, b)\subset (x_\Left,x_\Right)$, then
\[
\cC( (a, b) ) \equiv \cC_{x_{i}} \times_{\cC(u_{i, i+1})}  \cC_{x_{i+1}} \times \ldots \times \cC_{x_{j}},
\]
is a fiber product. Here, the objects are tuples $(\xi_i, \xi_{i+1}, \ldots, \xi_j; f_{i, i+1}, \ldots, f_{j-1, j})$ where
$\xi_k \in \cC_{x_{k}}$ and $f_{k, k+1}: \fg_\Right(\xi_k) \to \fg_\Left(\xi_{k+1})$ is an isomorphism in $\cC$, i.e. a closed degree $0$ morphism in $\hom^*_{\cC(u_{k, k+1})}(\fg_\Right(\xi_k),\fg_\Left(\xi_{k+1}))$, whose cohomology is invertible by passing to the cohomological category of $\cC(u_{k, k+1})$.

On the other hand, there is a  full subcategory of this fiber product, called \emph{strict fiber product}:
\[
(\cC_{x_{i}} \times_{\cC(u_{i, i+1})}  \cC_{x_{i+1}} \times \ldots \times \cC_{x_{j}})_{strict},
\]
in which all the $f_{k, k+1}$'s are identity morphisms, that is, $\fg_{R}(\xi_k) = \fg_\Left(\xi_{k+1})$.
Recall that:
\begin{lemma}[{\cite[Lem.7.4]{NRSSZ2015}}]\label{lem:strict fiber product}
If all $\fg_\Left$'s satisfy the \emph{isomorphism lifting property}:
any isomorphism $\phi: \fg_\Left(\xi) \sim \eta'$ is the image under $\fg_\Left$ of some isomorphism
$\psi: \xi \sim \xi'$.  Then the inclusion of the strict fiber product in the actual fiber product is an $A_{\infty}$-equivalence.
\end{lemma}

In our case for augmentation categories, we have
\begin{lemma}\label{lem:isomorphism lifting property for aug cat}
In the front projection picture, let $(T,\mu)$ be any of the following elementary bordered Legendrian graph: $n$ parallel strands, a single crossing, a single left cusp, a single marked right cusp, or a single vertex as in Example \ref{ex:aug cat for a vertex}. Then $\Aug_+(T_\Left;\field)\xleftarrow[]{\fr_\Left}\Aug_+(T;\field)$ satisfies the isomorphism lifting property as in Lemma \ref{lem:strict fiber product}. 
\end{lemma}

\begin{proof}
The smooth cases are covered in \cite[\S7.2]{NRSSZ2015}. The only remaining case is when $T=V$ is a single vertex as in Example \ref{ex:aug cat for a vertex}. By Lemma \ref{lem:aug cat for a vertex}, the restriction functor $\fr_\Left$ is just the obvious projection $\fp_1:\Aug_+(V_\Left;\field)\times\Aug_+(v;\field)\rightarrow \Aug_+(V_\Left;\field)$, which clearly satisfies the isomorphism lifting property.
\end{proof}

As a consequence, we obtain:
\begin{proposition}\label{prop:sheaf property for aug cat}
Let $(T,\mu)$ be any bordered Legendrian graph in $J^1I_x$ with $I_x=(x_\Left,x_\Right)$, such that the $x$-coordinates of the singularities in its front projection are all distinct, denoted by $x_1<x_2<\ldots<x_N$. Assume each right cusp is marked, and each vertex has no left half-edges. 
Denote $x_0\coloneqq x_\Left$ and $x_{N+1}\coloneqq x_\Right$, and $u_{k,k+1}=(x_k,x_{k+1})$. 
\emph{Define} a constructible sheaf $\cC$ of unital $A_{\infty}$-categories on $I_x$ via:
\begin{eqnarray*}
&&(\cC(u_{x_{k-1,k}})\xleftarrow[]{\fg_\Left}\cC_{x_k}\xrightarrow[]{\fg_\Right}\cC(u_{k,k+1}))\\
&\coloneqq&(\Aug_+(T|_{u_{k-1,k}})\xleftarrow[]{\fr_\Left}\Aug_+(T|_{(x_{k-1},x_{k+1})};\field)\xrightarrow[]{\fg_\Right}\Aug_+(T|_{u_{k,k+1}}))
\end{eqnarray*}
Then we have the following equivalences between the rows of diagrams of unital $A_{\infty}$-categories: 
\[
\begin{tikzcd}[column sep=4pc, row sep=2pc]
\Aug_+(T_\Left;\field)\arrow[equal,d]  & \Aug_+(T;\field)\arrow[l,"\fr_\Left"']\arrow[r,"\fr_\Right"]\arrow[equal,d] & \Aug_+(T_\Right;\field)\arrow[equal,d]\\
\cC(u_{0,1})\arrow[equal,d]& (\cC_{x_1} \times_{\cC(u_{1, 2})}  \cC_{x_2} \times \ldots \times \cC_{x_N})_{strict}\ar[l,"\fg_\Left"']\arrow[r,"\fg_\Right"]\arrow[hook,d,"\simeq"] & \cC(u_{N,N+1})\arrow[equal,d]\\
\cC(u_{0,1})& \cC_{x_1} \times_{\cC(u_{1,2})}  \cC_{x_2} \times \ldots \times \cC_{x_N}\arrow[l,"\fg_\Left"']\arrow[r,"\fg_\Right"] & \cC(u_{N,N+1})
\end{tikzcd}
\]
In particular, the augmentation categories $\{\Aug_+(T|_{(a,b)};\field)\}_{(a,b)\subset I_x}$ form a sheaf.
\end{proposition}

\begin{proof}
Similar to \cite[\S7.2]{NRSSZ2015}, the identification between the first two rows follows directly from the definition of diagrams of augmentation categories, and the co-sheaf property of LCH DGAs associated to bordered Legendrian graphs. It remains to show the middle-lower inclusion is an equivalence, but this is done by Lemma \ref{lem:isomorphism lifting property for aug cat} and Lemma \ref{lem:strict fiber product} above.
\end{proof}

\subsection{Augmentations are sheaves}
Let us start by recalling the result and strategy from \cite{NRSSZ2015}.

\begin{theorem}\cite{NRSSZ2015}
Let $\Lambda\subset \RR^3$ be a Legendrian knot, and let $\field$ be a field. Then there is an  $A_\infty$ equivalence between two categories  
\[
\Aug_+(\Lambda; \field)\stackrel{\isomorphic}{\longrightarrow}\cC_1(\Lambda;\field).
\]
\end{theorem}

The first step is to cut the front $\pi_{\front}(\Lambda)\subset \RR_{xz}$ diagram into elementary pieces, bordered Legendrians in $(0,1)\times \RR_z$ only having a cusp, a crossing, or a basepoint.
Next show the equivalence for each pieces, and then argue the sheaf property in both sides to conclude the global equivalence.

Among many ingredients in the proof, let us recall the equivalence when a bordered Legendrian is $I_n$, the trivial $n$-strands.

\begin{proposition}\cite{NRSSZ2015}\label{prop:equiv MC and C_1}
Let $I_n \subset (0,1)\times \RR_z$ be a trivial tangle with $n$-strands, and let $\field$ be a field. Then there are DG equivalences:
\[
\Aug_+(I_n; \field) \stackrel{\isomorphic}{\longrightarrow}  \cM\cC(I_n; \field),\qquad \cM\cC(I_n; \field)\stackrel{\homotopic}{\longrightarrow}\cC_1(I_n; \field).
\]
\end{proposition}
The first equivalence has been described in Lemma \ref{lem:aug cat for trivial tangle}. 
The second equivalence can be described as follows: 

\noindent{}Recall that the strands of $I_n$ are labelled from top to bottom by $1,2,\ldots,n$, and $C=C(I_n)\coloneqq\oplus_{i=1}^n\field\cdot e_i$, with $|e_i|=-\mu(i)$. $C$ is equipped with a filtration $F^{\bullet}:F^i=F^iC\coloneqq\oplus_{j>i}\field\cdot e_j$. In particular, $F^0=C$ and $F^n=0$. In addition, $\cS_{I_n}$ is the regular cell complex induced by $I_n$, and by Definition \ref{def:poset for legible model}, $\cR(\cS)$ is the $A_{n+1}$-quiver:
\[
n~\rightarrow~n-1~\rightarrow~\ldots~\rightarrow~0
\]
where the node $i$ corresponds to the region immediately below the strand $i$ for $1\leq i\leq n$, and $0$ corresponds to the upper region. Recall that the indecomposable projectives of the quiver $A_{n+1}$ are:
\[
P_i\coloneqq(0\rightarrow\ldots\rightarrow\field\xrightarrow[]{id}\field\xrightarrow[]{id}\ldots\field\xrightarrow[]{id}\field)
\] 
i.e. $P_i$ consists of a copy of $\field$ at all nodes $k\leq i$. The indecomposables of $A_{n+1}$ are:
\[
S_{i,j}\coloneqq P_j/P_{i-1}=(0\rightarrow\ldots\rightarrow 0\rightarrow\field\xrightarrow[]{id}\field\xrightarrow[]{id}\ldots\xrightarrow[]{id}\field\rightarrow 0\rightarrow\ldots\rightarrow 0),
\] 
i.e. $S_{i,j}$ consists of a copy of $\field$ at all nodes $k$ for $i\leq k\leq j$. In particular, $S_{i,j}'\coloneqq(P_{i-1}\hookrightarrow P_j)\twoheadrightarrow S_{i,j}$ is a projective resolution of $S_{i,j}$. Moreover, the path algebra $\field\langle A_{n+1}\rangle$ of $A_{n+1}$ is of global cohomological dimension one, hence every sub-module of a projective for $\field\langle A_{n+1}\rangle$ is still projective. It then follows that objects in the DG category $\Fun(A_{n+1},\field)$ (Definition \ref{def:rep cat}) split, hence quasi-isomorphic to a direct sum of $S_{i,j}[s]$'s.
As an immediate consequence, we obtain: 

\begin{lemma}\cite[Lem.7.38]{NRSSZ2015}\label{lem:cofibrant repr}
Every representation $R$ in $\Fun(A_{n+1},\field)$ is quasi-isomorphic to a representation $R'$:
\[
R_n'\hookrightarrow R_{n-1}'\hookrightarrow\ldots\hookrightarrow R_0'
\]
such that, by denoting $R_{n+1}'\coloneqq 0$, we have:
\begin{enumerate}
\item
The maps $R_{i+1}'\hookrightarrow R_i'$ are injective;
\item
The quiver representation $R'$ in each cohomological degree is projective;
\item
The differential on each $R_i'/R_{i+1}'$ is zero.
\end{enumerate}
\end{lemma}

There is a natural DG equivalence $\cR:\MC(I_n;\field)\xrightarrow[]{\sim}\Fun_{I_n,1}(\cR(\cS)=A_{n+1},\field)$, such that, on objects, we have:
\[
\cR(C,d)\coloneqq ((F^n,d|_{F^n})\hookrightarrow(F^{n-1},d|_{F^{n-1}})\hookrightarrow\ldots\hookrightarrow(F^0,d)),
\] 
and $\cR$ sends Hom complexes in $\MC(I_n;\field)$ literally to the corresponding identical Hom complexes in $\Fun_{I_n,1}(A_{n+1},\field)$. It follows immediately that $\cR$ is fully faithful. The essential surjectivity of $\cR$ is a direct corollary of Lemma \ref{lem:cofibrant repr} above.
Now, the second equivalence in Proposition \ref{prop:equiv MC and C_1} is just the composition of the equivalence in Corollary \ref{cor:legible model for C_1} with $\cR$.

Finally, we come to our main theorem:
\begin{theorem}[Augmentations are sheaves]\label{thm:augs are sheaves for Legendrian graphs}
Let $\cT$ be a bordered Legendrian graph, and let $\field$ be a field. Then there is an $A_\infty$-equivalence:
\[
\Aug_+(\cT;\field)\xrightarrow[]{\sim}\cC_1(\cT;\field)
\]
\end{theorem}

\begin{proof}
Let us recall the invariance property Theorem~\ref{thm:invariance aug}, Corollary~\ref{cor:invariance of C_1} of the both sides, and sheaf property Proposition~\ref{prop:sheaf property for aug cat} of $\Aug_+(\cT:\field)$. 
Thus it suffices to check the equivalence for a new type of elementary pieces $\cV=(V_\Left \stackrel{\iota_\Left}{\longrightarrow}V \stackrel{\iota_\Right}{\longleftarrow}V_\Right)$ in $T^{\infty,-}(I_x\times\RR_z)$, a bordered Legendrian with one vertex $v$ of type $(0,m)$ and trivial $k$-strands as follows:

\begin{align*}
\vcenter{\hbox{
\begingroup%
  \makeatletter%
  \providecommand\color[2][]{%
    \errmessage{(Inkscape) Color is used for the text in Inkscape, but the package 'color.sty' is not loaded}%
    \renewcommand\color[2][]{}%
  }%
  \providecommand\transparent[1]{%
    \errmessage{(Inkscape) Transparency is used (non-zero) for the text in Inkscape, but the package 'transparent.sty' is not loaded}%
    \renewcommand\transparent[1]{}%
  }%
  \providecommand\rotatebox[2]{#2}%
  \ifx\svgwidth\undefined%
    \setlength{\unitlength}{63.05535634bp}%
    \ifx\svgscale\undefined%
      \relax%
    \else%
      \setlength{\unitlength}{\unitlength * \real{\svgscale}}%
    \fi%
  \else%
    \setlength{\unitlength}{\svgwidth}%
  \fi%
  \global\let\svgwidth\undefined%
  \global\let\svgscale\undefined%
  \makeatother%
  \begin{picture}(1,1.79004296)%
    \put(0,0){\includegraphics[width=\unitlength,page=1]{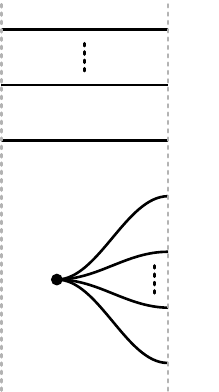}}%
    \put(0.83035619,1.68654629){\color[rgb]{0,0,0}\makebox(0,0)[lt]{\begin{minipage}{0.12687267\unitlength}\raggedright $_1$\end{minipage}}}%
    \put(0.83035618,1.17905566){\color[rgb]{0,0,0}\makebox(0,0)[lt]{\begin{minipage}{0.12687267\unitlength}\raggedright $_k$\end{minipage}}}%
    \put(0.83035618,0.92531034){\color[rgb]{0,0,0}\makebox(0,0)[lt]{\begin{minipage}{0.19096926\unitlength}\raggedright $_{k+1}$\end{minipage}}}%
    \put(0.83035618,0.16407447){\color[rgb]{0,0,0}\makebox(0,0)[lt]{\begin{minipage}{0.19096926\unitlength}\raggedright $_{k+m}$\end{minipage}}}%
    \put(0.19599296,0.41781975){\color[rgb]{0,0,0}\makebox(0,0)[lt]{\begin{minipage}{0.19096926\unitlength}\raggedright $_v$\end{minipage}}}%
  \end{picture}%
\endgroup%
}}
\end{align*}

As in Section \ref{subsubsec:legible model}, let $\cS_V$ be the Whitney stratification of $I_x\times\RR_z$ induced by $V$. Let $\cS$ be the regular cell complex for $I_x\times\RR_z$ refining $\cS_V$, by adding one left half-edge at $v$. Clearly, $\cS$ satisfies Assumption \ref{ass:legible model}. Then by Lemma \ref{lem:aug cat for a vertex} and Corollary \ref{cor:legible model for C_1}, it suffices to show the equivalence between the diagram of Morse complex categories
\[
\cMC\coloneqq(\cMC(V_\Left;\field)\xleftarrow[]{\fp_1}\cMC(V_\Left;\field)\times\cMC(v_\Right;\field)\xhookrightarrow[]{\fri}\cMC(V_\Right;\field))
\]
and $\Fun_{\cV,1}(\cR(\cS),\field)$. 

Notice that, $\cR(\cS|_{V_\Right})=A_{n+1}$ with $n=k+m$. And,
\[
\Fun_{V,1}(\cR(\cS),\field)\hookrightarrow\Fun_{V_\Right,1}(\cR(\cS|_{V_\Right}),\field)
\]
is just the fully faithful embedding into the full DG subcategory of $\Fun_{V_\Right,1}(A_{n+1},\field)$, consisting of functors $F$ such that $F(k+m\rightarrow k)$ is a quasi-isomorphism.

On the other hand, by the discussion above, we have the DG equivalence $\cR:\MC(V_\Right;\field)\xrightarrow[]{\sim}\Fun_{V_\Right,1}(A_{n+1},\field)$. Notice that the image of $\fri$ of $\MC(V_\Left;\field)\times\MC(v_\Right;\field)$ consists of Morse complexes $(C(V_\Right),d)=(C(V_\Left),d_\Left)\oplus(C(v_\Right),d_{v_\Right})$, such that each summand is acyclic. 
Then, $\cR\circ\fri$ induces a fully faithful DG functor 
\[
\cR_V:\MC(V_\Left;\field)\times\MC(v_\Right;\field)\hookrightarrow\Fun_{V,1}(\cR(\cS),\field).
\]
$\cR_V$ is also essentially surjective.
In fact, $\cR$ is essentially surjective implies that, any functor $F\in\Fun_{V,1}(\cR(\cS),\field)$, is quasi-isomorphic to $\cR((C(V_\Right),d))$, for some Morse complex $(C(V_\Right),d)$, which we can assume to be canonical (or Barannikov's normal form). 
In addition, the subquotient $(C(v_\Right),d_{v_\Right})$ of $(C(V_\Right),d)$ is acyclic, since $F(k+m\rightarrow k)$ is a quasi-isomorphism. It follows that $(C(V_\Right),d)=(C(V_\Left),d_\Left)\oplus(C(v_\Right),d_{v_\Right})$ for some canonical Morse complexes $(C(V_\Left),d_\Left)$ and $(C(v_\Right),d_{v_\Right})$. We are done.

Finally, it is direct to check that $\cR_v$ commutes with the DG equivalence $\cR_\Left:\MC(V_\Left;\field)\xrightarrow[]{\sim}\Fun_{V_\Left,1}(\cR(\cS|_{V_\Left}),\field)$ defined as in the discussion above, up to a specified natural isomorphism. 
Thus, we obtain a DG equivalence $\cR_V:\cMC\xrightarrow[]{\sim}\Fun_{\cV,1}(\cR(\cS),\field)$, as desired. This finishes the proof of Theorem \ref{thm:augs are sheaves for Legendrian graphs}.
\end{proof}

\subsection{Example}\label{sec:example augmentation category}

Let us consider a Legendrian graph $\Lambda$ having the following front with Maslov potential and Lagrangian projection.
Note that $\Lambda$ has one vertex $v$ and two (oriented) edges $e_1,e_2$:

\begin{figure}[ht]
\begin{align*}
\begin{tikzpicture}[baseline=-.5ex,scale=0.8]
\draw[thick] (-2.5,0) to[out=0,in=180] (0,1.5) node[above] {$_1$} node{$>$} to[out=0,in=180] (2.5,0) node{$_|$} node[right]{$_{e_2}$};
\draw[thick] (-2.5,0) to[out=0,in=180] (-1,-1) node[below] {$_0$} to[out=0,in=180] (0,-0.5) to[out=0,in=180] (1,-1) node[below] {$_0$} to[out=0,in=180] (2.5,0);
\draw[thick] (-1,0) node[below right] {$_0$} to[out=0,in=180] (0,0.5) node[below] {$_1$} node{$>$} to[out=0,in=180] (1,0) node[below left] {$_0$} node{$_|$} node[right]{$_{e_1}$};
\draw[thick] (-1,0) to[out=0,in=180] (0,-0.5) to[out=0,in=180] (1,0);
\draw[fill] (0,-0.5) node[below] {$_v$} circle (3pt);
\end{tikzpicture}
&&
\vcenter{\hbox{\scalebox{.9}{\input{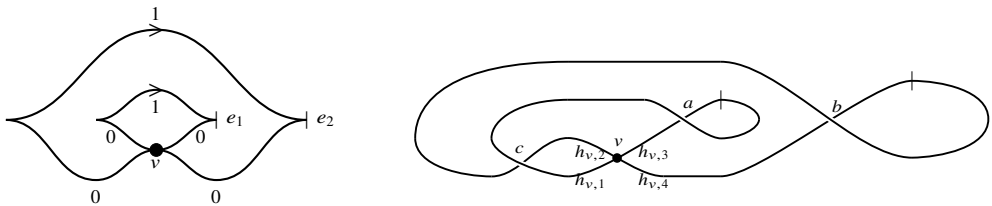}}}}
\end{align*}
\caption{The front and Lagrangian projection of $\Lambda$}
\label{fig:example front Lag}
\end{figure}

\subsubsection{Augmentation category}
Fix the base field $\field=\Zmod{2}$. Firstly, let us consider the augmentation category $\Aug_+=\Aug_+(\Lambda;\field)$.
We consider CE DGA $A(\Lambda)$ with $\Zmod{2}[t_1^{\pm1},t_2^{\pm1}]$-coefficient.
Here, $t_i$ corresponds to the basepoint on the edge $e_i$ for $i=1,2$ in Figure~\ref{fig:example front Lag}. 
Then the DGA $(\alg,|\cdot|,\differential)$ is
\begin{align*}
\alg&=\Zmod{2}[t_1^{\pm1},t_2^{\pm1}]\langle \sfG \rangle;\\ 
\sfG&=\{ a, b, c\} \sqcup \{ v_{i,\ell} \mid i\in \Zmod{4}, \ell\in \NN \},
\end{align*}
and the grading is given by
\begin{align*}
|a|&=|b|=1,& |c|&=0,& |v_{i,\ell}|&=\mu(i)-\mu(i+\ell)+N(v;i,\ell)-1.
\end{align*}
The differential on the crossing generators are as follows:
\begin{align*}
\differential a&=t_1 + v_{1,2} + c v_{2,1};\\
\differential b&=t_2 + v_{2,2} + v_{2,1} t_1^{-1}(v_{1,3} + c v_{2,2}+ a v_{3,1});\\
\differential c&=v_{1,1}.
\end{align*}
The following are the list of degree zero generators:
\[
c, v_{1,2}, v_{1,3}, v_{2,1}, v_{2,2}, v_{3,2}, v_{3,3}, v_{4,1}, v_{4,2}.
\]
Among infinitely many equation of differential on vertex generators, let us list relations coming from the differential of degree one generators that augmentation $\bar{g}$ of a generator $g$ should satisfy:
\begin{align*}
\bar{v_{1,2} v_{3,2}} + \bar{v_{1,3} v_{4,1}}&=1,&
\bar{v_{2,1} v_{3,3}} + \bar{v_{2,2} v_{4,2}}&=1,&
\bar{v_{3,2} v_{1,2}} + \bar{v_{3,3} v_{2,1}}&=1,&
\bar{v_{4,1} v_{1,3}} + \bar{v_{4,2} v_{2,2}}&=1;\\
\bar{v_{1,2} v_{3,3}} + \bar{v_{1,3} v_{4,2}}&=0,&
\bar{v_{3,2} v_{1,3}} + \bar{v_{3,3} v_{2,2}}&=0,&
\bar{v_{2,1} v_{3,2}} + \bar{v_{2,2} v_{4,1}}&=0,&
\bar{v_{4,1} v_{1,2}} + \bar{v_{4,2} v_{2,1}}&=0.\\
\end{align*}

Then by direct computation, we can check that there are eight possible ($\Zmod{2}$-valued) augmentations:
\begin{center}
 \begin{tabular}{c || c | c | c | c | c | c | c | c | c} 
$i$ & $\epsilon_i(c)$ & $\epsilon_i(v_{1,2})$ & $\epsilon_i(v_{1,3})$ & $\epsilon_i(v_{2,1})$ & $\epsilon_i(v_{2,2})$ & $\epsilon_i(v_{3,2})$ & $\epsilon_i(v_{3,3})$ & $\epsilon_i(v_{4,1})$ & $\epsilon_i(v_{4,2})$ \\ [0.5ex] 
 \hline\hline
1& 1&0&1&1&0&0&1&1&0 \\ 
 \hline
2&  0&1&0&0&1&1&0&0&1\\
 \hline
 3 & 0&1&1&1&0&0&1&1&1 \\
 \hline
 4 & 1&0&1&1&1&1&1&1&0 \\
 \hline
 5 & 0&1&1&0&1&1&1&0&1 \\
 \hline
 6& 1&1&0&0&1&1&0&0&1\\
 \hline
 7& 1&1&1&0&1&1&1&0&1\\
 \hline
 8& 0&1&0&1&1&1&0&1&1
\end{tabular}
\end{center}
Obviously $\epsilon(t_1)=\epsilon(t_2)=1$.

\begin{remark}
Note that there are only two equivalence classes of augmentations up to isomorphisms.
One can check that
\begin{align*}
\epsilon_1 &\sim \epsilon_3 \sim \epsilon _4 \sim \epsilon_8,&
\epsilon_2 &\sim \epsilon_5 \sim \epsilon _6 \sim \epsilon_7.
\end{align*}
In the rest of example, we mainly consider two non-equivalent augmentations $\epsilon_1$ and $\epsilon_2$.
Moreover, two (equivalence classes of) augmentations corresponds to two possible resolutions at vertex $v$
\begin{align*}
\epsilon_1 &\in \aug^{\ruling_v^1}(\Lambda),& \epsilon_2&\in \aug^{\ruling_v^2}(\Lambda),
\end{align*}
where $\ruling_v^1=\{\{1,4\},\{2,3\}\}$, and $\ruling_v^2=\{\{1,3\},\{2,4\}\}$. See \cite[\S5]{ABS2019count} for the details.
\end{remark}

\begin{figure}[ht]
\input{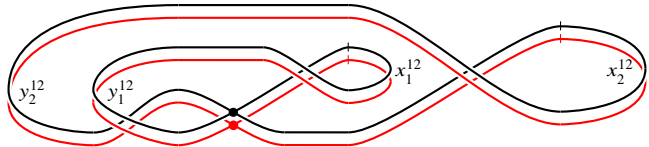}
\caption{Two copy of $\Lambda$ in the Lagrangian projection}
\end{figure}

Now consider two copies $\Lambda^{(2)}$. The labeling of the crossing is given as in Section~\ref{section:augmentation category of normal form}. Then the corresponding algebra $\alg^{(2)}$ is generated by vertex generators
\[
\{v^{m}_{k,\ell} \mid m=1,2, k\in \Zmod{4}, \ell\in \NN\},
\]
pure Reeb chords
\[
\{ a^{mm}, b^{mm}, c^{mm} \mid m=1,2\},
\]
and mixed Reeb chords
\[
\{ a^{12}, b^{12}, c^{12}, v_{1,1}^{12}, v_{4,-1}^{12}, x_1^{12}, x_2^{12}, y_1^{12}, y_2^{12} \}\sqcup \{a^{21}, b^{21}, c^{21}\}.
\]
Here $v_{1,1}^{12}, v_{4,-1}^{12}$ are the crossings near two-copies of vertex, see (\ref{eqn:m_copy_vertexlabel}) for the label convention near $m$-copy of the vertex, and the crossings $x_i^{12}, y_i^{12}$ arise from the right- and left cusps, respectively.
The grading for the generator is straight forward from the DGA $A(\Lambda)$ except the following
\begin{align*}
|x_i^{12}|&=0,& |y_i^{12}|&=-1 \quad\text{ for } i=a,b;\\
|v_{1,1}^{12}|&=-1,& |v_{4,-1}^{12}|&=0.
\end{align*}

Now let us consider the mixed chords in $\Lambda^{(2)}$, especially from the second to the first copy.
Note that there are such Reeb chords $a^{12},b^{12}$ of degree 1, $c^{12}, v_{4,-1}^{12}, x_1^{12}, x_2^{12}$ of degree 0, and $y_1^{12}, y_2^{12}, v_{1,1}^{12}$ of degree $-1$.

For convenience of the computation, let us restrict the differential to the following submodule
\[
\bfM_0^{(2)} \coloneqq (\sfM_0^{11})^{\otimes n_1}\otimes \sfM^{12} \otimes (\sfM_0^{22})^{\otimes n_2},
\]
where $n_i$ is a nonnegative integer, and $\sfM_0^{ii}$ denotes the submodule spanned by degree $0$ elements for $i=1,2$.
Then the restriction of the differential becomes
\begin{align*}
\differential a^{12}|_{\bfM_0^{(2)}} &= x_1^{12} t_1^{2} + c^{11}v_{2,2}^{11} v_{4,-1}^{12} + c^{12} v_{2,1}^{22} + v_{1,3}^{11}v_{4,-1}^{12};\\
\differential b^{12}|_{\bfM_0^{(2)}} &= x_2^{12} t_2^{2} + v_{2,1}^{11}(t_1^1)^{-1} c^{12} v_{2,2}^{22} +\left( v_{2,1}^{11} (t_1^1)^{-1} x_1^{12} + v_{2,2}^{11} v_{4,-1}^{12} (t_1^2)^{-1} \right) \left( c^{22} v_{2,2}^{22} + v_{1,3}^{22}  \right)\\
\differential c^{12}|_{\bfM_0^{(2)}} &= y_1^{12} c^{22} + c^{11} y_2^{12} + v_{1,1}^{12};\\
\differential v_{4,-1}^{12}|_{\bfM_0^{(2)}} &= v_{4,1}^{11} v_{1,1}^{12} v_{2,1}^{22} + v_{4,2}^{11} y_2^{12} v_{2,1}^{22} + v_{4,1}^{11} y_1^{12} v_{1,2}^{22};\\
\differential x_1^{12}|_{\bfM_0^{(2)}} &= t_1^1 \left( v_{3,2}^{11} v_{1,1}^{12} v_{2,1}^{22} + v_{3,2}^{11} y_1^{12} v_{1,2}^{22} + v_{3,3}^{11} y_2^{12} v_{2,1}^{22}  \right) (t_1^2)^{-1} + y_1^{12};\\
\differential x_2^{12}|_{\bfM_0^{(2)}} &= t_2^1 \left( v_{4,1}^{11} v_{1,1}^{12} v_{2,2}^{22} + v_{4,2}^{11} y_2^{12} v_{2,2}^{22} + v_{4,1}^{11} y_1^{12} v_{1,3}^{22}  \right) (t_2^2)^{-1} + y_2^{12};\\
\differential v_{1,1}^{12}|_{\bfM_0^{(2)}} &=0;\\
\differential y_1^{12}|_{\bfM_0^{(2)}} &=0;\\
\differential y_2^{12}|_{\bfM_0^{(2)}} &=0.
\end{align*}

Note that the above restriction only contribute the bi-linearized differentials. For example, we have the following bi-linearized differential with respect to a pair $(\epsilon_1,\epsilon_1)$:
\begin{align*}
\differential_{\epsilon_1,\epsilon_1}a^{12}&=x_1^{12}+c^{12}+v_{4,-1}^{12},& 
\differential_{\epsilon_1,\epsilon_1}b^{12}&=x_1^{12} + x_2^{12},&
\differential_{\epsilon_1,\epsilon_1}c^{12}&=y_1^{12} + y_2^{12} + v_{1,1}^{12};\\
\differential_{\epsilon_1,\epsilon_1}v_{4,-1}^{12}&=v_{1,1}^{12},&
\differential_{\epsilon_1,\epsilon_1}x_1^{12}&=y_1^{12} + y_2^{12},&
\differential_{\epsilon_1,\epsilon_1}x_2^{12}&=y_1^{12} + y_2^{12};\\
\differential_{\epsilon_1,\epsilon_1}y_1^{12}&=0,&
\differential_{\epsilon_1,\epsilon_1}y_2^{12}&=0,&
\differential_{\epsilon_1,\epsilon_1}v_{1,1}^{12}&=0.
\end{align*}
For a pair $(\epsilon_1,\epsilon_2)$, we have
\begin{align*}
\differential_{\epsilon_1,\epsilon_2}a^{12}&=x_1^{12}+v_{4,-1}^{12},& 
\differential_{\epsilon_1,\epsilon_2}b^{12}&=x_2^{12} + c^{12},&
\differential_{\epsilon_1,\epsilon_2}c^{12}&=y_2^{12} + v_{1,1}^{12};\\
\differential_{\epsilon_1,\epsilon_2}v_{4,-1}^{12}&=y_1^{12},&
\differential_{\epsilon_1,\epsilon_2}x_1^{12}&=y_1^{12},&
\differential_{\epsilon_1,\epsilon_2}x_2^{12}&=v_{1,1}^{12} + y_2^{12};\\
\differential_{\epsilon_1,\epsilon_2}y_1^{12}&=0,&
\differential_{\epsilon_1,\epsilon_2}y_2^{12}&=0,&
\differential_{\epsilon_1,\epsilon_2}v_{1,1}^{12}&=0.
\end{align*}
In a case of $(\epsilon_2, \epsilon_1)$, we deduce
\begin{align*}
\differential_{\epsilon_2,\epsilon_1}a^{12}&=x_1^{12} + c^{12},& 
\differential_{\epsilon_2,\epsilon_1}b^{12}&=x_2^{12} + v_{4,-1}^{12},&
\differential_{\epsilon_2,\epsilon_1}c^{12}&=y_1^{12} + v_{1,1}^{12};\\
\differential_{\epsilon_2,\epsilon_1}v_{4,-1}^{12}&=y_2^{12},&
\differential_{\epsilon_2,\epsilon_1}x_1^{12}&= v_{1,1}^{12} + y_1^{12},&
\differential_{\epsilon_2,\epsilon_1}x_2^{12}&= y_2^{12};\\
\differential_{\epsilon_2,\epsilon_1}y_1^{12}&= 0,&
\differential_{\epsilon_2,\epsilon_1}y_2^{12}&= 0,&
\differential_{\epsilon_2,\epsilon_1}v_{1,1}^{12}&=0.
\end{align*}

When a pair is $(\epsilon_2, \epsilon_2)$, then
\begin{align*}
\differential_{\epsilon_2,\epsilon_2}a^{12}&=x_1^{12},& 
\differential_{\epsilon_2,\epsilon_2}b^{12}&=x_2^{12},&
\differential_{\epsilon_2,\epsilon_2}c^{12}&=v_{1,1}^{12},
\end{align*}
with all other differentials are trivial.

For any pair of augmentation $(\epsilon_i,\epsilon_j)$, we have
\begin{align*}
\hom_{\Aug_+}(\epsilon_i,\epsilon_j)&=\hom_{\Aug_+}^0(\epsilon_i,\epsilon_j)\oplus\hom_{\Aug_+}^1(\epsilon_i,\epsilon_j)\oplus\hom_{\Aug_+}^2(\epsilon_i,\epsilon_j);\\
\hom_{\Aug_+}^0(\epsilon_i,\epsilon_j)&=\Zmod{2}\langle y_1^{12}, y_2^{12}, v_{1,1}^{12}  \rangle^\vee;\\
\hom_{\Aug_+}^1(\epsilon_i,\epsilon_j)&=\Zmod{2}\langle c^{12}, v_{4,-1}^{12}, x_1^{12}, x_2^{12}  \rangle^\vee;\\
\hom_{\Aug_+}^2(\epsilon_i,\epsilon_j)&=\Zmod{2}\langle a^{12}, b^{12}  \rangle^\vee.
\end{align*}

By dualizing the bi-linearized differential $\differential_{\epsilon_1,\epsilon_1}$, we have
\begin{align*}
m_1^{\epsilon_1,\epsilon_1}(y_{1}^{12})^\vee &= (c^{12})^\vee + (x_1^{12})^\vee + (x_2^{12})^\vee,&
m_1^{\epsilon_1,\epsilon_1}(x_1^{12})^\vee &= (a^{12})^\vee + (b^{12})^\vee,&
m_1^{\epsilon_1,\epsilon_1}(a^{12})^\vee &= 0;\\
m_1^{\epsilon_1,\epsilon_1}(y_{2}^{12})^\vee &= (c^{12})^\vee + (x_1^{12})^\vee + (x_2^{12})^\vee,&
m_1^{\epsilon_1,\epsilon_1}(x_2^{12})^\vee &= (b^{12})^\vee,&
m_1^{\epsilon_1,\epsilon_1}(b^{12})^\vee &= 0;\\
m_1^{\epsilon_1,\epsilon_1}(v_{1,1}^{12})^\vee &= (c^{12})^\vee + (v_{4,-1}^{12})^\vee,&
m_1^{\epsilon_1,\epsilon_1}(v_{4,-1}^{12})^\vee &= (a^{12})^\vee,&
m_1^{\epsilon_1,\epsilon_1}(c^{12})^\vee &= (a^{12})^\vee.
\end{align*}
The only non-trivial cohomology class with respect to $m_1^{\epsilon_1,\epsilon_1}$ is $\alpha\coloneqq[(y_1^{12}+y_2^{12})^\vee]\in H^0 \hom_{\Aug_+}(\epsilon_1,\epsilon_1)$.

By the similar computation for $(\epsilon_1,\epsilon_2)$, $(\epsilon_2,\epsilon_1)$, and $(\epsilon_2, \epsilon_2)$, we have the following non-trivial cohomology classes 
\begin{align}\label{eqn:aug_example_cohomology_class}
\beta \coloneqq [(y_2^{12} + v_{1,1}^{12})^\vee] & \in H^0 \hom_{\Aug_+}(\epsilon_1,\epsilon_2),& \gamma  \coloneqq [(y_1^{12} + v_{1,1}^{12})^\vee] &\in H^0 \hom_{\Aug_+}(\epsilon_2,\epsilon_1);\\
\delta_1  \coloneqq [(y_1^{12})^\vee], \delta_2 \coloneqq [(y_2^{12})^\vee] &\in H^0 \hom_{\Aug_+}(\epsilon_2,\epsilon_2),& \delta_3  \coloneqq [(v_{4,-1}^{12})^\vee] &\in H^1 \hom_{\Aug_+}(\epsilon_2,\epsilon_2),
\end{align}
respectively.

For the $A_\infty$ structure, especially $m_2$, we consider three copy of $\Lambda$ as in Figure~\ref{fig:three copy of example}.
\begin{figure}[ht]\label{fig:three copy of example}
\input{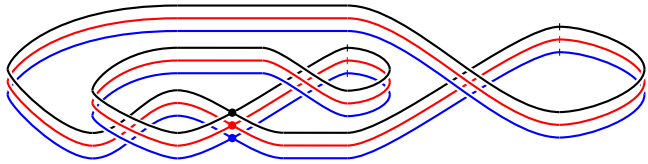}
\caption{Three copy of $\Lambda$}
\end{figure}
Again the labeling and grading convention for generators are as in Section~\ref{section:augmentation category of normal form}.
Since the DGA $A(\Lambda^{(3)})$ is complicated, for simplicity, let us list the terms in the differential which are relevant to our construction of $A_\infty$ structure. Especially, we only list term of type $\bfM_0^{(3)}$, where
\[
\bfM_0^{(k+1)} \coloneqq \bfM_0^{(k)} \otimes \sfM^{k\, k+1} \otimes (\sfM_0^{k+1\, k+1})^{\otimes n_{k+1}}.
\]
Here $n_{k+1}$ is a nonnegative integer, and $\sfM_0^{k+1\, k+1}$ denotes the submodule spanned by degree $0$ elements.

Then the restriction of the differential becomes
\begin{align*}
\differential a^{13}|_{\bfM_0^{(3)}} &= c^{12} v_{2,2}^{22} v_{4,-1}^{23};\\
\differential b^{13}|_{\bfM_0^{(3)}} &= 
\left( v_{2,2}^{11} v_{4,-1}^{12} (t_1^2)^{-1} + v_{2,1}^{11} (t_1^1)^{-1} x_1^{12} \right) \left(x_1^{23} c^{33} v_{2,2}^{33} + c^{23} v_{2,2}^{33} +  x_1^{23} v_{1,3}^{33}  \right);\\
\differential c^{13}|_{\bfM_0^{(3)}} &= y_1^{12} c^{23} + c^{12} y_1^{23} ;\\
\differential v_{4,-1}^{13}|_{\bfM_0^{(3)}} &= v_{4,2}^{11} y_2^{12} v_{2,2}^{22} v_{4,-1}^{23} + v_{4,1}^{11} \left( y_1^{12} v_{1,3}^{22} +{v_{1,1}^{12} v_{2,2}^{22} } \right) v_{4,-1}^{23} + v_{4,-1}^{12} v_{3,2}^{22} y_1^{23} v_{1,2}^{33}\\ 
&\mathrel{\hphantom{=}} + v_{4,-1}^{12} \left( v_{3,3}^{22} y_2^{23} + { v_{3,2}^{22} v_{1,1}^{23}} \right) v_{2,1}^{33};\\
\differential x_1^{13}|_{\bfM_0^{(3)}} &= y_1^{12} x_1^{23} + x_1^{12} t_1^2 \left( v_{3,2}^{22} y_1^{23} v_{1,2}^{33} + { v_{3,2}^{22} v_{1,1}^{23} v_{2,1}^{33}} + v_{3,3}^{22} y_2^{23} v_{2,1}^{33}\right) (t_1^3)^{-1} \\
&\mathrel{\hphantom{=}}+t_1^1 \left( v_{3,2}^{11} y_1^{12} v_{1,3}^{22} + { v_{3,2}^{11} v_{1,1}^{12} v_{2,2}^{22} } + v_{3,3}^{11} y_2^{12} v_{2,2}^{22} \right) v_{4,-1}^{23} (t_1^3)^{-1}\\
\differential x_2^{13}|_{\bfM_0^{(3)}} &= y_2^{12} x_2^{23} + x_2^{12} t_2^2 \left( v_{4,1}^{22} y_1^{23} v_{1,3}^{33} + {v_{4,1}^{22} v_{1,1}^{23} v_{2,2}^{33}} + v_{4,2}^{22} y_2^{23} v_{2,2}^{33} \right) (t_2^3)^{-1}\\
\differential v_{1,1}^{13}|_{\bfM_0^{(3)}} &= v_{1,1}^{12} y_2^{23} + y_1^{12} v_{1,1}^{23};\\
\differential y_1^{13}|_{\bfM_0^{(3)}} &= y_1^{12} y_1^{23};\\
\differential y_2^{13}|_{\bfM_0^{(3)}} &= y_2^{12} y_2^{23}.\\
\end{align*}

If we choose the triple $(\epsilon_1,\epsilon_1,\epsilon_1)$, then $\differential_{(\epsilon_1,\epsilon_1,\epsilon_1)}|_{\bfM_0^{(3)}}$ 
induces a multiplication on cohomology classes
\[
\bar{m}_2^{(\epsilon_1,\epsilon_1,\epsilon_1)}:H^*\hom_{\Aug_+}(\epsilon_1,\epsilon_1)\otimes H^*\hom_{\Aug_+}(\epsilon_1,\epsilon_1) \to H^*\hom_{\Aug_+}(\epsilon_1,\epsilon_1)
\]
is given by
\[
\bar{m}_2^{(\epsilon_1,\epsilon_1,\epsilon_1)}(\alpha \otimes \alpha)=\alpha.
\]
By the similar computation,
\begin{align*}
\bar{m}_2^{(\epsilon_1,\epsilon_1,\epsilon_2)}(\beta \otimes \alpha)&=\beta,&
\bar{m}_2^{(\epsilon_1,\epsilon_2,\epsilon_1)}(\gamma \otimes \beta)&=0,&
\bar{m}_2^{(\epsilon_2,\epsilon_1,\epsilon_1)}(\alpha \otimes \gamma)&=\gamma,&
\bar{m}_2^{(\epsilon_2,\epsilon_1,\epsilon_2)}(\beta \otimes \gamma)&=0,
\end{align*}
and
\begin{align*}
\bar{m}_2^{(\epsilon_1,\epsilon_2,\epsilon_2)}(\delta_1 \otimes \beta)&=0,&
\bar{m}_2^{(\epsilon_1,\epsilon_2,\epsilon_2)}(\delta_2 \otimes \beta)&=\beta,&
\bar{m}_2^{(\epsilon_1,\epsilon_2,\epsilon_2)}(\delta_3 \otimes \beta)&=0;\\
\bar{m}_2^{(\epsilon_2,\epsilon_2,\epsilon_1)}(\gamma \otimes \delta_1)&=\gamma,&
\bar{m}_2^{(\epsilon_2,\epsilon_2,\epsilon_1)}(\gamma \otimes \delta_2)&=0,&
\bar{m}_2^{(\epsilon_2,\epsilon_2,\epsilon_1)}(\gamma \otimes \delta_3)&=0.
\end{align*}
For the triple $(\epsilon_2, \epsilon_2,\epsilon_2)$, we deduce the following table of multiplication:
\begin{center}
 \begin{tabular}{c || c | c | c } 
$\bar{m}_2$ & $\delta_1$ & $\delta_2$ & $\delta_3$  \\ [0.7ex] 
 \hline\hline
$\delta_1$ & $\delta_1$ &0& $\delta_3$ \\[0.7ex] 
 \hline
$\delta_2$ & 0 & $\delta_2$ & 0 \\[0.7ex] 
 \hline
$\delta_3$ & 0 & $\delta_3$ & 0
\end{tabular}
\end{center}

For the higher multiplication $m_k,$ $k\geq 3$, let us consider  $(k+1)$-copy of $\Lambda$, the induced differential, and its restriction to $\bfM_0^{(k+1)}$. The followings are possible terms which may contribute to $m_k,$ $k\geq 3$:
\begin{align*}
\differential b^{14}|_{\bfM_0^{(4)}} &= \left( v_{2,2}^{11} v_{4,-1}^{12} (t_1^{2})^{-1}  + v_{2,1}^{11} (t_1^1)^{-1} x_1^{12} \right) x_1^{23} \left( c^{34} v_{2,2}^{44} + x_1^{34} v_{1,3}^{44} + x_1^{34} c^{44} v_{2,2}^{44} \right);\\
\differential v_{4,-1}^{14}|_{\bfM_0^{(4)}} &=  v_{4,-1}^{12} \left( v_{3,2}^{22} y_1^{23} v_{1,3}^{33} + v_{3,2}^{22} v_{1,1}^{23} v_{2,2}^{33} + v_{3,3}^{22} y_2^{23} v_{2,2}^{33} \right) v_{4,-1}^{34};\\
\differential x_1^{14}|_{\bfM_0^{(4)}} &= x_1^{12} t_1^2 \left( v_{3,2}^{22} y_1^{23} v_{1,3}^{33} + v_{3,2}^{22} v_{1,1}^{23} v_{2,2}^{33} + v_{3,3}^{22} y_2^{23} v_{2,2}^{33} \right) v_{4,-1}^{34} (t_1^4)^{-1};\\
\differential b^{1m}|_{\bfM_0^{(m)}} &=\left( v_{2,2}^{11} v_{4,-1}^{12} (t_1^{2})^{-1}  + v_{2,1}^{11} (t_1^1)^{-1} x_1^{12} \right) x_1^{23} x_1^{34} \cdots x_1^{m-2\, m-1} \\
&\mathrel{\hphantom{=}}\left( c^{m-1\, m} v_{2,2}^{m\, m} + x_1^{m-1\, m} v_{1,3}^{m\, m} + x_1^{m-1\, m} c^{m\, m} v_{2,2}^{m\, m} \right)
\end{align*}
For example, a term in $\differential b^{14}|_{\bfM_0^{(4)}}$ has a chain level contribution
\[
\langle m_3^{(\epsilon_2,\epsilon_1,\epsilon_1,\epsilon_2)}(c^{12} \otimes x_1^{12} \otimes v_{4,-1}^{12}), b^{12}\rangle=1.
\]

\subsubsection{Sheaf category}
On the other hand, let us consider the sheaf category $\cC_1(\Lambda;\field)$.
By the discussion in Section~\ref{subsubsec:comb model} and \ref{subsubsec:legible model},
a sheaf $\cF\in \cC_1(\Lambda;\field)\subset \Sh_\Lambda(\RR^2;\field)$ of micro-local rank 1 whose singular support lies in $\Lambda$ can be identified with a representation of the following quiver diagram:

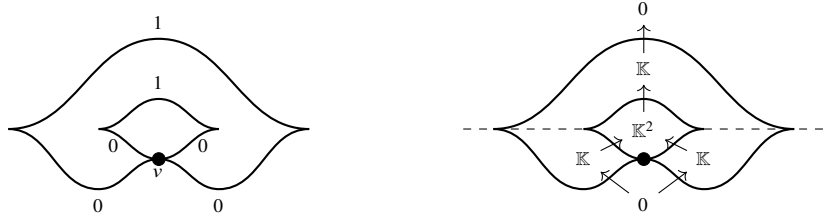
\begin{figure}[ht]
\begin{align*}
\begin{tikzpicture}[baseline=-.5ex,scale=0.8]
\draw[thick] (-2.5,0) to[out=0,in=180] (0,1.5) node[above] {$_1$} to[out=0,in=180] (2.5,0);
\draw[thick] (-2.5,0) to[out=0,in=180] (-1,-1) node[below] {$_0$} to[out=0,in=180] (0,-0.5) to[out=0,in=180] (1,-1) node[below] {$_0$} to[out=0,in=180] (2.5,0);
\draw[thick] (-1,0) node[below right] {$_0$} to[out=0,in=180] (0,0.5) node[above] {$_1$} to[out=0,in=180] (1,0) node[below left] {$_0$};
\draw[thick] (-1,0) to[out=0,in=180] (0,-0.5) to[out=0,in=180] (1,0);
\draw[fill] (0,-0.5) node[below] {$_v$} circle (3pt);
\end{tikzpicture}
&&
\begin{tikzpicture}[baseline=-.5ex,scale=0.8]
\draw[thick] (-2.5,0) to[out=0,in=180] (0,1.5)  to[out=0,in=180] (2.5,0);
\draw[thick] (-2.5,0) to[out=0,in=180] (-1,-1) to[out=0,in=180] (0,-0.5) to[out=0,in=180] (1,-1) to[out=0,in=180] (2.5,0);
\draw[thick] (-1,0) to[out=0,in=180] (0,0.5) to[out=0,in=180] (1,0);
\draw[thick] (-1,0) to[out=0,in=180] (0,-0.5) to[out=0,in=180] (1,0);
\draw[dashed] (-3,0) to (-1,0);
\draw[dashed] (1,0) to (3,0);
\draw[fill] (0,-0.5) circle (3pt);
\node (P0) at (0,-1.25) {$_0$};
\node (P1) at (-1,-0.5) {$_\field$};
\node (P2) at (1,-0.5) {$_\field$};
\node (P3) at (0,0) {$_{\field^2}$};
\node (P4) at (0,1) {$_\field$};
\node (P5) at (0,2) {$_0$};
\path
    (P0) edge[->] (P1)
    (P0) edge[->] (P2)
    (P1) edge[->] (P3)
    (P2) edge[->] (P3)
    (P3) edge[->] (P4)
    (P4) edge[->] (P5);
\end{tikzpicture}
\end{align*}
\caption{The front diagram of $\Lambda$, and the induced legible model.}
\end{figure}

More precisely, the representation of the above quiver diagram is determined by
\begin{figure}[ht]
\begin{tikzpicture}[baseline=-.5ex,scale=0.8]
\node (P1) at (-1,-1) {$\field$};
\node (P2) at (1,-1) {$\field$};
\node (P3) at (0,0) {$\field^2$};
\node (P4) at (0,1.25) {$\field$};
\path
    (P1) edge[->] node[above left] {$_{\ell_1}$} (P3)
    (P2) edge[->] node[above right] {$_{\ell_2}$} (P3)
    (P3) edge[->] node[right] {$_{\pi}$} (P4);
\end{tikzpicture}
\end{figure}
three lines 
\begin{align*}
[\ell_1] &\coloneqq \im\ell_1 \hookrightarrow \field^2,&
[\ell_2] &\coloneqq \im\ell_2 \hookrightarrow \field^2,&
[\ell_0] &\coloneqq \ker\pi \hookrightarrow \field^2,
\end{align*}
which are isomorphic to $\PP^1(\field)$ satisfying that $\pi\circ \ell_i$ is quasi-isomorphism for $i=1,2$, in other words,
\begin{align*}
[\ell_0] &\neq [\ell_1],& [\ell_0] &\neq [\ell_2].
\end{align*}
There are two possible cases, $[\ell_1]\neq[\ell_2]$ and $[\ell_1]= [\ell_2]$, and let us denote the corresponding representation of quiver by $Q_1$ and $Q_2$, respectively. Note that $Q_2\isomorphic Q_2^1 \oplus Q_2^2$, where
\begin{align*}
Q_2^1 &\coloneqq
\begin{tikzpicture}[baseline=-.5ex,scale=0.8]
\node (P1) at (-1,-1) {$\field$};
\node (P2) at (1,-1) {$\field$};
\node (P3) at (0,0) {$\field$};
\node (P4) at (0,1.25) {$\field$};
\path
    (P1) edge[->] (P3)
    (P2) edge[->] (P3)
    (P3) edge[->] (P4);
\end{tikzpicture}
&
Q_2^2 &\coloneqq
\begin{tikzpicture}[baseline=-.5ex,scale=0.8]
\node (P1) at (-1,-1) {$0$};
\node (P2) at (1,-1) {$0$};
\node (P3) at (0,0) {$\field$};
\node (P4) at (0,1.25) {$0$};
\path
    (P1) edge[->] (P3)
    (P2) edge[->] (P3)
    (P3) edge[->] (P4);
\end{tikzpicture}
\end{align*}

Firstly, let us list all indecomposable representations of the above quiver of type $D_4$ as follows:
\begin{align*}
P_1 &\coloneqq
\begin{tikzpicture}[baseline=-.5ex,scale=0.8]
\node (P1) at (-1,-1) {$0$};
\node (P2) at (1,-1) {$0$};
\node (P3) at (0,0) {$0$};
\node (P4) at (0,1.25) {$\field$};
\path
    (P1) edge[->] (P3)
    (P2) edge[->] (P3)
    (P3) edge[->] (P4);
\end{tikzpicture}
&
P_2 &\coloneqq
\begin{tikzpicture}[baseline=-.5ex,scale=0.8]
\node (P1) at (-1,-1) {$0$};
\node (P2) at (1,-1) {$0$};
\node (P3) at (0,0) {$\field$};
\node (P4) at (0,1.25) {$\field$};
\path
    (P1) edge[->] (P3)
    (P2) edge[->] (P3)
    (P3) edge[->] (P4);
\end{tikzpicture}
&
P_3 &\coloneqq
\begin{tikzpicture}[baseline=-.5ex,scale=0.8]
\node (P1) at (-1,-1) {$\field$};
\node (P2) at (1,-1) {$0$};
\node (P3) at (0,0) {$\field$};
\node (P4) at (0,1.25) {$\field$};
\path
    (P1) edge[->] (P3)
    (P2) edge[->] (P3)
    (P3) edge[->] (P4);
\end{tikzpicture}
&
P_4 &\coloneqq
\begin{tikzpicture}[baseline=-.5ex,scale=0.8]
\node (P1) at (-1,-1) {$0$};
\node (P2) at (1,-1) {$\field$};
\node (P3) at (0,0) {$\field$};
\node (P4) at (0,1.25) {$\field$};
\path
    (P1) edge[->] (P3)
    (P2) edge[->] (P3)
    (P3) edge[->] (P4);
\end{tikzpicture}
\end{align*}

\subsubsection{Computation of differential}
In the rest of example, we assume that our base field is $\Zmod{2}$.
Recall for a vector space $V$ that $V[i]$ means the degree shift by $-i$.\footnote{We omit the degree shift notation $[\ell]$ when it is clear.}

Note that
\[
\hom(P_i[k],P_j[\ell])\isomorphic
\begin{cases}
\field[\ell-k] &\text{ if } (i,j)=(1,2), (1,3), (1,4), (2,4);\\
0&\text{ otherwise}.
\end{cases}
\]
We sometimes use $\field_{P_i P_j}$ instead of $\hom(P_i,P_j)$ when it is non-trivial.

The quivers $Q_1$, $Q_2^1$ and $Q_2^2$ admit the following projective resolutions which will be denoted by $\tilde Q_1$, $\tilde Q_2^1$ and $\tilde Q_2^2$, respectively:
\[
\begin{tikzcd}[row sep=0.3pc]
0 \arrow[r]& P_1 \arrow[r, "f"] & P_3 \oplus P_4 \arrow[r]& Q_1 \arrow[r]& 0; \\
0 \arrow[r]& P_2 \arrow[r, "g"] & P_3 \oplus P_4 \arrow[r]& Q_2^1 \arrow[r]& 0; \\
0 \arrow[r]& P_1 \arrow[r, "h"] & P_2 \arrow[r]& Q_2^2 \arrow[r]& 0. \\
\end{tikzcd}
\]
Here a matrix form of $f, g$ and $h$ are given by $1 \choose 1$, $1 \choose 1$, and $(1)$ with respect to a canonical choice of basis of each ordered summand.
Then we compute that
\[
R\hom(Q_1, Q_1) = \hom(\tilde Q_1, Q_1) \isomorphic \left( 
\begin{tikzcd}
\field_{P_3 Q_1} \oplus \field_{P_4 Q_1} \ar[r, "(1\ 1)"] & \field_{P_1 Q_1}[-1]
\end{tikzcd}
 \right) \isomorphic \field.
\]
Let $b_1$ and $b_2$ be the non-zero element of $\field_{P_3 Q_1}$ and $\field_{P_4 Q_1}$, then
the only non-trivial homology class in $H^*\hom(\tilde Q_1, Q_1)$ is $[b_1+b_2]$.

Similarly we verify that
\begin{align*}
\hom(\tilde Q_2^1, Q_2^1) &\isomorphic \left( \field^2 \twoheadrightarrow \field[-1] \right) \isomorphic \field,&
\hom(\tilde Q_2^1, Q_2^2) &\isomorphic \left( 0 \to \field[-1] \right) \isomorphic \field[-1] ;\\
\hom(\tilde Q_2^2, Q_2^1) &\isomorphic \left( \field \stackrel{\isomorphic}{\to} \field[-1] \right) \isomorphic 0,&
\hom(\tilde Q_2^2, Q_2^1) &\isomorphic \left( \field \to 0[-1] \right) \isomorphic \field,
\end{align*}
which conclude
\[
R\hom(Q_2, Q_2) = \hom(\tilde Q_2, Q_2) \isomorphic \bigoplus_{i,j=1,2}\hom(\tilde Q_2^i,Q_2^j) \isomorphic \field^2 \oplus \field[-1].
\]
Also compute
\begin{align*}
\hom(\tilde Q_1,Q_2^1) &\isomorphic \left( \field^2 \twoheadrightarrow \field[-1] \right) \isomorphic \field,&
\hom(\tilde Q_1,Q_2^2) &\isomorphic \left( 0 \to 0[-1] \right) \isomorphic 0;\\
\hom(\tilde Q_2^1, Q_1) &\isomorphic \left( \field^2 \twoheadrightarrow \field[-1] \right) \isomorphic \field,&
\hom(\tilde Q_2^1, Q_1) &\isomorphic \left( \field \stackrel{\isomorphic}{\to} \field[-1] \right),
\end{align*}
which imply
\begin{align*}
R\hom(Q_1, Q_2) &\isomorphic \hom(\tilde Q_1,Q_2^1) \oplus \hom(\tilde Q_1,Q_2^2) \isomorphic \field;\\
R\hom(Q_2, Q_1) &\isomorphic \hom(\tilde Q_2^1, Q_1) \oplus \hom(\tilde Q_2^1, Q_1) \isomorphic \field.
\end{align*}
Note that the above computation coincide with the one in $\Aug_+(\Lambda)$, see (\ref{eqn:aug_example_cohomology_class}).

\subsubsection{Computation of multiplication}
Now consider the compositions between $\hom$ spaces. 
Especially look at 
\[
m : R\hom(Q_1, Q_1) \otimes R\hom(Q_1, Q_1) \to R\hom(Q_1, Q_1).
\]
The domain is isomorphic to $\hom(\tilde Q_1, Q_1) \otimes \hom(\tilde Q_1, \tilde Q_1)$ and hence becomes the following:
\[
\left(
\begin{tikzcd}
\field_{P_3 Q_1}\oplus \field_{P_4  Q_1} \ar[r, "(1\ 1)"]& \field_{P_1 Q_1}[-1]
\end{tikzcd}
\right)
\otimes
\left(
\begin{tikzcd}
\field_{P_1 P_1}\oplus \field_{P_3  P_3} \oplus \field_{P_4  P_4} \ar[r, "{1\ 1\ 0 \choose 1\ 0\ 1}"]& \field_{P_1 P_3}[-1] \oplus \field_{P_1 P_4}[-1]
\end{tikzcd}
\right).
\]
Let $c_1$, $c_2$, and $c_3$ be the non-zero element in $\field_{P_1 P_1}$, $\field_{P_3  P_3}$, and $\field_{P_4  P_4}$, respectively, then $[c_1+c_2+c_3]$ be the unique non-trivial cohomology class in $H^*\hom(\tilde Q_1, \tilde Q_1)$.

Let us denote the domain of $m_2$ by $D^\bullet$,
then the degree $i$-th part $D^i$ becomes as follows for $i=0,1,2$:
\begin{align*}
D^0
&\isomorphic
\left( \field_{P_3 Q_1} \otimes \field_{P_1 P_1} \right) \oplus
\left( \field_{P_3 Q_1} \otimes \field_{P_3 P_3} \right) \oplus
\left( \field_{P_3 Q_1} \otimes \field_{P_4 P_4} \right) \\
&\mathrel{\hphantom{\isomorphic}}\oplus
\left( \field_{P_4 Q_1} \otimes \field_{P_1 P_1} \right) \oplus
\left( \field_{P_4 Q_1} \otimes \field_{P_3 P_3} \right) \oplus
\left( \field_{P_4 Q_1} \otimes \field_{P_4 P_4} \right);\\
D^1
&\isomorphic
\left( \field_{P_3 Q_1} \otimes \field_{P_1 P_3} \right) \oplus
\left( \field_{P_3 Q_1} \otimes \field_{P_1 P_4} \right) \oplus
\left( \field_{P_4 Q_1} \otimes \field_{P_1 P_3} \right) \\
&\mathrel{\hphantom{\isomorphic}}\oplus
\left( \field_{P_4 Q_1} \otimes \field_{P_1 P_4} \right) \oplus
\left( \field_{P_1 Q_1} \otimes \field_{P_1 P_1} \right) \oplus
\left( \field_{P_1 Q_1} \otimes \field_{P_3 P_3} \right) \oplus
\left( \field_{P_1 Q_1} \otimes \field_{P_4 P_4} \right);\\
D^2
&\isomorphic
\left( \field_{P_1 Q_1} \otimes \field_{P_1 P_3} \right) \oplus
\left( \field_{P_1 Q_1} \otimes \field_{P_1 P_4} \right) 
\end{align*}
The differential $d_i:D^i\to D^{i+1}$ for $i=0,1$ can be expressed by the following matrix form
\begin{align*}
d_0&=
\begin{pmatrix}
1&1&0&0&0&0\\
1&0&1&0&0&0\\
0&0&0&1&1&0\\
0&0&0&1&0&1\\
1&0&0&1&0&0\\
0&1&0&0&1&0\\
0&0&1&0&0&1
\end{pmatrix},&
d_1&=
\begin{pmatrix}
1&0&1&0&1&1&0\\
0&1&0&1&1&0&1
\end{pmatrix},
\end{align*} 
with respect to basis $\{e_i^0\}_{i=1,\dots,6}$, $\{e_j^1\}_{j=1,\dots,7}$, $\{e_k^2\}_{k=1,2}$ for $D^0$, $D^1$, $D^2$. These are chosen by  a canonical element in each ordered summand.
Note that the only non-trivial cohomology class in the domain $D^\bullet$ is 
\[
[e_1^1+\dots+e_6^1] \isomorphic [b_1+b_2] \otimes [c_1 + c_2 + c_3].
\]
In the codomain of $m$, 
\begin{tikzcd}
\field_{P_3 Q_1} \oplus \field_{P_4 Q_1} \ar[r,"(1\ 1)"] & \field_{P_1 Q_1}[-1],
\end{tikzcd}
let $f_1$ and $f_2$ be the non-zero element in $\field_{P_3 Q_1}$ and $\field_{P_4 Q_1}$, respectively.
Then $[f_1+f_2]$ be the only non-trivial cohomology class.

The multiplication $m$ for each summand is defined by
\[
m\left( \field_{W R} \otimes \field_{S T} \right)=
\begin{cases}
\field_{S R} &\text{ if } T=W;\\
0 &\text{ otherwise}.
\end{cases}
\]
and the induced multiplication
\[
\bar m:H^*\hom(\tilde Q_1,Q_1) \otimes H^*\hom(\tilde Q_1, \tilde Q_1)\to H^*\hom(\tilde Q_1,Q_1)
\]
sends $[e_1^1+\dots+e_6^1]$ to $[f_1+f_2]$ which coincides with 
\[
\bar m_2: H^*\hom(\epsilon_1,\epsilon_1) \otimes H^*\hom(\epsilon_1,\epsilon_1) \to H^*\hom(\epsilon_1,\epsilon_1).
\]
One can also verify the other cases coincide with the computation of multiplication in $\Aug_+(\Lambda)$.

\bibliographystyle{abbrv}
\bibliography{references}

\end{document}